%% file: book_Almeida_Tavares_Torres.tex
\documentclass[envcountsame,envcountchap,sectrefs]{my-svmono}

\usepackage{makeidx}
\usepackage{multicol}
\usepackage[bottom]{footmisc}
\usepackage{amsmath,amssymb}
\usepackage{url}
\usepackage{enumerate}
\usepackage{xspace}
\usepackage{cases}
\usepackage{mathtools}
\usepackage{mathrsfs}
\usepackage{graphicx, psfrag, subfigure}
\usepackage[round,sectionbib,sort]{natbib}

\renewcommand{\cite}{\citep}
\newtheorem{Theorem}{Theorem}

\newtheorem{Lemma}[Theorem]{Lemma}
\newtheorem{Remark}[Theorem]{Remark}
\newtheorem{Definition}{Definition} 
\newtheorem{Problem}{Problem} 
\def\bR{\mathbb{R}}
\def\CP{\clearpage{\thispagestyle{empty}\cleardoublepage}}

\newcommand{\C}{\mathbb{C}}
\def\a{{\alpha(\cdot,\cdot)}}
\def\an{{\alpha_n(\cdot,\cdot)}}
\def\ai{{\alpha_i(\cdot,\cdot)}}
\def\ao{{\alpha_1(\cdot,\cdot)}}
\def\at{{\alpha_2(\cdot,\cdot)}}

\def\b{{\beta(\cdot,\cdot)}}
\def\bn{{\beta_n(\cdot,\cdot)}}
\def\bi{{\beta_i(\cdot,\cdot)}}
\def\bo{{\beta_1(\cdot,\cdot)}}
\def\bt{{\beta_2(\cdot,\cdot)}}

\def\t{\tau}
\def\e{\epsilon}
\def\s{\sigma}
\def\xs{_\s[x]_\gamma^{\alpha, \beta}}
\def\LI{{_aI_t^{\a}}}
\def\RI{{_tI_b^{\a}}}
\def\LDa{{_aD_t^{\a}}}
\def\LDan{{_aD_t^{\an}}}
\def\LDb{{_aD_t^{\b}}}
\def\RDa{{_tD_b^{\a}}}
\def\RDan{{_tD_b^{\an}}}
\def\RDb{{_tD_b^{\b}}}
\def\RDbn{{_tD_b^{\bn}}}
\def\LC{{^C_aD_t^{\a}}}
\def\LCan{{^C_aD_t^{\an}}}
\def\LCao{{^C_aD_t^{\ao}}}
\def\LCat{{^C_aD_t^{\at}}}
\def\RCa{{^C_tD_b^{\a}}}
\def\RCan{{^C_tD_b^{\an}}}

\def\RCb{{^C_tD_b^{\b}}}
\def\RCbn{{^C_tD_b^{\bn}}}
\def\RCbo{{^C_tD_b^{\bo}}}
\def\RCbt{{^C_tD_b^{\bt}}}
\def\DC{{^CD_\gamma^{\a,\b}}}
\def\DCi{{^CD_{\gamma^i}^{\ai,\bi}}}
\def\LIan{{_aI_t^{\an}}}
\def\RIan{{_tI_b^{\an}}}
\def\LDan{{_aD_t^{\an}}}
\def\RDan{{_tD_b^{\an}}}
\def\RDbn{{_tD_b^{\bn}}}
\def\ati{{\alpha(t)}}
\def\akn{{\alpha_k(t_k)}}
\def\Dak{{\alpha'_k(t_k)}}
\def\LCI{{^C_aD_t^{\ati}}}
\def\RCI{{^C_tD_b^{\ati}}}
\def\LCII{{^C_a\mathcal{D}_t^{\ati}}}
\def\RCII{{^C_t\mathcal{D}_b^{\ati}}}
\def\LCIII{{^C_a\mathbb{D}_t^{\ati}}}
\def\RCIII{{^C_t\mathbb{D}_b^{\ati}}}
\def\PLCI{{^C_{a_k}D_{t_k}^{\akn}}}
\def\PRCI{{^C_{t_k}D_{b_k}^{\akn}}}
\def\PLCII{{^C_{a_k}\mathcal{D}_{t_k}^{\akn}}}
\def\PRCII{{^C_{t_k}\mathcal{D}_{b_k}^{\akn}}}
\def\PLCIII{{^C_{a_k}\mathbb{D}_{t_k}^{\akn}}}
\def\PRCIII{{^C_{t_k}\mathbb{D}_{b_k}^{\akn}}}
\def\C{\left(^{n-\akn}_{\quad p}\right)}
\def\D{\left(^{1-\akn}_{\quad p}\right)}
\def\DS{\displaystyle}

\makeindex

\begin{document}

\author{Ricardo Almeida \and Dina Tavares \and \\ Delfim F. M. Torres}

\title{The Variable-Order Fractional Calculus of Variations}

\date{June 17, 2018}

\maketitle


\frontmatter

\include{0Preface}

\clearpage{\thispagestyle{empty}\cleardoublepage}


\chapter*{Acknowledgements}

This work was supported by Portuguese funds through the 
\emph{Center for Research and Development in Mathematics and Applications} (CIDMA), 
and the \emph{Portuguese Foundation for Science and Technology} (FCT), 
within project UID/MAT/04106/2013. 

\bigskip

Any comments or suggestions related to the material here contained
are more than welcome, and may be submitted by post
or by electronic mail to the authors:

\medskip

\begin{flushleft}
Ricardo Almeida \verb!<ricardo.almeida@ua.pt>!\\
Center for Research and Development in Mathematics and Applications\\
Department of Mathematics, University of Aveiro\\
3810-193 Aveiro, Portugal\\[0.3cm]
Dina Tavares \verb!<dtavares@ipleiria.pt>!\\
ESECS, Polytechnic Institute of Leiria\\
2410--272 Leiria, Portugal\\[0.3cm]
Delfim F. M. Torres \verb!<delfim@ua.pt>!\\
Center for Research and Development in Mathematics and Applications\\
Department of Mathematics, University of Aveiro\\
3810-193 Aveiro, Portugal
\end{flushleft}


\clearpage{\thispagestyle{empty}\cleardoublepage}


\tableofcontents


\mainmatter

\include{1FractionalCalculus}

\CP
\include{2FCV}

\CP
\include{3NumericalApprox}

\CP
\include{4FractionalCalculusofvariations}

\CP
\include{Appendix}

\CP

\backmatter

\printindex
\addcontentsline{toc}{part}{Index}


\end{document}

%% file: 0Preface.tex
\chapter*{Preface}

This book intends to deepen the study of the fractional calculus, giving special emphasis to variable-order operators.

Fractional calculus is a recent field of mathematical analysis and it is a generalization of integer differential calculus, involving derivatives and
integrals of real or complex order \cite{Pref:Kilbas,Pref:Podlubny}. The first note about this ideia of differentiation, for non-integer numbers, dates back to
1695, with a famous correspondence between Leibniz and L'H\^opital. In a letter, L'H\^opital asked Leibniz about the possibility of the order
$n$ in the notation $d^ny/dx^n$, for the $n$th derivative of the function $y$, to be a non-integer, $n=1/2$. Since then, several mathematicians
investigated this approach, like Lacroix, Fourier, Liouville, Riemann, Letnikov, Gr\"unwald, Caputo, and contributed to the grown development of
this field. Currently, this is one of the most intensively developing areas of mathematical analysis as a result of its numerous applications. The first
book devoted to the fractional calculus was published by Oldham and Spanier in 1974, where the authors systematized the main ideas, methods and
applications about this field \cite{Pref:Mainardi}.

In the recent years, fractional calculus has attracted the attention of many mathematicians, but also some researchers in other areas like physics,
chemistry and engineering. As it is well known, several physical phenomena are often better described by fractional derivatives \cite{Pref:Herr,Pref:Od2012a,Pref:Zheng}. This is mainly due to the fact that fractional operators take into consideration the evolution of the system, by taking the global correlation, and not only local characteristics. Moreover, integer-order calculus sometimes contradict the experimental results and therefore derivatives of fractional order may be more suitable \cite{Pref:Hilfer}.

In 1993, Samko and Ross devoted themselves to investigate operators when the order $\alpha$ is not a constant during the process, but variable on time: $\alpha(t)$ \cite{Pref:Samko:Ross}. An interesting recent generalization of the theory of fractional calculus is developed to allow the fractional order of the
derivative to be non-constant, depending on time \cite{Pref:Chen:Liu,Pref:Od2012c,Pref:Od2013}. With this approach of variable-order fractional calculus, the non-local
properties are more evident and numerous applications have been found in physics, mechanics, control and signal processing \cite{Pref:Coimbra2,Pref:Ingman,Pref:Od2013a,Pref:Ostal,Pref:Ramirez,Pref:Rapaic,Pref:Valerio}.

Although there are many definitions of fractional derivative, the most commonly used are the Riemann--Liouville, the Caputo, and the Gr\"unwald-Letnikov derivatives. For more about the development of fractional calculus, we suggest \cite{Pref:Samko:Kilbas}, \cite{Pref:Samko:Ross}, \cite{Pref:Podlubny},  \cite{Pref:Kilbas} or \cite{Pref:Mainardi}.

One difficult issue that usually arises when dealing with such fractional operators, is the extreme difficulty in solving analytically such problems \cite{Pref:Atangana,Pref:Zhuang}. Thus, in most cases, we do not know the exact solution for the problem and one needs to seek a numerical approximation. Several numerical methods can be found in the literature, typically applying some discretization over time or replacing the fractional operators by a proper decomposition \cite{Pref:Atangana,Pref:Zhuang}.

Recently, new approximation formulas were given for fractional constant order operators, with the advantage that higher-order derivatives are not
required to obtain a good accuracy of the method \cite{Pref:Atana:Janev,Pref:Pooseh2012,Pref:Pooseh2013a}. These decompositions only depend on integer-order derivatives, and by replacing the fractional operators that appear in the problem by them, one leaves the fractional context ending up in the presence of a standard problem, where numerous tools are available to solve them \cite{Pref:book2:FCV}.

The first goal of this book is to extend such decompositions to Caputo fractional problems of variable-order.  
For three types of Caputo derivatives with variable-order, we obtain approximation formulas for the fractional operators and respective upper bounds for the errors.

Then, we focus our attention on a special operator introduced by Malinowska and Torres: 
the combined Caputo fractional derivative, which is an extension
of the left and the right fractional Caputo derivatives \cite{Pref:Malin:Tor2010}. Considering $\alpha,\beta \in (0, 1)$ and $\gamma \in [0,1]$, the combined
Caputo fractional derivative operator $^CD_\gamma^{\alpha,\beta}$ is a convex combination of the left and the right Caputo fractional derivatives, defined
by $$^CD_\gamma^{\alpha,\beta}=\gamma \, _a^CD_t^\alpha + (1-\gamma) \, _t^CD_b^\beta.$$
We consider this fractional operator with variable fractional order, i.e., 
the combined Caputo fractional derivative of variable-order:
$$
^{C}D_{\gamma}^{\a,\b}x(t)=\gamma_1 \, \LC x(t)+\gamma_2 \, {^C_tD_b^{\b}} x(t),
$$ 
where $\gamma =\left(\gamma_1,\gamma_2 \right)\in [0,1]^2$, with $\gamma_1$ and $\gamma_2$ not both zero.
With this fractional operator, we study different types of fractional calculus of variations problems, where the Lagrangian depends on the referred derivative. 

The calculus of variations is a mathematical subject that appeared formally in the XVII century, with the solution to the bachistochrone problem, that deals
with the extremization (minimization or maximization) of functionals \cite{Pref:Brunt}. Usually, functionals are given by an integral that involves one or more functions or/and its derivatives. This branch of mathematics has proved to be relevant because of the numerous applications existing in real situations.

The fractional variational calculus is a recent mathematical field that consists in minimizing or maximizing functionals that depend on fractional operators
(integrals or/and derivatives). This subject was introduced by Riewe in 1996, where the author generalizes the classical calculus of variations, by using fractional derivatives, and allows to obtain conservations laws with nonconservative forces such as friction \cite{Pref:Riewe96,Pref:Riewe97}. Later appeared several works on various aspects of the fractional calculus of variations and involving different fractional operators, like the Riemann--Liouville, the Caputo, the Gru\"nwald--Letnikov, the Weyl, the Marchaud or the Hadamard fractional derivatives  \cite{Pref:Agrawal2002,Pref:Almeidadelay,Pref:Askari,Pref:Atanackovic,Pref:Baleanu1,Pref:Fraser,Pref:Georgieva,Pref:Jaraddelay}.
For the state of the art of the fractional calculus of variations, we refer the readers to the books \cite{Pref:book:FCV,Pref:book2:FCV,Pref:book:MOT}.

Specifically, here we study some problems of the calculus of variations with integrands depending on the independent variable $t$, an arbitrary function $x$
and a fractional derivative $^CD_\gamma^{\a,\b}x$. The endpoint of the cost integral, as well the terminal state, are considered to be free. The fractional problem of the calculus of variations consists in finding the maximizers or minimizers to the functional
$$\mathcal{J}(x,T)=\int_a^T L\left(t, x(t), \DC x(t)\right)dt+\phi(T,x(T)),$$
where $\DC x(t)$ stands for the combined Caputo fractional derivative of variable fractional order, subject to the boundary condition $x(a)=x_a$.
For all variational problems presented here, we establish necessary optimality conditions and transversality optimality conditions.

The book is organized in two parts, as follows. In the first part, we review the basic concepts of fractional calculus (Chapter~\ref{PartI_FC}) and of the fractional calculus of variations (Chapter~\ref{PartI_FCV}). In Chapter~\ref{PartI_FC}, we start with a brief overview about fractional calculus and an introduction to the theory of some special functions in fractional calculus. Then, we recall several fractional operators (integrals and derivatives)
definitions and some properties of the considered fractional derivatives and integrals are introduced. In the end of this chapter, we review integration
by parts formulas for different operators. Chapter~\ref{PartI_FCV} presents a short introduction to the classical calculus of variations and review different variational problems, like the isoperimetric problems or problems with variable endpoints. In the end of this chapter, we introduce the theory of the fractional calculus of variations and some fractional variational problems with variable-order.

In the second part, we systematize some new recent results on variable-order fractional calculus 
of \cite{Pref:Tavares2016,Pref:Tavares2015,Pref:Tavares2017,Pref:Tavares_4,Pref:Tavares_5}.
In Chapter~\ref{PartII_Expan}, considering three types of fractional Caputo derivatives of variable-order, we present new approximation formulas for those
fractional derivatives and prove upper bound formulas for the errors. In Chapter~\ref{PartII_FCV}, we introduce the combined Caputo fractional derivative
of variable-order and corresponding higher-order operators. Some properties are also given. Then, we prove fractional Euler--Lagrange equations for several types of fractional problems of the calculus of variations, with or without constraints.

\begingroup
\renewcommand{\addcontentsline}[3]{}

\endgroup

%% file: 1FractionalCalculus.tex
\chapter{Fractional calculus}
\label{PartI_FC}

In this chapter, a brief introduction to the theory of fractional calculus
is presented. We start with a historical perspective of the theory, with a
strong connection with the development of classical calculus (Section~
\ref{sec:history}).
Then, in Section~\ref{sec:specFun}, we review some definitions and properties
about a few special functions that will be needed. We end with a review on
fractional integrals and fractional derivatives of noninteger order and with
some formulas of integration by parts, involving fractional operators (Section~\ref{sec:FID}).

The content of this chapter  can be found in some classical books on fractional
calculus, for example \cite{Cap1:Kilbas,Cap1:Podlubny,Cap1:Samko:Kilbas,Cap1:book:FCV,Cap1:book2:FCV}.


\section{Historical perspective}
\label{sec:history}
Fractional Calculus (FC) is considered as a branch of mathematical analysis
which deals with the investigation and applications of integrals and
derivatives of arbitrary order. Therefore, FC is an extension of the
integer-order calculus that considers integrals and derivatives of any
real or complex order \cite{Cap1:Kilbas,Cap1:Samko:Kilbas}, i.e., unify and generalize
the notions of integer-order differentiation and $n$-fold integration.

FC was born in 1695 with a letter that L'H\^opital wrote to Leibniz,
where the derivative of order $1/2$ is suggested \cite{Cap1:Oldham}.
After Leibniz had introduced in his publications the notation for the $n$th
derivative of a function $y$,
$$ \frac{d^ny}{dx^n},$$
L'H\^opital wrote a letter to Leibniz to ask him about the possibility of a
derivative of integer order to be extended in order to have a meaning when the
order is a fraction: "What if $n$ be $1/2$?" \cite{Cap1:Ross}.
In his answer, dated on 30 September 1695, Leibniz replied that "This is an
apparent paradox from which, one day, useful consequences will be drawn" and,
today, we know it is truth.
Then, Leibniz still wrote about derivatives of general order and in 1730, Euler investigated the result of the derivative when the order $n$ is a fraction.
But, only in 1819, with Lacroix, appeared the first definition of fractional
derivative based on the expression for the $n$th derivative of the power function.
Considering $y=x^m$, with $m$ a positive integer, Lacroix developed the $n$th
derivative
$$
\frac{d^ny}{dx^n}=\frac{m !}{(m-n) !}\, \,x^{m-n},\quad m\geq n,
$$
and using the definition of Gamma function, for the generalized factorial, he got
$$
\frac{d^ny}{dx^n}=\frac{\Gamma(m+1)}{\Gamma(m-n+1)} \, x^{m-n}.
$$
Lacroix also studied the following example, for $n=\frac{1}{2}$ and $m=1$:
\begin{equation}
\label{orederonehalf}
\frac{d^{1/2} x}{dx^{1/2}}=\frac{\Gamma(2)}{\Gamma(3/2)} \, x^{\frac{1}{2}}= \frac{2\sqrt{x}}{\sqrt{\pi}}.
\end{equation}
Since then, many mathematicians, like Fourier, Abel, Riemann, Liouville, among
others, contributed to the development of this subject. One of the first
applications of fractional calculus appear in 1823 by Niels Abel, through the
solution of an integral equation of the form
$$
\int_0^t (t-\t)^{-\alpha}x(\t) d\t=k,
$$
used in the formulation of the Tautochrone problem \cite{Cap1:Abel,Cap1:Ross}.
\index{tautochrone problem}

Different forms of fractional operators have been introduced along
time, like the Riemann--Liouville, the Gr\"unwald-Letnikov, the Weyl, the Caputo,
the Marchaud or the Hadamard fractional derivatives 
\cite{Cap1:Oliveira,Cap1:Kilbas,Cap1:Podlubny,Cap1:Oldham}.
The first approach is the Riemann-Liouville, which is based on iterating the
classical integral operator $n$ times and then considering the Cauchy's formula
where $n!$ is replaced by the Gamma function and hence the fractional integral
of noninteger order is defined. Then, using this operator, 
some of the fractional derivatives mentioned above are defined.

During three centuries, FC was developed but as a pure theoretical subject
of mathematics.
In recent times, FC had an increasing of importance due to its applications
in various fields, not only in mathematics, but also in physics, mechanics,
engineering, chemistry, biology, finance, and others areas of science
\cite{Cap1:Herr,Cap1:Li,Cap1:Mainardi,Cap1:Hilfer,Cap1:Od_2013b,Cap1:Pinto,Cap1:Sie,Cap1:Sun}.
In some of these applications, many real world phenomena are better described
by noninteger order derivatives, if we compare with the usual integer-order
calculus.
In fact, fractional order derivatives have unique characteristics that may model
certain dynamics more efficiently. Firstly, we can consider any real order for the
derivatives, and thus we are not restricted to integer order derivatives only;
secondly, they are nonlocal operators, in opposite to the usual derivatives,
thus containing memory. With the memory property, one can take into account the past
of the processes. Signal processing, modeling and control are some areas that
have been the object of more intensive publishing in the last decades.

In most applications of the FC, the order of the derivative is assumed to be
fixed along the process, that is, when determining what is the order $\alpha >0$
such that the solution of the fractional differential equation
$D^{\alpha}y(t)=f(t, y(t))$ better approaches the experimental data, we consider
the order $\alpha$ to be a fixed constant. Of course, this may not be the best
option, since trajectories are a dynamic process, and the order may vary. More
interesting possibilities arise when one considers the order $\alpha$ of the
fractional integrals and derivatives not constant during the process but to be a
function $\alpha (t)$, depending on time.
Then, we may seek what is the best function $\alpha(\cdot)$ such that the 
variable-order fractional differential equation $D^{\alpha(t)}y(t) = f(t,y(t))$ better
describes the process under study. This approach is very recent. One such fractional calculus
of variable-order was introduced in \cite{Cap1:Samko:Ross}.
Afterwards, several mathematicians obtained important results about 
variable-order fractional calculus, and some applications appeared, like in mechanics,
in the modeling of linear and nonlinear viscoelasticity oscillators and in other
phenomena where the order of the derivative varies with time.
See, for instance, \cite{Cap1:Alm:Torres2013,Cap1:Atana,Cap1:Coimbra1,Cap1:Sheng,Cap1:Od2013,Cap1:Ramirez,Cap1:Samko:1995}.

The most common fractional operators considered in the literature take into account
the past of the process. They are usually called left fractional operators. But in
some cases we may be also interested in the future of the process, and the
computation of $\alpha (t)$ to be influenced by it.
In that case, right fractional derivatives are then considered. Recently, in some
works, the main goal is to develop a theory where both fractional operators are
taken into account.
For that, some combined fractional operators are introduced, like the symmetric
fractional derivative, the Riesz fractional integral and derivative,
the Riesz--Caputo fractional derivative and the combined Caputo fractional
derivative that consists in a linear combination of the left and right fractional
operators.
For studies with fixed fractional order, see
\cite{Cap1:Klimek,Cap1:Malin:Tor2011,Cap1:Malin:Tor2012,Cap1:book:FCV}.

Due to the growing number of applications of fractional calculus in science and engineering, 
numerical approaches are being developed to provide tools for solving such problems.
At present, there are already vast studies on numerical approximate formulas
\cite{Cap1:Li_Chen,Cap1:Kumar}. For example, for numerical modeling
of time fractional diffusion equations, we refer the reader to \cite{Cap1:Fu}.


\section{Special functions}
\label{sec:specFun}
Before introducing the basic facts on fractional operators, we recall four
types of functions that are important in Fractional Calculus: the \textit{Gamma},
\textit{Psi}, \textit{Beta} and \textit{Mittag-Leffler} functions. Some properties
of these functions are also recalled.

\begin{Definition}
\index{gamma function}
\label{Gammaf}
The Euler Gamma function is an extension of the factorial function to real
numbers, and it is defined by
$$
\Gamma(t)=\int_0^\infty \t^{t-1}\exp(-\t)\,d\t, \quad t>0.
$$
\end{Definition}

For example, $\Gamma(1) = 1$, $ \Gamma(2) = 1$ and $ \Gamma(3/2)=\frac{\sqrt{\pi}}{2}$.
For positive integers $n$, we get $\Gamma(n) = (n - 1)!$.
We mention that other definitions for the Gamma function exist, and it is
possible to define it for complex numbers, except for the non-positive integers.

The Gamma function is considered the most important Eulerian function used in
fractional calculus, because  it appears in almost every fractional integral
and derivative definitions.
A basic but fundamental property of  $\Gamma$, that we will use later, is
obtained using integration by parts:
$$
\Gamma(t+1)=t\, \Gamma(t).
$$

\begin{Definition}
\index{psi function}
\label{Psif}
The Psi function is the derivative of the logarithm of the Gamma function:
$$
\Psi(t)=\frac{d}{dt}\ln\left(\Gamma(t)\right)=\frac{\Gamma'(t)}{\Gamma(t)}.
$$
\end{Definition}

The follow function is used sometimes for convenience to replace a combination
of Gamma functions.
It is important in FC because it shares a form that is similar to the fractional
derivative or integral of many functions, particularly power functions.

\begin{Definition}
\label{Betaf}
\index{beta function}
The Beta function $B$ is defined by
$$
B(t,u)=\int_0^1 s^{t-1} (1-s)^{u-1}ds, \quad t,u > 0.
$$
\end{Definition}

This function satisfies an important property:
$$B(t,u)=\frac{\Gamma(t)\Gamma(u)}{\Gamma(t+u)}.$$
With this property, it is obvious that the Beta function is symmetric, i.e.,
$$B(t,u)=B(u,t).$$
The next function is a direct generalization of the exponential series and it was defined by the mathematician Mittag-Leffler in 1903 \cite{Cap1:Podlubny}.

\begin{Definition}
\index{Mittag-Leffler function! one parameter}
Let $\alpha > 0$. The function $E_\alpha$ defined by
$$
E_\alpha(t)=\sum_{k=0}^\infty \frac{t^k}{\Gamma(\alpha k+1)}, \quad t \in \mathbb{R},
$$
is called the one parameter Mittag-Leffler function.
\end{Definition}

For $\alpha=1$, this function coincides with the series expansion of $e^t$,
i.e.,
$$
E_1(t)=\sum_{k=0}^\infty \frac{t^k}{\Gamma(k+1)}= \sum_{k=0}^\infty
\frac{t^k}{k!}= e^t.
$$
While linear ordinary differential equations present in general the
exponential function as a solution, the Mitttag-Leffler function occurs
naturally in the solution of fractional order differential equations
\cite{Cap1:Kilbas}. For this reason, in recent times, the Mittag-Leffler
function has become an important function in the theory of the fractional
calculus and its applications.

It is also common to represent the Mittag-Leffler function in two arguments.
This generalization of Mittag-Leffler function was studied by Wiman  in 1905
\cite{Cap1:Mainardi}.

\begin{Definition}
\index{Mittag-Leffler function! two parameters}
The two-parameter function of the Mittag-Leffler type with parameters
$\alpha, \beta >0$ is defined by
$$
E_{\alpha,\beta}(t)=\sum_{k=0}^\infty \frac{t^k}{\Gamma(\alpha k+\beta)},
\quad t \in \mathbb{R}.
$$
\end{Definition}

If $\beta=1$, this function coincides with the classical Mittag-Leffler
function, i.e., $E_{\alpha,1}(t)=E_\alpha(t)$.


\section{Fractional integrals and derivatives}
\label{sec:FID}
In this section, we recall some definitions of fractional integral and
fractional differential operators, that includes all we use throughout
this book. In the end, we present some integration by parts formulas
because they have a crucial role in deriving Euler--Lagrange equations.


\subsection{Classical operators}

As it was seen in Section~\ref{sec:history}, there are more than one way
to generalize integer-order operations to the non-integer case. Here, we
present several definitions and properties about fractional operators,
omitting some details about the conditions that ensure the existence of
such fractional operators.

In general, the fractional derivatives are defined using fractional integrals.
We present only two fractional integrals operators, but there are several
known forms of the fractional integrals.

Let $x:[a,b]\to \mathbb{R}$ be an integrable function  and $\alpha >0$ a
real number.

Starting with Cauchy's formula for a $n$-fold iterated integral, given by
\begin{equation}
\label{nIntegral}
\begin{split}
_aI_t^n x(t)&=\int_a^t d\tau_1\int_a^{\tau_1}d\tau_2\cdots \int_a^{\tau_{n-1}}
x(\tau_n)d\tau_n\\
& = \frac{1}{(n-1)!}\int_a^t (t-\tau)^{n-1}x(\tau) d\tau,
\end{split}
\end{equation}
where $n \in \mathbb{N}$, Liouville and Riemann defined fractional integration,
generalizing equation~\eqref{nIntegral} to noninteger values of $n$ and using the definition of Gamma function $\Gamma$.
With this, we introduce two important concepts: the left and the right
Riemann--Liouville fractional integrals.

\begin{Definition}
\label{R-LIntegrals}
\index{Riemann--Liouville fractional integral! left}
\index{Riemann--Liouville fractional integral! right}
We define the left and right Riemann--Liouville fractional integrals of order
$\alpha$, respectively, by
$$
{_aI_t}^{\alpha} x(t) = \frac{1}{\Gamma(\alpha)} \int_a^t (t-\t)^{\alpha-1}x(\t)
d\t, \quad t>a
$$
and
$$
{_tI_b}^{\alpha} x(t) = \frac{1}{\Gamma(\alpha)} \int_t^b (\t-t)^{\alpha-1}x(\t)
d\t, \quad t<b.
$$
\end{Definition}

The constants $a$ and $b$ determine, respectively, the lower and upper boundary
of the integral domain.
Additionally, if $x$ is a continuous function, as $\alpha\rightarrow 0$, ${_aI_t}^{\alpha}={_tI_b}^{\alpha}=I$, with $I$ the
identity operator, i.e., ${_aI_t}^{\alpha} x(t)={_tI_b}^{\alpha} x(t) = x(t)$.

We present the second fractional integral operator, introduced by J. Hadamard
in 1892 \cite{Cap1:Kilbas}.

\begin{Definition}
\label{HadIntegrals}
\index{Hadamard fractional integral! left}
\index{Hadamard fractional integral! right}
We define the left and right Hadamard fractional integrals of order $\alpha$,
respectively, by
$$
_aJ_t^{\alpha} x(t) = \frac{1}{\Gamma(\alpha)} \int_a^t \left(\ln \frac{t}{\t} \right)^{\alpha-1} \frac{x(\t)}{\t} d\t, \quad t>a
$$
and
$$
_tJ_b^{\alpha} x(t) = \frac{1}{\Gamma(\alpha)} \int_t^b \left(\ln \frac{\t}{t} \right)^{\alpha-1} \frac{x(\t)}{\t} d\t, \quad t<b.
$$
\end{Definition}

The three most frequently used definitions for fractional derivatives are:
the Gr\"unwald-Letnikov, the Riemann--Liouville  and the Caputo fractional
derivatives \cite{Cap1:Oldham,Cap1:Podlubny}. Other definitions were introduced by
others mathematicians, as for instance Weyl, Fourier, Cauchy, Abel.

Let $x\in AC([a,b];\bR)$ be an absolutely continuous functions on the interval $[a,b]$, and $\alpha$ a positive real number.
Using Definition~\ref{R-LIntegrals} of Riemann--Liouville fractional integrals,
we define the left and the right Riemann--Liouville and Caputo derivatives as
follows.

\begin{Definition}
\label{R-Lderivatives}
\index{Riemann--Liouville fractional derivative! left}
\index{Riemann--Liouville fractional derivative! right}
We define the left and right Riemann--Liouville fractional derivatives of order
$\alpha>0$, respectively, by
\begin{equation*}
\begin{split}
_aD_t^{\alpha} x(t) =&  \left(\frac{d}{dt}\right)^n \,  _aI_t^{n-\alpha} x(t)\\
=& \frac{1}{\Gamma(n-\alpha)} \left( \frac{d}{dt}\right)^n \int_a^t (t-\t)^{n-\alpha-1}x(\t)d\t, \quad t>a
\end{split}
\end{equation*}
and
\begin{equation*}
\begin{split}
_tD_b^{\alpha} x(t) =& \left(-\frac{d}{dt}\right)^n \,  _tI_b^{n-\alpha} x(t)\\
=& \frac{(-1)^n}{\Gamma(n-\alpha)}  \left(\frac{d}{dt}\right)^n \int_t^b (\t-t)^{n-\alpha-1}x(\t)d\t, \quad t<b,
\end{split}
\end{equation*}
where $n=[\alpha]+1$.
\end{Definition}

The following definition was introduced in \cite{Cap1:Caputo}. 
The Caputo fractional derivatives, in general, are more applicable and interesting in fields
like physics and engineering, for its properties like the initial conditions.

\begin{Definition}
\label{CapDer}
\index{Caputo fractional derivative! left}
\index{Caputo fractional derivative! right}
We define the left and right Caputo fractional derivatives of order $\alpha$,
respectively, by
\begin{equation*}
\begin{split}
_a^CD_t^{\alpha} x(t)& =  {_aI_t^{n-\alpha}} \left( \frac{d}{dt}\right)^n x(t) \\
&= \frac{1}{\Gamma(n-\alpha)} \int_a^t (t-\t)^{n-\alpha-1}  x^{(n)}(\t)d\t, \quad t>a
\end{split}
\end{equation*}
and
\begin{equation*}
\begin{split}
_t^CD_b^{\alpha} x(t) &=  {_tI_b^{n-\alpha}} \left(-\frac{d}{dt} \right)^n x(t)\\
&= \frac{(-1)^n}{\Gamma(n-\alpha)} \int_t^b (\t-t)^{n-\alpha-1} x^{(n)}(\t)d\t, \quad t<b,
\end{split}
\end{equation*}
where $n=[\alpha]+1$ if $\alpha\notin \mathbb N$ and
$n=\alpha$ if $\alpha\in\mathbb N$.
\end{Definition}

Obviously, the above defined operators are linear. From these definitions,
it is clear that the Caputo fractional derivative of a constant $C$ is zero,
which is false when we consider the Riemann--Liouville fractional derivative.
If $x(t)=C$, with $C$ a constant, then we get
$$_a^CD_t^{\alpha} x(t) = {_t^CD_b^{\alpha} x(t)}= 0$$
and
$$_aD_t^{\alpha} x(t) =\frac{C\, (t-a)^{-\alpha}}{\Gamma(1-\alpha)},
\quad _tD_b^{\alpha} x(t)
=\frac{C\, (b-t)^{-\alpha}}{\Gamma(1-\alpha)}.$$
For this reason, in some applications, Caputo fractional derivatives seem to
be more natural than the Riemann--Liouville fractional derivatives.

\begin{Remark}
If $\alpha$ goes to $n^{-}$, with $n\in\mathbb{N}$, then the fractional operators introduced above
coincide with the standard derivatives:
$$ _aD_t^{\alpha} = {_a^CD_t^{\alpha}} = \left(\frac{d}{dt}\right)^n$$
and
$$_tD_b^{\alpha} ={_t^CD_b^{\alpha}}=\left(-\frac{d}{dt}\right)^n.$$
\end{Remark}

The Riemann--Liouville fractional integral and differential operators of order
$\alpha>0$ of power functions return power functions, as we can see below.

\begin{Lemma} Let $x$ be the power function $x(t)=(t-a)^\gamma$.
Then, we have
$$_aI_t^{\alpha} x(t) = \frac{\Gamma(\gamma+1)}{\Gamma(\gamma+\alpha+1)}
(t-a)^{\gamma+\alpha}, \quad  \gamma>-1$$
and
$${_aD_t^{\alpha}} x(t)= \frac{\Gamma(\gamma+1)}{\Gamma(\gamma-\alpha+1)}
(t-a)^{\gamma-\alpha}, \quad  \gamma>-1.$$
\end{Lemma}

In particular, if we consider $\gamma=1$, $a=0$ and $\alpha=1/2$, then the left
Riemann--Liouville fractional derivative of $x(t)=t$ is $\frac{2\sqrt{t}}{\sqrt{\pi}}$,
the same result \eqref{orederonehalf} as Lacroix obtained in 1819.

Gr\"unwald and Letnikov, respectively in 1867 and 1868, 
returned to the original sources and
started the formulation by the fundamental definition of a
derivative, as a limit,
$$
x^{(1)} (t)=\lim_{h\rightarrow0}{\frac{x(t+h)-x(t)}{h}}
$$
and considering the iteration at the $n$th order derivative formula:
\begin{equation}
\label{HO_natural}
x^{(n)} (t)=\lim_{h\rightarrow0}{\frac{1}{h^n}} \sum_{k=0}^{n}(-1)^k
{n \choose k} x(t-kh), \quad n \in \mathbb{N},
\end{equation}
where ${n \choose k}=\frac{n(n-1)(n-2)...(n-k+1)}{k!}$, with
$n, k \in \mathbb{N}$, the usual notation for the binomial coefficients.
The Gr\"unwald-Letnikov definition of fractional derivative consists in a
generalization of \eqref{HO_natural} to derivatives of arbitrary order
$\alpha$ \cite{Cap1:Podlubny}.

\begin{Definition}
\index{Gr\"unwald-Letnikov fractional derivative}
The $\alpha$th order Gr\"unwald-Let\-nikov fractional derivative of
function $x$ is given by
$$
_a^{GL}D_t^{\alpha} x(t) =  \lim_{h\longrightarrow0}{h^{-\alpha} \sum_{k=0}^{n}(-1)^k
{\alpha \choose k} x(t-kh),}
$$
where ${\alpha \choose k}=\frac{\Gamma (\alpha+1)}{k !  \Gamma(\alpha-k+1)}$
and $nh=t-a$.
\end{Definition}

\begin{Lemma}
Let $x$ be the power function $x(t)=(t-a)^\gamma$, where $\gamma$ is a real
number. Then, for $0<\alpha<1$ and $\gamma>0$, we have
$$
{_a^{GL}D_t^{\alpha}} x(t)=\frac{\Gamma(\gamma+1)}{\Gamma(\gamma-\alpha+1)}
(t-a)^{\gamma-\alpha}.\\
$$
\end{Lemma}

\begin{Definition}
\label{HadDerivatives}
\index{Hadamard fractional derivative! left}
\index{Hadamard fractional derivative! right}
We define the left and right Hadamard fractional derivatives of order $\alpha$,
respectively, by
$$
_a\mathsf{D}_t^{\alpha} x(t) =\left(t \frac{d}{dt}\right)^n \frac{1}{\Gamma(n-\alpha)} \int_a^t \left( \ln \frac{t}{\t} \right)^{n-\alpha-1} \frac{x(\t)}{\t} d\t
$$
and
$$
_t\mathsf{D}_b^{\alpha} x(t) = \left(-t \frac{d}{dt}\right)^n \frac{1}{\Gamma(n-\alpha)} \int_t^b \left(\ln \frac{\t}{t} \right)^{n-\alpha-1} \frac{x(\t)}{\t} d\t,
$$
for all $t \in (a,b)$, where $n=[\alpha]+1$.
\end{Definition}

Observe that, for all types of derivative operators, if variable $t$ is
the time-variable, the left fractional derivative of $x$ is interpreted
as a past state of the process, while the right fractional derivative
of $x$ is interpreted as a future state of the process.


\subsection{Some properties of the Caputo derivative}

For $\alpha>0$ and $x \in AC ([a,b];\bR)$, the Riemann--Liouville and
Caputo derivatives are related by the following formulas \cite{Cap1:Kilbas}:
\begin{equation}
\label{relRL-Cl}
{_a^CD_t^\alpha}x(t)={_aD_t^\alpha}\left[ x(t)-\sum_{k=0}^{n-1} \frac{x^{(k)}(a)(t-a)^k}{k!} \right]
\end{equation}
and
\begin{equation}
\label{relRL-Cr}
{_t^CD_b^\alpha}x(t)={_tD_b^\alpha}\left[ x(t)-\sum_{k=0}^{n-1} \frac{x^{(k)}(b)(b-t)^k}{k!} \right],
\end{equation}
where $n=[\alpha]+1$ if $\alpha\notin \mathbb N$ and $n=\alpha$ if $\alpha\in\mathbb N$.\\
In particular, when $\alpha \in (0,1)$, the relations \eqref{relRL-Cl} and
\eqref{relRL-Cr} take the form:
\begin{equation}
\label{relRL-Clp}
\begin{array}{ll}
{_a^CD_t^\alpha}x(t)&={_aD_t^\alpha}(x(t)-x(a))\\
&={_aD_t^\alpha}x(t)-\frac{x(a)}{\Gamma(1-\alpha)}(t-a)^{-\alpha}
\end{array}
\end{equation}
and
\begin{equation}
\label{relRL-Crp}
\begin{array}{ll}
{_t^CD_b^\alpha}x(t)&={_tD_b^\alpha}(x(t)-x(b))\\
&={_tD_b^\alpha}x(t)-\frac{x(b)}{\Gamma(1-\alpha)}(b-t)^{-\alpha}.
\end{array}
\end{equation}

It follows from \eqref{relRL-Clp} and \eqref{relRL-Crp} that the left
Riemann--Liouville derivative equals the left Caputo fractional derivative in
the case $x(a)=0$ and the analogue holds for the right derivatives
under the assumption $x(b)=0$.

In Theorem~\ref{composition}, we see that the Caputo fractional derivatives
provides a left inverse operator to the Riemann--Liouville fractional integration 
(cf. Lemma~2.21 in \cite{Cap1:Kilbas}).
\begin{Theorem}
\label{composition}
Let $\alpha>0$ and let $x \in C\left([a,b];\bR^n\right)$. For
Caputo fractional operators, the next rules hold:
$$
_a^CD_t^\alpha \, _aI_t^\alpha x(t) = x(t)
$$
and
$$
 _t^CD_b^\alpha  \, _tI_b^\alpha x(t) = x(t).
$$
\end{Theorem}

The next statement characterizes the composition of the Riemann--Liouville
fractional integration operators with the Caputo fractional differentiation
operators (cf. Lemma~2.22 in \cite{Cap1:Kilbas}).

\begin{Theorem}
Let $\alpha>0$. If $x \in AC^n \left([a,b];\bR\right)$, then
$$
_aI_t^\alpha \, _a^CD_t^\alpha x(t) =x(t) - \sum_{k=0}^{n-1}
\frac{x^{(k)}(a)}{k!}(t-a)^k
$$
and
$$
_tI_b^\alpha \, _t^CD_b^\alpha x(t) =x(t) - \sum_{k=0}^{n-1}
\frac{(-1)^k x^{(k)}(b)}{k!}(b-t)^k,
$$
with $n=[\alpha]+1$ if $\alpha\notin \mathbb N$ and $n=\alpha$ if
$\alpha\in\mathbb N$.\\
In particular, when $\alpha \in (0,1)$, then
$$ _aI_t^\alpha \, _a^CD_t^\alpha x(t) =x(t) - x(a) \quad \mbox{and}
\quad  _tI_b^\alpha \, _t^CD_b^\alpha x(t) =x(t) -x(b).$$
\end{Theorem}

Similarly to the Riemann--Liouville fractional derivative, the Caputo fractional
derivative of a power function yields a power function of the same form.

\begin{Lemma}
Let $\alpha>0$. Then, the following relations hold:
$${_a^CD_t^{\alpha}} (t-a)^\gamma =\frac{\Gamma(\gamma+1)}
{\Gamma(\gamma-\alpha+1)}(t-a)^{\gamma-\alpha}, \quad  \gamma>n-1$$
and
$${_t^CD_b^{\alpha}}(b-t)^\gamma =\frac{\Gamma(\gamma+1)}
{\Gamma(\gamma-\alpha+1)}(b-t)^{\gamma-\alpha}, \quad  \gamma>n-1,$$
with $n=[\alpha]+1$ if $\alpha\notin \mathbb N$ and $n=\alpha$
if $\alpha\in\mathbb N$.
\end{Lemma}


\subsection{Combined Caputo derivative}
\label{sec:combC_classic}

In this section, we introduce a special operator for our work, the
combined Caputo fractional derivative.
We extend the notion of the Caputo fractional derivative to the fractional
derivative $^CD_\gamma^{\alpha,\beta}$, that involves the left and the right
Caputo fractional derivative, i.e., it combines the past and the future
of the process into one single operator.

This operator was introduced in
\cite{Cap1:Malin:Tor2010}, motivated by the ideia of the symmetric fractional
derivative introduced in \cite{Cap1:Klimek}:

\begin{Definition}
\index{symmetric fractional derivative}
Considering the left and right Riemann--Liouville derivatives, the symmetric
fractional derivative is given by
\begin{equation}
_aD_b^\alpha x(t)=\frac{1}{2}\left[_aD_t^\alpha x(t) +{_tD_b^\alpha}x(t)\right].
\end{equation}
\end{Definition}

Other combined operators were studied. For example, we have the Riesz and
the Riesz--Caputo operators \cite{Cap1:book:FCV}.

\begin{Definition}
\index{Riesz fractional integral}
Let $x:[a,b] \to \bR$ be a function of class $C^1$ and $\alpha \in (0,1)$.
For $t \in [a,b]$, the Riesz fractional integral of order $\alpha$, is
defined by
\begin{equation*}
\begin{split}
_a^RI_b^\alpha x(t)&=\frac{1}{2\Gamma(\alpha)} \int_a^b |t-\t|^{\alpha-1} x(\t)d\t\\
&=\frac{1}{2}\left[_aI_t^\alpha x(t) +{_tI_b^\alpha}x(t)\right],
\end{split}
\end{equation*}
and the Riesz fractional derivative of order $\alpha$, is defined by
\index{Riesz fractional derivative}
\begin{equation*}
\begin{split}
_a^RD_b^\alpha x(t)&=\frac{1}{\Gamma(1-\alpha)} \frac{d}{dt} \int_a^b
|t-\t|^{-\alpha} x(\t)d\t\\
&=\frac{1}{2}\left[_aD_t^\alpha x(t)- {_tD_b^\alpha}x(t) \right].
\end{split}
\end{equation*}
\end{Definition}

\begin{Definition}
\index{Riesz--Caputo fractional derivative}
Considering the left and right Caputo derivatives, Riesz--Caputo fractional
derivative is given by
\begin{equation}
\label{RCder}
\begin{split}
_a^{RC}D_b^{\alpha}x(t)&=\frac{1}{\Gamma(1-\alpha)} \int_a^b |t-\t|^{-\alpha}
\frac{d}{d\t} x(\t)d\t\\
&=\frac{1}{2}\left[_a^CD_t^\alpha x(t)- {_t^CD_b^\alpha}x(t) \right].
\end{split}
\end{equation}
\end{Definition}

\begin{Remark}[\citet{Cap1:book:FCV}]
If $\alpha$ goes to $1$, then the fractional derivatives introduced above coincide
with the standard derivative:
$$ _a^RD_b^{\alpha} = {_a^{RC}D_b^{\alpha}} = \frac{d}{dt}.$$
\end{Remark}

Similarly to the last operator \eqref{RCder}, the combined Caputo derivative is
a convex combination of the left and the right Caputo fractional derivatives.
But, in this operator, we also consider other coefficients 
for the convex combination besides $1/2$.  
Moreover, the orders $\alpha$ and $\beta$ of the left-
and right-sided fractional derivatives can be different.
Therefore, the combined Caputo derivative
is a convex combination of the left Caputo fractional derivative of order $\alpha$
and the right Caputo fractional derivative of order $\beta$.

\begin{Definition}
\label{CombClassical}
\index{combined Caputo fractional derivative! classical}
Let $\alpha,\beta \in (0, 1)$ and $\gamma \in [0,1]$.
The combined Caputo fractional derivative operator $^CD_\gamma^{\alpha,\beta}$
is defined by
\begin{equation}
\label{Defcomb_1}
^CD_\gamma^{\alpha,\beta}=\gamma \, _a^CD_t^\alpha + (1-\gamma) \, _t^CD_b^\beta,
\end{equation}
which acts on $x\in AC\left([a,b];\bR\right)$ in the following way:
$$
^CD_\gamma^{\alpha,\beta}x(t)=\gamma \, _a^CD_t^\alpha x(t) + (1-\gamma)
\, _t^CD_b^\beta x(t).
$$
\end{Definition}

The operator \eqref{Defcomb_1} is obviously linear.
Observe that
$$^CD_0^{\alpha,\beta}={_t^CD_b^{\beta}} \quad \mbox{and} \quad
^CD_1^{\alpha,\beta}={_a^CD_t^{\alpha}}.$$

The symmetric fractional derivative and the Riesz fractional derivative are
useful tools to describe some nonconservative models.
But those types of differentiation do not seem suitable for all kinds of
variational problems because they are based on the Riemann--Liouville fractional derivatives and therefore the possibility that admissible trajectories $x$ have continuous fractional derivatives imply that $x(a) = x(b) = 0$
\cite{Cap1:Samko:Ross}.
For more details about the combined Caputo fractional derivative, 
see \cite{Cap1:Malin:Tor2011,Cap1:Malin:Tor2012,Cap1:Od2012b,Cap1:Malin:Tor2012b}.


\subsection{Variable-order operators}
\label{VariableOperators}

Very useful physical applications have given birth to the variable-order
fractional calculus, for example in modeling mechanical behaviors
\cite{Cap1:Fu,Cap1:Sun:Hu}.
Nowadays, variable-order fractional calculus is particularly recognized
as a useful and promising approach in the modelling of diffusion processes,
in order to characterize time-dependent or concentration-dependent anomalous
diffusion, or diffusion processes in inhomogeneous porous media
\cite{Cap1:Sun:Chen:Li}.

Now, we present the fundamental notions of the fractional calculus of
variable-order \cite{Cap1:book:MOT}.
We consider the fractional order of the derivative and of the
integral to be a continuous function of two variables, $\alpha(\cdot,\cdot)$
with domain $[a,b]^2$, taking values on the open interval $(0,1)$.
Let $x:[a,b]\to\bR$ be a function.

First, we recall the generalization of fractional integrals for 
a variable-order $\a$.

\begin{Definition}
\index{Riemann--Liouville fractional integral! left (variable-order)}
\index{Riemann--Liouville fractional integral! right (variable-order)}
The left and right Riemann--Liouville fractional integrals of order $\a$ are defined by
$$
\LI x(t)=\int_a^t \frac{1}{\Gamma(\alpha(t,\t))}(t-\t)^{\alpha(t,\t)-1}x(\t)d\t, \quad t>a
$$
and
$$
\RI x(t)=\int_t^b\frac{1}{\Gamma(\alpha(\t,t))}(\t-t)^{\alpha(\t,t)-1} x(\t)d\t, \quad t<b,
$$
respectively.
\end{Definition}

We remark that, in contrast to the fixed fractional order case,
variable-order fractional integrals are not the inverse operation
of the variable-order fractional derivatives.

For fractional derivatives, we consider two types: the Riemann--Liouville and
the Caputo fractional derivatives.

\begin{Definition}
\index{Riemann--Liouville fractional derivative! left (variable-order)}
\index{Riemann--Liouville fractional derivative! right (variable-order)}
The left and right Riemann--Liouville fractional derivatives of order $\a$ are
defined by
\begin{equation}
\begin{split}
\LDa x(t) = & \frac{d}{dt} \LI x(t)\\
= &\frac{d}{dt} \int_a^t \frac{1}{\Gamma(1-\alpha(t,\t))}(t-\t)^{-\alpha(t,\t)}x(\t)d\t, \quad t>a
\end{split}
\end{equation}
and
\begin{equation}
\begin{split}
\RDa x(t)= & - \frac{d}{dt} \RI x(t)\\
=&\frac{d}{dt}\int_t^b\frac{-1}{\Gamma(1-\alpha(\t,t))}(\t-t)^{-\alpha(\t,t)} x(\t)d\t, \quad t<b
\end{split}
\end{equation}
respectively.
\end{Definition}

Lemma \ref{pf_RL-VO} gives a Riemann--Liouville variable-order fractional integral and fractional
derivative for the power function $x(t)=(t-a)^\gamma$, where we assume that the
fractional order depends only on the first variable:
$\alpha(t, \t):= \overline{\alpha}(t)$, where  $\overline{\alpha}: [a,b]
\rightarrow (0,1)$ is a given function.

\begin{Lemma}
\label{pf_RL-VO}
Let $x$ be the power function $x(t)=(t-a)^\gamma$.
Then, for $\gamma>-1$, we have
$$
_aI_t^{\overline{\alpha}(t)} x(t)= \frac{\Gamma(\gamma+1)}
{\Gamma(\gamma+\overline{\alpha}(t)+1)}(t-a)^{\gamma+\overline{\alpha}(t)}
$$
and
$$\begin{array}{ll}
_aD_t^{\overline{\alpha}(t)} x(t) &=\frac{\Gamma(\gamma+1)}
{\Gamma(\gamma-\overline{\alpha}(t)+1)}(t-a)^{\gamma-\overline{\alpha}(t)}\\
&-\overline{\alpha}^{(1)}(t)\frac{\Gamma(\gamma+1)}
{\Gamma(\gamma-\overline{\alpha}(t)+2)}(t-a)^{\gamma-\overline{\alpha}(t)+1}\\
&\times\left[\ln(t-a)-\Psi(\gamma-\overline{\alpha}(t)+2)
+\Psi(1-\overline{\alpha}(t))\right].\\
\end{array}$$
\end{Lemma}

\begin{Definition}
\index{Caputo fractional derivative! left (variable-order)}
\index{Caputo fractional derivative! right (variable-order)}
The left and right Caputo fractional derivatives of order $\a$
are defined by
\begin{equation}
\LC x(t)=\int_a^t\frac{1}{\Gamma(1-\alpha(t,\t))}(t-\t)^{-\alpha(t,\t)}x^{(1)}(\t)d\t, \quad t>a
\end{equation}
and
\begin{equation}
\RCa x(t)=\int_t^b\frac{-1}{\Gamma(1-\alpha(\t,t))}(\t-t)^{-\alpha(\t,t)}x^{(1)}(\t)d\t, \quad t<b,
\end{equation}
respectively.
\end{Definition}
Of course, the fractional derivatives just defined are linear operators.


\subsection{Generalized fractional operators}

In this section, we present three definitions of one-dimensional generalized
fractional operators that depend on a general kernel, studied by Odzijewicz,
Malinowska and Torres (see e.g. \cite{Cap1:book:MOT,Cap1:Od2012a,Cap1:Od_2013b}, 
although \cite{Cap1:Agrawal2010} had introduced this generalized fractional operators).

Let $\Delta:=\lbrace (t,\tau) \in \bR^2 : a\leq\tau < t \leq b \rbrace$.

\begin{Definition}
\index{generalized fractional integrals! of Riemann--Liouville}
\label{gfi-RL}
Let $k^\alpha$ be a function defined almost everywhere on $\Delta$ with values in $\bR$.
For all $x:[a,b]\rightarrow\bR$, the generalized fractional integral operator
$K_P$ is defined by
\begin{equation}
\label{GenFIRL}
K_P [x](t) = \lambda \int_a^t k^\alpha(t,\tau) x(\t)d\t + \mu \int_t^b k^\alpha(\tau,t) x(\t)d\t,
\end{equation}
with $P = \langle a, t, b, \lambda, \mu \rangle$, where $\lambda$ and $\mu$ are
real numbers.
\end{Definition}

In particular, if we choose  special cases for the kernel, we can obtain
standard fractional operators or variable-order.

\begin{Remark}
\label{example_kernel} For special chosen kernels $k^\alpha$ and parameters $P$,
the operator $K_P$ can be reduced  to the classical or variable-order
Riemann--Liouville fractional integrals:
\begin{itemize}

\item Let $k^\alpha(t,\t)=\frac{1}{\Gamma(\alpha)}(t-\t)^{\alpha-1}$ and
$0<\alpha<1$. If $P = \langle a, t, b, 1, 0 \rangle$, then
$$K_P [x](t) = \frac{1}{\Gamma(\alpha)} \int_a^t (t-\t)^{\alpha-1} x(\t)d\t
=: {_aI_t^\alpha [x](t)}$$
is the left Riemann--Liouville fractional integral of order $\alpha$;

if $P=\langle a, t, b, 0, 1 \rangle$, then
$$K_P [x](t) = \frac{1}{\Gamma(\alpha)} \int_t^b (\t-t)^{\alpha-1} x(\t)d\t
=: {_tI_b^\alpha [x](t)}$$
is the right Riemann--Liouville fractional integral of order $\alpha$.

\item If $k^\alpha(t,\t)=\frac{1}{\Gamma(\alpha(t,\t))}(t-\t)^{\alpha(t,\t)-1}$
and $P = \langle a, t, b, 1, 0 \rangle$, then
$$K_P [x](t) = \int_a^t \frac{1}{\Gamma(\alpha(t,\t))} (t-\t)^{\alpha(t,\t)-1}
x(\t)d\t=: {_aI_t^\a [x](t)}$$
is the left Riemann--Liouville fractional integral of variable-order $\a$; 

if $P=\langle a, t, b, 0, 1 \rangle$, then
$$K_P [x](t) = \int_t^b \frac{1}{\Gamma(\alpha(t,\t))}  (\t-t)^{\alpha(t,\t)-1}
x(\t)d\t=: {_tI_b^\a [x](t)}$$
is the right Riemann--Liouville fractional integral of variable-order $\a$.
\end{itemize}
\end{Remark}

Some other fractional operators can be obtained with the generalized fractional
integrals, for example, Hadamard, Riesz, Katugampola fractional operators
\cite{Cap1:book:MOT,Cap1:Agrawal2010}.

The following two news operators, the generalized fractional Riemann--Liouville
and Caputo derivatives, are defined as a composition of classical derivatives
and generalized fractional integrals.

\begin{Definition}
\index{generalized fractional derivative! of Riemann--Liouville}
The gene\-ralized fractional derivative of Riemann--Liouville type, denoted
by $A_P$, is
defined by
$$A_P=\frac{d}{dt} \circ K_P.$$
\end{Definition}

\begin{Definition}
\index{generalized fractional derivative! of Caputo}
The generalized fractional derivative of Caputo type, denoted by $B_P$, is
defined by
$$
B_P= K_P \circ  \frac{d}{dt}.
$$
\end{Definition}

Considering $k^\alpha(t,\t)=\frac{1}{\Gamma(1-\alpha)}(t-\t)^{-\alpha}$
with $0<\alpha<1$ and appropriate sets $P$, this two general kernel
operators $A_P$ and $B_P$ can be reduced to the standard Riemann--Liouville
and Caputo fractional derivatives, respectively \cite{Cap1:book:MOT}.


\subsection{Integration by parts} 
\label{subsec:IP}

In this section, we summarize formulas of integration by parts because
they are important results to find necessary optimality conditions
when dealing with variational problems.

First, we present the rule of fractional integration by parts 
for the Riemann--Liouville fractional integral.

\begin{Theorem}
\label{thm:FIP:int}
\index{integration by parts! Riemann--Liouville fractional integrals}
Let $0<\alpha<1$, $p \geq 1$, $q \geq 1$ and $1/p+1/q \leq 1+\alpha$.
If $y\in L_p([a,b];\bR)$ and $x\in L_q([a,b];\bR)$, then the following formula
for integration by parts hold:
$$
\int_{a}^{b}y(t) \, _aI_t^{\alpha} x(t)dt =\int_a^b x(t) \, _tI_b^{\alpha}y(t)dt.
$$
\end{Theorem}

For Caputo fractional derivatives, the integration by parts formulas
are presented below \cite{Cap1:Alm:Malin}.

\begin{Theorem}
\label{thm:FIP:cap}
\index{integration by parts! Caputo fractional derivatives}
Let $0<\alpha<1$. The following relations hold:
\begin{equation}
\label{FIP:cap_l}
\int_{a}^{b}y(t) \, _a^CD_t^{\alpha} x(t)dt
=\int_a^b x(t) \, {_tD_b^{\alpha}}y(t)dt
+\left[x(t) \, {_tI_b^{1-\alpha}}y(t) \right]_{t=a}^{t=b}
\end{equation}
and
\begin{equation}
\label{FIP:cap_r}
\int_{a}^{b}y(t) \, _t^CD_b^{\alpha}x(t)dt
=\int_a^b x(t) \, {_aD_t^{\alpha}} y(t)dt
-\left[x(t) \, {_aI_t^{1-\alpha}}y(t)\right]_{t=a}^{t=b}.
\end{equation}
\end{Theorem}

When $\alpha \rightarrow 1$, we get $_a^CD_t^{\alpha}={_aD_t^{\alpha}}=\frac{d}{dt}$, $_t^CD_b^{\alpha}={_tD_b^{\alpha}}=-\frac{d}{dt}$, $_aI_t^{\alpha}={_tI_b^{\alpha}}=I$,
and formulas \eqref{FIP:cap_l} and \eqref{FIP:cap_r} give the classical
formulas of integration by parts.

Then, we introduce the integration by parts formulas for variable-order 
fractional integrals \cite{Cap1:Od2013a}.

\begin{Theorem}
\index{integration by parts! fractional integrals of variable-order}
\label{thm:FIP_VO_Int}
Let $\frac{1}{n}<\alpha (t,\tau)<1$ for all $t,\tau \in [a,b]$ and a certain
$n\in \mathbb{N}$ greater or equal than two, and $x,y \in C([a,b];\bR)$.
Then the following formula for integration by parts hold:
$$
\int_{a}^{b} y(t) \, _aI_t^{\a} x(t)dt =\int_a^b x(t) \, _tI_b^{\a}y(t)dt.
$$
\end{Theorem}

In the following theorem, we present the formulas involving the Caputo
fractional derivative of variable-order. The theorem was proved 
in \cite{Cap1:Od2013a} and gives a generalization of the standard
fractional formulas of integration by parts for a constant $\alpha$.

\begin{Theorem}
\label{thm:FIP}
\index{integration by parts! Caputo fractional derivatives of variable-order}
Let $0<\alpha (t,\tau)<1-\frac{1}{n}$ for all $t,\tau \in [a,b]$ and a certain
$n\in \mathbb{N}$ greater or equal than two. If $x,y \in C^1 \left([a,b];\bR\right)$, then the
fractional integration by parts formulas
$$
\int_{a}^{b}y(t) \, \LC x(t)dt
=\int_a^b x(t) \, {\RDa}y(t)dt
+\left[x(t) \, {_tI_b^{1-\a}}y(t) \right]_{t=a}^{t=b}
$$
and
$$
\int_{a}^{b}y(t) \, {\RCa}x(t)dt
=\int_a^b x(t) \, {\LDa} y(t)dt
-\left[x(t) \, {_aI_t^{1-\a}}y(t)\right]_{t=a}^{t=b}
$$
hold.
\end{Theorem}

This last theorem have an important role in this work to the proof of the
generalized Euler--Lagrange equations.

In the end of this chapter, we present integration by parts formulas for
generalized fractional operators \cite{Cap1:book:MOT}. For that, we need the
following definition:

\begin{Definition}
\index{dual parameter set}
Let $P=\langle a,t,b,\lambda,\mu\rangle$. We denote by $P^*$ the parameter
set $P^*=\langle a,t,b,\mu,\lambda\rangle$. The parameter $P^*$ is called
the dual of $P$.
\end{Definition}

Let $1<p<\infty$ and $q$ be the adjoint of $p$, that is $\frac{1}{p}+\frac{1}{q}=1$. 
A proof of the next result can be found in \cite{Cap1:book:MOT}.

\begin{Theorem}
Let $k \in L_q (\Delta;\bR)$. Then the operator $K_{P^*}$ is a linear bounded
operator from $L_p([a,b];\bR)$ to $L_q([a,b];\bR)$. Moreover, the following
integration by parts formula holds:
$$\int_a^b x(t)\cdot K_P[y](t) dt =  \int_a^b y(t)\cdot K_{P^*}[x](t) dt$$
for all $x,y \in L_p([a,b];\bR)$.
\end{Theorem}

\begingroup
\renewcommand{\addcontentsline}[3]{}

\endgroup

%% file: 2FCV.tex
\chapter{The calculus of variations}
\label{PartI_FCV}

As part of this book is devoted to the fractional calculus of variations,
in this chapter we introduce the basic concepts about the classical calculus
of variations and the fractional calculus of variations. The study of fractional
problems of the calculus of variations and respective Euler--Lagrange type
equations is a subject of current strong research.

In Section \ref{sec:cv}, we introduce some concepts and important results from
the classical theory. Afterwards, in Section~\ref{sec:fcv}, we start with a
brief historical introduction to the non--integer calculus of variations and
then we present recent results on the fractional calculus of variations.

For more information about this subject, we refer the reader to the books
\cite{Cap2:Brunt,Cap2:book:FCV,Cap2:book2:FCV,Cap2:book:MOT}.


\section{The classical calculus of variations}
\label{sec:cv}
The calculus of variations is a field of mathematical analysis that concerns
with finding extrema (maxima or minima) for functionals, i.e., concerns with
the problem of finding a function for which the value of a certain integral
is either the largest or the smallest possible. \\

In this context, a functional is a mapping from a set of functions to the real
numbers, i.e., it receives a function and produces a real number.
Let $D \subseteq C^2 ([a,b];\bR)$ be a linear space endowed with a
norm $\|\cdot\|$. The cost functional $\mathcal{J}:D \rightarrow \bR$ is
generally of the form
\begin{equation}
\label{functional_RN}
\mathcal{J}(x)=\int_a^b L\left(t, x(t), x'(t)\right)dt,
\end{equation}
where $t \in [a,b]$ is the independent variable, usually called time, and
$x(t) \in \bR$ is a function. The integrand $L: [a,b] \times
\bR^2 \rightarrow \bR$, that depends on the function $x$, its
derivative $x'$ and the independent variable $t$, is a real-valued function,
called the Lagrangian.

The roots of the calculus of variations appear in works of Greek thinkers, such
as Queen Dido or Aristotle in the late of the 1st century BC.
During the $17$th century, some physicists and mathematicians (Galileo, Fermat,
Newton, among others) investigated some variational problems, but in general they did
not use variational methods to solve them. The development of the calculus of
variations began with a problem posed by Johann Bernoulli in 1696, called the
brachistochrone problem: given two points $A$ and $B$ in a vertical plane,
what is the curve traced out by a point acted on only by gravity, which starts
at $A$ and reaches $B$ in minimal time?\index{brachistochrone}
The curve that solves the problem is called the \textit{brachistochrone}.
This problem caught the attention of some mathematicians including Jakob
Bernoulli, Leibniz, L'H\^opital and Newton, which presented also a solution
for the brachistochrone problem.
Integer variational calculus is still, nowadays, a revelant area of research.
It plays a significant role in many areas of science, physics, engineering,
economics, and applied mathematics.

The classical variational problem, considered by Leonhard Euler, is stated as
follows.\\
Let $a, b \in \bR$. Among all functions $x\in D$, find 
the ones that minimize (or maximize) the functional $\mathcal{J}:D\rightarrow\bR$, where
\begin{equation}
\label{class_functJ}
\mathcal{J}(x)=\int_a^b L\left(t, x(t), x'(t)\right)dt,
\end{equation}
subject to the boundary conditions \index{boundary conditions}
\begin{equation}
\label{bondcond}
x(a)=x_a, \quad \mbox \quad x(b)=x_b,
\end{equation}
with $x_a, x_b$ fixed reals and the Lagrangian $L$ satisfying some smoothness
properties. Usually, we say that a function is "sufficiently smooth" for a
particular development if all required actions (integration, differentiation,
$\ldots$) are possible.

\begin{Definition}
\index{admissible trajectory}
A trajectory $x \in C^2 ([a,b];\bR)$ is said to be an admissible trajectory
if it satisfies all the constraints of the problem along the interval $[a, b]$.
The set of admissible trajectories is denoted by $D$.
\end{Definition}

To discuss maxima and minima of functionals, we need to introduce the following
definition.

\begin{Definition}
We say that $x^\star \in D$ is a local extremizer to the functional 
$\mathcal{J}:D \rightarrow \bR$  if there exists  some real $\epsilon >0$, such that
$$\forall x \in D: \quad \Vert x^\star - x \Vert < \epsilon \,\, 
\Rightarrow \,\,\mathcal{J}(x^\star)-\mathcal{J}(x)\leq 0 \, \vee \, \mathcal{J}(x^\star)
-\mathcal{J}(x)\geq 0.$$
\end{Definition}

In this context, as we are dealing with functionals defined on functions,
we need to clarify the term of directional derivatives, here called variations.
The concept of variation of a functional is central to obtain the solution of
variational problems.

\begin{Definition}
\index{first variation}
Let $\mathcal{J}$ be defined on $D$. The first variation of a functional
$\mathcal{J}$ at $x \in D$ in the direction $h \in D$ is defined by
$$
\delta \mathcal{J} (x,h)= \lim_{\e \rightarrow 0}
\frac{\mathcal{J}(x+\e h) - \mathcal{J}(x)}{\e}=
\left. \frac{d}{d\e}  \mathcal{J}(x+\e h) \right]_{\e=0},
$$
where $x$ and $h$ are functions and $\e$ is a scalar, whenever
the limit exists.
\end{Definition}

\begin{Definition}
\index{admissible variation}
A direction $h \in D$, $h\neq 0$, is said to be an admissible
variation for $\mathcal{J}$ at $y\in D$ if
\begin{enumerate}
\item $\delta \mathcal{J}(x,h)$ exists;
\item $x+\e h \in D$ for all sufficiently small $\e$.
\end{enumerate}
\end{Definition}

With the condition that $\mathcal{J}(x)$ be a local extremum and the
definition of variation, we have the following result that offers a
necessary optimality condition for problems of calculus of variations
\cite{Cap2:Brunt}.

\begin{Theorem}
Let $\mathcal{J}$ be a functional defined on $D$. If $x^\star$ minimizes
(or maximizes) the functional $\mathcal{J}$ over all functions
$x: [a,b] \rightarrow \bR$ satisfying boundary conditions \eqref{bondcond},
then
$$
\delta \mathcal{J}(x^\star,h)=0
$$
for all admissible variations $h$ at $x^\star$.
\end{Theorem}


\subsection{Euler--Lagrange equations}
\index{Euler--Lagrange equations}

Although the calculus of variations was born with Johann's problem, it was
with the work of Euler in 1742 and the one of Lagrange in 1755 
that a systematic theory was developed. 
The common procedure to address such variational problems
consists in solving a differential equation, called the Euler--Lagrange equation,
which every minimizer/maximizer of the functional must satisfy.

In Lemma \ref{flcv}, we review an important result to transform the necessary
condition of extremum in a differential equation, free of integration with an
arbitrary function. In literature, it is known as the fundamental lemma of the calculus of variations.

\begin{Lemma}
\label{flcv}
Let $x$ be continuous in $[a,b]$ an let $h$ be an arbitrary function on $[a,b]$
such that it is continuous and $h(a)=h(b)=0$. If
$$
\int_a^b x(t)h(t)dt=0
$$
for all such $h$, then $x(t)=0$ for all $t \in [a,b]$.
\end{Lemma}

For the sequel, we denote by  $\partial_i z$, $i\in \{1,2,\ldots, M\}$, with
$M\in \mathbb{N}$, the partial derivative of a function
$z:\bR^{M} \rightarrow \bR$ with respect to its $i$th argument.
Now we can formulate the necessary optimality condition for the classical
variational problem \cite{Cap2:Brunt}.

\begin{Theorem}
\label{thm:17:rp}
If $x$ is an extremizing of the functional \eqref{class_functJ} on $D$,
subject to \eqref{bondcond}, then $x$ satisfies
\begin{equation}
\label{eqELclass}
\partial_2 L\left(t,x(t), x'(t) \right)-\frac{d}{dt}
\partial_3 L\left(t,x(t), x'(t) \right)=0
\end{equation}
for all $t\in [a,b]$.
\end{Theorem}

To solve this second order differential equation, the two given boundary
conditions \eqref{bondcond} provide sufficient information to determine the
two arbitrary constants.

\begin{Definition}
\index{extremal}
A curve $x$ that is a solution of the Euler--Lagrange differential equation
will be called an extremal of $\mathcal{J}$.
\end{Definition}


\subsection{Problems with variable endpoints}

In the basic variational problem considered previously, the functional
$\mathcal{J}$ to minimize (or maximize) is subject to given boundary conditions
of the form
$$
x(a)=x_a, \quad \quad x(b)=x_b,
$$
where $x_a,x_b \in \bR$ are fixed. It means that the solution of the problem,
$x$, needs to pass through the prescribed points. This variational problem is
called a fixed endpoints variational problem.
The Euler--Lagrange equation \eqref{eqELclass} is normally a second-order
differential equation containing two arbitrary constants, so with two given
boundary conditions provided, they are sufficient to determine the two constants.

However, in some areas, like physics and geometry, the variational problems
do not impose the appropriate number of boundary conditions. In these cases,
when one or both boundary conditions are missing, that is, when the set of
admissible functions may take any value at one or both of the boundaries, 
then one or two auxiliary conditions, known as the natural boundary 
conditions or transversality 
conditions\index{boundary conditions}\index{transversality conditions}, 
need to be obtained in order to solve the equation \cite{Cap2:Brunt}:
\begin{equation}
\left[\frac{\partial L(t,x(t), x'(t))}{\partial x'} \right]_{t=a}=0 \quad
\mbox{and/or} \quad \left[\frac{\partial L(t,x(t), x'(t))}{\partial x'} \right]_{t=b}=0.
\end{equation}

There are different types of variational problems with variable endpoints:
\begin{itemize}
\item Free terminal point -- one boundary condition at the initial time ($x(a)=x_a$).
The terminal point is free ($x(b) \in \mathbb{R}$);
\item Free initial point -- one boundary condition at the final time ($x(b)=x_b$).
The initial point is free ($x(a) \in \mathbb{R}$);
\item Free endpoints -- both endpoints are free ($x(a)\in \mathbb{R}$, $x(b) \in \mathbb{R}$);
\item Variable endpoints -- the initial point $x(a)$ or/and the endpoint $x(b)$ 
is variable on a certain set, for example, on a prescribed curve.
\end{itemize}

Another generalization of the variational problem 
consists to find an optimal curve $x$ and the optimal 
final time $T$ of the variational integral, $T \in [a,b]$. 
This problem is known in the literature as a free-time problem \cite{Cap2:Chiang}.
An example is the following free-time problem with free terminal point.
Let $D$ denote the subset $C^2([a,b];\bR)\times [a,b]$ endowed with a norm
$\Vert(\cdot, \cdot)\Vert$.
Find the local minimizers of the functional $\mathcal{J}:D\rightarrow \bR$,
with
\begin{equation}
\label{Prob_endpointClass}
\mathcal{J}(x,T)=\int_a^T L(t,x(t), x'(t))dt,
\end{equation}
over all $(x,T) \in D$ satisfying the boundary condition $x(a)=x_a$, with
$x_a \in \bR$ fixed. The terminal time $T$ and the terminal state $x(T)$ 
are here both free.

\begin{Definition}
We say that $(x^\star,T^\star)\in D$ is a local extremizer (minimizer or maximizer)
to the functional $\mathcal{J}:D\rightarrow \bR$ as in \eqref{Prob_endpointClass} if there exists some $\e>0$
such that, for all $(x,T)\in D$,
$$\|(x^\star,T^\star)-(x,T)\|<\e \Rightarrow J(x^\star,T^\star)\leq J(x,T) \, \vee \, J(x^\star,T^\star)\geq J(x,T).$$
\end{Definition}

To develop a necessary optimality condition to problem \eqref{Prob_endpointClass}
for an extremizer $(x^\star,T^\star)$, we need to consider 
an admissible variation of the form:
$$
(x^\star + \e h, \quad  T^\star + \e \Delta T),
$$
where $h \in C^1 ([a,b];\bR)$ is a perturbing curve that satisfies the condition
$h(a)=0$, $\e$ represents a small real number, and $\Delta T$ represents an
arbitrarily chosen small change in $T$. Considering the functional
$\mathcal{J}(x,T)$ in this admissible variation, we get a function of $\e$,
where the upper limit of integration will also vary with $\e$:
\begin{equation}
\label{functioneps}
\mathcal{J} (x^\star + \e h, \quad  T^\star + \e \Delta T) =
\int_a^{T^\star + \e \Delta T} L(t,(x^\star+\e h)(t), (x^\star+\e h)'(t))dt.
\end{equation}

To find the first-order necessary optimality condition, we need to determine the derivative
of \eqref{functioneps} with respect to $\e$ and set it equal to zero. 
By doing it, we obtain three terms on the equation, where 
the Euler--Lagrange equation emerges from the first term,
and the other two terms, which depend only on the
terminal time $T$, give the transversality conditions.\index{transversality conditions}


\subsection{Constrained variational problems}

Variational problems are often subject to one or more constraints (holonomic
constraints, integral constraints, dynamic constraints, $\ldots$).
Isoperimetric problems are a special class of constrained variational problems
for which the admissible functions are needed to satisfy an integral constraint. \index{isoperimetric problem}

Here we review the classical isoperimetric variational problem.
The classical variational problem, already defined, may be modified by
demanding that the class of potential extremizing functions also satisfy a
new condition, called an isoperimetric constraint, of the form
\begin{equation}
\label{iso_cond}
\int_a^b g(t,x(t),x'(t))dt=C,
\end{equation}
where $g$ is a given function of $t$, $x$ and $x'$, and $C$ is a given
real number.

The new problem is called an isoperimetric problem and encompasses an important
family of variational problems. In this case, the variational problems
are often subject to one or more constraints involving an integral of a
given function \cite{Cap2:Fraser}.
Some classical examples of isoperimetric problems appear in geometry.
The most famous example consists in finding the curve of a given perimeter
that bounds the greatest area and the answer is the circle.
Isoperimetric problems are an important type of variational problems,
with applications in different areas, like geometry, astronomy, physics,
algebra or analysis.

In the next theorem, we present a necessary condition for a function to be an
extremizer to a classical isoperimetric problem, obtained via the concept of
Lagrange multiplier \cite{Cap2:Brunt}.

\begin{Theorem}
Consider the problem of minimizing (or maximizing) the functional
$\mathcal{J}$, defined by \eqref{class_functJ}, on $D$ given by those
$x\in C^2 \left([a,b];\bR\right)$ satisfying the boundary conditions
\eqref{bondcond} and an integral constraint of the form
$$
\mathcal{G}=\int_a^b g(t,x(t),x'(t))dt=C,
$$
where $g:[a,b]\times \bR^2 \rightarrow \bR$ is a twice
continuously differentiable function.
Suppose that $x$ gives a local minimum (or maximum) to this problem.
Assume that $\delta \mathcal{G}(x,h)$ does not vanish for all $h \in D$.
Then, there exists a constant $\lambda$ such that $x$ satisfies the
Euler--Lagrange equation
\begin{equation}
\partial_2 F\left(t,x(t), x'(t), \lambda \right)-\frac{d}{dt} \partial_3
F\left(t,x(t), x'(t),\lambda \right)=0,
\end{equation}
where $F\left(t,x, x',\lambda\right)=L(t,x,x')-\lambda g(t,x,x')$.
\end{Theorem}

\begin{Remark}
The constant $\lambda$ is called a Lagrange multiplier.
\index{Lagrange multiplier}
\end{Remark}

Observe that $\delta \mathcal{G}(x,h)$ does not vanish for all $h \in D$
if $x$ does not satisfies the Euler--Lagrange equation with respect to
the isoperimetric constraint, that is, $x$ is not an extremal for $\mathcal{G}$.


\section{Fractional calculus of variations}
\label{sec:fcv}
The first connection between fractional calculus and the calculus of
variations appeared in the XIX century, with Niels Abel \cite{Cap2:Abel}. In
1823, Abel applied fractional calculus in the solution of an integral
equation involved in a generalization of the tautochrone problem\index{tautochrone problem}. 
Only in the XX century, however, both areas were joined in an unique
research field: the fractional calculus of variations.

The fractional calculus of variations deals with problems in which the
functional, the constraint conditions, or both, depend on some fractional
operator \cite{Cap2:book:FCV,Cap2:book2:FCV,Cap2:book:MOT} and the main goal is to find
functions that extremize such a fractional functional.
By inserting fractional operators that are non-local in variational problems,
they are suitable for developing some models possessing memory effects.

This is a fast growing subject, and different approaches have been developed
by considering different types of Lagrangians, e.g., depending on
Riemann--Liouville or Caputo fractional derivatives, fractional integrals,
and mixed integer-fractional order operators (see, for example,  \cite{Cap2:Alm:Torres2009,Cap2:Askari,Cap2:Atanackovic,Cap2:Baleanu1,Cap2:Baleanu2,Cap2:Cresson,Cap2:Rabei,Cap2:Tarasov}).
In recent years, there has been a growing interest in the area of fractional
variational calculus and its applications, which include classical and quantum
mechanics, field theory and optimal control.

Although the origin of fractional calculus goes back more than three centuries,
the calculus of variations with fractional derivatives has born 
only in 1996-1997 with the works of F. Riewe. In his works, Riewe obtained a version of the
Euler--Lagrange equations for problems of the calculus of variations with
fractional derivatives, when investigating non-conservative Lagrangian and
Hamiltonian mechanics \cite{Cap2:Riewe96,Cap2:Riewe97}.
Agrawal continued the study of the fractional Euler--Lagrange equations 
\cite{Cap2:Agrawal2002,Cap2:Agrawal2006,Cap2:Agrawal2007}, for some kinds of fractional
variational problems, for example, problems with only one dependent variable,
for functionals with different orders of fractional derivatives, for several
functions, involving both Riemann--Liouville and Caputo derivatives, etc.
The most common fractional operators considered in the literature take into
account the past of the process, that is, one usually uses 
left fractional operators.
But, in some cases, we may be also interested in the future of the process,
and the computation of $\alpha(\cdot)$ to be influenced by it. In that case,
right fractional derivatives are then considered.


\subsection{Fractional Euler--Lagrange equations}

Similarly to the classical variational calculus, the common procedure to
address such fractional variational problems consists in solving a fractional
differential equation, called the fractional Euler--Lagrange equation,
which every minimizer/maximizer of the functional must satisfy.
With the help of the boundary conditions imposed on the problem
at the initial time $t=a$ and at the terminal time $t=b$,
one solves, often with the help of some numerical procedure, the fractional
differential equation and obtain a possible solution to the problem 
\cite{Cap2:Pooseh2013b,Cap2:Lotfi,Cap2:Blasz:Cie,Cap2:Xu,Cap2:Sumelka}.

Referring again to Riewe's works \cite{Cap2:Riewe96,Cap2:Riewe97}, friction forces are
described with Lagrangians that contain fractional derivatives. Precisely, for
$r, N$ and $N'$ natural numbers and assuming $x:[a,b]\rightarrow \bR^r$, $\alpha_i,
\beta_j \in [0,1]$ with $i=1,\ldots,N, j=1,\ldots,N'$, the functional defined by Riewe is
\begin{equation}
\label{functional_Riewe}
\mathcal{J}(x)=
\int_a^b L \left({_aD_t^{\alpha_1}[x](t)},\ldots,{_aD_t^{\alpha_N}}[x](t), {_tD_b^{\beta_1}[x](t)},\ldots,{_tD_b^{\beta_{N'}}[x](t)},x(t),t\right)dt.
\end{equation}
He proved that any solution $x$ of the variational problem of extremizing the functional \eqref{functional_Riewe}, satisfies the following necessary condition:
\index{variational fractional problem}
\begin{Theorem}
\index{Euler--Lagrange equations}
Considering the variational problem of minimizing (or maximizing) the
functional \eqref{functional_Riewe}, the fractional Euler--Lagrange
equation is
$$
\sum_{i=1}^N {_tD_b^{\alpha_i}}[\partial_iL]+\sum_{i=1}^{N'}
{_aD_t^{\beta_i}}[\partial_{i+N}L] +\partial_{N'+N+1}L=0.
$$
\end{Theorem}

Riewe also illustrated his results considering the classical problem of linear
friction \cite{Cap2:Riewe96}.

In what follows, we are concerned with problems of the fractional calculus of
variations where the functional depends on a combined fractional Caputo
derivative with constant orders $\alpha$ and $\beta$
\cite[Definition~\ref{CombClassical}]{Cap2:Malin:Tor2011}.

Let $D$ denote the set of all functions $x:[a,b] \rightarrow \bR^N$, endowed
with a norm $\Vert \cdot \Vert$ in $\bR^N$.
Consider the following problem: find a function $x \in D$ for which the
functional
\begin{equation}
\label{funct_comb_CO}
\mathcal{J}(x)=\int_a^b L(t, x(t),{^CD_\gamma^{\alpha,\beta}x(t)})dt
\end{equation}
subject to given boundary conditions
$$
x(a)=x_a, \quad \quad x(b)=x_b
$$
archives a minimum, where $t \in [a,b]$, $x_a,x_b \in \bR^N$, $\gamma \in [0,1]$
and the Lagrangian $L$ satisfies some smoothness properties.
\index{variational fractional problem}

\begin{Theorem}
Let $x = (x_1,\ldots, x_N)$ be a local minimizer to the problem with
the functional \eqref{funct_comb_CO} subject to two boundary conditions, as
defined before. Then, $x$ satisfies the system of $N$ fractional Euler--Lagrange
equations
\begin{equation}
\partial_i L(t, x(t),{^CD_\gamma^{\alpha,\beta}}x(t)) + D_{1-\gamma}^{\beta,\alpha} \partial_{N+i} L(t, x(t),{^CD_\gamma^{\alpha,\beta}}x(t)) = 0,
\end{equation}
$i=2,\ldots,N+1$, for all $t \in [a,b]$.
\end{Theorem}

For a proof of the last result, see \cite{Cap2:Malin:Tor2010}.
Observe that, if the orders $\alpha$ and $\beta$ go to $1$, and if $\gamma=0$
or $\gamma=1$, we obtain a corresponding result in the classical context of
the calculus of variations. In fact, considering $\alpha$ and $\beta$ going to
$1$, the fractional derivatives $_a^CD_t^\alpha$ and $_aD_t^\alpha$ coincide
with the classical derivative $\frac{d}{dt}$; and similarly, $_t^CD_b^\beta$
and $_tD_b^\beta$ coincide with the classical derivative $-\frac{d}{dt}$.

Variational problems with free endpoints and transversality conditions\index{transversality conditions}
are also relevant subjects in the fractional
calculus of variations. The subject of free boundary points in fractional
variational problems was first considered by Agrawal in \cite{Cap2:Agrawal2006}.
In that work, he studied the Euler--Lagrange equation and transversality
conditions for the case when both initial and final times are given and the 
admissible functions are specified at the initial time but are unspecified at the final time.
After that, some free-time variational problems involving fractional
derivatives or/and fractional integrals were studied \cite{Cap2:Alm:Malin,Cap2:Od2012b}.

Sometimes, the analytic solution of the fractional Euler--Lagrange equation
is very difficult to obtain and, in this case, some numerical methods have
been developed to solve the variational problem \cite{Cap2:book2:FCV}.


\subsection{Fractional variational problems of variable-order}

In recent years, motivated by the works of Samko and Ross, where they
investigated integrals and derivatives not of a constant but of variable-order
\cite{Cap2:Samko:1995,Cap2:Samko:Ross},  some problems of the calculus of variations
involving derivatives of variable fractional order have appeared 
\cite{Cap2:Od2012c,Cap2:Atana}.

Considering Definition \ref{gfi-RL} of the generalized fractional integral
of operator $K_P$, \cite{Cap2:book:MOT} presented a new
variational problem, where the functional was defined by a given kernel.
For appropriate choices of the kernel $k$ and the set $P$, we can obtain
a variable-order fractional variational problem
(see Remark~\ref{example_kernel}).

Let $P = \langle a,t,b,\lambda,\mu \rangle$. Consider the functional
$\mathcal{J}$ in $\mathbf{A}(x_a,x_b)$ defined by:

\begin{equation}
\label{funct_kernel}
\mathcal{J}[x]=\int_a^b L \left(x(t), K_P[x](t), x'(t), B_P[x](t),t \right)dt,
\end{equation}
where $\mathbf{A}(x_a,x_b)$ is the set
$$\lbrace x\in C^1([a,b];\bR):\, x(a)=x_a, x(b)=x_b,\, K_P[x], B_P[x] \in C([a,b];\bR)\rbrace,$$
and $K_P$ is the generalized fractional integral operator with kernel
belonging to $L_q(\Delta; \bR)$ and $B_P$ the generalized fractional
derivative of Caputo type.

The optimality condition for the problem that consists to determine a
function that minimize (or maximize) the functional \eqref{funct_kernel}
is given in the following theorem \cite{Cap2:book:MOT}.\index{variational fractional problem}

\begin{Theorem}
\label{Theor_ELKernel}
\index{Euler--Lagrange equations}
Let $x \in \mathbf{A}(x_a,x_b)$ be a minimizer of functional \eqref{funct_kernel}.
Then, $x$ satisfies the following Euler--Lagrange equation:
\begin{equation}
\begin{split}
\frac{d}{dt}&\left[\partial_3L(\star_x)(t) \right] + A_{P^*}\left[\t\longmapsto\partial_4L(\star_x)(\t) \right](t) \\
&= \partial_1 L(\star_x)(t) + K_{P^*}\left[\t\longmapsto\partial_2L(\star_x)(\t)
\right](t),
\end{split}
\end{equation}
where $(\star_x)(t)=\left(x(t), K_P[x](t), x'(t), B_P[x](t), t\right)$,
for $t\in(a,b)$.
\end{Theorem}

Observe that, if functional \eqref{funct_kernel} does not depends on the
generalized fractional operators $K_P$ and $B_P$, this problem coincide with
the classical variational problem and Theorem~\ref{Theor_ELKernel} reduces to
Theorem~\ref{thm:17:rp}.

Let $\vartriangle:=\lbrace(t,\t) \in \mathbb{R}^2: a\leq\t < t\leq b \rbrace$ and let $1<p<\infty$ and $q$ be the adjoint of $p$. A special case of this problem is obtained when we consider
$\alpha:\Delta \rightarrow [0,1-\delta]$ with $\delta > 1/p$ and the
kernel is defined by
$$k^\alpha (t,\t)=\frac{1}{\Gamma(1-\alpha(t,\t))}(t-\t)^{-\alpha(t,\t)}$$
in $L_q(\Delta;\bR)$. The next results provides necessary conditions 
of optimality \cite{Cap2:book:MOT}.

\begin{Theorem}
Consider the problem of minimizing a functional
\begin{equation}
\label{fract_MOT}
\mathcal{J}[x]=\int_a^b L \left(x(t), {_aI_t^{1-\a}}[x](t), x'(t), {_a^CD_t^{\a}}[x](t),t \right)dt
\end{equation}
subject to boundary conditions
\begin{equation}
\label{bc_MOT}
x(a)=x_a, \quad x(b)=x_b,
\end{equation}
where $x', {_aI_t^{1-\a}[x]}, {_a^CD_t^{\a}[x]} \in C([a,b];\bR)$.
Then, if $x \in C^1([a,b];\bR)$ minimizes (or maximizes) the functional
\eqref{fract_MOT} subject to \eqref{bc_MOT}, then it satisfies the
following Euler--Lagrange equation:
\begin{equation*}
\begin{split}
\partial_1 L & \left(x(t),{_aI_t^{1-\a}[x](t)},x'(t),{_a^CD_t^{\a}[x](t),t}\right)\\
&-\frac{d}{dt}\partial_3 L \left(x(t),{_aI_t^{1-\a}[x]}(t),x'(t),{_a^CD_t^{\a}}[x](t),t\right)\\
&+{_tI_b^{1-\a}}\left[ \partial_2 L \left(x(\t),{_aI_\t^{1-\a}}[x](\t),x'(\t),{_a^CD_\t^{\a}}[x](\t),\t \right) \right](t)\\
&+{_tD_b^{\a}}\left[ \partial_4 L \left(x(\t),{_aI_\t^{1-\a}}[x](\t),x'(\t),{_a^CD_\t^{\a}}[x](\t),\t\right) \right](t)=0.
\end{split}
\end{equation*}
\end{Theorem}

In fact, the use of fractional derivatives of constant order in variational problems 
may not be the best option, since trajectories are a dynamic process, and the order may vary. 
Therefore, it is important to consider the order to be a function, $\alpha(\cdot)$, 
depending on time. Then we may seek what is the best function $\alpha(\cdot)$ such that 
the variable-order fractional differential equation 
$D^{\alpha(\cdot)}x(t) = f(t, x(t))$ better describes the process under study.

This approach is very recent, and many work has to be done for a complete study of the subject (see, e.g., \cite{Cap2:Atangana:Kili,Cap2:Coimbra2,Cap2:Samko:Ross,Cap2:Sheng,Cap2:Valerio:Vinagre}).

\begingroup
\renewcommand{\addcontentsline}[3]{}

\endgroup

%% file: 3NumericalApprox.tex
\chapter{Expansion formulas for fractional derivatives}
\label{PartII_Expan}

In this chapter, we present a new numerical tool to solve differential
equations involving three types of Caputo  derivatives of fractional
variable-order. For each one of them, an approximation formula is 
obtained, which is written in terms of standard (integer order) derivatives only.
Estimations for the error of the approximations are also provided. Then,
we compare the numerical approximation of some test function with its exact
fractional derivative. We present an exemplification of how the presented
methods can be used to solve partial fractional differential equations of
variable-order.

Let us briefly describe the main contents of the chapter.
We begin this chapter by formulating the needed definitions
(Section~\ref{sec:defs}). Namely, we present three types of Caputo derivatives
of variable fractional order. First, we consider one independent variable only;
then we generalize for several independent variables.
The following Section~\ref{sec:NumApprox} is the main core of the chapter,
where we prove approximation formulas for the given fractional operators of
variable-order and respectively upper bound formulas for the errors.
To test the efficiency of the proposed method, in Section~\ref{sec:ex_numer}
we compare the exact fractional derivative of some test function with the
numerical approximations obtained from the decomposition formulas given in Section~\ref{sec:NumApprox}.
To end, in Section~\ref{sec:2_Appl} we apply our method to approximate two
physical problems involving Caputo fractional operators of variable-order
(a time-fractional diffusion equation and a fractional Burgers' partial
differential equation in fluid mechanics) by classical problems that may
be solved by well-known standard techniques.

The results of this chapter first appeared in \cite{Cap3:Tavares2016}.


\section{Caputo-type fractional operators of variable-order}
\label{sec:defs}

In the literature of fractional calculus, several different
definitions  of derivatives are found \cite{Cap3:Samko:Kilbas}.
One of those, introduced by  \cite{Cap3:Caputo} and studied independently by other authors,
like \cite{Cap3:Dzherbashyan} and \cite{Cap3:Rabotnov},
has found many applications and seems to be more suitable to model physical phenomena 
\cite{Cap3:Dalir,Cap3:Diethelm,Cap3:Machado,Cap3:Murio,Cap3:Singh,Cap3:Sweilam,Cap3:Yajima}.


\subsection{Caputo derivatives for functions of one variable}
\label{subsec:one_variable}

Our goal is to consider fractional derivatives of variable-order, with $\alpha$ depending on time.
In fact, some phenomena in physics are better described when the order of the fractional operator is not constant,
for example, in the diffusion process in an inhomogeneous or heterogeneous medium, or processes
where the changes in the environment modify the dynamic of the particle \cite{Cap3:Chechkin,Cap3:Santamaria,Cap3:Sun}.
Motivated by the above considerations, we introduce three types of Caputo fractional derivatives.
The order of the derivative is considered as a function $\ati$ taking values on the open interval $(0,1)$.
To start, we define two different kinds of Riemann--Liouville fractional derivatives.

\begin{Definition}
Given a function  $x:[a,b]\to\mathbb{R}$,
\begin{enumerate}
\item the type I left Riemann--Liouville fractional derivative of order $\ati$ is defined by
$$
{_aD_t^\ati}x(t)=\frac{1}{\Gamma(1-\ati)}\frac{d}{dt}\int_a^t(t-\t)^{-\ati}x(\t)d\t;
$$
\item the type I right Riemann--Liouville fractional derivative of order $\ati$ is defined by
$$
{_tD_b^\ati}x(t)=\frac{-1}{\Gamma(1-\ati)}\frac{d}{dt}\int_t^b(\t-t)^{-\ati}x(\t)d\t;
$$
\item the type II left Riemann--Liouville fractional derivative of order $\ati$ is defined by
$$
{_a\mathcal{D}_t^\ati}x(t)=\frac{d}{dt}\left(\frac{1}{\Gamma(1-\ati)}\int_a^t(t-\t)^{-\ati}x(\t)d\t\right);
$$
\item the type II right Riemann--Liouville fractional derivative of order $\ati$ is defined by
$$
{_t\mathcal{D}_b^\ati}x(t)=\frac{d}{dt}\left(\frac{-1}{\Gamma(1-\ati)}\int_t^b(\t-t)^{-\ati}x(\t)d\t\right).
$$
\end{enumerate}
\end{Definition}

The Caputo derivatives are given using the previous Riemann--Liouville fractional deri\-vatives.

\begin{Definition}
Given a function  $x:[a,b]\to\mathbb{R}$,
\begin{enumerate}
\item the type I left Caputo derivative of order $\ati$ is defined by
\begin{align*}\LCI x(t)&={_aD_t^\ati}(x(t)-x(a))\\
&=\frac{1}{\Gamma(1-\ati)}\frac{d}{dt}\int_a^t(t-\t)^{-\ati}[x(\t)-x(a)]d\t;\end{align*}
\item the type I right Caputo derivative of order $\ati$ is defined by
\begin{align*}
\RCI x(t)&={_tD_b^\ati}(x(t)-x(b))\\
&=\frac{-1}{\Gamma(1-\ati)}\frac{d}{dt}\int_t^b(\t-t)^{-\ati}[x(\t)-x(b)]d\t;
\end{align*}
\item the type II left Caputo derivative of order $\ati$ is defined by
\begin{align*}
\LCII x(t)&= {_a\mathcal{D}_t^\ati} (x(t)-x(a))\\
&=\frac{d}{dt}\left(\frac{1}{\Gamma(1-\ati)}\int_a^t(t-\t)^{-\ati}[x(\t)-x(a)]d\t\right);
\end{align*}
\item the type II right Caputo derivative of order $\ati$ is defined by
\begin{align*}
\RCII x(t)&= {_t\mathcal{D}_b^\ati}(x(t)-x(b))\\
&=\frac{d}{dt}\left(\frac{-1}{\Gamma(1-\ati)}\int_t^b(\t-t)^{-\ati}[x(\t)-x(b)]d\t\right);
\end{align*}
\item the type III left Caputo derivative of order $\ati$ is defined by
$$
\LCIII x(t)=\frac{1}{\Gamma(1-\ati)}\int_a^t(t-\t)^{-\ati}x'(\t)d\t;
$$
\item the type III right Caputo derivative of order $\ati$ is defined by
$$
\RCIII x(t)=\frac{-1}{\Gamma(1-\ati)}\int_t^b(\t-t)^{-\ati}x'(\t)d\t.
$$
\end{enumerate}
\end{Definition}

In contrast with the case when $\alpha$ is a constant, definitions of different types do not coincide.

\begin{Theorem}
The following relations between the left fractional operators hold:
\begin{multline}
\label{eq1}
\LCI x(t)=\LCIII x(t)+\frac{\alpha'(t)}{\Gamma(2-\ati)}\\
\times \int_a^t(t-\t)^{1-\ati}x'(\t)\left[\frac{1}{1-\ati}-\ln(t-\t)\right]d\t\end{multline}
and
\begin{multline}
\label{eq2}
\LCI x(t)=\LCII x(t)-\frac{\alpha'(t)\Psi(1-\ati)}{\Gamma(1-\ati)}\\
\times \int_a^t(t-\t)^{-\ati}[x(\t)-x(a)]d\t.\end{multline}
\end{Theorem}

\begin{proof}
Integrating by parts, one gets
\begin{equation*}
\begin{split}
\LCI x(t)&= \DS\frac{1}{\Gamma(1-\ati)}
\frac{d}{dt}\int_a^t(t-\t)^{-\ati}[x(\t)-x(a)]d\t\\
&= \DS\frac{1}{\Gamma(1-\ati)}\frac{d}{dt}
\left[\frac{1}{1-\ati}\int_a^t(t-\t)^{1-\ati}x'(\t)d\t\right].
\end{split}
\end{equation*}
Differentiating the integral, it follows that
\begin{equation*}
\begin{split}
&\LCI x(t)\DS=\frac{1}{\Gamma(1-\ati)}\left[
\frac{\alpha'(t)}{(1-\ati)^2}\int_a^t(t-\t)^{1-\ati}x'(\t)d\t\right.\\
&\DS\quad\left.+\frac{1}{1-\ati}\int_a^t(t-\t)^{1-\ati}x'(\t)\left[
-\alpha'(t)\ln(t-\t)+\frac{1-\ati}{t-\t}\right]d\t\right]\\
&=\DS\LCIII x(t)+\frac{\alpha'(t)}{\Gamma(2-\ati)}\int_a^t
(t-\t)^{1-\ati}x'(\t)\left[\frac{1}{1-\ati}-\ln(t-\t)\right]d\t.
\end{split}
\end{equation*}
The second formula follows from direct calculations.
\end{proof}

Therefore, when the order $\ati \equiv c$ is a constant, or for constant functions $x(t) \equiv k$, we have
$$
\LCI x(t)=\LCII x(t)=\LCIII x(t).
$$
Similarly, we obtain the next result.

\begin{Theorem}
The following relations between the right fractional operators hold:
\begin{multline*}
\RCI x(t) \DS=\DS\RCIII x(t)+\frac{\alpha'(t)}{\Gamma(2-\ati)}\\
\times\int_t^b(\t-t)^{1-\ati}x'(\t)\left[\frac{1}{1-\ati}-\ln(\t-t)\right]d\t\end{multline*}
and
$$
\RCI x(t)=\RCII x(t)+\frac{\alpha'(t)\Psi(1-\ati)}{\Gamma(1-\ati)}\int_t^b(\t-t)^{-\ati}[x(\t)-x(b)]d\t.
$$
\end{Theorem}

\begin{Theorem}
\label{initialpoint}
Let $x\in C^1\left([a,b],\mathbb{R}\right)$. At $t=a$
$$
\LCI x(t)=\LCII x(t)=\LCIII x(t)=0;
$$
at $t=b$
$$
\RCI x(t)=\RCII x(t)=\RCIII x(t)=0.
$$
\end{Theorem}

\begin{proof}
We start proving the third equality at the initial time $t=a$. We simply note that
$$
\left|\LCIII x(t)\right|\leq \frac{\|x'\|}{\Gamma(1-\ati)}\int_a^t(t-\t)^{-\ati}d\t
=\frac{\|x'\|}{\Gamma(2-\ati)}(t-a)^{1-\ati},
$$
which is zero at $t=a$. For the first equality at $t=a$,
using equation \eqref{eq1}, and the two next relations
$$
\left|\int_a^t(t-\t)^{1-\ati}\frac{x'(\t)}{1-\ati}d\t\right|
\leq \frac{\|x'\|}{(1-\ati)(2-\ati)}(t-a)^{2-\ati}
$$
and
\begin{align*}
\Bigg|\int_a^t(t-\t)^{1-\ati}  x'(\t)\ln(t-\t)&d\t\Bigg|\\
\leq &\frac{\|x'\|}{2-\ati}(t-a)^{2-\ati}\Bigg|\ln(t-a)-\frac{1}{2-\ati}\Bigg|,\end{align*}
this latter inequality obtained from integration by parts,
we prove that $\LCI x(t)=0$ at $t=a$. Finally, we prove the second equality at $t=a$
by considering equation \eqref{eq2}: performing an integration by parts, we get
$$
\left|\int_a^t(t-\t)^{-\ati}[x(\t)-x(a)]d\t\right|\leq \frac{\|x'\|}{(1-\ati)(2-\ati)}(t-a)^{2-\ati}
$$
and so $\LCII x(t)=0$ at $t=a$. The proof that the right
fractional operators also vanish at the end point $t=b$
follows by similar arguments.
\end{proof}

With some computations, a relationship between the Riemann--Liouville
and the Caputo fractional derivatives is easily deduced:
\begin{equation*}
\begin{split}
{_aD_t^\ati}x(t)&=\displaystyle\LCI x(t)
+\frac{x(a)}{\Gamma(1-\ati)}\frac{d}{dt}\int_a^t(t-\t)^{-\ati}d\t\\
&=\displaystyle\LCI x(t)+\frac{x(a)}{\Gamma(1-\ati)}(t-a)^{-\ati}\\
&\qquad\displaystyle+\frac{x(a)\alpha'(t)}{\Gamma(2-\ati)}
(t-a)^{1-\ati}\left[\frac{1}{1-\ati}-\ln(t-a)\right]
\end{split}
\end{equation*}
and
\begin{equation*}
\begin{split}
{_a\mathcal{D}_t^\ati}x(t)&=\displaystyle\LCII x(t)
+x(a)\frac{d}{dt}\left(\frac{1}{\Gamma(1-\ati)}\int_a^t(t-\t)^{-\ati}d\t\right)\\
&=\displaystyle\LCII x(t)+\frac{x(a)}{\Gamma(1-\ati)}(t-a)^{-\ati}\\
&\qquad\displaystyle+\frac{x(a)\alpha'(t)}{\Gamma(2-\ati)}
(t-a)^{1-\ati}\left[\Psi(2-\ati)-\ln(t-a)\right].
\end{split}
\end{equation*}
For the right fractional operators, we have
\begin{multline*}
{_tD_b^\ati}x(t)=\displaystyle\RCI x(t)+\frac{x(b)}{\Gamma(1-\ati)}(b-t)^{-\ati}\\
\displaystyle-\frac{x(b)\alpha'(t)}{\Gamma(2-\ati)}(b-t)^{1-\ati}\left[\frac{1}{1-\ati}-\ln(b-t)\right]
\end{multline*}
and
\begin{multline*}
{_t\mathcal{D}_b^\ati}x(t)=\displaystyle\RCII x(t)+\frac{x(b)}{\Gamma(1-\ati)}(b-t)^{-\ati}\\
\displaystyle-\frac{x(b)\alpha'(t)}{\Gamma(2-\ati)}(b-t)^{1-\ati}\left[\Psi(2-\ati)-\ln(b-t)\right].
\end{multline*}
Thus, it is immediate to conclude that if $x(a)=0$, then
$${_aD_t^\ati}x(t)=\LCI x(t) \quad \mbox{and} \quad {_a\mathcal{D}_t^\ati}x(t)=\LCII x(t)$$
and if $x(b)=0$, then
$${_tD_b^\ati}x(t)=\RCI x(t) \quad \mbox{and} \quad {_t\mathcal{D}_b^\ati}x(t)=\RCII x(t).$$

Next we obtain formulas for the Caputo fractional derivatives of a power function.

\begin{Lemma}
\label{2_1_Lemma_power}
Let $x(t)=(t-a)^\gamma$ with $\gamma>0$. Then,
\begin{equation*}
\begin{split}
\LCI x(t) &=\DS \frac{\Gamma(\gamma+1)}{\Gamma(\gamma-\ati+1)}(t-a)^{\gamma-\ati}\\
&\DS\quad -\alpha'(t)\frac{\Gamma(\gamma+1)}{\Gamma(\gamma-\ati+2)}(t-a)^{\gamma-\ati+1}\\
&\DS \quad \times \left[\ln(t-a)-\Psi(\gamma-\ati+2)+\Psi(1-\ati)\right],\\
\LCII x(t) &=\DS \frac{\Gamma(\gamma+1)}{\Gamma(\gamma-\ati+1)}(t-a)^{\gamma-\ati}\\
&\DS\quad -\alpha'(t)\frac{\Gamma(\gamma+1)}{\Gamma(\gamma-\ati+2)}(t-a)^{\gamma-\ati+1}\\
&\DS \quad \times \left[\ln(t-a)-\Psi(\gamma-\ati+2)\right],\\
\LCIII x(t) &=\DS  \frac{\Gamma(\gamma+1)}{\Gamma(\gamma-\ati+1)}(t-a)^{\gamma-\ati}.
\end{split}
\end{equation*}
\end{Lemma}

\begin{proof}
The formula for $\LCI x(t)$ follows immediately
from \cite{Cap3:Samko:Ross}. For the second equality, one has
\begin{equation*}
\begin{split}
\LCII x(t) &=\DS\frac{d}{dt}\left(\frac{1}{\Gamma(1-\ati)}
\int_a^t(t-\t)^{-\ati}(\t-a)^{\gamma} d\t\right)\\
 &=\DS\frac{d}{dt}\left(\frac{1}{\Gamma(1-\ati)}\int_a^t
 (t-a)^{-\ati}\left(1-\frac{\t-a}{t-a}\right)^{-\ati}(\t-a)^{\gamma} d\t\right).
\end{split}
\end{equation*}
With the change of variables $\t-a=s(t-a)$,
and with the help of the Beta function $B(\cdot,\cdot)$ (see Definition \ref{Betaf}), we prove that \index{beta function}
\begin{equation*}
\begin{split}
\LCII x(t) &=\DS\frac{d}{dt}\left(\frac{(t-a)^{-\ati}}{\Gamma(1-\ati)}
\int_0^1(1-s)^{-\ati}s^{\gamma}(t-a)^{\gamma+1} ds\right)\\
&=\DS\frac{d}{dt}\left(\frac{(t-a)^{\gamma-\ati+1}}{\Gamma(1-\ati)}B(\gamma+1,1-\ati)\right)\\
 &= \DS\frac{d}{dt}\left(\frac{\Gamma(\gamma+1)}{\Gamma(\gamma-\ati+2)}(t-a)^{\gamma-\ati+1}\right).
\end{split}
\end{equation*}
We obtain the desired formula by differentiating this latter expression.
The last equality follows in a similar way.
\end{proof}

Analogous relations to those of Lemma~\ref{2_1_Lemma_power},
for the right Caputo fractional derivatives of variable-order,
are easily obtained.

\begin{Lemma}
\label{2_1_Lemma_power_r}
Let $x(t)=(b-t)^\gamma$ with $\gamma>0$. Then,
\begin{equation*}
\begin{split}
\RCI x(t) &=\DS \frac{\Gamma(\gamma+1)}{\Gamma(\gamma-\ati+1)}(b-t)^{\gamma-\ati}\\
&\DS\quad +\alpha'(t)\frac{\Gamma(\gamma+1)}{\Gamma(\gamma-\ati+2)}(b-t)^{\gamma-\ati+1}\\
&\DS \quad \times \left[\ln(b-t)-\Psi(\gamma-\ati+2)+\Psi(1-\ati)\right],\\
\RCII x(t) &=\DS \frac{\Gamma(\gamma+1)}{\Gamma(\gamma-\ati+1)}(b-t)^{\gamma-\ati}\\
&\DS\quad +\alpha'(t)\frac{\Gamma(\gamma+1)}{\Gamma(\gamma-\ati+2)}(b-t)^{\gamma-\ati+1}\\
&\DS \quad \times \left[\ln(b-t)-\Psi(\gamma-\ati+2)\right],\\
\RCIII x(t) &=\DS  \frac{\Gamma(\gamma+1)}{\Gamma(\gamma-\ati+1)}(b-t)^{\gamma-\ati}.
\end{split}
\end{equation*}
\end{Lemma}

With Lemma~\ref{2_1_Lemma_power} in mind, we immediately see that
$$\LCI x(t)\not=\LCII x(t)\not=\LCIII x(t).$$
Also, at least for the power function, it suggests that $\LCIII x(t)$
may be a more suitable inverse operation
of the fractional integral when the order is variable.
For example, consider functions $x(t)=t^2$ and $y(t)=(1-t)^2$,
and the fractional order $\ati=\frac{5t+1}{10}$, $t\in[0,1]$.
Then, $0.1\leq \ati \leq 0.6$ for all $t$. Next we compare
the fractional derivatives of $x$ and $y$ of order $\ati$
with the fractional derivatives of constant order $\alpha=0.1$ and $\alpha=0.6$.
By Lemma~\ref{2_1_Lemma_power}, we know that the left Caputo fractional
derivatives of order $\ati$ of $x$ are given by
\begin{equation*}
\begin{split}
{^C_0D_t^{\ati}} x(t)
&=\DS \frac{2}{\Gamma(3-\ati)}t^{2-\ati}\\
&\quad -\frac{t^{3-\ati}}{\Gamma(4-\ati)}\left[\ln(t)-\Psi(4-\ati)+\Psi(1-\ati)\right],\\
{^C_0\mathcal{D}_t^{\ati}} x(t) &=\DS \frac{2}{\Gamma(3-\ati)}t^{2-\ati}
-\frac{t^{3-\ati}}{\Gamma(4-\ati)}\left[\ln(t)-\Psi(4-\ati)\right],\\
{^C_0\mathbb{D}_t^{\ati}} x(t) &=\DS  \frac{2}{\Gamma(3-\ati)}t^{2-\ati},
\end{split}
\end{equation*}
while by Lemma~\ref{2_1_Lemma_power_r}, the right Caputo
fractional derivatives of order $\ati$ of $y$  are given by
\begin{equation*}
\begin{split}
{^C_tD_1^{\ati}} y(t) &=\DS \frac{2(1-t)^{2-\ati}}{\Gamma(3-\ati)}\\
&\quad+\frac{(1-t)^{3-\ati}}{\Gamma(4-\ati)}\left[\ln(1-t)-\Psi(4-\ati)+\Psi(1-\ati)\right],\\
{^C_t\mathcal{D}_1^{\ati}} y(t) &=\DS \frac{2(1-t)^{2-\ati}}{\Gamma(3-\ati)}
+\frac{(1-t)^{3-\ati}}{\Gamma(4-\ati)}\left[\ln(1-t)-\Psi(4-\ati)\right],\\
{^C_t\mathbb{D}_1^{\ati}} y(t) &=\DS  \frac{2(1-t)^{2-\ati}}{\Gamma(3-\ati)}.
\end{split}
\end{equation*}
For a constant order $\alpha$, we have
$$
{^C_0D_t^{\alpha}} x(t) = \frac{2}{\Gamma(3-\alpha)}t^{2-\alpha}
\quad \mbox{and} \quad {^C_tD_1^{\alpha}} y(t)
=\frac{2}{\Gamma(3-\alpha)}(1-t)^{2-\alpha}.
$$
The results can be seen in Figure~\ref{comparison}.

\begin{figure}[p]
\begin{center}
\subfigure[${^C_0{D}_t^{\alpha(t)}} x(t)$]{\includegraphics[scale=0.25]{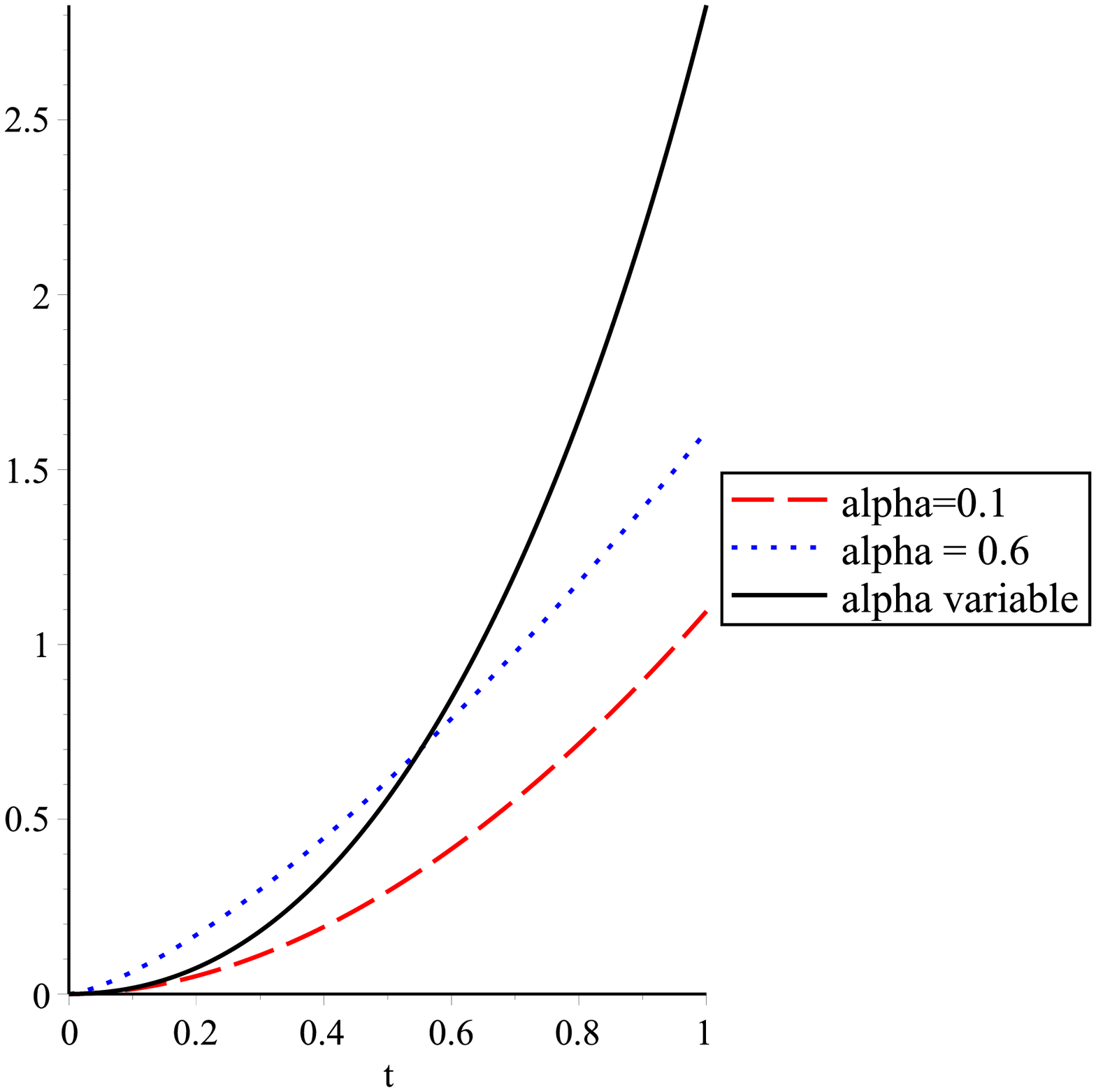}} \hspace{1cm}
\subfigure[${^C_0\mathcal{D}_t^{\alpha(t)}} x(t)$]{\includegraphics[scale=0.25]{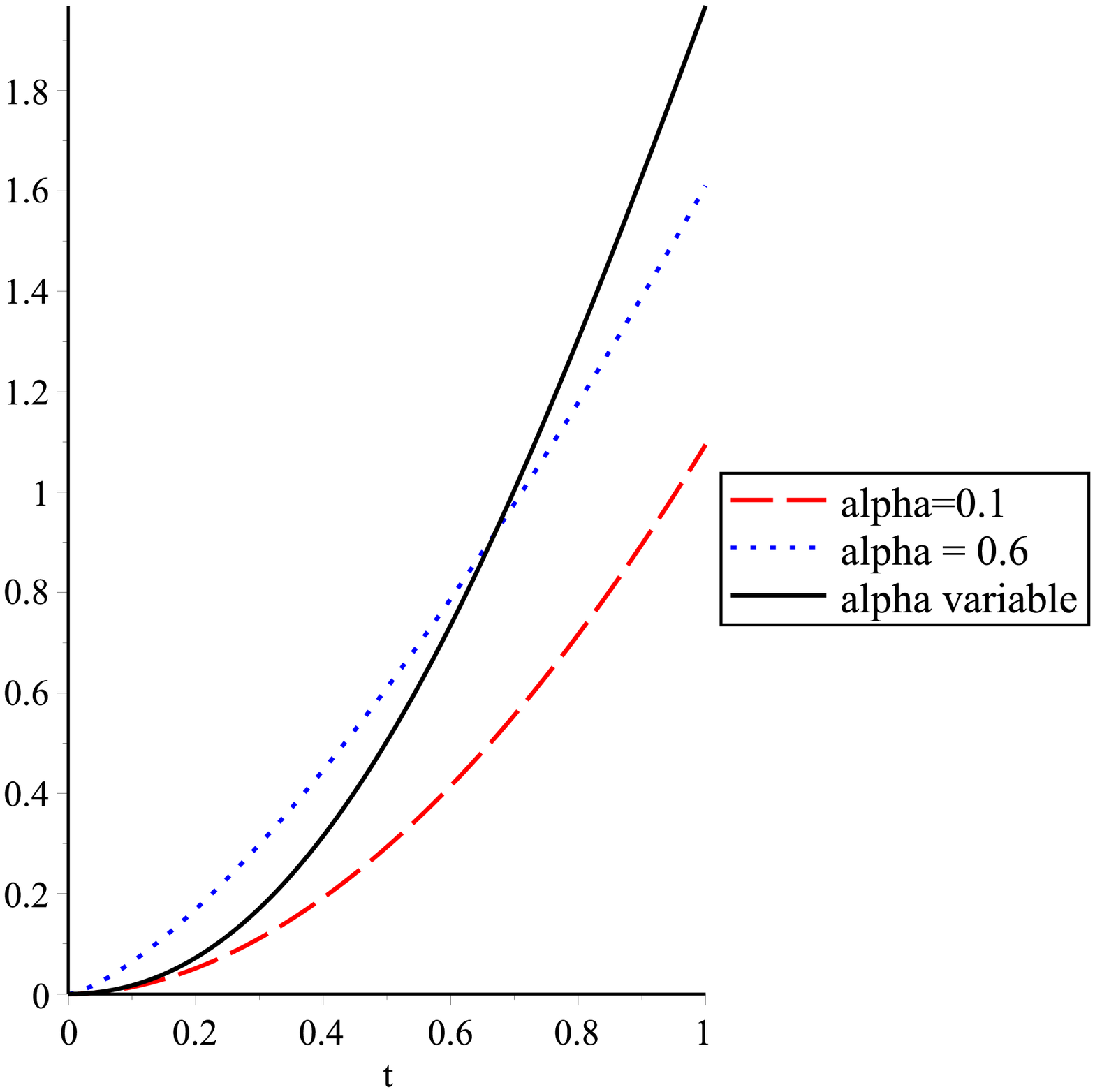}}
\subfigure[${^C_0\mathbb{D}_t^{\alpha(t)}} x(t)$]{\includegraphics[scale=0.25]{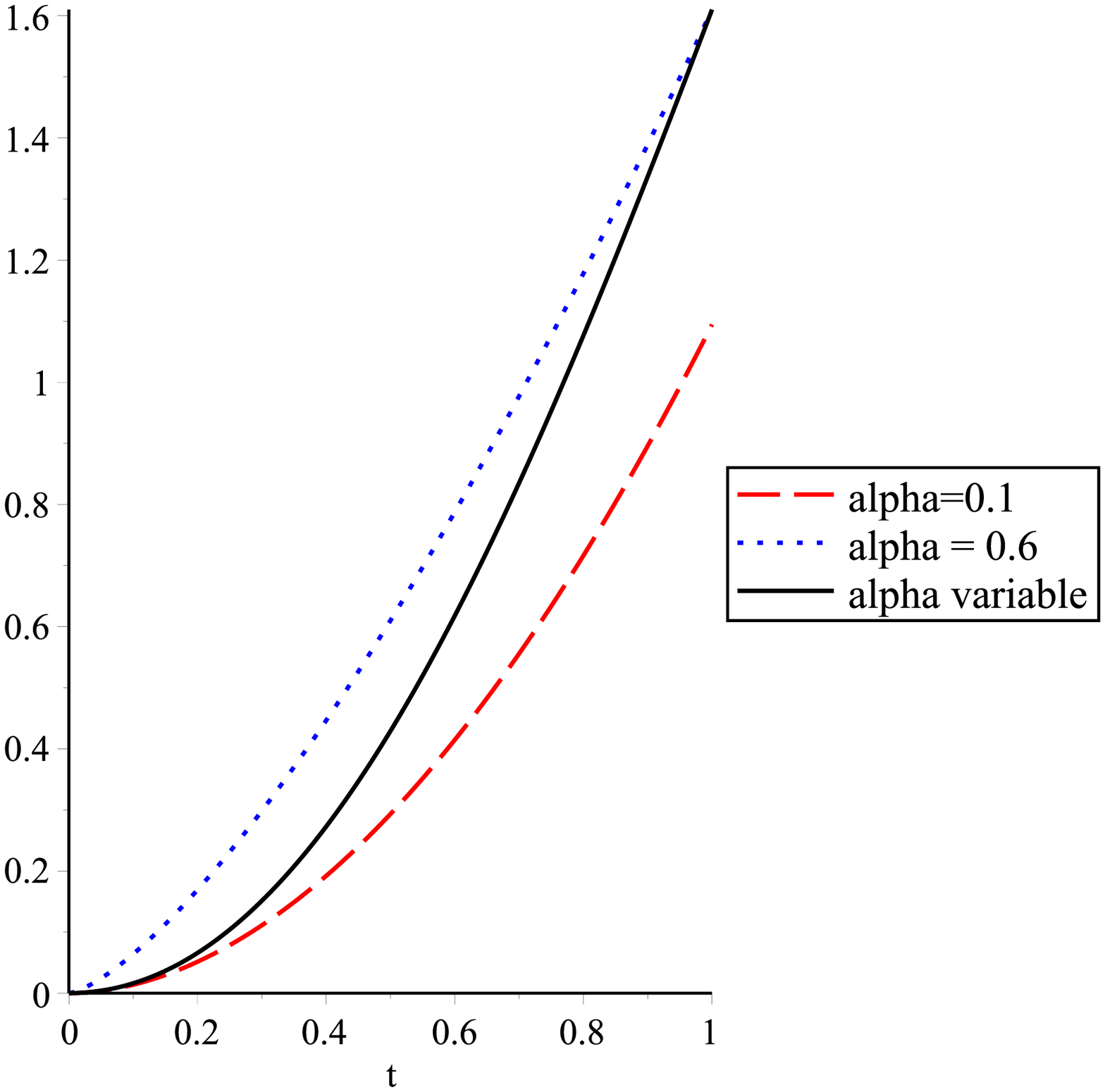}} \hspace{1cm}
\subfigure[${^C_t{D}_1^{\alpha(t)}} y(t)$]{\includegraphics[scale=0.25]{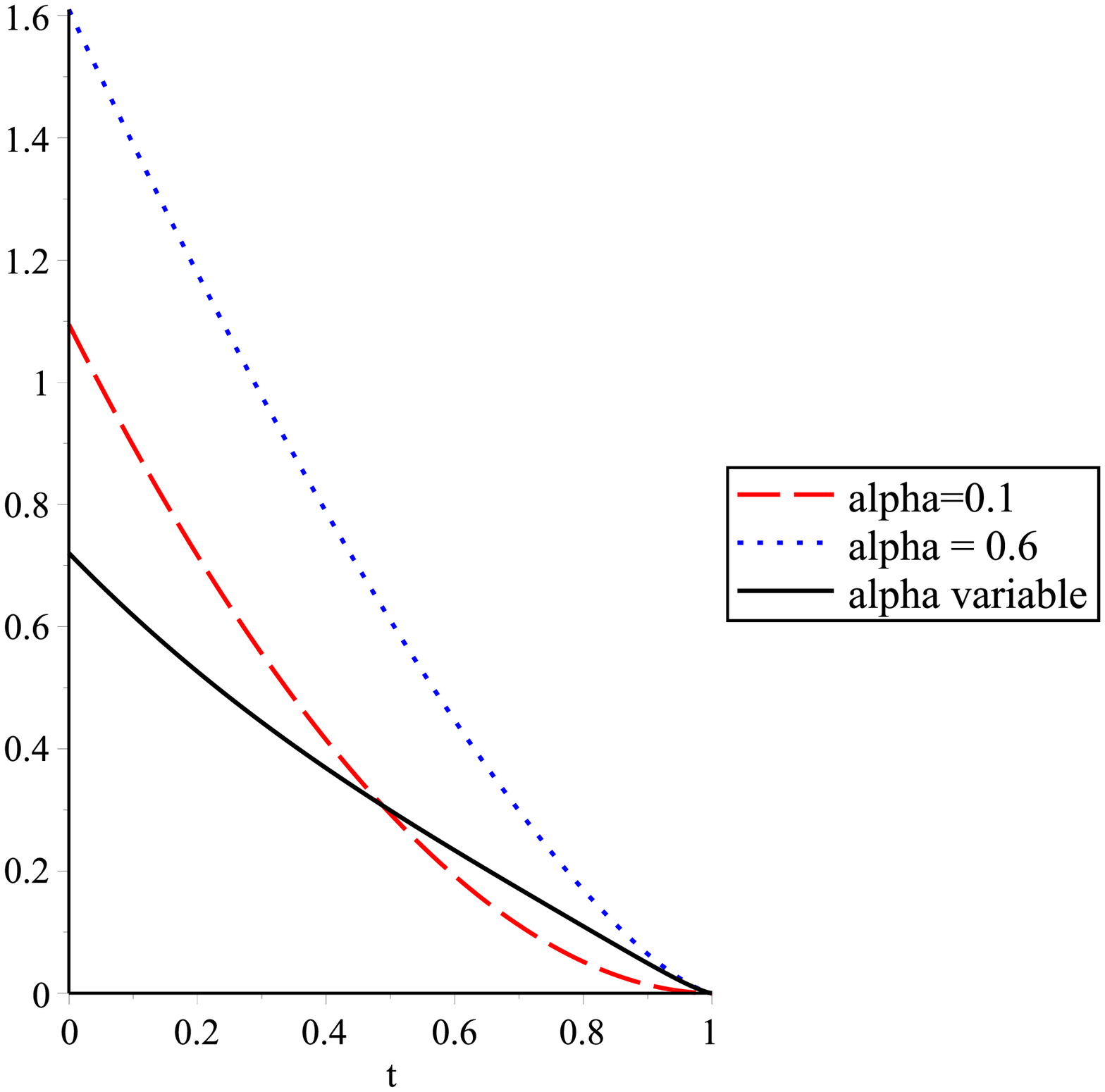}}
\subfigure[${^C_t\mathcal{D}_1^{\alpha(t)}} y(t)$]{\includegraphics[scale=0.25]{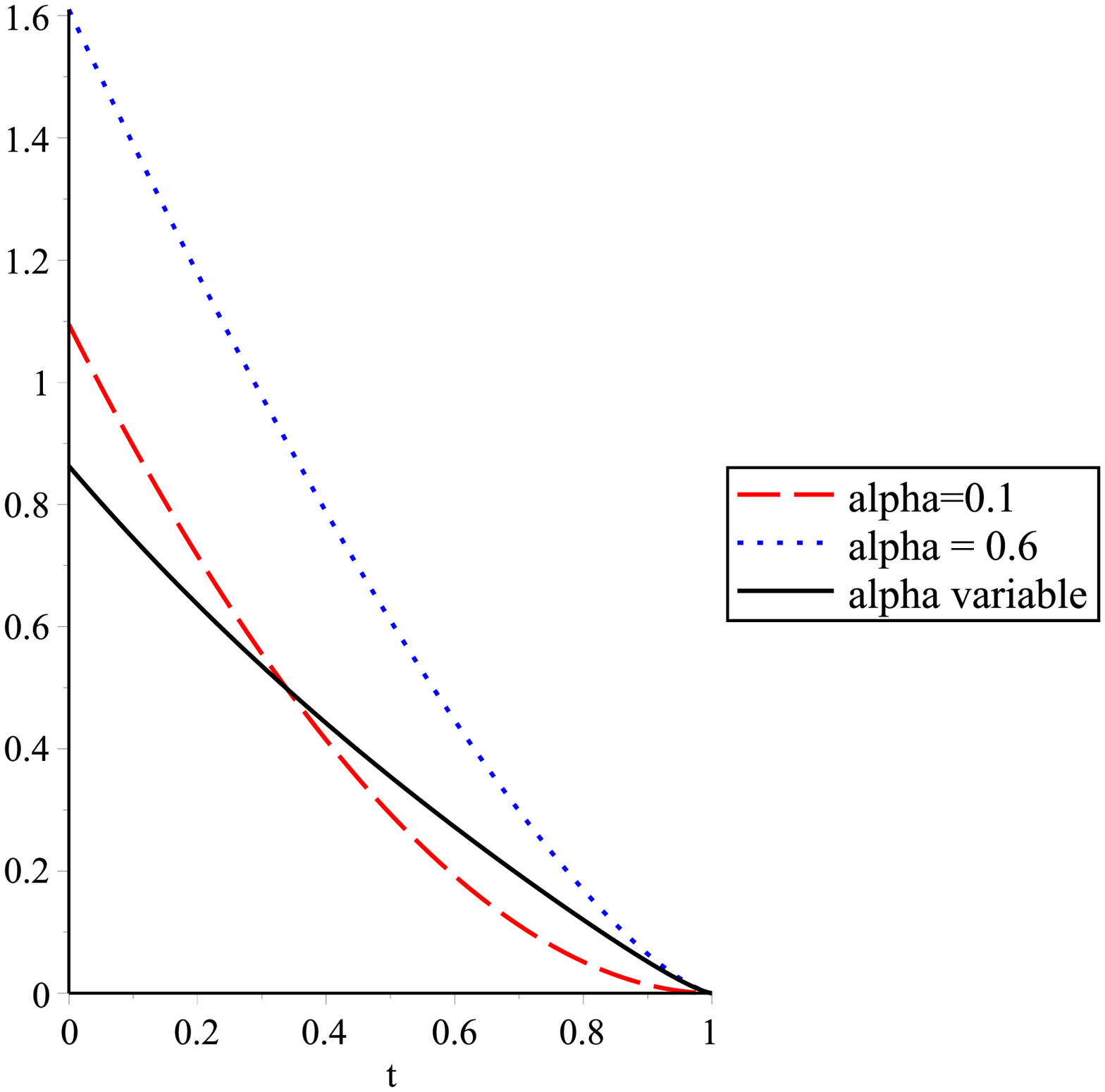}} \hspace{1cm}
\subfigure[${^C_t\mathbb{D}_1^{\alpha(t)}} y(t)$]{\includegraphics[scale=0.25]{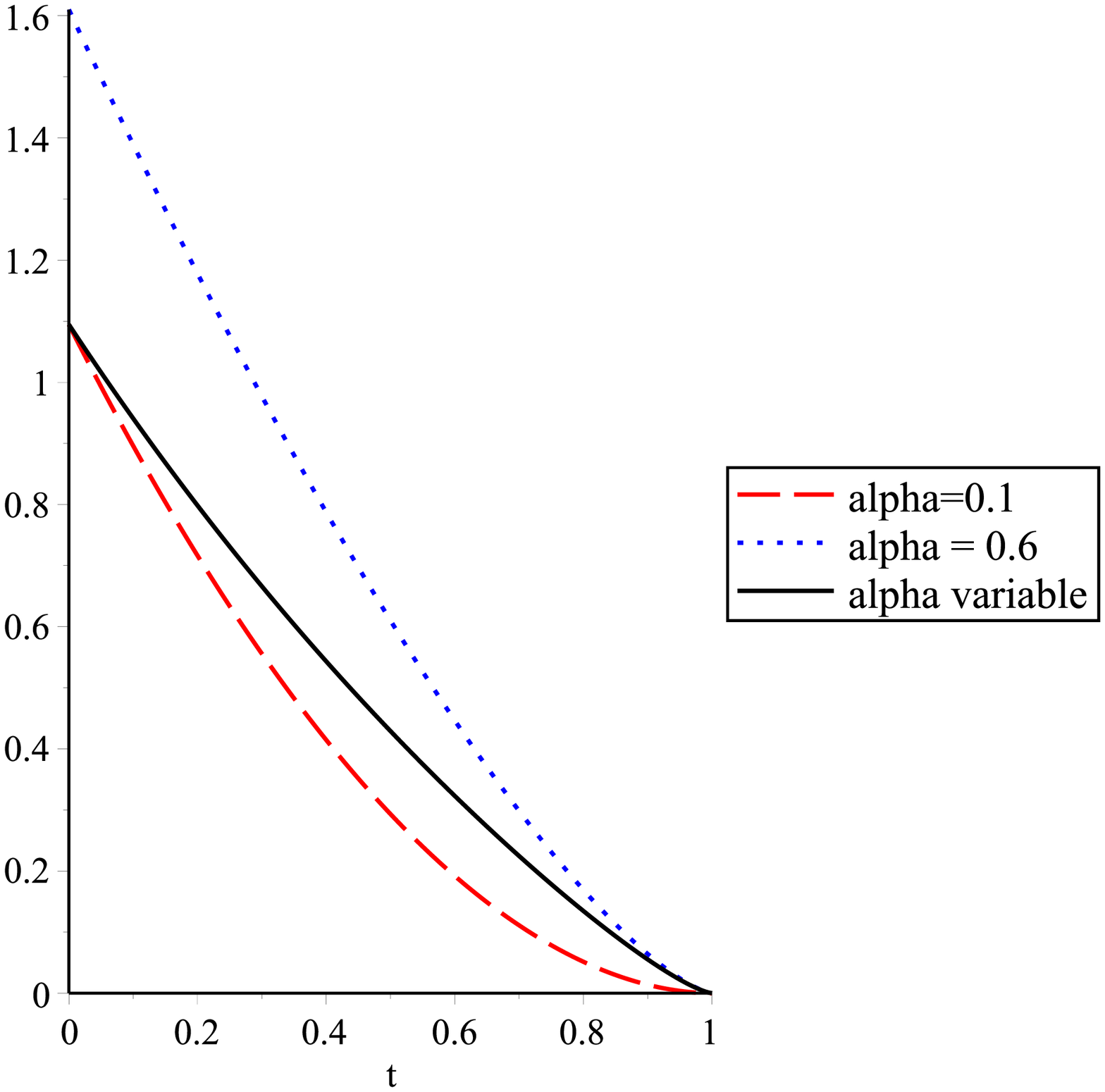}}
\end{center}
\caption{Comparison between variable-order and constant-order fractional derivatives.}\label{comparison}
\label{IntExp1}
\end{figure}


\subsection{Caputo derivatives for functions of several variables}
\label{subsec:sev_variables}
\index{Caputo fractional derivative! several variables}

Partial fractional derivatives are a natural extension
and are defined in a similar way. Let $m\in\mathbb{N}$,
$k\in\{1,\ldots,m\}$, and consider a function
$\DS x:\prod_{i=1}^m[a_i,b_i]\to\bR$
with $m$ variables. For simplicity, we define the vectors
$$
[\t]_k=(t_1,\ldots,t_{k-1},\t,t_{k+1},\ldots,t_m)\in\bR^m
$$
and
$$
(\overline t)=(t_1,\ldots,t_m)\in\bR^m.
$$

\begin{Definition}
\index{Caputo fractional derivative! partial left}
\index{Caputo fractional derivative! partial right}
\label{def:8}
Given a function $x:\prod_{i=1}^m[a_i,b_i]\to\bR$
and fractional orders $\alpha_k:[a_k,b_k]\to(0,1)$,
$k\in\{1,\ldots,m\}$,
\begin{enumerate}
\item the type I partial left Caputo derivative of order $\akn$ is defined by
$$
\PLCI x(\overline t)=\frac{1}{\Gamma(1-\akn)}\frac{\partial}{\partial t_k}
\int_{a_k}^{t_k}(t_k-\t)^{-\akn}\left(x[\t]_k-x[a_k]_k\right)d\t;
$$
\item the type I partial right Caputo derivative of order $\akn$ is defined by
$$
\PRCI x(\overline t)=\frac{-1}{\Gamma(1-\akn)}\frac{\partial }{\partial t_k}
\int_{t_k}^{b_k}(\t-t_k)^{-\akn}\left(x[\t]_k-x[b_k]_k\right)d\t;
$$
\item the type II partial left Caputo derivative of order $\akn$ is defined by
\begin{align*}
\PLCII& x(\overline t)\\
&=\frac{\partial }{\partial t_k}\left(\frac{1}{\Gamma(1-\akn)}
\int_{a_k}^{t_k}(t_k-\t)^{-\akn}\left(x[\t]_k-x[a_k]_k\right)d\t\right);
\end{align*}
\item the type II partial right Caputo derivative of order $\akn$ is defined by
\begin{align*}
\PRCII& x(\overline t)\\
&=\frac{\partial }{\partial t_k}\left(\frac{-1}{\Gamma(1-\akn)}
\int_{t_k}^{b_k}(\t-t_k)^{-\akn}\left(x[\t]_k-x[b_k]_k\right)d\t\right);
\end{align*}
\item the type III partial left Caputo derivative of order $\akn$ is defined by
$$
\PLCIII x(\overline t)=\frac{1}{\Gamma(1-\akn)}
\int_{a_k}^{t_k}(t_k-\t)^{-\akn}\frac{\partial x}{\partial t_k}[\t]_kd\t;
$$
\item the type III partial right Caputo derivative of order $\akn$ is defined by
$$
\PRCIII x(\overline t)=\frac{-1}{\Gamma(1-\akn)}
\int_{t_k}^{b_k}(\t-t_k)^{-\akn}\frac{\partial x}{\partial t_k}[\t]_kd\t.
$$
\end{enumerate}
\end{Definition}

Similarly as done before, relations between these definitions can be proven.

\begin{Theorem}
The following four formulas hold:
\begin{multline}
\label{relation}
\PLCI x(\overline t)=\PLCIII x(\overline t)\\
+\frac{\Dak}{\Gamma(2-\akn)}\int_{a_k}^{t_k}(t_k-\t)^{1-\akn}
\frac{\partial x}{\partial t_k}[\t]_k\left[\frac{1}{1-\akn}-\ln(t_k-\t)\right]d\t,
\end{multline}
\begin{multline}
\label{relation3}
\PLCI x(\overline t)=\PLCII x(\overline t)\\
-\frac{\Dak\Psi(1-\akn)}{\Gamma(1-\akn)}\int_{a_k}^{t_k}(t_k-\t)^{-\akn}[x[\t]_k-x[a_k]_k]d\t,
\end{multline}
\begin{multline*}
\PRCI x(\overline t)=\PRCIII x(\overline t)\\
+\frac{\Dak}{\Gamma(2-\akn)}\int_{t_k}^{b_k}(\t-t_k)^{1-\akn}
\frac{\partial x}{\partial t_k}[\t]_k\left[\frac{1}{1-\akn}-\ln(\t-t_k)\right]d\t
\end{multline*}
and
\begin{multline*}
\PRCI x(\overline t)=\PRCII x(\overline t)\\
+\frac{\Dak\Psi(1-\akn)}{\Gamma(1-\akn)}
\int_{t_k}^{b_k}(\t-t_k)^{-\akn}[x[\t]_k-x[b_k]_k]d\t.
\end{multline*}
\end{Theorem}


\section{Numerical approximations}
\label{sec:NumApprox}

Let $p\in\mathbb{N}$. We define
\begin{equation*}
\begin{split}
A_p &=\DS \frac{1}{\Gamma(p+1-\akn)}\left[1+\sum_{l=n-p+1}^N
\frac{\Gamma(\akn-n+l)}{\Gamma(\akn-p)(l-n+p)!}  \right],\\
B_p &=  \DS\frac{\Gamma(\akn-n+p)}{\Gamma(1-\akn)\Gamma(\akn)(p-n)!},\\
V_p(\overline t) &= \DS\int_{a_k}^{t_k}(\t-a_k)^{p-n}\frac{\partial x}{\partial t_k}[\t]_kd\t,\\
L_{p}(\overline t) &=\DS \max_{\t\in[a_k,t_k]}\left| \frac{\partial^{p} x}{\partial t_k^{p}}[\t]_k \right|.
\end{split}
\end{equation*}

\begin{Theorem}
\label{N_teo1}
\index{expansion formulas}
Let $x\in C^{n+1}\left(\prod_{i=1}^m[a_i,b_i],\bR\right)$ with $n\in\mathbb{N}$.
Then, for all $k\in\{1,\ldots,m\}$ and for all $N \in \mathbb{N}$ such that $N \geq n$, we have
\begin{align*}
\PLCIII x(\overline t)& =\DS\sum_{p=1}^{n}A_p (t_k-a_k)^{p-\akn}
\frac{\partial^p x}{\partial t_k^p}[t_k]_k \DS\\
& \quad +\sum_{p=n}^N B_p (t_k-a_k)^{n-p-\akn} V_p(\overline t)+E(\overline t).
\end{align*}
The approximation error $E(\overline t)$ is bounded by
\index{approximation error}
$$
E(\overline t)\leq L_{n+1}(\overline t)
\frac{\exp((n-\akn)^2+n-\akn)}{\Gamma(n+1-\akn)N^{n-\akn}(n-\akn)}(t_k-a_k)^{n+1-\akn}.
$$
\end{Theorem}

\begin{proof}
By definition,
$$
\PLCIII x(\overline t)=\DS\frac{1}{\Gamma(1-\akn)}
\int_{a_k}^{t_k}(t_k-\t)^{-\akn}\frac{\partial x}{\partial t_k}[\t]_kd\t
$$
and, integrating by parts with $u'(\t)=(t_k-\t)^{-\akn}$
and $v(\t)=\frac{\partial x}{\partial t_k}[\t]_k$, we deduce that
\begin{align*}
\PLCIII x(\overline t)&=\DS\frac{(t_k-a_k)^{1-\akn}}{\Gamma(2-\akn)}
\frac{\partial x}{\partial t_k}[a_k]_k\\
& \quad +\frac{1}{\Gamma(2-\akn)}
\int_{a_k}^{t_k}(t_k-\t)^{1-\akn}\frac{\partial^2 x}{\partial t_k^2}[\t]_kd\t.
\end{align*}
Integrating again by parts, taking $u'(\t)=(t_k-\t)^{1-\akn}$
and $v(\t)=\frac{\partial^2 x}{\partial t_k^2}[\t]_k$, we get
\begin{multline*}
\PLCIII x(\overline t)
=\DS\frac{(t_k-a_k)^{1-\akn}}{\Gamma(2-\akn)}\frac{\partial x}{\partial t_k}[a_k]_k
+\frac{(t_k-a_k)^{2-\akn}}{\Gamma(3-\akn)}\frac{\partial^2 x}{\partial t_k^2}[a_k]_k\\
\DS +\frac{1}{\Gamma(3-\akn)}\int_{a_k}^{t_k}(t_k-\t)^{2-\akn}
\frac{\partial^3 x}{\partial t_k^3}[\t]_kd\t.
\end{multline*}
Repeating the same procedure $n-2$ more times, we get the expansion formula
\begin{multline*}
\PLCIII x(\overline t) = \DS\sum_{p=1}^n \frac{(t_k-a_k)^{p-\akn}}{\Gamma(p+1-\akn)}
\frac{\partial^p x}{\partial t_k^p}[a_k]_k\\
\DS +\frac{1}{\Gamma(n+1-\akn)}\int_{a_k}^{t_k}(t_k-\t)^{n-\akn}
\frac{\partial^{n+1} x}{\partial t_k^{n+1}}[\t]_kd\t.
\end{multline*}
Using the equalities
\begin{equation*}
\begin{split}
(t_k-\t)^{n-\akn}&=\DS(t_k-a_k)^{n-\akn}\left(
1-\frac{\t-a_k}{t_k-a_k}\right)^{n-\akn}\\
&=\DS(t_k-a_k)^{n-\akn}\left[\sum_{p=0}^N \C (-1)^p
\frac{(\t-a_k)^p}{(t_k-a_k)^p}+\overline E(\overline t)\right]
\end{split}
\end{equation*}
with
$$
\overline E(\overline t)
=\sum_{p=N+1}^\infty \C (-1)^p \frac{(\t-a_k)^p}{(t_k-a_k)^p},
$$
we arrive at
\begin{equation*}
\begin{split}
&\PLCIII  x(\overline t)
=\DS\sum_{p=1}^n \frac{(t_k-a_k)^{p-\akn}}{\Gamma(p+1-\akn)}\frac{\partial^p x}{\partial t_k^p}[a_k]_k\\
&\DS +\frac{(t_k-a_k)^{n-\akn}}{\Gamma(n+1-\akn)}\int_{a_k}^{t_k}
\sum_{p=0}^N \C (-1)^p \frac{(\t-a_k)^p}{(t_k-a_k)^p}
\frac{\partial^{n+1} x}{\partial t_k^{n+1}}[\t]_kd\t+E(\overline t)\\
&\quad \quad   =\DS \sum_{p=1}^n \frac{(t_k-a_k)^{p-\akn}}{\Gamma(p+1-\akn)}
\frac{\partial^p x}{\partial t_k^p}[a_k]_k+\frac{(t_k-a_k)^{n-\akn}}{\Gamma(n+1-\akn)}\\
&\quad \quad \quad  \DS \times\sum_{p=0}^N \C
\frac{(-1)^p}{(t_k-a_k)^p}\int_{a_k}^{t_k}(\t-a_k)^p
\frac{\partial^{n+1} x}{\partial t_k^{n+1}}[\t]_kd\t+E(\overline t)
\end{split}
\end{equation*}
with
$$
E(\overline t)=\frac{(t_k-a_k)^{n-\akn}}{\Gamma(n+1-\akn)}\int_{a_k}^{t_k}
\overline E(\overline t)\frac{\partial^{n+1} x}{\partial t_k^{n+1}}[\t]_kd\t.
$$
Now, we split the last sum into $p=0$ and the remaining terms $p=1,\ldots,N$
and integrate by parts with $u(\t)=(\t-a_k)^p$ and
$v'(\t)=\frac{\partial^{n+1} x}{\partial t_k^{n+1}}[\t]_k$.
Observing that
$$
\C(-1)^p=\frac{\Gamma(\akn-n+p)}{\Gamma(\akn-n)p!},
$$
we obtain:
\begin{equation*}
\begin{split}
&\frac{(t_k-a_k)^{n-\akn}}{\Gamma(n+1-\akn)}\sum_{p=0}^N \C \frac{(-1)^p}{(t_k-a_k)^p}
\int_{a_k}^{t_k}(\t-a_k)^p\frac{\partial^{n+1} x}{\partial t_k^{n+1}}[\t]_kd\t\\
&=\frac{(t_k-a_k)^{n-\akn}}{\Gamma(n+1-\akn)}\left[
\frac{\partial^n x}{\partial t_k^n}[t_k]_k-\frac{\partial^n x}{\partial t_k^n}[a_k]_k\right]\\
&\quad+\frac{(t_k-a_k)^{n-\akn}}{\Gamma(n+1-\akn)}\sum_{p=1}^N
\frac{\Gamma(\akn-n+p)}{\Gamma(\akn-n)p!(t_k-a_k)^{p}}\\
&\quad \times \left[(t_k-a_k)^p \frac{\partial^n x}{\partial t_k^n}[t_k]_k
-\int_{a_k}^{t_k}p(\t-a_k)^{p-1}\frac{\partial^n x}{\partial t_k^n}[\t]_kd\t\right]\\
&=-\frac{(t_k-a_k)^{n-\akn}}{\Gamma(n+1-\akn)}\frac{\partial^n x}{\partial t_k^n}[a_k]_k
+\frac{(t_k-a_k)^{n-\akn}}{\Gamma(n+1-\akn)}\frac{\partial^n x}{\partial t_k^n}[t_k]_k\\
&\quad \times\left[1+\sum_{p=1}^N \frac{\Gamma(\akn-n+p)}{\Gamma(\akn-n)p!}\right]
+\frac{(t_k-a_k)^{n-\akn-1}}{\Gamma(n-\akn)}\\
&\quad \times\sum_{p=1}^N\frac{\Gamma(\akn-n+p)}{\Gamma(\akn+1-n)(p-1)!(t_k-a_k)^{p-1}}
\int_{a_k}^{t_k}(\t-a_k)^{p-1}\frac{\partial^{n} x}{\partial t_k^{n}}[\t]_kd\t.
\end{split}
\end{equation*}
Thus, we get
\begin{equation*}
\begin{split}
\PLCIII &x(\overline t) =\DS\sum_{p=1}^{n-1}
\frac{(t_k-a_k)^{p-\akn}}{\Gamma(p+1-\akn)}\frac{\partial^p x}{\partial t_k^p}[a_k]_k\\
&  \DS +\frac{(t_k-a_k)^{n-\akn}}{\Gamma(n+1-\akn)}
\frac{\partial^n x}{\partial t_k^n}[t_k]_k\left[1+\sum_{p=1}^N
\frac{\Gamma(\akn-n+p)}{\Gamma(\akn-n)p!}  \right]\\
&  \DS +\frac{(t_k-a_k)^{n-\akn-1}}{\Gamma(n-\akn)}\sum_{p=1}^N
\frac{\Gamma(\akn-n+p)}{\Gamma(\akn+1-n)(p-1)!(t_k-a_k)^{p-1}} \\
& \DS\times\int_{a_k}^{t_k}(\t-a_k)^{p-1}
\frac{\partial^{n} x}{\partial t_k^{n}}[\t]_kd\t+E(\overline t).
\end{split}
\end{equation*}
Repeating the process $n-1$ more times with respect to the last sum, that is,
splitting the first term of the sum and integrating by parts the obtained result,
we arrive to
\begin{equation*}
\begin{split}
\PLCIII x(\overline t)& =\DS\sum_{p=1}^{n}
\frac{(t_k-a_k)^{p-\akn}}{\Gamma(p+1-\akn)}\frac{\partial^p x}{\partial t_k^p}[t_k]_k\\
& \quad \times\left[
1+\sum_{l=n-p+1}^N \frac{\Gamma(\akn-n+l)}{\Gamma(\akn-p)(l-n+p)!}  \right]\\
& \quad \DS +\sum_{p=n}^N \frac{\Gamma(\akn-n+p)}{\Gamma(1-\akn)\Gamma(\akn)(p-n)!}(t_k-a_k)^{n-p-\akn} \\
& \quad \DS\times\int_{a_k}^{t_k}(\t-a_k)^{p-n}\frac{\partial x}{\partial t_k}[\t]_kd\t+E(\overline t).
\end{split}
\end{equation*}
We now seek the upper bound formula for $E(\overline t)$.
Using the two relations
$$
\left|  \frac{\t-a_k}{t_k-a_k}\right|\leq 1,
\, \mbox{ if } \, \t\in[a_k,t_k]$$
and
$$
\left| \C \right|\leq \frac{\exp((n-\akn)^2+n-\akn)}{p^{n+1-\akn}},
$$
we get
\begin{equation*}
\begin{split}
\overline E(\overline t)& \leq\DS \sum_{p=N+1}^\infty
\frac{\exp((n-\akn)^2+n-\akn)}{p^{n+1-\akn}} \\
& \DS\leq \int_N^\infty \frac{\exp((n-\akn)^2+n-\akn)}{p^{n+1-\akn}}\, dp\\
&=\frac{\exp((n-\akn)^2+n-\akn)}{N^{n-\akn}(n-\akn)}.
\end{split}
\end{equation*}
Then,
$$
E(\overline t)\leq L_{n+1}(\overline t)
\frac{\exp((n-\akn)^2+n-\akn)}{\Gamma(n+1-\akn)N^{n-\akn}(n-\akn)}(t_k-a_k)^{n+1-\akn}.
$$
This concludes the proof.
\end{proof}

\begin{Remark}
In Theorem~\ref{N_teo1} we have
$$
\lim_{N\to\infty}E(\overline t)=0
$$
for all $\overline t\in \prod_{i=1}^m[a_i,b_i]$ and $n\in\mathbb{N}$.
\end{Remark}

\begin{Theorem}
\label{N_teo2}
\index{expansion formulas}
Let $x\in C^{n+1}\left(\prod_{i=1}^m[a_i,b_i],\mathbb{R}\right)$ with $n\in\mathbb{N}$.
Then, for all $k\in\{1,\ldots,m\}$ and for all $N \in \mathbb{N}$ such that $N \geq n$, we have
\begin{align*}
\PLCI &x(\overline t) =\DS\sum_{p=1}^{n}A_p (t_k-a_k)^{p-\akn}
\frac{\partial^p x}{\partial t_k^p}[t_k]_k \\
&\DS+\sum_{p=n}^N B_p (t_k-a_k)^{n-p-\akn} V_p(\overline t)+\frac{\Dak(t_k-a_k)^{1-\akn}}{\Gamma(2-\akn)}\\
&\times \left[\left(\frac{1}{1-\akn}-\ln(t_k-a_k)\right)\sum_{p=0}^N\D\frac{(-1)^p}{(t_k-a_k)^{p}} V_{n+p}(\overline t)\right.\\
&\left.+\sum_{p=0}^N\D(-1)^p\sum_{r=1}^N\frac{1}{r(t_k-a_k)^{p+r}}
V_{n+p+r}(\overline t)\right]+E(\overline t).
\end{align*}
The approximation error $E(\overline t)$ is bounded by
\index{approximation error}
\begin{equation*}
\begin{split}
E(\overline t)& \leq\DS L_{n+1}(\overline t)
\frac{\exp((n-\akn)^2+n-\akn)}{\Gamma(n+1-\akn)N^{n-\akn}(n-\akn)}(t_k-a_k)^{n+1-\akn}\\
&\quad \DS+\left|\Dak\right|L_1(\overline t)\frac{{\exp((1-\akn)^2+1-\akn)}}{\Gamma(2-\akn)N^{1-\akn}(1-\akn)}\\
&\quad \DS \times \left[\left|\frac{1}{1-\akn}-\ln(t_k-a_k)\right|+\frac{1}{N}\right](t_k-a_k)^{2-\akn}.
\end{split}
\end{equation*}
\end{Theorem}

\begin{proof}
Taking into account relation \eqref{relation} and Theorem~\ref{N_teo1},
we only need to expand the term
\begin{equation}
\label{relation2}
\frac{\Dak}{\Gamma(2-\akn)}\int_{a_k}^{t_k}(t_k-\t)^{1-\akn}
\frac{\partial x}{\partial t_k}[\t]_k\left[\frac{1}{1-\akn}-\ln(t_k-\t)\right]d\t.
\end{equation}
Splitting the integral, and using the expansion formulas
\begin{equation*}
\begin{split}
(t_k-\t)^{1-\akn}&=\DS(t_k-a_k)^{1-\akn}\left(1-\frac{\t-a_k}{t_k-a_k}\right)^{1-\akn}\\
&=\DS(t_k-a_k)^{1-\akn}\left[\sum_{p=0}^N \D (-1)^p
\frac{(\t-a_k)^p}{(t_k-a_k)^p}+\overline E_1(\overline t)\right]
\end{split}
\end{equation*}
with
$$
\overline E_1(\overline t)
=\sum_{p=N+1}^\infty \D (-1)^p \frac{(\t-a_k)^p}{(t_k-a_k)^p}
$$
and
\begin{equation*}
\begin{split}
\ln(t_k-\t)&=\DS\ln(t_k-a_k)+\ln\left(1-\frac{\t-a_k}{t_k-a_k}\right)\\
&=\DS\ln(t_k-a_k)-\sum_{r=1}^N \frac{1}{r}
\frac{(\t-a_k)^r}{(t_k-a_k)^r}-\overline E_2(\overline t)
\end{split}
\end{equation*}
with
$$
\overline E_2(\overline t)=\sum_{r=N+1}^\infty
\frac{1}{r} \frac{(\t-a_k)^r}{(t_k-a_k)^r},
$$
we conclude that \eqref{relation2} is equivalent to
\begin{equation*}
\begin{split}
&\frac{\Dak}{\Gamma(2-\akn)}\left[\left(\frac{1}{1-\akn}-\ln(t_k-a_k)\right)
\int_{a_k}^{t_k}(t_k-\t)^{1-\akn}\frac{\partial x}{\partial t_k}[\t]_kd\t\right.\\
&\qquad \left. -\int_{a_k}^{t_k}(t_k-\t)^{1-\akn} \ln\left(1-\frac{\t-a_k}{t_k-a_k}\right)
\frac{\partial x}{\partial t_k}[\t]_kd\t\right]\\
&=\frac{\Dak}{\Gamma(2-\akn)}\left[
\left(\frac{1}{1-\akn}-\ln(t_k-a_k)\right)\right.\\
&\qquad \times\int_{a_k}^{t_k}(t_k-a_k)^{1-\akn}\sum_{p=0}^N \D (-1)^p
\frac{(\t-a_k)^p}{(t_k-a_k)^p}\frac{\partial x}{\partial t_k}[\t]_kd\t\\
&\qquad +\int_{a_k}^{t_k}(t_k-a_k)^{1-\akn}
\sum_{p=0}^N \D (-1)^p \frac{(\t-a_k)^p}{(t_k-a_k)^p}\\
& \qquad\left. \times \sum_{r=1}^N \frac{1}{r}\frac{(\t-a_k)^r}{(t_k-a_k)^r}
\frac{\partial x}{\partial t_k}[\t]_kd\t  \right]
+\frac{\Dak}{\Gamma(2-\akn)}\left[\left(\frac{1}{1-\akn}-\ln(t_k-a_k)\right)\right.\\
&\qquad\times\int_{a_k}^{t_k}(t_k-a_k)^{1-\akn}\overline E_1(\overline t)
\frac{\partial x}{\partial t_k}[\t]_kd\t\\
&\qquad\left. +\int_{a_k}^{t_k}(t_k-a_k)^{1-\akn}\overline E_1(\overline t)
\overline E_2(\overline t)\frac{\partial x}{\partial t_k}[\t]_kd\t \right]\\
&=\frac{\Dak(t_k-a_k)^{1-\akn}}{\Gamma(2-\akn)}\left[\left(\frac{1}{1-\akn}
-\ln(t_k-a_k)\right)\sum_{p=0}^N\D\right.\\
& \qquad\times \frac{(-1)^p}{(t_k-a_k)^{p}} V_{n+p}(\overline t)\left.
+\sum_{p=0}^N\D(-1)^p\sum_{r=1}^N\frac{1}{r(t_k-a_k)^{p+r}}
V_{n+p+r}(\overline t)\right]
\end{split}
\end{equation*}
\begin{equation*}
\begin{split}
&\qquad+\frac{\Dak (t_k-a_k)^{1-\akn}}{\Gamma(2-\akn)}\left[\left(\frac{1}{1-\akn}
-\ln(t_k-a_k)\right)\right.\\
&\qquad \times \left.\int_{a_k}^{t_k}\overline E_1(\overline t)
\frac{\partial x}{\partial t_k}[\t]_kd\t
+\int_{a_k}^{t_k}\overline E_1(\overline t)\overline E_2(\overline t)
\frac{\partial x}{\partial t_k}[\t]_kd\t  \right].
\end{split}
\end{equation*}
For the error analysis, we know from Theorem~\ref{N_teo1} that
$$
\overline E_1(\overline t)
\leq\frac{\exp((1-\akn)^2+1-\akn)}{N^{1-\akn}(1-\akn)}.
$$
Then,
\begin{equation}
\label{error1}
\begin{split}
\left| \int_{a_k}^{t_k}\right. &(t_k-a_k)^{1-\akn} \left. \overline E_1(\overline t)
\frac{\partial x}{\partial t_k}[\t]_kd\t \right| \\
& \leq L_1(\overline t)\frac{{\exp((1-\akn)^2+1-\akn)}}{N^{1-\akn}(1-\akn)}(t_k-a_k)^{2-\akn}.
\end{split}
\end{equation}
On the other hand, we have
\begin{equation}
\label{error2}
\begin{split}
&\left|\int_{a_k}^{t_k}(t_k-a_k)^{1-\akn}\overline E_1(\overline t)\overline E_2(\overline t)
\frac{\partial x}{\partial t_k}[\t]_kd\t \right|\\
&\leq L_1(\overline t)\frac{{\exp((1-\akn)^2+1-\akn)}}{N^{1-\akn}(1-\akn)}(t_k-a_k)^{1-\akn}\\
& \qquad \times\sum_{r=N+1}^\infty\frac{1}{r(t_k-a_k)^r}\int_{a_k}^{t_k}(\t-a_k)^rd\t\\
&= L_1(\overline t)\frac{{\exp((1-\akn)^2+1-\akn)}}{N^{1-\akn}(1-\akn)}(t_k-a_k)^{1-\akn}
\sum_{r=N+1}^\infty\frac{t_k-a_k}{r(r+1)}\\
&\leq L_1(\overline t)\frac{{\exp((1-\akn)^2+1-\akn)}}{N^{2-\akn}(1-\akn)}(t_k-a_k)^{2-\akn}.
\end{split}
\end{equation}
We get the desired result by combining inequalities \eqref{error1} and \eqref{error2}.
\end{proof}

\begin{Theorem}
\label{N_teo3}
\index{expansion formulas}
Let $x\in C^{n+1}(\prod_{i=1}^m[a_i,b_i],\bR)$ with $n\in\mathbb{N}$. Then,
for all $k\in\{1,\ldots,m\}$ and for all $N \in \mathbb{N}$ such that $N \geq n$, we have
\begin{equation*}
\begin{split}
&\PLCII x(\overline t) =\DS\sum_{p=1}^{n}A_p (t_k-a_k)^{p-\akn}\frac{\partial^p x}{\partial t_k^p}[t_k]_k \\
& \qquad \DS+\sum_{p=n}^N B_p (t_k-a_k)^{n-p-\akn} V_p(\overline t)+\frac{\Dak(t_k-a_k)^{1-\akn}}{\Gamma(2-\akn)}\\
&\quad\times\left[\left(\Psi(2-\akn)-\ln(t_k-a_k)\right)
\sum_{p=0}^N\D\frac{(-1)^p}{(t_k-a_k)^{p}} V_{n+p}(\overline t)\right.\\
&\quad \left.+\sum_{p=0}^N\D(-1)^p\sum_{r=1}^N\frac{1}{r(t_k-a_k)^{p+r}}
V_{n+p+r}(\overline t)\right]+E(\overline t).
\end{split}
\end{equation*}
The approximation error $E(\overline t)$ is bounded by
\index{approximation error}
\begin{equation*}
\begin{split}
E(\overline t)&\leq\DS L_{n+1}(\overline t)
\frac{\exp((n-\akn)^2+n-\akn)}{\Gamma(n+1-\akn)N^{n-\akn}(n-\akn)}(t_k-a_k)^{n+1-\akn}\\
& \quad \DS+\left|\Dak\right|L_1(\overline t)
\frac{{\exp((1-\akn)^2+1-\akn)}}{\Gamma(2-\akn)N^{1-\akn}(1-\akn)}\\
& \quad \DS\times \left[\left|\Psi(2-\akn)-\ln(t_k-a_k)\right|
+\frac{1}{N}\right](t_k-a_k)^{2-\akn}.
\end{split}
\end{equation*}
\end{Theorem}

\begin{proof}
Starting with relation \eqref{relation3},
and integrating by parts the integral, we obtain that
\begin{multline*}
\PLCII x(\overline t)=\PLCI x(\overline t)\\
+\frac{\Dak\Psi(1-\akn)}{\Gamma(2-\akn)}
\int_{a_k}^{t_k}(t_k-\t)^{1-\akn}\frac{\partial x}{\partial t_k}[\t]_k d\t.
\end{multline*}
The rest of the proof is similar to the one of Theorem~\ref{N_teo2}.
\end{proof}

\begin{Remark}
As particular cases of Theorems~\ref{N_teo1}, \ref{N_teo2} and \ref{N_teo3},
we obtain expansion formulas for $\LCI x(t)$, $\LCII x(t)$ and $\LCIII x(t)$.
\end{Remark}

With respect to the three right fractional operators of Definition~\ref{def:8},
we set, for $p\in\mathbb{N}$,
\begin{equation*}
\begin{split}
C_p &=\DS \frac{(-1)^{p}}{\Gamma(p+1-\akn)}\left[1+\sum_{l=n-p+1}^N
\frac{\Gamma(\akn-n+l)}{\Gamma(\akn-p)(l-n+p)!}  \right],\\
D_p &= \DS\frac{-\Gamma(\akn-n+p)}{\Gamma(1-\akn)\Gamma(\akn)(p-n)!},\\
W_p(\overline t)
&= \DS\int_{t_k}^{b_k}(b_k-\t)^{p-n}\frac{\partial x}{\partial t_k}[\t]_kd\t,\\
M_{p}(\overline t)
&=\DS \max_{\t\in[t_k,b_k]}\left| \frac{\partial^{p} x}{\partial t_k^{p}}[\t]_k \right|.
\end{split}
\end{equation*}
The expansion formulas are given in Theorems~\ref{thm:15}, \ref{thm:16} and \ref{thm:17}.
We omit the proofs since they are similar to the corresponding left ones.

\begin{Theorem}
\label{thm:15}
\index{expansion formulas}
Let $x\in C^{n+1}\left(\prod_{i=1}^m[a_i,b_i],\bR\right)$ with $n\in\mathbb{N}$.
Then, for all $k\in\{1,\ldots,m\}$ and for all $N \in \mathbb{N}$ such that $N \geq n$, we have
\begin{align*}
\PRCIII x(\overline t) =&\DS\sum_{p=1}^{n}C_p (b_k-t_k)^{p-\akn}\frac{\partial^p x}{\partial t_k^p}[t_k]_k\\
& \quad \DS +\sum_{p=n}^N D_p (b_k-t_k)^{n-p-\akn} W_p(\overline t)+E(\overline t).
\end{align*}
The approximation error $E(\overline t)$ is bounded by
\index{approximation error}
$$
E(\overline t)\leq M_{n+1}(\overline t)
\frac{\exp((n-\akn)^2+n-\akn)}{\Gamma(n+1-\akn)N^{n-\akn}(n-\akn)}(b_k-t_k)^{n+1-\akn}.
$$
\end{Theorem}

\begin{Theorem}
\label{thm:16}
\index{expansion formulas}
Let $x\in C^{n+1}\left(\prod_{i=1}^m[a_i,b_i],\bR\right)$ with $n\in\mathbb{N}$.
Then, for all $k\in\{1,\ldots,m\}$ and for all $N \in \mathbb{N}$ such that $N \geq n$, we have
\begin{equation*}
\begin{split}
&\PRCI x(\overline t) =\DS\sum_{p=1}^{n}C_p (b_k-t_k)^{p-\akn}\frac{\partial^p x}{\partial t_k^p}[t_k]_k \\
&\quad \DS+\sum_{p=n}^N D_p (b_k-t_k)^{n-p-\akn} W_p(\overline t)+\frac{\Dak(b_k-t_k)^{1-\akn}}{\Gamma(2-\akn)}\\
&\quad\quad \times \left[\left(\frac{1}{1-\akn}-\ln(b_k-t_k)\right)\sum_{p=0}^N\D\frac{(-1)^p}{(b_k-t_k)^{p}} W_{n+p}(\overline t)\right.\\
&\quad \left.+\sum_{p=0}^N\D(-1)^p\sum_{r=1}^N\frac{1}{r(b_k-t_k)^{p+r}}
W_{n+p+r}(\overline t)\right]+E(\overline t).
\end{split}
\end{equation*}
The approximation error $E(\overline t)$ is bounded by
\index{approximation error}
\begin{equation*}
\begin{split}
E(\overline t)&\leq M_{n+1}(\overline t)
\frac{\exp((n-\akn)^2+n-\akn)}{\Gamma(n+1-\akn)N^{n-\akn}(n-\akn)}(b_k-t_k)^{n+1-\akn}\\
&\quad \DS +\left|\Dak\right|M_1(\overline t)
\frac{{\exp((1-\akn)^2+1-\akn)}}{\Gamma(2-\akn)N^{1-\akn}(1-\akn)}\\
&\quad \DS \times \left[\left|\frac{1}{1-\akn}-\ln(b_k-t_k)\right|
+\frac{1}{N}\right](b_k-t_k)^{2-\akn}.
\end{split}
\end{equation*}
\end{Theorem}

\begin{Theorem}
\label{thm:17}
\index{expansion formulas}
Let $x\in C^{n+1}\left(\prod_{i=1}^m[a_i,b_i],\bR\right)$ with $n\in\mathbb{N}$.
Then, for all $k\in\{1,\ldots,m\}$ and for all $N \in \mathbb{N}$ such that $N \geq n$, we have
\begin{equation*}
\begin{split}
&\PRCI x(\overline t) =\DS\sum_{p=1}^{n}C_p (b_k-t_k)^{p-\akn}\frac{\partial^p x}{\partial t_k^p}[t_k]_k \DS\\
& \quad +\sum_{p=n}^N D_p (b_k-t_k)^{n-p-\akn} W_p(\overline t)+\frac{\Dak(b_k-t_k)^{1-\akn}}{\Gamma(2-\akn)}\\
&\quad\times\left[\left(\Psi(2-\akn)-\ln(b_k-t_k)\right)\sum_{p=0}^N\D\frac{(-1)^p}{(b_k-t_k)^{p}}
W_{n+p}(\overline t)\right.\\
&\quad \left.+\sum_{p=0}^N\D(-1)^p\sum_{r=1}^N\frac{1}{r(b_k-t_k)^{p+r}}
W_{n+p+r}(\overline t)\right]+E(\overline t).
\end{split}
\end{equation*}
The approximation error $E(\overline t)$ is bounded by
\index{approximation error}
\begin{equation*}
\begin{split}
E(\overline t)&\leq M_{n+1}(\overline t)
\frac{\exp((n-\akn)^2+n-\akn)}{\Gamma(n+1-\akn)N^{n-\akn}(n-\akn)}(b_k-t_k)^{n+1-\akn}\\
&\quad \DS +\left|\Dak\right|M_1(\overline t)\frac{{\exp((1-\akn)^2+1-\akn)}}{\Gamma(2-\akn)N^{1-\akn}(1-\akn)}\\
&\quad \DS \times \left[\left|\Psi(2-\akn)-\ln(b_k-t_k)\right|+\frac{1}{N}\right](b_k-t_k)^{2-\akn}.
\end{split}
\end{equation*}
\end{Theorem}


\section{Example}
\label{sec:ex_numer}

To test the accuracy of the proposed method, we compare the fractional derivative
of a concrete given function with some numerical approximations of it.
For $t\in[0,1]$, let $x(t)=t^2$ be the test function. For the order
of the fractional derivatives we consider two cases:
$$
\ati=\frac{50t+49}{100} \quad \mbox{and} \quad \beta(t)=\frac{t+5}{10}.
$$
We consider the approximations given in Theorems~\ref{N_teo1}, \ref{N_teo2} and \ref{N_teo3},
with a fixed $n=1$ and $N\in\{2,4,6\}$. The error of approximating
$f(t)$ by $\tilde{f}(t)$ is measured by $|f(t)-\tilde{f}(t)|$.
See Figures~\ref{IntExp2}--\ref{IntExp7}.
\begin{figure}[!ht]
\begin{center}
\subfigure[${^C_0\mathbb{D}_t^{\ati}} x(t)$]{\includegraphics[scale=0.25]{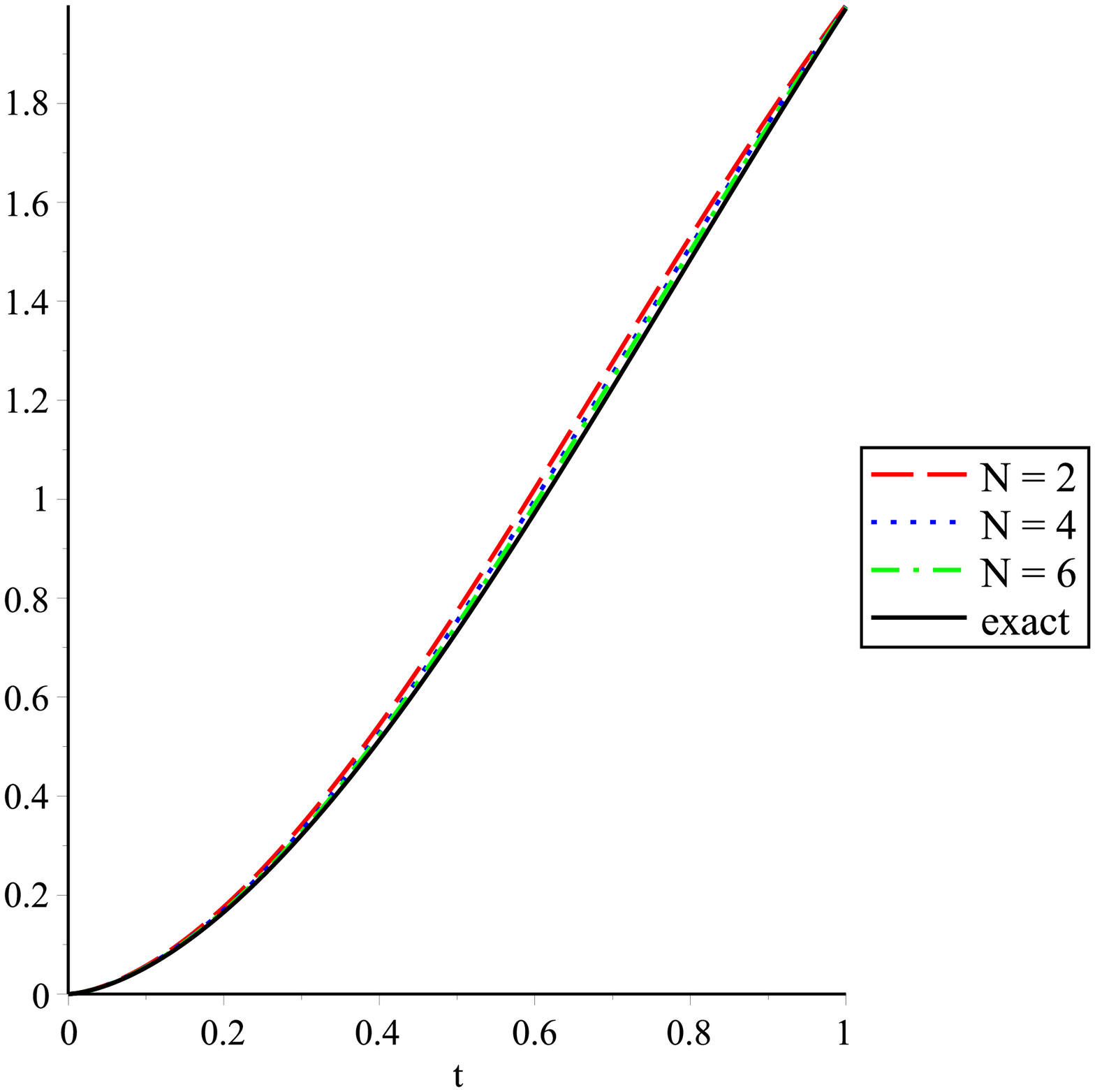}}\hspace{1cm}
\subfigure[Error]{\includegraphics[scale=0.25]{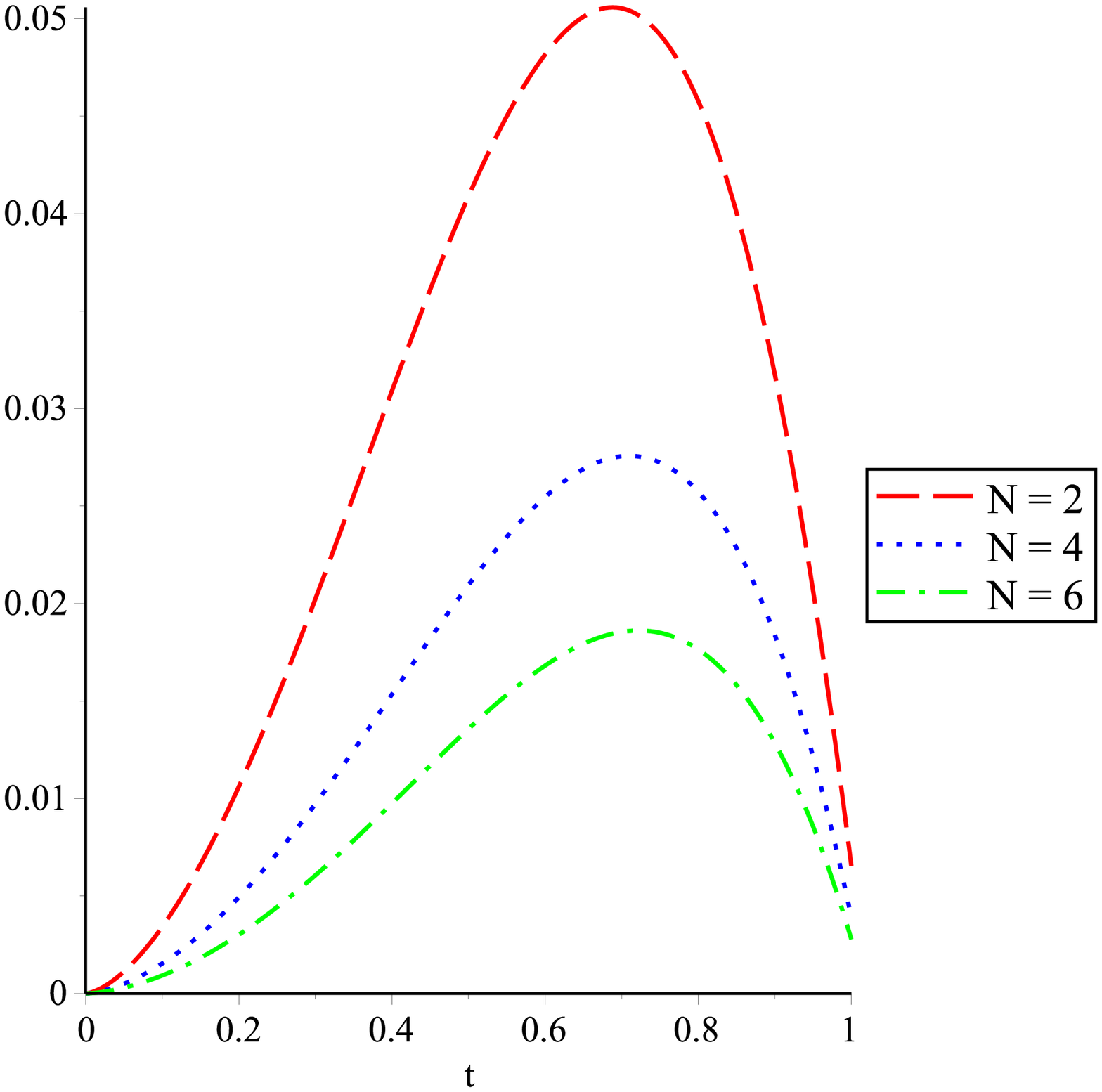}}
\end{center}
\caption{Type III left Caputo derivative of order $\ati$
for the example of Section~\ref{sec:ex_numer}---analytic
versus numerical approximations obtained from Theorem~\ref{N_teo1}.}
\label{IntExp2}
\end{figure}
\begin{figure}[!ht]
\begin{center}
\subfigure[${^C_0D_t^{\ati}} x(t)$]{\includegraphics[scale=0.25]{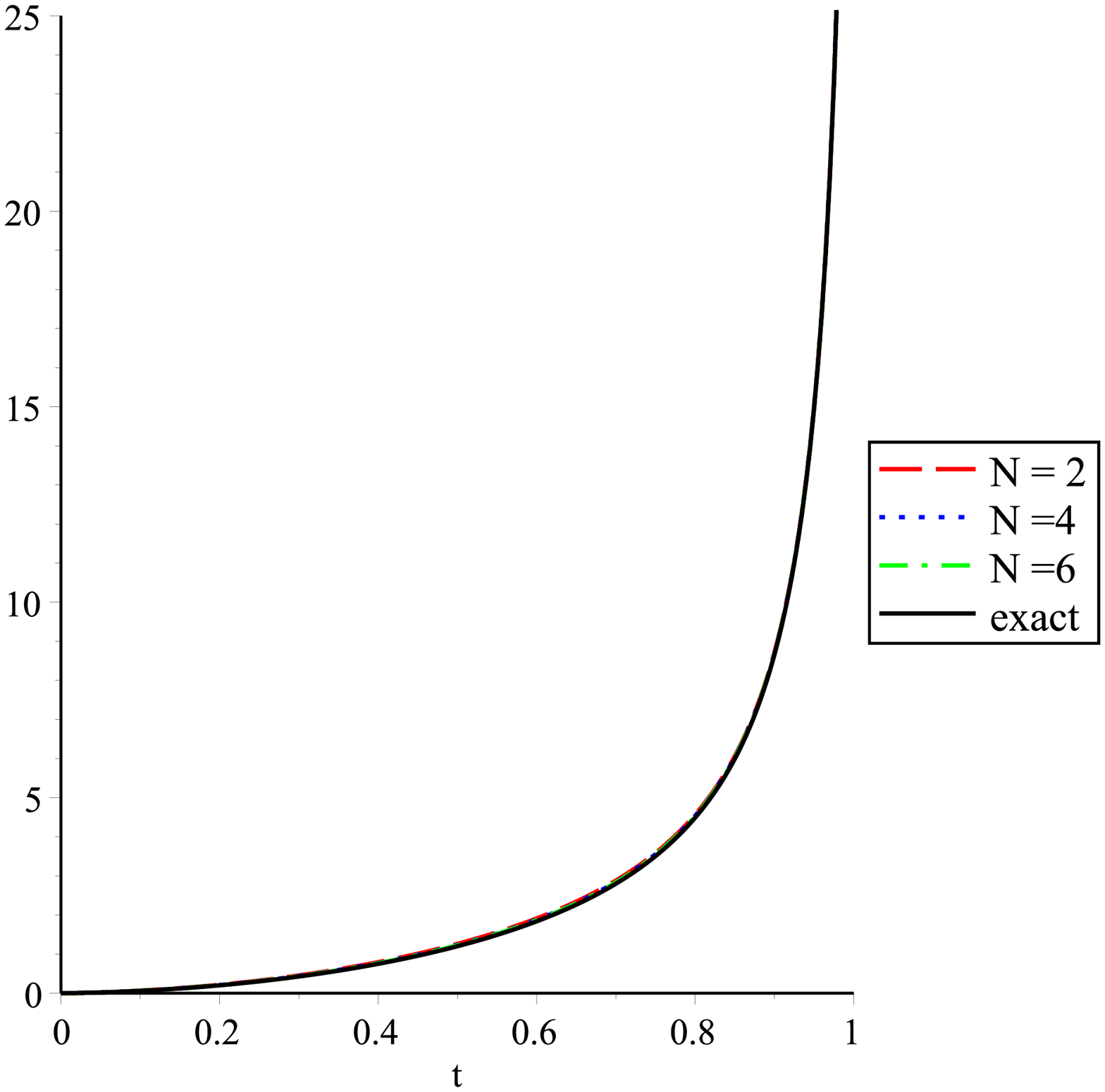}}\hspace{1cm}
\subfigure[Error]{\includegraphics[scale=0.25]{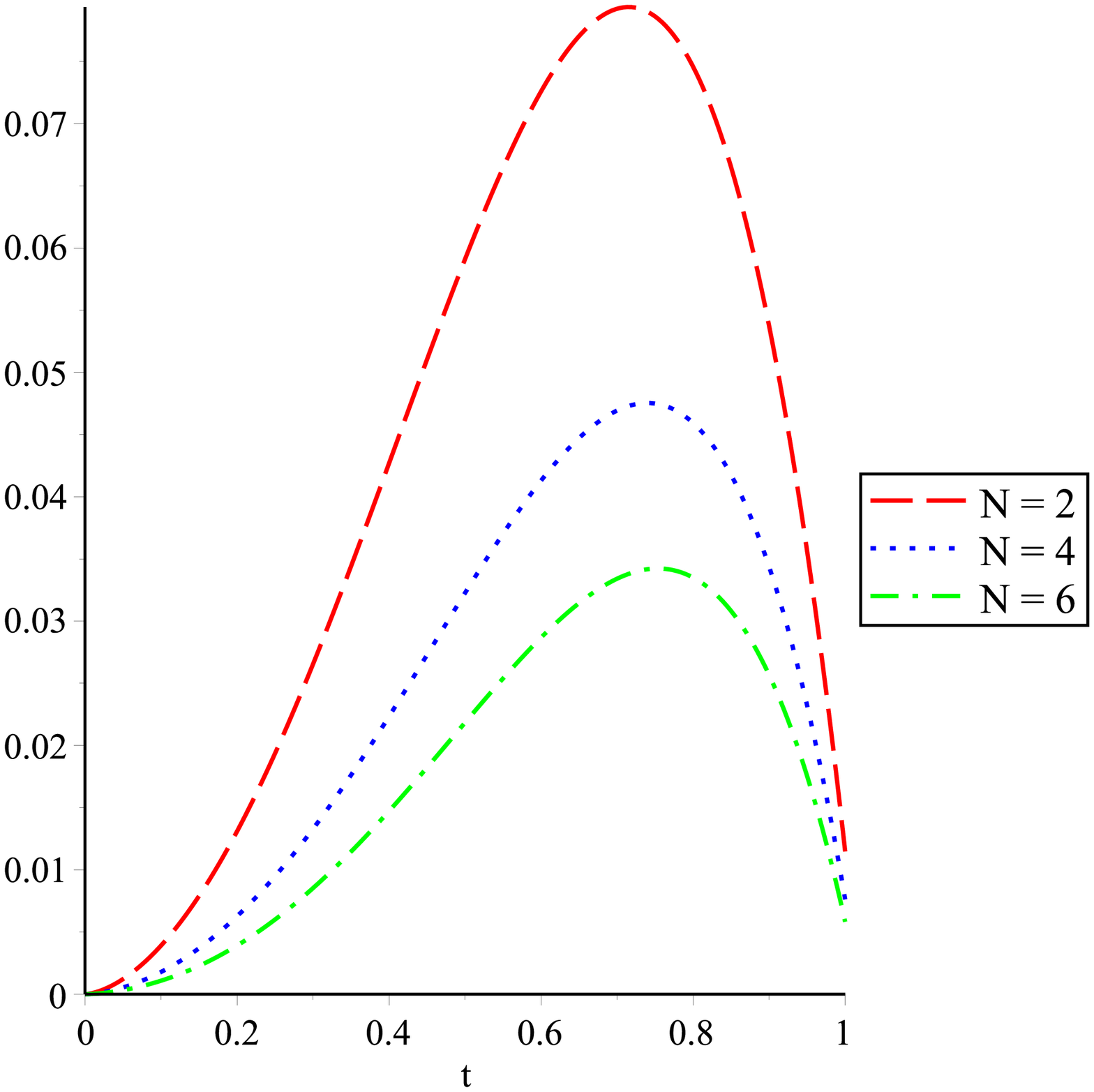}}
\end{center}
\caption{Type I left Caputo derivative of order $\ati$
for the example of Section~\ref{sec:ex_numer}---analytic
versus numerical approximations obtained from Theorem~\ref{N_teo2}.}
\label{IntExp3}
\end{figure}
\begin{figure}[!ht]
\begin{center}
\subfigure[${^C_0\mathcal{D}_t^{\ati}} x(t)$]{\includegraphics[scale=0.25]{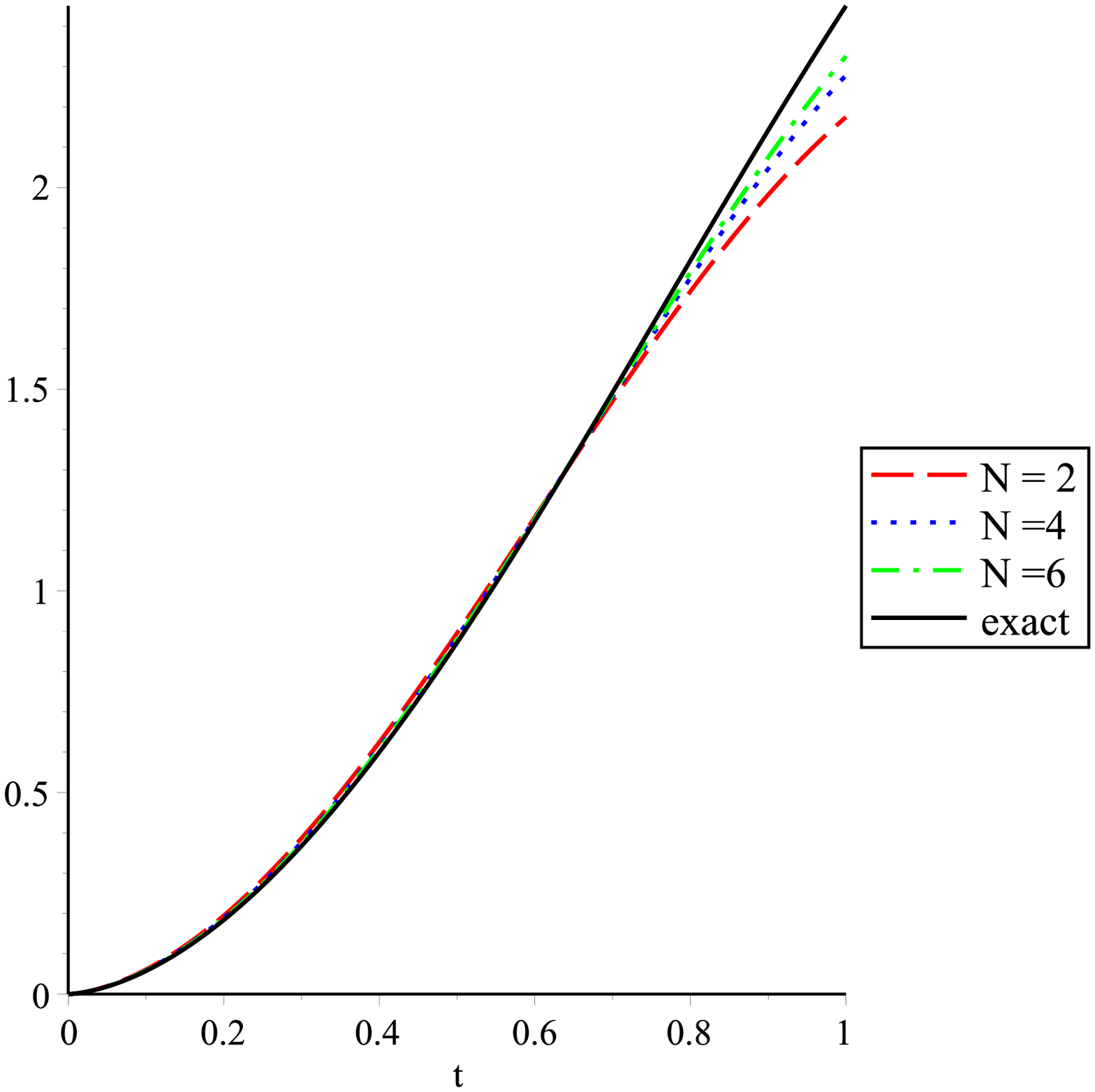}} \hspace{1cm}
\subfigure[Error]{\includegraphics[scale=0.25]{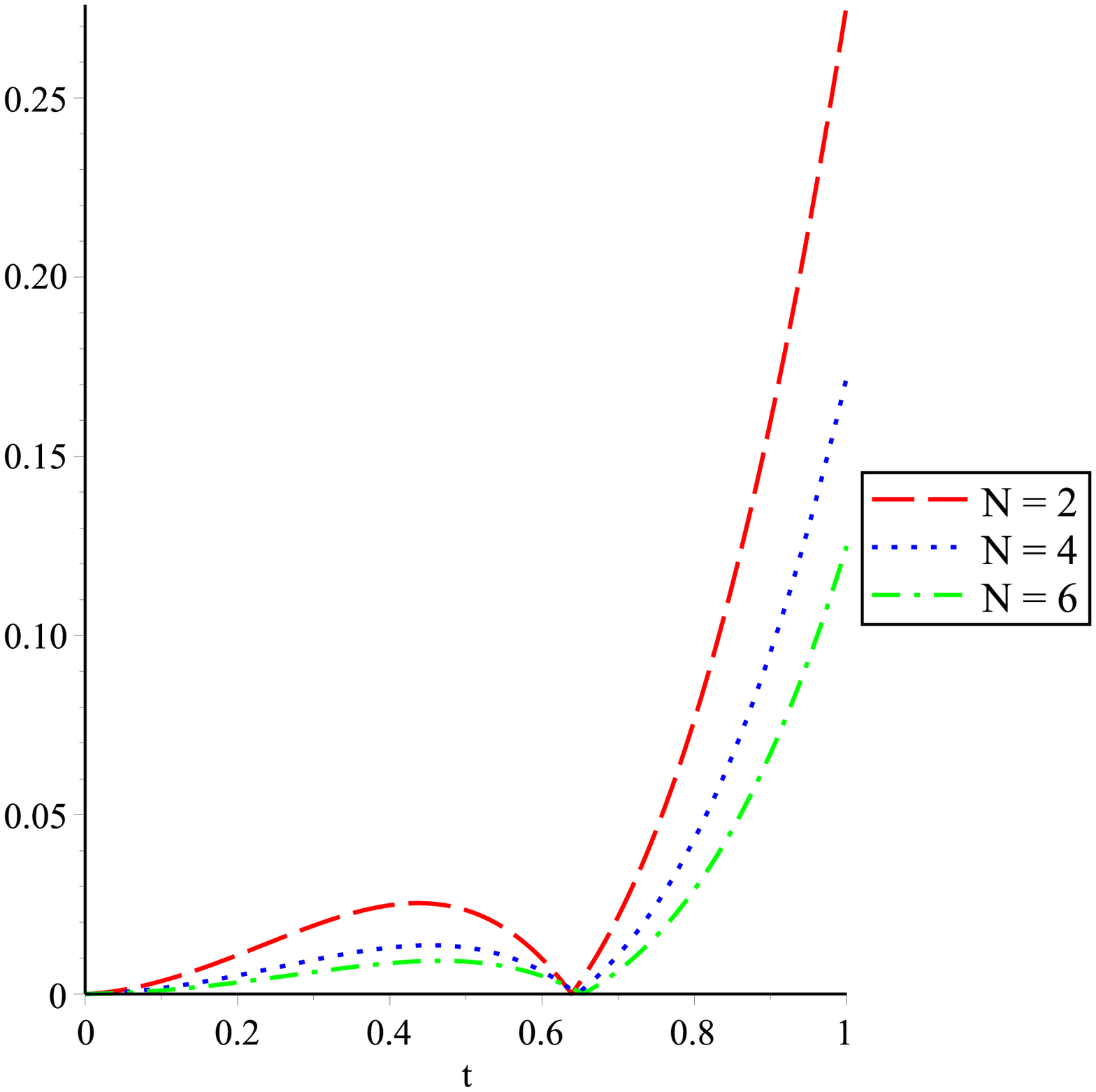}}
\end{center}
\caption{Type II left Caputo derivative of order $\ati$
for the example of Section~\ref{sec:ex_numer}---analytic
versus numerical approximations obtained from Theorem~\ref{N_teo3}.}
\label{IntExp4}
\end{figure}
\begin{figure}[!ht]
\begin{center}
\subfigure[${^C_0\mathbb{D}_t^{\beta(t)}} x(t)$]{\includegraphics[scale=0.25]{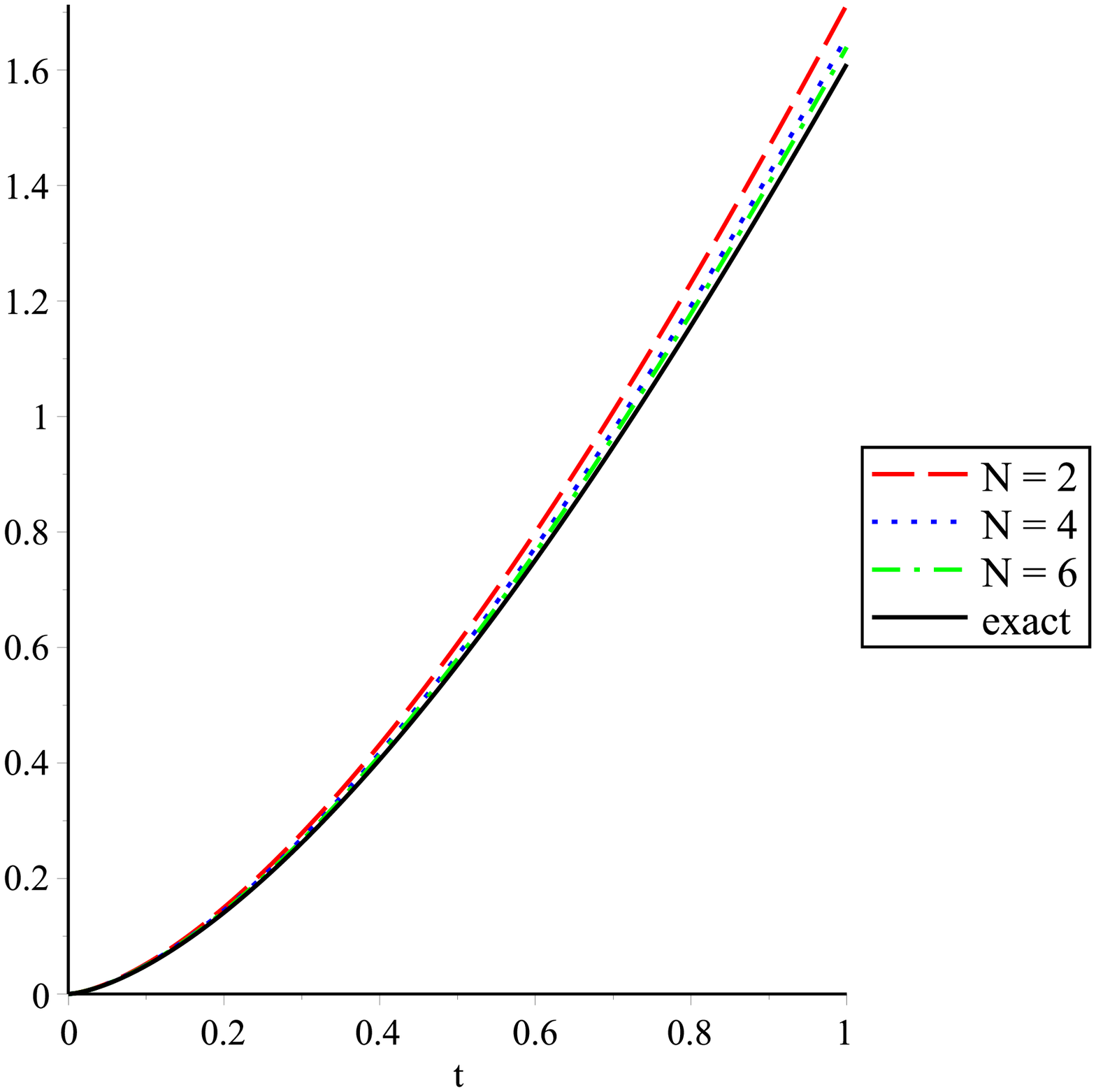}}\hspace{1cm}
\subfigure[Error]{\includegraphics[scale=0.25]{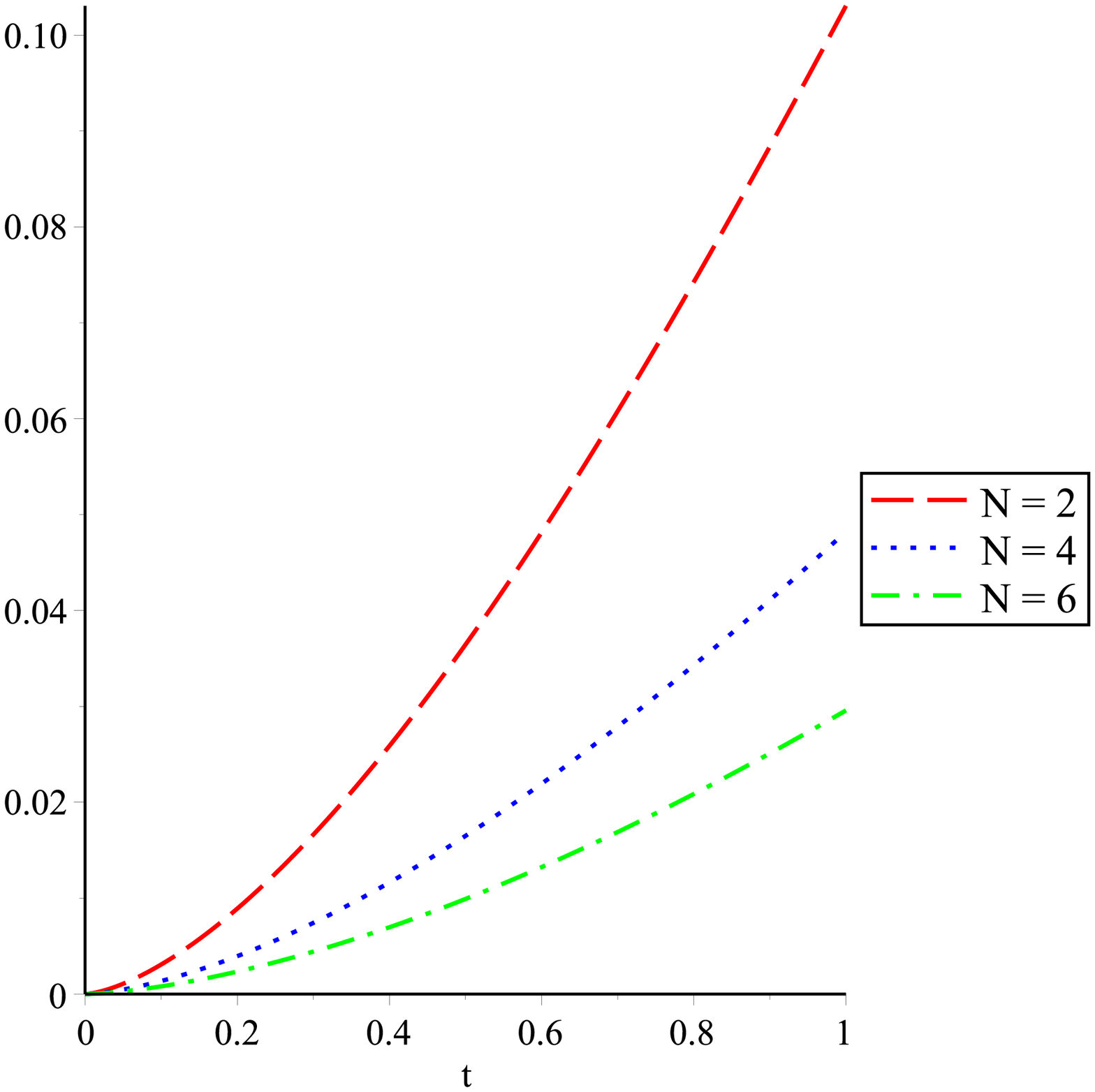}}
\end{center}
\caption{Type III left Caputo derivative of order $\beta(t)$
for the example of Section~\ref{sec:ex_numer}---analytic
versus numerical approximations obtained from Theorem~\ref{N_teo1}.}
\label{IntExp5}
\end{figure}
\begin{figure}[!ht]
\begin{center}
\subfigure[${^C_0D_t^{\beta(t)}} x(t)$]{\includegraphics[scale=0.25]{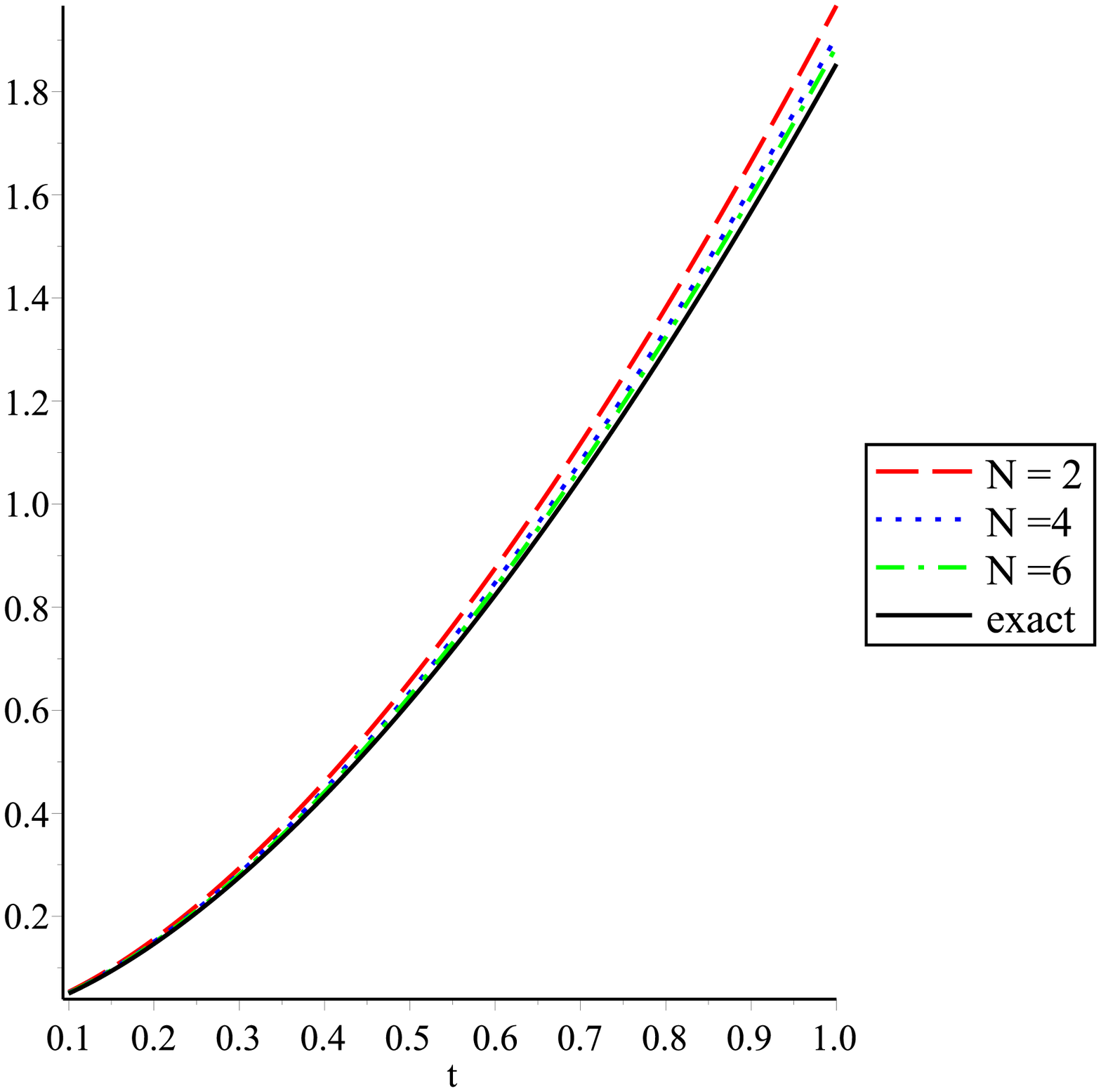}}\hspace{1cm}
\subfigure[Error]{\includegraphics[scale=0.25]{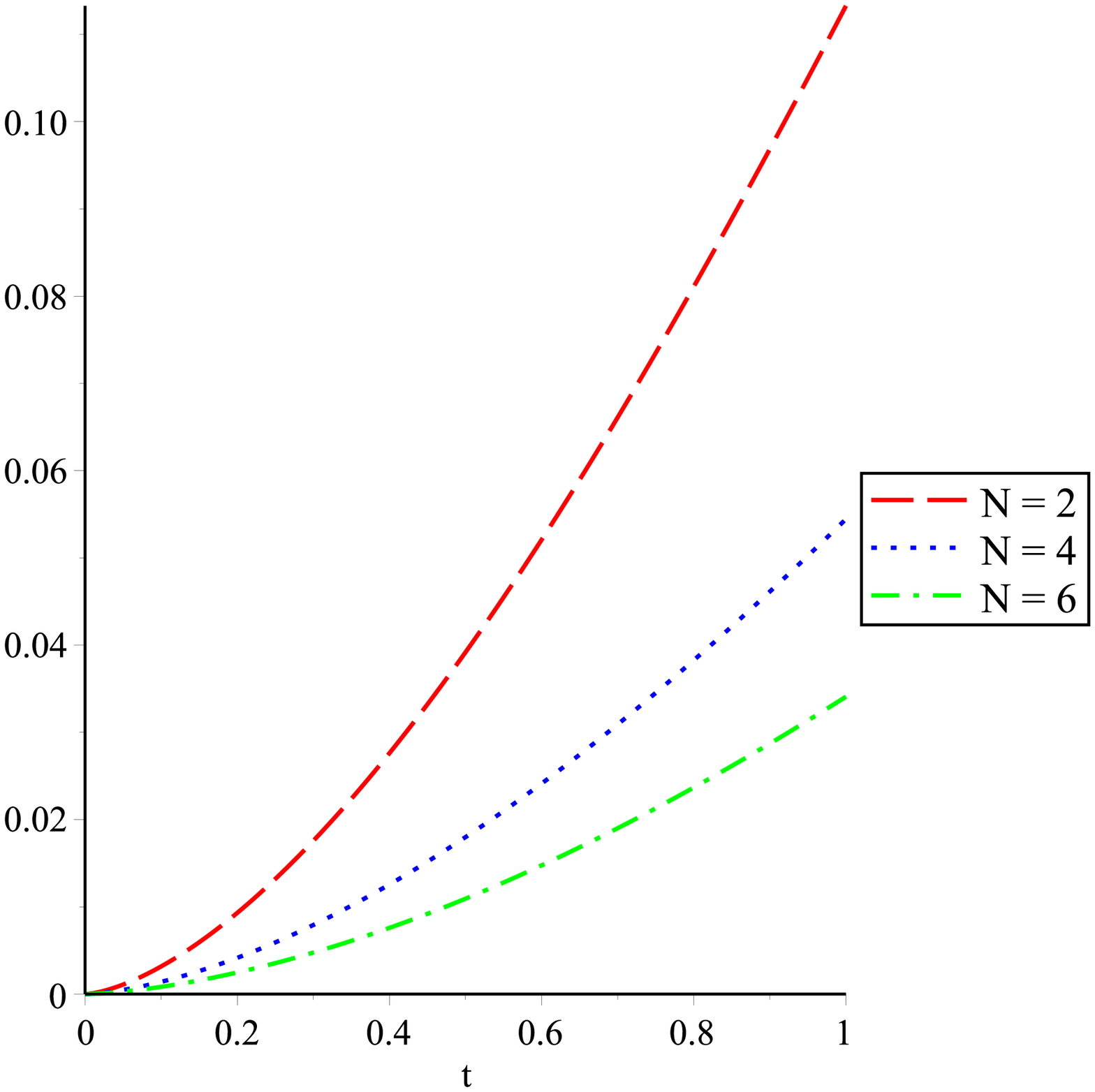}}
\end{center}
\caption{Type I left Caputo derivative of order $\beta(t)$
for the example of Section~\ref{sec:ex_numer}---analytic
versus numerical approximations obtained from Theorem~\ref{N_teo2}.}
\label{IntExp6}
\end{figure}
\begin{figure}[!ht]
\begin{center}
\subfigure[${^C_0\mathcal{D}_t^{\beta(t)}} x(t)$]{\includegraphics[scale=0.25]{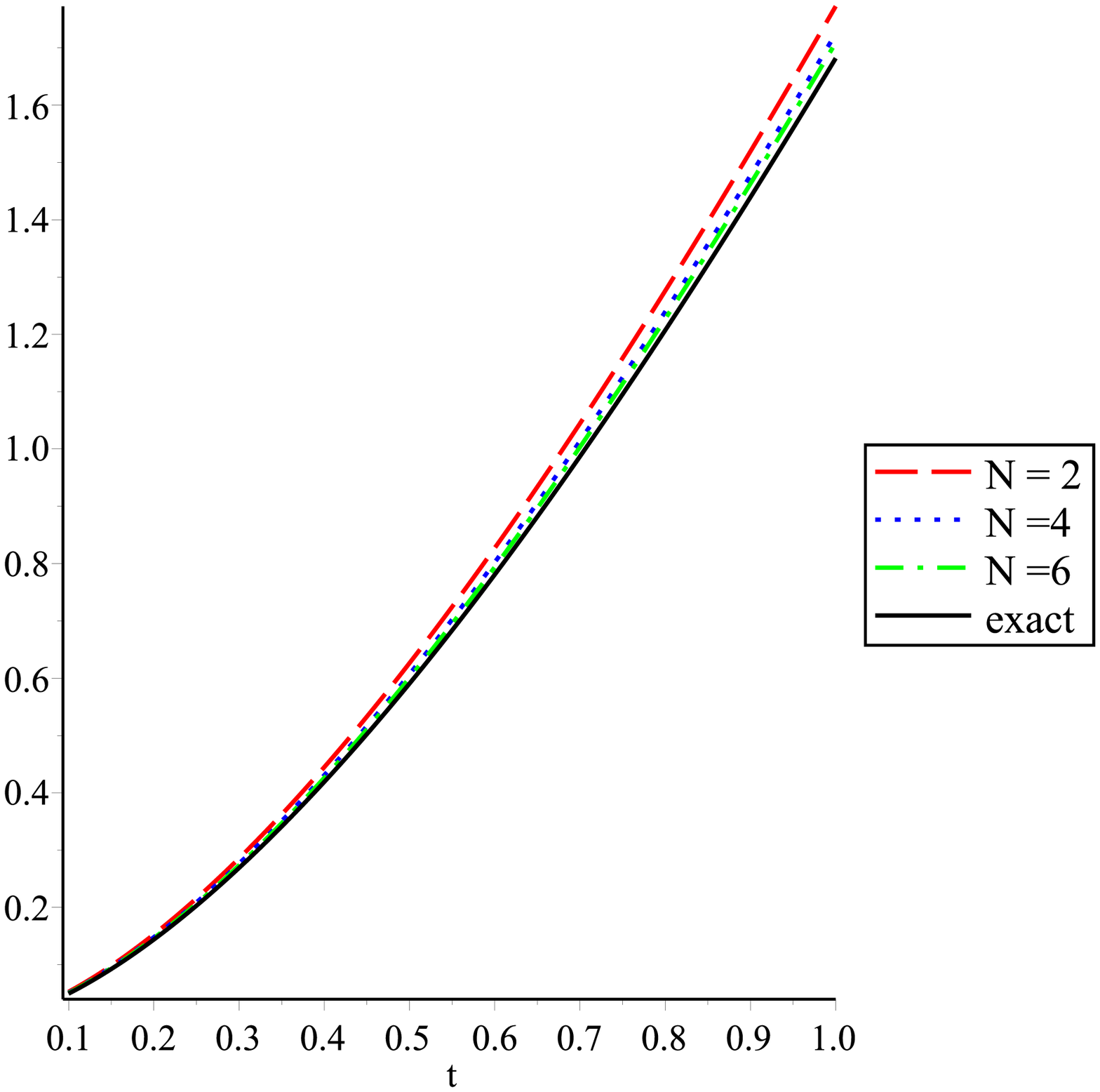}} \hspace{1cm}
\subfigure[Error]{\includegraphics[scale=0.25]{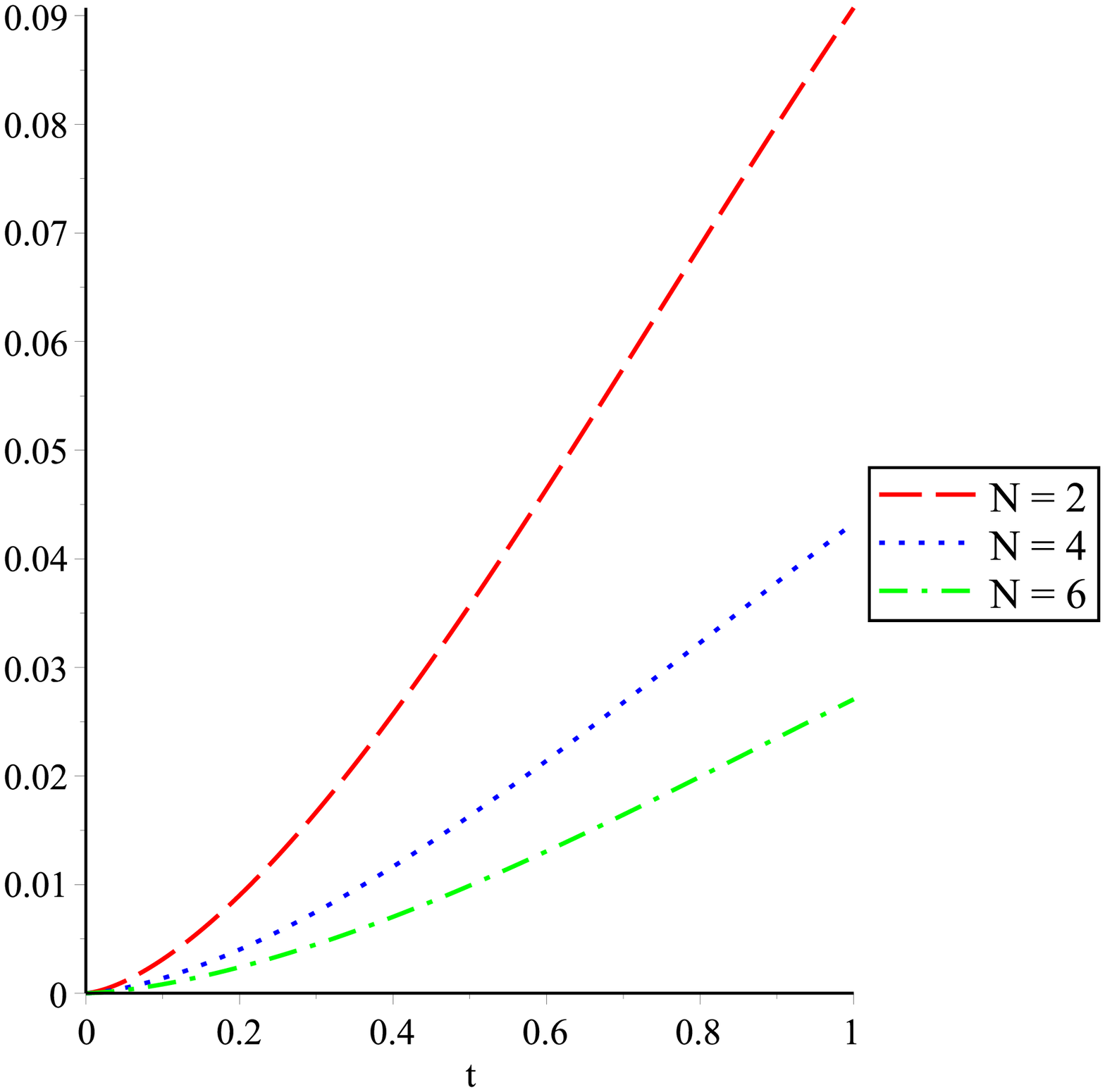}}
\end{center}
\caption{Type II left Caputo derivative of order $\beta(t)$
for the example of Section~\ref{sec:ex_numer}---analytic
versus numerical approximations obtained from Theorem~\ref{N_teo3}.}
\label{IntExp7}
\end{figure}


\section{Applications}
\label{sec:2_Appl}

In this section we apply the proposed technique to some concrete
fractional differential equations of physical relevance.


\subsection{A time-fractional diffusion equation}
\label{subsec:example1}

We extend the one-dimensional time-fractional diffusion equation \cite{Cap3:Lin}
to the variable-order case. Consider
$u=u(x,t)$ with domain $[0,1]^2$. The partial fractional differential
equation of order $\ati$ is the following:
\begin{equation}
\label{eq:diffeq}
{^C_{0}\mathbb{D}_{t}^{\ati}u}(x,t)-\frac{\partial^2u}{\partial x^2}(x,t)
=f(x,t) \quad \mbox{ for } \, x\in[0,1],
\quad t\in [0,1],
\end{equation}
subject to the boundary conditions \index{boundary conditions}
\begin{equation}
\label{eq:diffeq:bc1}
u(x,0)=g(x), \quad \mbox{ for } \, x\in(0,1),
\end{equation}
and
\begin{equation}
\label{eq:diffeq:bc2}
u(0,t)=u(1,t)=0,\quad \mbox{ for } \, t\in[0,1].
\end{equation}
We mention that when $\ati\equiv1$, one obtains the classical diffusion equation,
and when $\ati \equiv 0$ one gets the classical Helmholtz elliptic equation.
Using Lemma~\ref{2_1_Lemma_power}, it is easy to check that
$$
u(x,t)=t^2\sin(2\pi x)
$$
is a solution to \eqref{eq:diffeq}--\eqref{eq:diffeq:bc2} with
$$
f(x,t)=\left(\frac{2}{\Gamma(3-\ati)}t^{2-\ati}
+4\pi^2t^2\right)\sin(2\pi x)
$$
and
$$
g(x)=0
$$
(compare with Example~1 in \cite{Cap3:Lin}).
The numerical procedure is the following: replace ${^C_{0}\mathbb{D}_{t}^{\ati}u}$
with the approximation given in Theorem~\ref{N_teo1},
taking $n=1$ and an arbitrary $N\geq 1$, that is,
$$
{^C_{0}\mathbb{D}_{t}^{\ati}u}(x,t)\approx A t^{1-\ati}
\frac{\partial u}{\partial t}(x,t) +\sum_{p=1}^N B_p t^{1-p-\ati} V_p(x,t)
$$
with
\begin{equation*}
\begin{split}
A& =\DS \frac{1}{\Gamma(2-\ati)}\left[
1+\sum_{l=1}^N \frac{\Gamma(\ati-1+l)}{\Gamma(\ati-1)l!}  \right],\\
B_p & =  \DS\frac{\Gamma(\ati-1+p)}{\Gamma(1-\ati)\Gamma(\ati)(p-1)!},\\
V_p(x,t)& = \DS\int_{0}^{t}\t^{p-1}\frac{\partial u}{\partial t}(x,\t)d\t.
\end{split}
\end{equation*}
Then, the initial fractional problem \eqref{eq:diffeq}--\eqref{eq:diffeq:bc2}
is approximated by the following system of second-order partial differential equations:
$$
A t^{1-\ati}\frac{\partial u}{\partial t}(x,t)
+\sum_{p=1}^N B_p t^{1-p-\ati} V_p(x,t)
-\frac{\partial^2u}{\partial x^2}(x,t)=f(x,t)
$$
and
$$
\frac{\partial V_p}{\partial t}(x,t)=t^{p-1}
\frac{\partial u}{\partial t}(x,t),
\quad p=1,\ldots, N,
$$
for $x\in[0,1]$ and for $t\in [0,1]$,
subject to the boundary conditions \index{boundary conditions}
$$
u(x,0)=0, \quad \mbox{ for } \, x\in(0,1),
$$
$$
u(0,t)=u(1,t)=0,\quad \mbox{ for } \, t\in[0,1],
$$
and
$$
V_p(x,0)=0, \quad \mbox{ for } \, x\in[0,1], \quad p=1,\ldots, N.
$$


\subsection{A fractional partial differential equation in fluid mechanics}
\label{subsec:fluid:mech}

We now apply our approximation techniques to the following
one-dimensional linear inhomogeneous fractional Burgers'
equation of variable-order (see \cite[Example~5.2]{Cap3:Odibat}):
\begin{equation}
\label{eq:fluid:mech}
{^C_{0}\mathbb{D}_{t}^{\ati}u}(x,t)+\frac{\partial u}{\partial x}(x,t)
-\frac{\partial^2u}{\partial x^2}(x,t)=\frac{2t^{2-\ati}}{\Gamma(3-\ati)}
+2x-2,
\end{equation}
for $x\in[0,1]$ and $t\in [0,1]$,
subject to the boundary condition \index{boundary conditions}
\begin{equation}
\label{eq:fluid:mech:bc}
u(x,0)=x^2, \quad \mbox{ for } \, x\in(0,1).
\end{equation}
Here,
$$
F(x,t)=\frac{2t^{2-\ati}}{\Gamma(3-\ati)}+2x-2
$$
is the external force field. Burgers' equation is used to model gas dynamics,
traffic flow, turbulence, fluid mechanics, etc. The exact solution is
$$
u(x,t)=x^2+t^2.
$$
The fractional problem \eqref{eq:fluid:mech}--\eqref{eq:fluid:mech:bc}
can be approximated by
\begin{multline*}
A t^{1-\ati}\frac{\partial u}{\partial t}(x,t) +\sum_{p=1}^N B_p t^{1-p-\ati} V_p(x,t)
+\frac{\partial u}{\partial x}(x,t)-\frac{\partial^2u}{\partial x^2}(x,t)\\
=\frac{2t^{2-\ati}}{\Gamma(3-\ati)}+2x-2
\end{multline*}
with $A$, $B_p$ and $V_p$, $p\in\{1,\ldots,N\}$, as in Section~\ref{subsec:example1}. 
The approximation error can be decreased as much as desired by increasing the value of $N$.

\begingroup
\renewcommand{\addcontentsline}[3]{}

\endgroup


%% file: 4FractionalCalculusofvariations.tex
\chapter{The fractional calculus of variations}
\label{PartII_FCV}

In this chapter, we consider general fractional problems of the calculus
of variations, where the Lagrangian depends on a combined Caputo fractional
derivative of variable fractional order $^CD_\gamma^{\a,\b}$ given as a
combination of the left and the right Caputo fractional derivatives of orders, 
respectively, $\a$ and $\b$. More specifically, here we study some problems
of the calculus of variations with integrands depending on the independent
variable $t$, an arbitrary function $x$ and a fractional derivative
$^CD_\gamma^{\a,\b}x$ defined by
$$
^{C}D_{\gamma}^{\a,\b}x(t)=\gamma_1 \, \LC x(t)+\gamma_2 \, {^C_tD_b^{\b}} x(t),
$$
where $\gamma =\left(\gamma_1,\gamma_2 \right)\in [0,1]^2$, 
with $\gamma_1$ and $\gamma_2$ not both zero.

Starting from the fundamental variational problem, we investigate different
types of variational problems: problems with time delay or with higher-order
derivatives, isoperimetric problems, problems with holonomic constraints and
problems of Herglotz and those depending on combined Caputo fractional
derivatives of variable-order.
Our variational problems are known as free-time problems because, in general,
we impose a boundary condition at the initial time $t=a$, but we consider
the endpoint $b$ of the integral and the terminal state $x(b)$ to be free
(variable). The main results provide necessary optimality conditions of
Euler--Lagrange type, described by fractional differential equations 
of variable-order, and different transversality optimality conditions.

Our main contributions are to consider the Lagrangian containing a fractional
operator where the order is not a constant, and may depend on time. Moreover,
we do not only assume that $x(b)$ is free, but the endpoint $b$ is also free.

In Section~\ref{sec:introd}, we introduce the combined Caputo fractional
derivative of variable-order and provide the necessary concepts and results
needed in the sequel (Section~\ref{subsec:Combined}). We deduce two formulas
of integration by parts that involve the combined Caputo fractional derivative
of higher order (Section~\ref{subsec:intbyparts}).

The fractional variational problems under our consideration are formulated
in terms of the fractional derivative $^CD_\gamma^{\a,\b}$. We discuss the
fundamental concepts of a variational calculus such as the Euler--Lagrange
equations and transversality conditions (Section~\ref{sec:FP}), variational
problems involving higher-order derivatives (Section~\ref{sec:HO}), variational
problems with time delay (Section~\ref{sec:delay}), isoperimetric problems
(Section~\ref{sec:Iso}), problems with holonomic constraints (Section~\ref{sec:Hol})
and the last Section \ref{sec:Herg} investigates fractional variational Herglotz
problems.

Some illustrative examples are presented for all considered variational problems.

The results of this chapter first appeared in
\cite{Cap4:Tavares2015,Cap4:Tavares2017,Cap4:Tavares_4,Cap4:Tavares_5}.


\section{Introduction}
\label{sec:introd}

In this section, we recall the fundamental definition of the combined Caputo
fractional derivative presented in Section \ref{sec:combC_classic}
(see Definition~\ref{CombClassical}), and generalize it for fractional 
variable-order.
In the end, we prove an integration by parts formula, involving the
higher-order Caputo fractional derivative of variable-order.

\subsection{Combined operators of variable-order}
\label{subsec:Combined}

Motivated by the combined fractional Caputo derivative, we propose the following
definitions about combined variable-order fractional calculus.

Let $\alpha, \, \beta: [a,b]^2\rightarrow(0,1)$ be two functions and
$\gamma=(\gamma_1,\gamma_2) \in [0,1]^{2}$ a vector, with $\gamma_1$ and $\gamma_2$ not both zero.

\begin{Definition}
\label{Comb_RL1}
\index{combined Riemann--Liouville fractional derivative! variable-order}
The combined Rie\-mann--Liouville fractional derivative of variable-order,
denoted by $D_\gamma^{\a,\b}$, is defined by
$$
D_\gamma^{\a,\b}=\gamma_1 \, \LDa + \gamma_2 \, \RDb,
$$
acting on $x\in C([a,b];\mathbb{R})$ in the following way:
$$
D_\gamma^{\a,\b} x(t)=\gamma_1 \, \LDa x(t)+\gamma_2 \, \RDb x(t).
$$
\end{Definition}

\begin{Definition}
\label{Comb_Cap}
\index{combined Caputo fractional derivative! variable-order}
The combined Caputo fractional derivative operator of variable-order,
denoted by $^{C}D_\gamma^{\a,\b}$, is defined by
$$
^CD_\gamma^{\a,\b}=\gamma_1 \, \LC+\gamma_2 \, \RCb,
$$
acting on $x\in C^{1}([a,b];\mathbb{R})$ in the following way:
$$
^{C}D_\gamma^{\a,\b}x(t)=\gamma_1 \, \LC x(t)+\gamma_2 \, \RCb x(t).
$$
\end{Definition}

In the sequel, we need the auxiliary notation of the dual fractional derivative,
defined by
\begin{equation}
D_{\overline{\gamma}}^{\b,\a}=\gamma_2 \, {_aD_t^{\b}}+\gamma_1 \, {_tD_T^{\a}},
\end{equation}
where $\overline{\gamma}=(\gamma_2,\gamma_1)$ and $T\in(a,b]$.

To generalize the concept of the combined fractional derivative to higher orders,
we need to review some definitions of higher-order operators.

Let $ n \in \mathbb{N}$ and $x: [a,b] \rightarrow \mathbb{R}$ be a function
of class $C^n$. The fractional order is a continuous function of two variables,
$\alpha_n:[a,b]^2 \to (n-1,n)$.

\begin{Definition}
\index{Riemann--Liouville fractional integral! left (higher-order)}
\index{Riemann--Liouville fractional integral! right (higher-order)}
\label{def:RL:fi}
The left and right Riemann--Liouville fractional integrals of order $\an$
are defined respectively by
$$
\LIan x(t)=\int_a^t  \frac{1}{\Gamma(\alpha_n(t,\t))}(t-\t)^{\alpha_n(t,\t)-1}x(\t)d\t
$$
and
$$
\RIan x(t)=\int_t^b\frac{1}{\Gamma(\alpha_n(\t,t))}(\t-t)^{\alpha_n(\t,t)-1}x(\t)d\t.
$$
\end{Definition}

With this definition of Riemann--Liouville fractional integrals of variable-order,
and considering $n=1$, in Appendix~A.1 we implemented two functions \texttt{leftFI(x,alpha,a)}
and \texttt{rightFI(x,alpha,b)} that approximate, respectively, 
the Riemann--Liouville fractional integrals $\LIan x$ and $\RIan x$, 
using the open source software package \textsf{Chebfun} \cite{Cap4:chebfun:book}. 
With these two functions we present also, in Appendix~A.1, an illustrative example 
where we determine computational approximations to the fractional integrals 
for a specific power function of the form
$x(t)=t^\gamma$ (see Example~\ref{ex:cheb02} in Appendix~A.1).
For this numerical computations, we have used \textsf{MATLAB} \cite{Cap4:Matlab}.

\begin{Definition}
\index{Riemann--Liouville fractional derivative! left (higher-order)}
\index{Riemann--Liouville fractional derivative! right (higher-order)}
\label{HORLFD}
The left and right Riemann--Liouville fractional derivatives of order $\an$ are
defined by
$$
\LDan x(t)=\frac{d^n}{dt^n}\int_a^t \frac{1}{\Gamma(n-\alpha_n(t,\t))}
(t-\t)^{n-1-\alpha_n(t,\t)}x(\t)d\t
$$
and
$$
\RDan x(t)=(-1)^{n}\frac{d^n}{dt^n}\int_t^b\frac{1}{\Gamma(n-\alpha_n(\t,t))}
(\t-t)^{n-1-\alpha_n(\t,t)} x(\t)d\t,
$$
respectively.
\end{Definition}

\begin{Definition}
\index{Caputo fractional derivative! left (higher-order)}
\index{Caputo fractional derivative! right (higher-order)}
\label{HOCFD}
The left and right Caputo fractional derivatives of order $\an$ are defined by
\begin{equation}
\label{eq:left:Cap:der}
\LCan x(t)=\int_a^t\frac{1}{\Gamma(n-\alpha_n(t,\t))}
(t-\t)^{n-1-\alpha_n(t,\t)}x^{(n)}(\t)d\t
\end{equation}
and
\begin{equation}
\label{eq:right:Cap:der}
\RCan x(t)=(-1)^{n}\int_t^b\frac{1}{\Gamma(n-\alpha_n(\t,t))}
(\t-t)^{n-1-\alpha_n(\t,t)}x^{(n)}(\t)d\t,
\end{equation}
respectively.
\end{Definition}

\begin{Remark}
Definitions~\ref{HORLFD} and \ref{HOCFD}, for the particular case
of order between 0 and 1, can be found in \cite{Cap4:book:MOT}.
They seem to be new for the higher-order case.
\end{Remark}

Considering Definition~\ref{HOCFD}, in Appendix~A.2 we  implement two new functions 
\texttt{leftCaputo(x,alpha,a,n)} and \texttt{rightCaputo(x,alpha,b,n)} 
that approximate the higher-order Caputo fractional derivatives $\LCan x$ 
and $\RCan x$, respectively.
With these two functions, we present, also in Appendix~A.2, an illustrative
example where we study approximations to the Caputo fractional derivatives
for a specific power function of the form $x(t)=t^\gamma$ (see Example~\ref{ex:cheb01}
in Appendix~A.2).

\begin{Remark}
From Definition~\ref{def:RL:fi}, it follows that
$$
\LDan x(t)=\frac{d^{n}}{dt^{n}}\, _aI_t^{n-\an} x(t),
\, \RDan x(t)=(-1)^{n}\frac{d^{n}}{dt^{n}}\, _tI_b^{n-\an} x(t)
$$
and
$$
\LCan x(t)=\, _aI_t^{n-\an} \frac{d^{n}}{dt^{n}}x(t),
\, \RCan x(t)=(-1)^{n}\, _tI_b^{n-\an} \frac{d^{n}}{dt^{n}}x(t).
$$
\end{Remark}

In Lemma~\ref{lemma:power}, we obtain higher-order Caputo fractional derivatives
of a power function.
For that, we assume that the fractional order depends only on the first variable:
$\alpha_n(t,\t) := \overline{\alpha}_n(t)$, where $\overline{\alpha}_n:[a,b]\to (n-1,n)$
is a given function.

\begin{Lemma}
\label{lemma:power}
Let $x(t)=(t-a)^\gamma$ with $\gamma>n-1$. Then,
$$
{^C_aD_t^{\overline{\alpha}_n(t)}}x(t)
= \frac{\Gamma(\gamma+1)}{\Gamma(\gamma-\overline{\alpha}_n(t)+1)}
(t-a)^{\gamma-\overline{\alpha}_n(t)}.
$$
\end{Lemma}

\begin{proof}
As $x(t)=(t-a)^\gamma$, if we differentiate it $n$ times, we obtain
$$
x^{(n)}(t)=\frac{\Gamma(\gamma + 1)}{\Gamma(\gamma-n+1)}(t-a)^{\gamma-n}.
$$
Using Definition~\ref{HOCFD} of the left Caputo fractional 
derivative,\index{Caputo fractional derivative! left (higher-order)} we get
$$
\begin{array}{ll}
{^C_aD_t^{\overline{\alpha}_n(t)}}x(t)
&= \DS \int_a^t \frac{1}{\Gamma(n-\overline{\alpha}_n(t))}
(t-\t)^{n-1-\overline{\alpha}_n(t)}x^{(n)}(\t) d\t\\
&=\DS \int_a^t \frac{\Gamma(\gamma + 1)}{\Gamma(\gamma-n+1)
\Gamma(n-\overline{\alpha}_n(t))}(t-\t)^{n-1-\overline{\alpha}_n(t)}(\t-a)^{\gamma-n} d\t.
\end{array}
$$
Now, we proceed with the change of variables $\t-a=s(t-a)$.
Using the Beta function $B(\cdot,\cdot)$, we obtain that
\index{beta function}
$$
\begin{array}{ll}
{^C_aD_t^{\overline{\alpha}_n(t)}}x(t)
&= \DS \frac{\Gamma(\gamma + 1)}{\Gamma(\gamma-n+1)\Gamma(n-\overline{\alpha}_n(t))}\\
& \quad \times \int_0^1 (1-s)^{n-1-\overline{\alpha}_n(t)} s^{\gamma-n} (t-a)^{\gamma-\overline{\alpha}_n(t)} ds\\
&= \DS \frac{\Gamma(\gamma + 1)(t-a)^{\gamma-\overline{\alpha}_n(t)}}{\Gamma(\gamma-n+1)
\Gamma(n-\overline{\alpha}_n(t))} B \left(\gamma-n+1, n-\overline{\alpha}_n(t)\right)\\
&= \DS \frac{\Gamma(\gamma + 1)}{\Gamma(\gamma-\overline{\alpha}_n(t)+1)}
\quad (t-a)^{\gamma-\overline{\alpha}_n(t)}.
\end{array}
$$
The proof is complete.
\end{proof}

Considering the higher-order left Caputo fractional derivative's formula of a
power function of the form $x(t)=(t-a)^\gamma$, deduced before, in Appendix~A.2
we determine the left Caputo fractional derivative of the particular function
$x(t)=t^4$ for several values of $t$ and compare them with the approximated values
obtained by the \textsf{Chebfun} function \texttt{leftCaputo(x,alpha,a,n)}
(see Example~\ref{ex:lemma:power} in Appendix A.2).

For our next result, we assume that the fractional order depends only on the
second variable:
$\alpha_n(\t,t) := \overline{\alpha}_n(t)$, where  $\overline{\alpha}_n:
[a,b] \to (n-1,n)$ is a given function. The proof is similar to that of
Lemma~\ref{lemma:power}, and so we omit it here.

\begin{Lemma}
\label{lemma:power:2}
Let $x(t)=(b-t)^\gamma$ with $\gamma>n-1$. Then,
$$
{^C_tD_b^{\overline{\alpha}_n(t)}}x(t)
= \frac{\Gamma(\gamma+1)}{\Gamma(\gamma-\overline{\alpha}_n(t)+1)}
(b-t)^{\gamma-\overline{\alpha}_n(t)}.
$$
\end{Lemma}

The next step is to consider a linear combination of the previous fractional
derivatives to define the combined fractional operators for higher-order.

Let $\alpha_n,\beta_n: [a,b]^2\rightarrow(n-1,n)$ be two variable fractional
orders, $\gamma^n =\left(\gamma_1^n,\gamma_2^n \right)\in [0,1]^2$ a vector, 
with $\gamma_1$ and $\gamma_2$ not both zero, 
and $x \in C^n\left([a,b];\mathbb{R}\right) $ a function.

\begin{Definition}
\label{def3}
\index{combined Riemann--Liouville fractional derivative! higher-order}
The hi\-gher-order combined Riemann--Liouville fractional derivative is defined by
$$D_{\gamma^n}^{\an,\bn}=\gamma_1^n \, \LDan+\gamma_2^n \, \RDbn,$$
acting on $x \in C^n \left([a,b];\mathbb{R}\right)$ in the following way:
$$D_{\gamma^n}^{\an,\bn}x(t)=\gamma_1^n \, \LDan x(t)+\gamma_2^n \, \RDbn x(t).$$
\end{Definition}

In our work, we use both Riemann--Liouville and Caputo derivatives definitions.
The emphasis is, however, in Caputo fractional derivatives.

\begin{Definition}
\label{def4}
\index{combined Caputo fractional derivative! higher-order}
The higher-order combined Caputo fractional derivative of $x$ at $t$ is defined by
$$
^{C}D_{\gamma^n}^{\an,\bn}x(t)=\gamma_1^n \, \LCan x(t)+\gamma_2^n \, \RCbn x(t).
$$
\end{Definition}

Similarly, in the sequel of this work, we need the auxiliary notation of the dual
fractional derivative:
\begin{equation}
D_{\overline{\gamma^i}}^{\bi,\ai}=\gamma_2^i \, {_aD_t^{\bi}}
+\gamma_1^i \, {_tD_T^{\ai}},
\end{equation}
where $\overline{\gamma^i}=(\gamma_2^i,\gamma_1^i)$
and $T\in(a,b]$.

Some computational aspects about the combined Caputo fractional derivative 
of variable-order, are also discussed in Appendix~A.3, using the software package
\textsf{Chebfun}. For that, we developed the new function (see Example~\ref{ex:matlab:comb} in Appendix~A.3) \texttt{combinedCaputo(x,alpha,beta,gamma1,gamma2,a,b,n)} and obtained approximated
values for a particular power function.


\subsection{Generalized fractional integration by parts}
\label{subsec:intbyparts}

When dealing with variational problems, one key property is integration by parts.
Formulas of integration by parts have an important role in the proof of Euler--Lagrange
conditions. In the following theorem, such formulas are proved for integrals
involving higher-order Caputo fractional derivatives of variable-order.

Let $n \in \mathbb{N}$ and $x,y \in C^n\left([a,b]; \bR\right)$ be two functions.
The fractional order is a continuous function of two variables, $\alpha_n: [a,b]^2 \to (n-1,n)$.

\begin{Theorem}
\label{thm:FIP_HO}
\index{integration by parts! Caputo fractional derivative of higher-order}
The higher-order Caputo fractional derivatives of variable-order satisfy the
integration by parts formulas
\begin{equation*}
\begin{split}
\int_{a}^{b}y(t) \, \LCan x(t)dt&=\int_a^b x(t) \, {\RDan}y(t)dt\\
&+\left[\sum_{k=0}^{n-1} (-1)^{k} x^{(n-1-k)}(t) \, \dfrac{d^{k}}
{dt^{k}}{_tI_b^{n-\an}}y(t) \right]_{t=a}^{t=b}
\end{split}
\end{equation*}
and
\begin{equation*}
\begin{split}
\int_{a}^{b}y(t) \, {\RCan}x(t)dt &=\int_a^b x(t) \, {\LDan} y(t)dt\\
&+\left[\sum_{k=0}^{n-1} (-1)^{n+k}x^{(n-1-k)}(t) \, \dfrac{d^{k}}{dt^{k}}
{_aI_t^{n-\an}}y(t)\right]_{t=a}^{t=b}.
\end{split}
\end{equation*}
\end{Theorem}

\begin{proof}
Considering the Definition \ref{HOCFD} of left Caputo fractional derivatives
of order $\an$, we obtain
\begin{multline*}\int_{a}^{b}y(t) \, \LCan x(t)dt\\
=\int_a^b \int_a^t y(t) \,\dfrac{1}{\Gamma(n-\alpha_n(t,\t))} (t-\t)^{n-1-\alpha_n(t,\t)}x^{(n)}(\t)d\t dt.\end{multline*}
Using Dirichelt's Formula, we rewrite it as
\begin{equation}
\label{eq:star}
\begin{split}
\int_a^b & \int_t^b y(\t) \, \dfrac{(\t-t)^{n-1-\alpha_n(\t,t)}}{\Gamma(n-\alpha_n(\t,t))} x^{(n)}(t)d\t dt\\
&=\int_a^b x^{(n)}(t) \int_t^b  \, \dfrac{(\t-t)^{n-1-\alpha_n(\t,t)}}{\Gamma(n-\alpha_n(\t,t))} y(\t)d\t dt
=\int_a^b x^{(n)}(t)  \, {_tI_b^{n-\an}} y(t) dt.
\end{split}
\end{equation}
Using the (usual) integrating by parts formula, we get that \eqref{eq:star} is
equal to
$$
-\int_a^b x^{(n-1)}(t) \dfrac{d}{dt}  {_tI_b^{n-\an}} y(t)dt + \left[ x^{(n-1)}(t) {_tI_b^{n-\an}}y(t)\right]_{t=a}^{t=b}.
$$
Integrating by parts again, we obtain
\begin{multline*}
\int_a^b x^{(n-2)}(t) \, \dfrac{d^2}{dt^2}  {_tI_b^{n-\an}} y(t)dt\\
+ \left[ x^{(n-1)}(t) \, {_tI_b^{n-\an}}y(t) - x^{(n-2)}(t) \dfrac{d}{dt} {_tI_b^{n-\an}}y(t)\right]_{t=a}^{t=b}.
\end{multline*}
If we repeat this process $n-2$ times more, we get
\begin{equation*}
\begin{split}
&\int_a^b x(t) (-1)^n \, \dfrac{d^n}{dt^n}  {_tI_b^{n-\an}} y(t)dt \\
&\quad+ \left[ \sum_{k=0}^{n-1} (-1)^{k} x^{(n-1-k)}(t) \, \dfrac{d^{k}}{dt^{k}} {_tI_b^{n-\an}}y(t)\right]_{t=a}^{t=b} \\
=&\int_a^b x(t) \RDan y(t) dt + \left[ \sum_{k=0}^{n-1} (-1)^{k} x^{(n-1-k)}(t) \, \dfrac{d^{k}}{dt^{k}} {_tI_b^{n-\an}}y(t)\right]_{t=a}^{t=b}.
\end{split}
\end{equation*}
The second relation of the theorem for the right Caputo fractional
derivative of order $\an$, follows directly from the first
by Caputo--Torres duality \cite{Cap4:Cap:Tor}.
\end{proof}

\begin{Remark}
If we consider in Theorem~\ref{thm:FIP_HO} the particular case when $n=1$,
then the fractional integration by parts formulas take the well-known forms
presented in Theorem~\ref{thm:FIP}.
\end{Remark}

\begin{Remark}
\index{integration by parts! Caputo fractional derivatives}
If $x$ is such that $x^{(i)}(a)=x^{(i)}(b)=0$, $i=0, \ldots, n-1$,
then the higher-order formulas of fractional integration by parts given
by Theorem~\ref{thm:FIP_HO} can be rewritten as
$$
\int_{a}^{b}y(t) \, \LCan x(t)dt=\int_a^b x(t) \, {\RDan}y(t)dt
$$
and
$$
\int_{a}^{b}y(t) \, {\RCan}x(t)dt=\int_a^b x(t) \, {\LDan} y(t)dt.
$$
\end{Remark}


\section{Fundamental variational problem}
\label{sec:FP}

This section is dedicated to establish necessary optimality conditions for
variational problems with a Lagrangian depending on a combined Caputo derivative
of variable fractional order.
The problem is then stated in Section~\ref{subsec:NecCond_FP}, consisting of
the variational functional
$$
\mathcal{J}(x,T)=\int_a^T L\left(t, x(t), \DC x(t)\right)dt+\phi(T,x(T)),
$$
where $\DC x(t)$ stands for the combined Caputo fractional derivative
of variable fractional order (Definition~\ref{Comb_Cap}), subject to the
boundary condition \index{boundary conditions} $x(a)=x_a$.

In this problem, we do not only assume that $x(T)$ is free, but the
endpoint $T$ is also variable. Therefore, we are interested in finding
an optimal curve $x(\cdot)$ and also the endpoint of the variational
integral, denoted in the sequel by $T$.

We begin by proving in Section~\ref{subsec:NecCond_FP} necessary optimality
conditions that every extremizer $(x,T)$ must satisfy.
The main results of this section provide necessary optimality conditions
of Euler--Lagrange type, described by fractional differential equations
of variable-order, and different transversality optimality conditions
(Theorems~\ref{FP_teo1} and \ref{FP_teo2}). Some particular cases of
interest are considered in Section~\ref{subsec:part:cases}.
We end with two illustrative examples (Section~\ref{subsec_FPex}).


\subsection{Necessary optimality conditions}
\label{subsec:NecCond_FP}

{Let $\alpha, \, \beta: [a,b]^2\rightarrow(0,1)$ be two functions. Let $D$ denote the set
\begin{equation}
\label{set}
D=\left\{ (x,t)\in  C^1([a,b])\times [a,b] : \DC x \in C([a,b])\right\},
\end{equation}
endowed with the norm $\Vert(\cdot,\cdot)\Vert$ defined on the linear space $ C^1([a,b])\times \bR$ by
$$
\|(x,t)\|:=\max_{a\leq t \leq  b}|x(t)|+\max_{a\leq t \leq b}\left| \DC x(t)\right|+|t|.
$$

\begin{Definition}
We say that $(x^\star,T^\star)\in D$ is a local minimizer to the functional
$\mathcal{J}:D\rightarrow \bR$ if there exists some $\epsilon>0$ such that
$$
\forall (x,T)\in D \, : \quad \|(x^\star,T^\star)-(x,T)\|<\epsilon\Rightarrow J(x^\star,T^\star)\leq J(x,T).
$$
\end{Definition}

Along the work, we denote by $\partial_i z$, $i\in \{1,2,3\}$, the partial
derivative of a function $z:\bR^{3} \rightarrow\bR$ with respect to its $i$th
argument, and by $L$ a differentiable Lagrangian $L: [a,b]\times \bR^2
\rightarrow \bR$.

Consider the following problem of the calculus of variations:

\index{variational fractional problem! fundamental}
\begin{Problem}
\label{P_fund}
Find the local minimizers of the functional $\mathcal{J}:D\rightarrow \bR$, with
\begin{equation}
\label{functional1}
\mathcal{J}(x,T)=\int_a^T L\left(t, x(t), \DC x(t)\right) dt + \phi(T,x(T)),
\end{equation}
over all $(x,T)\in D$ satisfying the boundary condition 
$x(a)=x_a$,\index{boundary conditions} for a fixed $x_a\in \bR$. 
The terminal time $T$ and terminal state $x(T)$
are free.\\
The terminal cost function \index{terminal cost function} $\phi:[a,b]\times \bR\to\bR$
is at least of class $C^1$.
\end{Problem}

For simplicity of notation, we introduce the operator $[\cdot]_\gamma^{\alpha, \beta}$ defined by
\begin{equation}
\label{operat_simpli}
[x]_\gamma^{\alpha, \beta}(t)=\left(t, x(t), \DC x(t)\right).
\end{equation}

With the new notation, one can write \eqref{functional1} simply as
$$\mathcal{J}(x,T)=\int_a^T L[x]_\gamma^{\alpha, \beta}(t) dt + \phi(T,x(T)).$$
The next theorem gives fractional necessary optimality conditions to
Problem~\ref{P_fund}.

\begin{Theorem}\label{FP_teo1}\index{Euler--Lagrange equations}
Suppose that $(x,T)$ is a local minimizer to the functional \eqref{functional1}
on $D$.
Then, $(x,T)$ satisfies the fractional Euler--Lagrange equations
\begin{equation}
\label{FP_ELeq_1}
\partial_2 L[x]_\gamma^{\alpha, \beta}(t)
+D{_{\overline{\gamma}}^{\b,\a}}\partial_3 L[x]_\gamma^{\alpha, \beta}(t)=0,
\end{equation}
on the interval $[a,T]$, and
\begin{equation}
\label{FP_ELeq_2}
\gamma_2\left({\LDb}\partial_3 L[x]_\gamma^{\alpha, \beta}(t)
-{ _TD{_t^{\b}}\partial_3 L[x]_\gamma^{\alpha, \beta}(t)}\right)=0,
\end{equation}
on the interval $[T,b]$. Moreover, $(x,T)$ satisfies the transversality conditions
\index{transversality conditions}
\begin{equation}
\label{FP_CT1}
\begin{cases}
L[x]_\gamma^{\alpha, \beta}(T)+\partial_1\phi(T,x(T))+\partial_2\phi(T,x(T))x'(T)=0,\\
\left[\gamma_1 \, {_tI_T^{1-\a}} \partial_3L[x]_\gamma^{\alpha, \beta}(t)
-\gamma_2 \, {_TI_t^{1-\b}} \partial_3 L[x]_\gamma^{\alpha, \beta}(t)\right]_{t=T}
+\partial_2 \phi(T,x(T))=0,\\
\gamma_2 \left[ {_TI_t^{1-\b}}\partial_3 L[x]_\gamma^{\alpha, \beta}(t)
-{_aI_t^{1-\b}\partial_3L[x]_\gamma^{\alpha, \beta}(t)}\right]_{t=b}=0.
\end{cases}
\end{equation}
\end{Theorem}

\begin{proof}
Let $(x,T)$ be a solution to the problem and
$\left(x+\e{h},T+\e\Delta{T}\right)$ be an admissible variation,
where $h\in C^1([a,b];\bR)$ is a perturbing curve,
$\triangle T \in\bR$ represents an arbitrarily chosen
small change in $T$ and $\e \in\bR$ represents
a small number.
The constraint $x(a)=x_a$ implies that all admissible variations
must fulfill the condition $h(a)=0$.
Define $j(\cdot)$ on a neighborhood of zero by
\begin{equation*}
\begin{split}
j(\e) &=\mathcal{J}(x+\e h,T+\e \triangle T)\\
&=\int_a^{T+\e \triangle T} L[x+\e h]_\gamma^{\alpha, \beta}(t)\,dt
+\phi\left(T+\e \triangle T,(x+\e h)(T+\e \triangle T)\right).
\end{split}
\end{equation*}
The derivative $j'(\e)$ is
\begin{equation*}
\begin{split}
&\int_a^{T+\e \triangle T} \left(
\partial_2 L[x+\e h]_\gamma^{\alpha, \beta}(t) h(t)
+ \partial_3 L[x+\e h]_\gamma^{\alpha, \beta}(t) \DC h(t) \right)dt\\
&\quad +L[x + {\e h}]_\gamma^{\alpha, \beta}(T + \e \Delta T)\Delta T
+ \partial_1 \phi\left(T+\e \triangle T,
(x+\e h)(T+\e \triangle T)\right) \, \Delta T\\
&\quad +\partial_2\phi\left(T+\e \triangle T,
(x+\e h)(T+\e \triangle T)\right) \, \left[(x+\e h)(T+\e \triangle T)\right]'.
\end{split}
\end{equation*}
Considering the differentiability properties of $j$, a necessary condition
for $(x,T)$ to be a local extremizer is given by
$\left. j'(\e) \right|_{\e=0}=0$, that is,
\begin{multline}
\label{FP_eq_derj}
\int_a^T \left( \partial_2 L[x]_\gamma^{\alpha, \beta}(t) h(t)
+ \partial_3 L[x]_\gamma^{\alpha, \beta}(t) \DC h(t) \right)dt
+L[x]_\gamma^{\alpha, \beta}(T)\Delta T\\
+\partial_1 \phi \left(T, x(T)\right)\Delta T
+ \partial_2\phi(T, x(T))\left[h(T)+x'(T) \triangle T \right]=0.
\end{multline}
The second addend of the integral function \eqref{FP_eq_derj},
\begin{equation}
\label{FP_term}
\int_a^T \partial_3 L[x]_\gamma^{\alpha, \beta}(t) \DC h(t) dt,
\end{equation}
can be written, using the definition of combined
Caputo fractional derivative, as
\begin{equation*}
\begin{split}
\int_a^T & \partial_3 L[x]_\gamma^{\alpha, \beta}(t) \DC h(t) dt\\
&=\int_a^T \partial_3 L[x]_\gamma^{\alpha, \beta}(t)\left[\gamma_1
\, \LC h(t)+\gamma_2 \, \RCb h(t)\right]dt\\
&=\gamma_1 \int_a^T \partial_3 L[x]_\gamma^{\alpha, \beta}(t)\LC h(t)dt  \\
&\quad + \gamma_2 \left[ \int_a^b \partial_3 L[x]_\gamma^{\alpha, \beta}(t) \RCb h(t)dt
- \int_T^b \partial_3 L[x]_\gamma^{\alpha, \beta}(t) \RCb h(t)dt \right].
\end{split}
\end{equation*}
Integrating by parts (see Theorem~\ref{thm:FIP}), and since $h(a)=0$,
the term \eqref{FP_term} can be written as
\begin{equation*}
\begin{split}
\gamma_1 & \left[ \int_a^T h(t) _tD_T^{\a} \partial_3
L[x]_\gamma^{\alpha, \beta}(t) dt
+\left[h(t) {_t I_T^{1-\a}\partial_3 L[x]_\gamma^{\alpha, \beta}(t)}\right]_{t=T} \right]\\
&+ \gamma_2 \Biggl[ \int_a^b h(t){_aD_t^{\b}} \partial_3 L[x]_\gamma^{\alpha, \beta}(t) dt
- \left[h(t){_a I_t^{1-\b}\partial_3 L[x]_\gamma^{\alpha, \beta}(t)}\right]_{t=b} \\
&\qquad \quad -\Biggl( \int_T^b h(t){_TD_t^{\b}} \partial_3 L[x]_\gamma^{\alpha, \beta}(t) dt
- \left[h(t){_TI_t^{1-\b}\partial_3 L[x]_\gamma^{\alpha, \beta}(t)}\right]_{t=b}\\
&\qquad \qquad \quad +\left[h(t){_TI_t^{1-\b}
\partial_3 L[x]_\gamma^{\alpha, \beta}(t)}\right]_{t=T} \Biggr) \Biggr].
\end{split}
\end{equation*}
Unfolding these integrals, and considering the fractional operator
$D_{\overline{\gamma}}^{\b,\a}$  with
$\overline{\gamma}=(\gamma_2,\gamma_1)$,
then \eqref{FP_term} is equivalent to
\begin{equation*}
\begin{split}
\int_a^T h(t) & D_{\overline{\gamma}}^{\b,\a}\partial_3L[x]_\gamma^{\alpha, \beta}(t)dt\\
&+ \int_T^b\gamma_2h(t)\left[_aD_t^{\b}\partial_3 L[x]_\gamma^{\alpha, \beta}(t)
-{_TD_t^{\b}\partial_3L[x]_\gamma^{\alpha, \beta}(t)}\right]dt\\
&+\left[h(t)\left(\gamma_1 \, {_tI_T^{1-\a}\partial_3L[x]_\gamma^{\alpha, \beta}(t)}
-{\gamma_2 \, {_TI_t^{1-\b}\partial_3L[x]_\gamma^{\alpha, \beta}(t)}}\right)\right]_{t=T}\\
&+\left[ h(t)\gamma_2\left({_TI_t^{1-\b}\partial_3 L[x]_\gamma^{\alpha, \beta}(t)}
-{_aI_t^{1-\b}\partial_3L[x]_\gamma^{\alpha, \beta}(t)}\right)\right]_{t=b}.
\end{split}
\end{equation*}
Substituting these relations into equation \eqref{FP_eq_derj}, we obtain
\begin{equation}
\label{FP_eq_derj2}
\begin{split}
0= &\int_a^Th(t)\left[\partial_2 L[x]_\gamma^{\alpha, \beta}(t)
+ D_{\overline{\gamma}}^{\b,\a}\partial_3 L[x]_\gamma^{\alpha, \beta}(t)\right]dt\\
&+ \int_T^b \gamma_2 h(t) \left[_aD_t^{\b}\partial_3 L[x]_\gamma^{\alpha, \beta}(t)
-{_TD_t^{\b}\partial_3 L[x]_\gamma^{\alpha, \beta}(t)}\right]dt\\
&+ h(T)\Bigl[\gamma_1 \, {_tI_T^{1-\a}\partial_3L[x]_\gamma^{\alpha, \beta}(t)}
-{\gamma_2 \, {_TI_t^{1-\b}\partial_3L[x]_\gamma^{\alpha, \beta}(t)}}\\
&\quad \quad +\partial_2 \phi(t,x(t))\Bigr]_{t=T}\\
&+\Delta T \left[ L[x]_\gamma^{\alpha, \beta}(t)+\partial_1\phi(t,x(t))
+\partial_2\phi(t,x(t))x'(t)  \right]_{t=T}\\
&+ h(b)\left[ \gamma_2  \left( _TI_t^{1-\b}\partial_3 L[x]_\gamma^{\alpha, \beta}(t)
-{_aI_t^{1-\b}\partial_3L[x]_\gamma^{\alpha, \beta}(t)}\right) \right]_{t=b}.
\end{split}
\end{equation}
As $h$ and $\triangle T$ are arbitrary, we can choose $\triangle T=0$ and $h(t)=0$,
for all $t\in[T,b]$, but $h$ is arbitrary in $t\in[a,T)$.
Then, for all $t\in[a,T]$, we obtain the first necessary condition \eqref{FP_ELeq_1}:
$$
\partial_2 L[x]_\gamma^{\alpha, \beta}(t)
+D{_{\overline{\gamma}}^{\b,\a}}\partial_3 L[x]_\gamma^{\alpha, \beta}(t)=0.
$$
Analogously, considering $\triangle T=0$, $h(t)=0$,
for all $t\in[a,T] \cup \{b\}$, and $h$ arbitrary on $(T,b)$,
we obtain the second necessary condition \eqref{FP_ELeq_2}:
$$
\gamma_2\left({\LDb}\partial_3L[x]_\gamma^{\alpha, \beta}(t)
-{ _TD{_t^{\b}}\partial_3L[x]_\gamma^{\alpha, \beta}(t)}\right)=0.
$$
As $(x,T)$ is a solution to the necessary conditions \eqref{FP_ELeq_1}
and \eqref{FP_ELeq_2}, then equation \eqref{FP_eq_derj2} takes the form
\begin{equation}
\label{FP_eq_derj3}
\begin{split}
0=&h(T)\left[\gamma_1 \, {_tI_T^{1-\a}\partial_3L[x]_\gamma^{\alpha, \beta}(t)}
-{\gamma_2 \, {_TI_t^{1-\b}\partial_3L[x]_\gamma^{\alpha, \beta}(t)}}
+\partial_2 \phi(t,x(t))\right]_{t=T}\\
&+\Delta T \left[ L[x]_\gamma^{\alpha, \beta}(t)+\partial_1\phi(t,x(t))
+\partial_2\phi(t,x(t))x'(t)  \right]_{t=T}\\
&+ h(b)\left[\gamma_2  \left( _TI_t^{1-\b}\partial_3 L[x]_\gamma^{\alpha, \beta}(t)
-{_aI_t^{1-\b}\partial_3L[x]_\gamma^{\alpha, \beta}(t)}\right) \right]_{t=b}.
\end{split}
\end{equation}
Transversality conditions \eqref{FP_CT1}
are obtained for appropriate choices of variations.
\end{proof}

In the next theorem, considering the same Problem \ref{P_fund},
we rewrite the transversality conditions \eqref{FP_CT1} in terms
of the increment $\Delta T$ on time and on the consequent increment
$\Delta x_T$ on $x$, given by
\begin{equation}
\label{FP_xt}
\Delta x_T = (x+h)\left( T + \Delta T \right) - x(T).
\end{equation}

\begin{Theorem}
\label{FP_teo2}
Let $(x,T)$ be a local minimizer to the functional \eqref{functional1} on $D$.
Then, the fractional Euler--Lagrange equations \eqref{FP_ELeq_1}
and \eqref{FP_ELeq_2} are satisfied together with
the following transversality conditions:
\index{transversality conditions}
\begin{equation}
\label{FP_CT2}
\begin{cases}
L[x]_\gamma^{\alpha, \beta}(T)+\partial_1\phi(T,x(T))\\
\qquad + x'(T) \left[ \gamma_2 {_TI}_t^{1-\b} \partial_3L[x]_\gamma^{\alpha, \beta}(t)
- \gamma_1 {_tI_T^{1-\a} \partial_3L[x]_\gamma^{\alpha, \beta}(t)} \right]_{t=T} =0,\\
\left[ \gamma_1\, {_tI_T^{1-\a}} \partial_3L[x]_\gamma^{\alpha, \beta}(t)
- \gamma_2\, {_TI_t^{1-\b}} \partial_3L[x]_\gamma^{\alpha, \beta}(t)\right]_{t=T}
+\partial_2 \phi(T,x(T))=0,\\
\gamma_2 \left[ _TI_t^{1-\b}\partial_3 L[x]_\gamma^{\alpha, \beta}(t)
-{_aI_t^{1-\b}\partial_3L[x]_\gamma^{\alpha, \beta}(t)}\right]_{t=b}=0.
\end{cases}
\end{equation}
\end{Theorem}

\begin{proof}
The Euler--Lagrange equations are deduced following
similar arguments as the ones presented in Theorem \ref{FP_teo1}.
We now focus our attention on the proof of the transversality conditions.
Using Taylor's expansion up to first order for a small $\Delta T$,
and restricting the set of variations to those for which $h'(T)=0$, we obtain
$$
(x+h)\left( T + \Delta T \right)=(x+h)(T)+x'(T) \Delta T + O(\Delta T)^{2}.
$$
Rearranging the relation \eqref{FP_xt} allows us to express $h(T)$
in terms of $\Delta T$ and $\Delta x_T$:
$$
h(T)= \Delta x_T - x'(T) \Delta T + O(\Delta T)^{2}.
$$
Substitution of this expression into \eqref{FP_eq_derj3} gives us
\begin{equation*}
\begin{split}
0=& \Delta x_T \left[\gamma_1 \, {_tI_T^{1-\a}\partial_3L[x]_\gamma^{\alpha, \beta}(t)}
-{\gamma_2 \, {_TI_t^{1-\b}\partial_3 L[x]_\gamma^{\alpha, \beta}(t)}}
+\partial_2 \phi(t,x(t))\right]_{t=T}\\
&+\Delta T \left[ L[x]_\gamma^{\alpha, \beta}(t)+\partial_1\phi(t,x(t))\right.\\
&\left.\qquad - x'(t)\left( \gamma_1 \, {_tI_T^{1-\a}\partial_3L[x]_\gamma^{\alpha, \beta}(t)}
-{\gamma_2 \, {_TI_t^{1-\b}\partial_3L[x]_\gamma^{\alpha, \beta}(t)}} \right)\right]_{t=T}\\
&+h(b)\left[\gamma_2\left( _TI_t^{1-\b}\partial_3L[x]_\gamma^{\alpha, \beta}(t)
-{_aI_t^{1-\b}\partial_3L[x]_\gamma^{\alpha, \beta}(t)}\right) \right]_{t=b}
+ O(\Delta T)^{2}.
\end{split}
\end{equation*}
Transversality conditions \eqref{FP_CT2} are obtained
using appropriate choices of variations.
\end{proof}


\subsection{Particular cases}
\label{subsec:part:cases}

Now, we specify our results to three particular
cases of variable terminal points.


\textbf{Vertical terminal line} \index{vertical terminal line}

This case involves a fixed upper bound $T$. Thus, $\Delta T=0$ and,
consequently, the second term in \eqref{FP_eq_derj3} drops out.
Since $\Delta x_T$ is arbitrary, we obtain the following
transversality conditions: if $T<b$, then
$$
\begin{cases}
\left[ \gamma_1 \, {_tI_T^{1-\a}} \partial_3L[x]_\gamma^{\alpha, \beta}(t)
- \gamma_2 \, {_TI_t^{1-\b}} \partial_3L[x]_\gamma^{\alpha, \beta}(t)\right]_{t=T}
+\partial_2 \phi(T,x(T))=0,\\
 \gamma_2 \left[ _TI_t^{1-\b}\partial_3L[x]_\gamma^{\alpha, \beta}(t)
 -{_aI_t^{1-\b}\partial_3L[x]_\gamma^{\alpha, \beta}(t)}\right]_{t=b}=0;
\end{cases}
$$
if $T=b$, then $\Delta x_T=h(b)$ and the transversality conditions reduce to
$$
\left[\gamma_1 \, {_tI_b^{1-\a}} \partial_3L[x]_\gamma^{\alpha, \beta}(t)
-\gamma_2 \, {_aI_t^{1-\b}\partial_3L[x]_\gamma^{\alpha, \beta}(t)}\right]_{t=b}
+\partial_2 \phi(b,x(b))=0.
$$


\textbf{Horizontal terminal line} \index{horizontal terminal line}

In this situation, we have $\Delta x_T=0$ but $\Delta T$ is arbitrary.
Thus, the transversality conditions are
$$
\begin{cases}
L[x]_\gamma^{\alpha, \beta}(T)+\partial_1\phi(T,x(T))\\
\qquad + x'(T) \left[ \gamma_2 {_TI_t^{1-\b} \partial_3L[x]_\gamma^{\alpha, \beta}(t)}
- \gamma_1 {_tI_T^{1-\a} \partial_3L[x]_\gamma^{\alpha, \beta}(t)} \right]_{t=T} =0,\\
\gamma_2 \left[ _TI_t^{1-\b}\partial_3L[x]_\gamma^{\alpha, \beta}(t)
-{_aI_t^{1-\b}\partial_3L[x]_\gamma^{\alpha, \beta}(t)}\right]_{t=b}=0.
\end{cases}
$$


\textbf{Terminal curve} \index{terminal curve}

Now the terminal point is described by a given curve
$\psi:C^{1}([a,b])\rightarrow \bR$,
in the sense that $x(T)=\psi (T)$. From Taylor's formula,
for a small arbitrary $\Delta T$, one has
$$
\Delta x(T)=\psi'(T) \Delta T+ O(\Delta T)^{2}.
$$
Hence, the transversality conditions are presented in the form
\begin{equation*}
\begin{cases}
L[x]_\gamma^{\alpha, \beta}(T)+\partial_1\phi(T,x(T))+ \partial_2 \phi(T, x(T))\psi'(T)+\left(  x'(T)-\psi'(T)\right)\\
\quad\quad\quad\times\left[ \gamma_2 \, {_TI_t^{1-\b} \partial_3L[x]_\gamma^{\alpha, \beta}(t)}
- \gamma_1 \, {_tI_T^{1-\a} \partial_3L[x]_\gamma^{\alpha, \beta}(t)} \right]_{t=T} =0,\\
\gamma_2 \left[ _TI_t^{1-\b} \partial_3L[x]_\gamma^{\alpha, \beta}(t)
-{_aI_t^{1-\b}\partial_3L[x]_\gamma^{\alpha, \beta}(t)}\right]_{t=b}=0.
 \end{cases}
\end{equation*}


\subsection{Examples}
\label{subsec_FPex}

In this section, we show two examples to illustrate the new results.
Let $\alpha(t,\t)=\alpha(t)$ and $\beta(t,\t)=\beta(\t)$ be two functions
depending on a variable $t$ and $\t$ only, respectively.
Consider the following fractional variational problem: to minimize the functional
\begin{multline*}
\mathcal{J}(x,T)=\int_0^T\Bigg[2 \alpha(t) -1\\+ \Bigg({^CD_\gamma^{\alpha(\cdot),\beta(\cdot)}} x(t)
-\frac{t^{1-\alpha(t)}}{2\Gamma(2-\alpha(t))}-\frac{(10-t)^{1-\beta(t)}}{2\Gamma(2-\beta(t))}\Bigg)^2\Bigg]dt
\end{multline*}
for $t\in[0,10]$, subject to the initial condition $x(0)=0$ and where $\gamma = (\gamma_1,\gamma_2)=(1/2,1/2)$.
Simple computations show that for $\overline{x}(t)=t$, with $t\in[0,10]$, we have
$$
^CD_\gamma^{\alpha(\cdot),\beta(\cdot)} \overline{x}(t)=\frac{t^{1-\alpha(t)}}{2\Gamma(2-\alpha(t))}+\frac{(10-t)^{1-\beta(t)}}{2\Gamma(2-\beta(t))}.
$$
For $\overline{x}(t)=t$ the functional reduces to
$$
\mathcal{J}(\overline{x},T)=\int_0^T\left(2 \alpha(t) -1\right)dt.
$$
In order to determine the optimal time $T$, we have to solve the equation
$2 \alpha(T) = 1$. For example, let $\alpha(t) = t^2/2$. In this case,
since $\mathcal{J}(x,T)\geq-2/3$ for all pairs $(x,T)$ and $\mathcal{J}(\overline{x},1)=-2/3$,
we conclude that the (global) minimum value of the functional is $-2/3$,
obtained for $\overline{x}$ and $T=1$.
It is obvious that the two Euler--Lagrange equations \eqref{FP_ELeq_1} and \eqref{FP_ELeq_2}
are satisfied when $x=\overline{x}$, since
$$
\partial_3L[\overline{x}]_\gamma^{\alpha, \beta}(t)=0
\quad \mbox{for all} \quad t\in[0,10].
$$
Using this relation, together with
$$
L[\overline{x}]_\gamma^{\alpha, \beta}(1)=0,
$$
the transversality conditions \eqref{FP_CT1} are also verified.

For our last example, consider the functional
\begin{multline*}
\mathcal{J}(x,T)=\int_0^T\Bigg[2 \alpha(t) -1\\
+\Bigg({^CD_\gamma^{\alpha(\cdot),\beta(\cdot)}} x(t)
-\frac{t^{1-\alpha(t)}}{2\Gamma(2-\alpha(t))}
-\frac{(10-t)^{1-\beta(t)}}{2\Gamma(2-\beta(t))}\Bigg)^3\Bigg]dt,
\end{multline*}
where the remaining assumptions and conditions are as in the previous example.
For this case, $\overline{x}(t)=t$ and $T=1$ still satisfy
the necessary optimality conditions. However, we cannot assure
that $(\overline x,1)$ is a local minimizer to the problem.


\section{Higher-order variational problems}
\label{sec:HO}

In this section, we intend to generalize the results obtained in Section
\ref{sec:FP} by considering higher-order variational problems with a
Lagrangian depending on a higher-order combined Caputo derivative of
variable fractional order, defined by
$$^{C}D_{\gamma^n}^{\an,\bn}x(t)=\gamma_1^n \, \LCan x(t)+\gamma_2^n \, \RCbn x(t),$$
subject to boundary conditions at the initial time $t=a$.

In Section~\ref{subsec:NecCond_HO}, we obtain higher-order Euler--Lagrange
equations and transversality conditions for the generalized variational problem
with a Lagrangian depending on a combined Caputo fractional derivative of
variable fractional order  (Theorems~\ref{HO_teo1} and \ref{HO_teo2}).

One illustrative example is discussed in Section~\ref{subsec:HO_example}.


\subsection{Necessary optimality conditions}
\label{subsec:NecCond_HO}

Let $n \in \mathbb{N}$ and $x : [a,b] \rightarrow \bR$ be a function of
class $C^n$. The fractional order is a continuous function of two variables,
$\alpha_n: [a,b]^2 \rightarrow (n-1,n)$.

Let $D$ denote the linear subspace of $C^n([a,b])\times [a,b]$
such that the fractional derivative of $x$,  $\DCi x(t)$, exists and is continuous on the interval $[a,b]$
for all $i \in \{1, \ldots,n\}$. We endow $D$ with the norm
$$
\|(x,t)\|=\max_{a\leq t \leq  b}|x(t)|+\max_{a\leq t \leq b}
\sum_{i=1}^{n} \left| \DCi  x(t)\right|+|t|.
$$
Consider the following higher-order problem of the calculus of variations:
\index{variational fractional problem! higher-order}
\begin{Problem}
\label{Problem_HO}
Minimize functional $\mathcal{J}:D\rightarrow \mathbb{R}$,
where
\begin{multline}
\label{funct_HO_1}
\mathcal{J}(x,T)
=\int_a^T L\Big(t, x(t), \,^CD_{\gamma^1}^{\alpha_1(\cdot,\cdot),
\beta_1(\cdot,\cdot)}x(t), \ldots, \,^CD_{\gamma^n}^{\alpha_n(\cdot,\cdot),
\beta_n(\cdot,\cdot)}x(t)  \Big) dt\\ + \phi(T,x(T)),
\end{multline}
over all $(x,T)\in D$ subject to boundary conditions
$$
x(a)=x_a, \quad x^{(i)}(a)=x^{i}_a, \quad \forall i \in \lbrace1,\ldots,n-1\rbrace,
$$
\index{boundary conditions}
for fixed $x_a, x^{1}_a, \ldots, x^{n-1}_a \in \mathbb{R}$.
Here the terminal time $T$ and terminal state $x(T)$ are both free.
For all $i \in \lbrace1, \ldots, n\rbrace$,
$\alpha_i, \beta_i  \left([a,b]^2\right) \subseteq (i-1,i)$
and $\gamma^i=\left(\gamma_1^i, \gamma_2^i\right)$ is a vector.
The terminal cost function $\phi:[a,b]\times \mathbb{R}\rightarrow \mathbb{R}$
is at least of class $C^1$. \index{terminal cost function}
\end{Problem}

For simplicity of notation, we introduce the operator $[\cdot]_\gamma^{\alpha, \beta}$ defined by
$$
[x]_\gamma^{\alpha, \beta}(t)
=\left(t, x(t), \,^CD_{\gamma^1}^{\alpha_1(\cdot,\cdot),
\beta_1(\cdot,\cdot)}x(t), \ldots,
\,^CD_{\gamma^n}^{\alpha_n(\cdot,\cdot),
\beta_n(\cdot,\cdot)}x(t)\right).
$$
We assume that the Lagrangian $L:[a,b]\times \mathbb{R}^{n+1} \to \mathbb{R}$
is a function of class  $C^{1}$. Along the work, we denote by $\partial_i L$,
$i \in \{1,\ldots, n+2\}$, the partial derivative of the Lagrangian $L$
with respect to its $i$th argument.

Now, we can rewrite functional \eqref{funct_HO_1} as
\begin{equation}
\label{funct_HO_1a}
\mathcal{J}(x,T)=\int_a^T L[x]_\gamma^{\alpha, \beta}(t) dt + \phi(T,x(T)).
\end{equation}

In the previous section, we obtained fractional necessary optimality conditions
that every local minimizer of functional $\mathcal{J}$, with $n=1$, must fulfill.
Here, we generalize those results to arbitrary values of $n$, $n \in \mathbb{N}$. 
Necessary optimality conditions for Problem~\ref{Problem_HO} are presented next.

\begin{Theorem}
\label{HO_teo1}
Suppose that $(x,T)$ gives a minimum to functional \eqref{funct_HO_1a} on $D$.
Then, $(x,T)$ satisfies the following fractional Euler--Lagrange equations:
\index{Euler--Lagrange equations}
\begin{equation}
\label{ELeqHO_1}
\partial_2 L[x]_\gamma^{\alpha, \beta}(t)
+\sum_{i=1}^{n} {D_{\overline{\gamma^i}}^{\bi,\ai}}
\partial_{i+2} L[x]_\gamma^{\alpha, \beta}(t)=0,
\end{equation}
on the interval $[a,T]$, and
\begin{equation}
\label{ELeqHO_2}
\sum_{i=1}^{n} \gamma_2^i \left({_aD_t^{\bi}}\partial_{i+2}
L[x]_\gamma^{\alpha, \beta}(t)-{ _TD{_t^{\bi}}
\partial_{i+2} L[x]_\gamma^{\alpha, \beta}(t)}\right)=0,
\end{equation}
on the interval $[T,b]$. Moreover, $(x,T)$ satisfies the following transversality conditions:
\index{transversality conditions}
\begin{equation}
\label{CTHO_1}
\begin{cases}
 L[x]_\gamma^{\alpha,\DS \beta}(T)+\partial_1\phi(T,x(T))+\partial_2\phi(T,x(T))x'(T)=0,\\
\DS \sum_{i=1}^{n}\left[\gamma_1^i (-1)^{i-1} \, \frac{d^{i-1}}{dt^{i-1}}
{_tI_T^{i-\ai}} \partial_{i+2} L[x]_\gamma^{\alpha, \beta}(t)\right.\\
 \DS - \left.\gamma_2^i \,\frac{d^{i-1}}{dt^{i-1}}  {_TI_t^{i-\bi}}
\partial_{i+2} L[x]_\gamma^{\alpha, \beta}(t)\right]_{t=T} +\partial_2 \phi(T,x(T))=0,\\
\DS \sum_{i=j+1}^{n} \left[ \gamma_1^i (-1)^{i-1-j} \,
\frac{d^{i-1-j}}{dt^{i-1-j}}  {_tI_T^{i-\ai}} \partial_{i+2} L[x]_\gamma^{\alpha, \beta}(t)\right.\\
\DS + \left.\gamma_2^i (-1)^{j+1}\,\frac{d^{i-1-j}}{dt^{i-1-j}}
{_TI_t^{i-\bi}} \partial_{i+2} L[x]_\gamma^{\alpha, \beta}(t)\right]_{t=T} =0, \quad \forall j=1,\ldots,n-1,\\
\DS \sum_{i=j+1}^{n} \left[\gamma_2^i (-1)^{j+1}\, \left[
\frac{d^{i-1-j}}{dt^{i-1-j}} {_aI_t^{i-\bi}}\partial_{i+2} L[x]_\gamma^{\alpha, \beta}(t)\right.\right.\\
\DS \left.\left. - \,\frac{d^{i-1-j}}{dt^{i-1-j}}   {_TI_t^{i-\bi}\partial_{i+2}
L[x]_\gamma^{\alpha, \beta}(t)}\right]\right]_{t=b}=0, \quad \forall j=0,\ldots,n-1.
\end{cases}
\end{equation}
\end{Theorem}

\begin{proof}
The proof is an extension of the one used in Theorem \ref{FP_teo1}.
Let $h\in C^n([a,b];\bR)$ be a perturbing curve and $\triangle T \in\mathbb{R}$
an arbitrarily chosen small change in $T$. For a small number $\e \in\mathbb{R}$,
if $(x,T)$ is a solution to the problem, we consider an admissible variation of $(x,T)$ of the form $\left(x+\e{h},T+\e\Delta{T}\right)$, and then, by the minimum condition, we have that
$$
\mathcal{J}(x,T) \leq \mathcal{J}(x+\e{h},T+\e\Delta{T}).
$$
The constraints $x^{(i)}(a)=x^{(i)}_a$ imply that all admissible variations
must fulfill the conditions $h^{(i)}(a)=0,$ for all $i=0, \ldots, n-1$.
We define function $j(\cdot)$ on a neighborhood of zero by
\begin{equation*}
\begin{split}
j(\e) &=\mathcal{J}(x+\e h,T+\e \triangle T)\\
&=\int_a^{T+\e \triangle T} L[x+\e h]_\gamma^{\alpha, \beta}(t)\,dt
+\phi\left(T+\e \triangle T,(x+\e h)(T+\e \triangle T)\right).
\end{split}
\end{equation*}
The derivative $j'(\e)$ is given by the expression
\begin{equation*}
\begin{split}
&\int_a^{T+\e \triangle T} \left(
\partial_2 L[x+\e h]_\gamma^{\alpha, \beta}(t) h(t)
+\sum_{i=1}^{n}  \partial_{i+2} L[x+\e h]_\gamma^{\alpha, \beta}(t) \DCi h(t) \right) dt  \\
&\quad \quad +L[x + {\e h}]_\gamma^{\alpha, \beta}(T + \e \Delta T)\Delta T
+ \partial_1 \phi\left(T+\e \triangle T,
(x+\e h)(T+\e \triangle T)\right) \, \Delta T\\
&\quad\quad +\partial_2\phi\left(T+\e \triangle T,
(x+\e h)(T+\e \triangle T)\right) \, (x+\e h)'(T+\e \triangle T).
\end{split}
\end{equation*}
Hence, by Fermat's theorem, a necessary condition for $(x,T)$
to be a local minimizer of $j$ is given by $j'(0) =0$, that is,
\begin{multline}
\label{eqHO_derj}
\int_a^T \left( \partial_2 L[x]_\gamma^{\alpha, \beta}(t) h(t)
+\sum_{i=1}^{n} \partial_{i+2} L[x]_\gamma^{\alpha, \beta}(t) \DCi h(t) \right)dt\\
+L[x]_\gamma^{\alpha, \beta}(T)\Delta T
+\partial_1 \phi \left(T, x(T)\right)\Delta T
+ \partial_2\phi(T, x(T))\left[h(t)+x'(T) \triangle T \right]=0.
\end{multline}
Considering the second addend of the integral function \eqref{eqHO_derj},
for $i=1$, we get
\begin{equation*}
\begin{split}
&\int_a^T \partial_3 L[x]_\gamma^{\alpha, \beta}(t) {^CD_{\gamma^1}^{\ao,\bo}} h(t) dt\\
&=\int_a^T \partial_3 L[x]_\gamma^{\alpha, \beta}(t)\left[\gamma_1^1
\, \LCao h(t)+\gamma_2^1 \, \RCbo h(t)\right]dt\\
&=\gamma_1^1 \int_a^T \partial_3 L[x]_\gamma^{\alpha, \beta}(t)\LCao h(t)dt  \\
&\quad + \gamma_2^1 \left[ \int_a^b \partial_3 L[x]_\gamma^{\alpha, \beta}(t) \RCbo h(t)dt
- \int_T^b \partial_3 L[x]_\gamma^{\alpha, \beta}(t) \RCbo h(t)dt \right].
\end{split}
\end{equation*}
Integrating by parts (see Theorem~\ref{thm:FIP_HO}), and since $h(a)=0$,
we obtain that
\begin{equation*}
\begin{split}
\gamma_1^1 & \left[ \int_a^T h(t) _tD_T^{\ao} \partial_3
L[x]_\gamma^{\alpha, \beta}(t) dt
+\left[h(t) {_t I_T^{1-\ao}\partial_3 L[x]_\gamma^{\alpha, \beta}(t)}\right]_{t=T} \right]\\
&+  \gamma_2^1 \Biggl[ \int_a^b h(t){_aD_t^{\bo}} \partial_3 L[x]_\gamma^{\alpha, \beta}(t) dt
-\left[h(t){_a I_t^{1-\bo}\partial_3 L[x]_\gamma^{\alpha, \beta}(t)}\right]_{t=b} \\
&-\Biggl( \int_T^b h(t){_TD_t^{\bo}} \partial_3 L[x]_\gamma^{\alpha, \beta}(t) dt
- \left[h(t){_TI_t^{1-\bo}\partial_3 L[x]_\gamma^{\alpha, \beta}(t)}\right]_{t=b}\\
&+\left[h(t){_TI_t^{1-\bo}\partial_3 L[x]_\gamma^{\alpha, \beta}(t)}\right]_{t=T} \Biggr) \Biggr].
\end{split}
\end{equation*}
Unfolding these integrals, and considering the fractional operator
$D_{\overline{\gamma^1}}^{\beta_1,\alpha_1}$ with
$\overline{\gamma^1}=(\gamma_2^1,\gamma_1^1)$,
then the previous term is equal to
\begin{equation*}
\begin{split}
\int_a^T h(t) & D_{\overline{\gamma^1}}^{\bo,\ao}\partial_3L[x]_\gamma^{\alpha, \beta}(t)dt\\
&+ \int_T^b\gamma_2^1 h(t)\left[_aD_t^{\bo}\partial_3 L[x]_\gamma^{\alpha, \beta}(t)
-{_TD_t^{\bo}\partial_3L[x]_\gamma^{\alpha, \beta}(t)}\right]dt\\
&+h(T) \left[\gamma_1^1 \, {_tI_T^{1-\ao}\partial_3L[x]_\gamma^{\alpha, \beta}(t)}
-{\gamma_2^1 \, {_TI_t^{1-\bo}\partial_3L[x]_\gamma^{\alpha, \beta}(t)}}\right]_{t=T}\\
&- h(b)\gamma_2^1\left[  {_aI_t^{1-\bo}\partial_3 L[x]_\gamma^{\alpha, \beta}(t)}
-{_TI_t^{1-\bo}\partial_3L[x]_\gamma^{\alpha, \beta}(t)}\right]_{t=b}.
\end{split}
\end{equation*}
Considering the third addend of the integral function \eqref{eqHO_derj},
for $i=2$, we get
\begin{equation*}
\begin{split}
&\int_a^T \partial_4 L[x]_\gamma^{\alpha, \beta}(t) {^CD_{\gamma^2}^{\at,\bt}} h(t) dt
=\gamma_1^2 \int_a^T \partial_4 L[x]_\gamma^{\alpha, \beta}(t)\LCat h(t)dt\\
&\quad + \gamma_2^2 \left[ \int_a^b \partial_4 L[x]_\gamma^{\alpha, \beta}(t) \RCbt h(t)dt
- \int_T^b \partial_4 L[x]_\gamma^{\alpha, \beta}(t) \RCbt h(t)dt \right]\\
&=\gamma_1^2 \left[ \int_a^T h(t) _tD_T^{\at} \partial_4L[x]_\gamma^{\alpha, \beta}(t) dt\right.\\
&\quad+\left. \left[h^{(1)}(t) {_t I_T^{2-\at}\partial_4 L[x]_\gamma^{\alpha, \beta}(t)}
-h(t) \frac{d}{dt} {_t I_T^{2-\at}\partial_4 L[x]_\gamma^{\alpha, \beta}(t)}\right]_{t=T} \right]\\
&\quad+  \gamma_2^2 \Biggl[ \int_a^b h(t){_aD_t^{\bt}} \partial_4 L[x]_\gamma^{\alpha, \beta}(t) dt \\
&\quad+\left. \left[h^{(1)}(t) {_a I_t^{2-\bt}\partial_4 L[x]_\gamma^{\alpha, \beta}(t)}
-h(t) \frac{d}{dt} {_a I_t^{2-\bt}\partial_4 L[x]_\gamma^{\alpha, \beta}(t)}\right]_{t=b} \right.\\
\end{split}
\end{equation*}
\begin{equation*}
\begin{split}
&\quad- \int_T^b h(t){_TD_t^{\bt}} \partial_4 L[x]_\gamma^{\alpha, \beta}(t) dt\\
&\quad- \left[h^{(1)}(t){_TI_t^{2-\bt}\partial_4 L[x]_\gamma^{\alpha, \beta}(t)}
-h(t)\frac{d}{dt} {_T I_t^{2-\bt}\partial_4 L[x]_\gamma^{\alpha, \beta}(t)}\right]_{t=b}\\
&\quad+ \left[h^{(1)}(t){_TI_t^{2-\bt} \partial_4 L[x]_\gamma^{\alpha, \beta}(t)
-h(t) \frac{d}{dt} {_T I_t^{2-\bt}\partial_4 L[x]_\gamma^{\alpha, \beta}(t)}}\right]_{t=T}\Biggr].
\end{split}
\end{equation*}
Again, with the auxiliary operator $D_{\overline{\gamma^2}}^{\beta_2,\alpha_2}$,
with $\overline{\gamma^2}=(\gamma_2^2,\gamma_1^2)$, we obtain
\begin{equation*}
\begin{split}
&\int_a^T h(t)  D_{\overline{\gamma^2}}^{\bt,\at}\partial_4L[x]_\gamma^{\alpha, \beta}(t)dt\\
&\quad+ \int_T^b\gamma_2^2 h(t)\left[_aD_t^{\bt}\partial_4 L[x]_\gamma^{\alpha, \beta}(t)
-{_TD_t^{\bt}\partial_4L[x]_\gamma^{\alpha, \beta}(t)}\right]dt\\
&\quad+\left[h^{(1)}(t)\left(\gamma_1^2 \, {_tI_T^{2-\at}\partial_4L[x]_\gamma^{\alpha, \beta}(t)}
+{\gamma_2^2 \, {_TI_t^{2-\bt}\partial_4L[x]_\gamma^{\alpha, \beta}(t)}}\right)\right]_{t=T}\\
&\quad-\left[h(t)\left(\gamma_1^2 \, \frac{d}{dt} {_tI_T^{2-\at}\partial_4L[x]_\gamma^{\alpha, \beta}(t)}
+{\gamma_2^2 \,\frac{d}{dt} {_TI_t^{2-\bt}\partial_4L[x]_\gamma^{\alpha, \beta}(t)}}\right)\right]_{t=T}\\
&\quad+\left[ h^{(1)}(t)\gamma_2^2 \left({_aI_t^{2-\bt}\partial_4 L[x]_\gamma^{\alpha, \beta}(t)}
-{_TI_t^{2-\bt}\partial_4L[x]_\gamma^{\alpha, \beta}(t)}\right)\right]_{t=b}\\
&\quad-\left[ h(t)\gamma_2^2 \left(\frac{d}{dt} {_aI_t^{2-\bt}\partial_4 L[x]_\gamma^{\alpha, \beta}(t)}
- \frac{d}{dt} {_TI_t^{2-\bt}\partial_4L[x]_\gamma^{\alpha, \beta}(t)}\right)\right]_{t=b}.
\end{split}
\end{equation*}
Now, consider the general case
\begin{equation*}
\int_a^T \partial_{i+2} L[x]_\gamma^{\alpha, \beta}(t) {^CD_{\gamma^i}^{\ai,\bi}} h(t) dt,
\end{equation*}
$i=3, \ldots, n$. Then, we obtain
\begin{equation*}
\begin{split}
& \gamma_1^i  \left[ \int_a^T h(t)  _tD_T^{\ai} \partial_{i+2}
L[x]_\gamma^{\alpha, \beta}(t) dt \right.\\
& \quad \quad + \left. \left[ \sum_{k=0}^{i-1}
(-1)^{k} h^{(i-1-k)}(t) \frac{d^k}{dt^k} {_t I_T^{i-\ai}
\partial_{i+2} L[x]_\gamma^{\alpha, \beta}(t) t)}\right]_{t=T} \right]\\
&+  \gamma_2^i \Biggl[ \int_a^b h(t){_aD_t^{\bi}} \partial_{i+2} L[x]_\gamma^{\alpha, \beta}(t) dt\\
&\quad \quad + \left[ \sum_{k=0}^{i-1} (-1)^{i+k} h^{(i-1-k)}(t)
\frac{d^k}{dt^k} {_a I_t^{i-\bi}\partial_{i+2} L[x]_\gamma^{\alpha, \beta}(t)}\right]_{t=b}\\
&- \left. \int_T^b h(t) {_TD_t^{\bi}} \partial_{i+2} L[x]_\gamma^{\alpha, \beta}(t) dt \right.\\
& \quad \quad - \left. \left[ \sum_{k=0}^{i-1} (-1)^{i+k} h^{(i-1-k)}(t)
\frac{d^k}{dt^k} {_T I_t^{i-\bi}\partial_{i+2}
L[x]_\gamma^{\alpha, \beta}(t)}\right]^{t=b}_{t=T}\right].
\end{split}
\end{equation*}
Unfolding these integrals, we obtain
\begin{equation*}
\begin{split}
&\int_a^T h(t)  D_{\overline{\gamma^i}}^{\bi,\ai}\partial_{i+2}L[x]_\gamma^{\alpha, \beta}(t)dt\\
&+ \int_T^b\gamma_2^i h(t)\left[_aD_t^{\bi}\partial_{i+2} L[x]_\gamma^{\alpha, \beta}(t)
-{_TD_t^{\bi}\partial_{i+2} L[x]_\gamma^{\alpha, \beta}(t)}\right]dt\\
&+h^{(i-1)}(T)\left[\gamma_1^i \, {_tI_T^{i-\ai}\partial_{i+2}L[x]_\gamma^{\alpha, \beta}(t)}
+{\gamma_2^i (-1)^i \, {_TI_t^{i-\bi}\partial_{i+2}L[x]_\gamma^{\alpha, \beta}(t)}}\right]_{t=T}\\
&+h^{(i-1)}(b) \gamma_2^i (-1)^i \left[ \, {_aI_t^{i-\bi}\partial_{i+2}L[x]_\gamma^{\alpha, \beta}(t)}
- {_TI_t^{i-\bi}\partial_{i+2}L[x]_\gamma^{\alpha, \beta}(t)}\right]_{t=b}\\
&+h^{(i-2)}(T)\left[\gamma_1^i (-1)^1 \frac{d}{dt} \, {_tI_T^{i-\ai}
\partial_{i+2}L[x]_\gamma^{\alpha, \beta}(t)}\right.\\
&\qquad \qquad \qquad \qquad \left.+{\gamma_2^i (-1)^{i+1}
\frac{d}{dt} \,{_TI_t^{i-\bi}\partial_{i+2}L[x]_\gamma^{\alpha, \beta}(t)}}\right]_{t=T}\\
&+h^{(i-2)}(b) \gamma_2^i (-1)^{i+1} \left[ \,
\frac{d}{dt} {_aI_t^{i-\bi}\partial_{i+2}L[x]_\gamma^{\alpha, \beta}(t)}\right.\\
&\qquad \qquad \qquad \qquad \left.- \frac{d}{dt} {_TI_t^{i-\bi}\partial_{i+2}L[x]_\gamma^{\alpha, \beta}(t)}\right]_{t=b}\\
&+ \ldots+ h(T)\left[ \gamma_1^i (-1)^{i-1} \frac{d^{i-1}}{dt^{i-1}} {_tI_T^{i-\ai}\partial_{i+2}
L[x]_\gamma^{\alpha, \beta}(t)}\right.\\
& \quad \quad \quad \quad + \left. \gamma_2^i (-1)^{2i-1}\frac{d^{i-1}}{dt^{i-1}}
{_TI_t^{i-\bi}\partial_{i+2}L[x]_\gamma^{\alpha, \beta}(t)}\right]_{t=T}\\
&+ h(b)\gamma_2^i (-1)^{2i-1}\left[  \frac{d^{i-1}}{dt^{i-1}} {_aI_t^{i-\bi}\partial_{i+2}
L[x]_\gamma^{\alpha, \beta}(t)}\right.\\
&\qquad \qquad \left.- \frac{d^{i-1}}{dt^{i-1}} {_TI_t^{i-\bi}
\partial_{i+2}L[x]_\gamma^{\alpha, \beta}(t)}\right]_{t=b}.
\end{split}
\end{equation*}
Substituting all the relations into equation \eqref{eqHO_derj}, we obtain that
\begin{equation*}
\begin{split}
0=&\int_a^T h(t)  \left( \partial_2 L[x]_\gamma^{\alpha, \beta}(t) + \sum_{i=1}^{n}
D_{\overline{\gamma^i}}^{\bi,\ai}\partial_{i+2}L[x]_\gamma^{\alpha, \beta}(t) \right) dt\\
&+ \int_T^b h(t) \sum_{i=1}^{n} \gamma_2^i \left[_aD_t^{\bi}\partial_{i+2} L[x]_\gamma^{\alpha, \beta}(t)
-{_TD_t^{\bi}\partial_{i+2} L[x]_\gamma^{\alpha, \beta}(t)}\right]dt\\
& +\sum_{j=0}^{n-1} h^{(j)}(T) \sum_{i=j+1}^{n}\left[ \gamma_1^i (-1)^{i-1-j}
\frac{d^{i-1-j}}{dt^{i-1-j}}{_tI_T^{i-\ai}\partial_{i+2} L[x]_\gamma^{\alpha, \beta}(t)} \right.\\
& \quad \quad \quad \quad \quad \quad \quad \quad \left.
+ \gamma_2^i (-1)^{j+1}\frac{d^{i-1-j}}{dt^{i-1-j}}{_TI_t^{i-\bi}\partial_{i+2}
L[x]_\gamma^{\alpha, \beta}(t)}\right]_{t=T}\\
\end{split}
\end{equation*}
\begin{equation}
\label{eq_HO_derj2}
\begin{split}
& +\sum_{j=0}^{n-1} h^{(j)}(b) \sum_{i=j+1}^{n} \gamma_2^i (-1)^{j+1}\left[
\frac{d^{i-1-j}}{dt^{i-1-j}} {_aI_t^{i-\bi}\partial_{i+2}
L[x]_\gamma^{\alpha, \beta}(t)} \right. \\
& \quad \quad \quad \quad \left.
- \frac{d^{i-1-j}}{dt^{i-1-j}}{_TI_t^{i-\bi}\partial_{i+2}
L[x]_\gamma^{\alpha, \beta}(t)} \right]_{t=b}+ h(T) \partial_2 \phi \left(T, x(T)\right)\\
&+ \Delta T \left[ L[x]_\gamma^{\alpha, \beta}(T)  +\partial_1 \phi \left(T, x(T)\right)
+ \partial_2\phi(T, x(T))x'(T) \right].
\end{split}
\end{equation}

We obtain the fractional Euler--Lagrange equations \eqref{ELeqHO_1}--\eqref{ELeqHO_2}
and the transversality conditions \eqref{CTHO_1} applying the fundamental lemma
of the calculus of variations (see, e.g., \cite{Cap4:Brunt}) for appropriate choices of
variations.
\end{proof}

\begin{Remark}
When $n=1$, functional \eqref{funct_HO_1} takes the form
$$
\mathcal{J}(x,T)=\int_a^T L\left(t, x(t), \DC x(t)\right)dt+\phi(T,x(T)),
$$
and the fractional Euler--Lagrange equations
\eqref{ELeqHO_1}--\eqref{ELeqHO_2} coincide with those
of Theorem~\ref{FP_teo1}.
\end{Remark}

Considering the increment $\Delta T$ on time $T$,
and the consequent increment $\Delta x_T$ on $x$, given by
\begin{equation}
\label{xt_HO}
\Delta x_T = (x+h)\left( T + \Delta T \right) - x(T),
\end{equation}
in the next theorem we rewrite the transversality conditions
\eqref{CTHO_1} in terms of these increments.

\begin{Theorem}
\label{HO_teo2}
If $(x,T)$ minimizes functional $\mathcal{J}$ defined by \eqref{funct_HO_1a},
then $(x,T)$ satisfies the Euler--Lagrange equations \eqref{ELeqHO_1}
and \eqref{ELeqHO_2}, and the following transversality conditions:
\index{transversality conditions}
\begin{equation*}
\begin{cases}
\partial_1\phi(T,x(T))-x'(T)\DS\sum_{i=1}^{n}\left[\gamma_1^i (-1)^{i-1} \, \frac{d^{i-1}}{dt^{i-1}}
{_tI_T^{i-\ai}} \partial_{i+2} L[x]_\gamma^{\alpha, \beta}(t) \right.\\ 
\DS \left. -\gamma_2^i \, \frac{d^{i-1}}{dt^{i-1}}
{_TI_t^{i-\bi}} \partial_{i+2} L[x]_\gamma^{\alpha, \beta}(t)+L[x]_\gamma^{\alpha,\beta}(T) \right]_{t=T}=0,\\

\DS\sum_{i=1}^{n}\left[\gamma_1^i (-1)^{i-1} \, \frac{d^{i-1}}{dt^{i-1}}
{_tI_T^{i-\ai}} \partial_{i+2} L[x]_\gamma^{\alpha, \beta}(t)\right.\\
\DS - \left.\gamma_2^i \,\frac{d^{i-1}}{dt^{i-1}}
{_TI_t^{i-\bi}} \partial_{i+2} L[x]_\gamma^{\alpha, \beta}(t)\right]_{t=T}
+\partial_2 \phi(T,x(T))=0,\\

\DS \sum_{i=j+1}^{n} \left[ \gamma_1^i (-1)^{i-1-j}
\, \frac{d^{i-1-j}}{dt^{i-1-j}}  {_tI_T^{i-\ai}}
\partial_{i+2} L[x]_\gamma^{\alpha, \beta}(t)\right.\\
\DS + \left.\gamma_2^i (-1)^{j+1}\,
\frac{d^{i-1-j}}{dt^{i-1-j}}  {_TI_t^{i-\bi}} \partial_{i+2}
L[x]_\gamma^{\alpha, \beta}(t)\right]_{t=T} =0, \quad \forall j=2,\ldots,n-1,\\

\DS \sum_{i=j+1}^{n} \left[\gamma_2^i (-1)^{j+1}\, \left[
\frac{d^{i-1-j}}{dt^{i-1-j}} {_aI_t^{i-\bi}}\partial_{i+2}
L[x]_\gamma^{\alpha, \beta}(t)\right.\right.\\
\DS \left.\left. - \,\frac{d^{i-1-j}}{dt^{i-1-j}}
{_TI_t^{i-\bi}\partial_{i+2}L[x]_\gamma^{\alpha, \beta}(t)}\right]\right]_{t=b}=0,
\quad \forall j=0,\ldots,n-1.
\end{cases}
\end{equation*}
\end{Theorem}

 \begin{proof}
Using Taylor's expansion up to first order, and restricting
the set of variations to those for which $h'(T)=0$, we obtain that
$$
(x+h)\left(T+\Delta T \right)=(x+h)(T)+x'(T) \Delta T + O(\Delta T)^{2}.
$$
According to the increment on $x$ given by \eqref{xt_HO}, we get
$$
h(T)= \Delta x_T - x'(T) \Delta T + O(\Delta T)^{2}.
$$
From substitution of this expression into \eqref{eq_HO_derj2},
and by using appropriate choices of variations, we obtain
the intended transversality conditions.
\end{proof}


\subsection{Example}
\label{subsec:HO_example}
We provide an illustrative example. It is covered by Theorem~\ref{HO_teo1} .

Let $p_{n-1}(t)$ be a polynomial of degree $n-1$. If $\alpha,\beta:[0,b]^2\to(n-1,n)$
are the fractional orders, then $\DC p_{n-1}(t)=0$ since $p_{n-1}^{(n)}(t)=0$
for all $t$.
Consider the functional
$$
\mathcal{J}(x,T)=\int_0^T\left[ \left( \DC x(t)\right)^2+(x(t)-p_{n-1}(t))^2-t-1\right]\,dt +T^2
$$
subject to the initial constraints
$$
x(0)=p_{n-1}(0) \quad \mbox{and} \quad x^{(k)}(0)=p_{n-1}^{(k)}(0) , \, k=1,\ldots,n-1.
$$
Observe that, for all $t\in[0,b]$,
$$
\partial_i L [p_{n-1}]_\gamma^{\alpha, \beta}(t)=0,
\quad i=2,3.
$$
Thus, function $x\equiv p_{n-1}$ and the final time $T=1$ satisfy
the necessary optimality conditions of Theorem~\ref{HO_teo1}. We also
remark that, for any curve $x$, one has
$$
\mathcal{J}(x,T)\geq\int_0^T\left[-t-1\right]\,dt +T^2=\frac{T^2}{2}-T,
$$
which attains a minimum value $-1/2$ at $T=1$. Since $\mathcal{J}(p_{n-1},1)=-1/2$,
we conclude that $(p_{n-1},1)$ is the (global) minimizer of $\mathcal{J}$.


\section{Variational problems with time delay}
\label{sec:delay}

In this section, we consider fractional variational problems with time
delay.\index{time delay}
As mentioned in \cite{Cap4:machado_delay}, ``we verify that a fractional
derivative requires an infinite number of samples capturing, therefore,
all the signal history, contrary to what happens with integer order
derivatives that are merely local operators. This fact motivates the
evaluation of calculation strategies based on delayed signal samples''.
This subject has already been studied for constant fractional order 
\cite{Cap4:Almeidadelay,Cap4:Baleanudelay,Cap4:Jaraddelay}. 
However, for a variable fractional order, it is, to the authors' 
best knowledge, an open question. We also refer to the works 
\cite{Cap4:Daftardar,Cap4:deng,Cap4:Laza,Cap4:wangdelay}, where
fractional differential equations are considered with a time delay.

In Section~\ref{subsec:NecCond_delay}, we deduce necessary optimality
conditions when the Lagrangian depends on a time delay
(Theorem~\ref{teo:delay}). One illustrative example is discussed
in Section~\ref{subsec:delay_example}.


\subsection{Necessary optimality conditions}
\label{subsec:NecCond_delay}

For simplicity of presentation, we consider fractional orders
$\alpha,\beta:[a,b]^2\to(0,1)$. Using similar arguments,
the problem can be easily generalized for higher-order derivatives.
Let $\s>0$ and define the vector
$$
\xs(t)= \left(t, x(t), \DC x(t),x(t-\s)\right).
$$
For the domain of the functional, we consider the set
$$
D_\s=\left\{(x,t)\in C^1([a-\s,b])\times [a,b]: \DC x\in C([a,b])\right\}.
$$
Let $\mathcal{J}:D_\s\to\mathbb R$ be the functional defined by
\begin{equation}
\label{funct:delay}
\mathcal{J}(x,T)=\int_a^T L\xs(t)+\phi(T,x(T)),
\end{equation}
where we assume again that the Lagrangian $L$ and the payoff term
$\phi$ are differentiable.

The optimization problem with time delay is the following:\\
\index{variational fractional problem! with time delay}
\begin{Problem} Minimize functional \eqref{funct:delay} on $D_\s$ subject
to the boundary condition
$$x(t)=\varphi(t)$$
\index{boundary conditions}
for all $t\in[a-\s,a]$, where $\varphi$ is a given (fixed) function.
\end{Problem}
We now state and prove the Euler--Lagrange equations for this problem.
\index{Euler--Lagrange equations}

\begin{Theorem}
\label{teo:delay}
Suppose that $(x,T)$ gives a local minimum to functional \eqref{funct:delay}
on $D_\s$.
If $\s\geq T-a$, then $(x,T)$ satisfies
\begin{equation}
\label{ELeq_1delay}
\partial_2 L\xs (t)+D{_{\overline{\gamma}}^{\b,\a}}\partial_3 L\xs (t)=0,
\end{equation}
for $t\in[a,T]$, and
\begin{equation}
\label{ELeq_2delay}
\gamma_2\left({\LDb}\partial_3 L\xs (t)-{ _TD{_t^{\b}}\partial_3 L\xs (t)}\right)=0,
\end{equation}
for $t\in[T,b]$. Moreover, $(x,T)$ satisfies
\begin{equation}
\label{CT1delay}
\begin{cases}
L\xs (T)+\partial_1\phi(T,x(T))+\partial_2\phi(T,x(T))x'(T)=0,\\
\left[\gamma_1 \, {_tI_T^{1-\a}} \partial_3L\xs (t)
-\gamma_2 \, {_TI_t^{1-\b}} \partial_3 L\xs (t)\right]_{t=T}\\
\qquad\qquad\qquad\qquad+\partial_2 \phi(T,x(T))=0,\\
\gamma_2 \left[ {_TI_t^{1-\b}}\partial_3 L\xs (t)
-{_aI_t^{1-\b}\partial_3L\xs (t)}\right]_{t=b}=0.
\end{cases}
\end{equation}
If $\s<T-a$, then Eq. \eqref{ELeq_1delay} is replaced by the two following ones:
\begin{equation}
\label{ELeq_3delay}
\partial_2 L\xs (t)+D{_{\overline{\gamma}}^{\b,\a}}\partial_3 L\xs (t)+\partial_4 L\xs (t+\s)=0,
\end{equation}
for $t\in[a,T-\s]$, and
\begin{equation}
\label{ELeq_4delay}
\partial_2 L\xs (t)+D{_{\overline{\gamma}}^{\b,\a}}\partial_3 L\xs (t)=0,
\end{equation}
for $t\in[T-\s,T]$.
\end{Theorem}

\begin{proof}
Consider variations of the solution $\left(x+\e h,T+\e\Delta{T}\right)$,
where $h\in C^1([a-\s,b];\bR)$ is such that $h(t)=0$ for all $t\in[a-\s,a]$,
and $\e,\triangle T$ are two reals. If we define
$j(\e)=\mathcal J\left(x+\e h,T+\e\Delta{T}\right)$, then $j'(0)=0$, that is,
\begin{multline}
\label{eq_derjdelay}
\int_a^T \Big( \partial_2 L\xs (t) h(t)+ \partial_3 L\xs (t) \DC h(t)\\
+ \partial_4 L\xs (t) h(t-\s) \Big)dt+L\xs (T)\Delta T\\
+\partial_1 \phi \left(T, x(T)\right)\Delta T
+ \partial_2\phi(T, x(T))\left[h(T)+x'(T) \triangle T \right]=0.
\end{multline}
First, suppose that $\s \geq T-a$. In this case, since
$$
\int_a^T  \partial_4 L\xs (t) h(t-\s) \, dt
=\int_{a-\s}^{T-\s}  \partial_4 L\xs (t+\s) h(t) \, dt
$$
and $h \equiv 0$ on $[a-\s,a]$, this term vanishes in \eqref{eq_derjdelay}
and we obtain Eq. \eqref{FP_eq_derj}. The rest of the proof is similar
to the one presented in Section \ref{sec:FP}, Theorem \ref{FP_teo1}, and we obtain
\eqref{ELeq_1delay}--\eqref{CT1delay}. Suppose now that $\s<T-a$. In this case,
we have that
\begin{equation*}
\begin{split}
\int_a^T  \partial_4 L\xs (t) h(t-\s) \, dt
&=\int_{a-\s}^{T-\s}  \partial_4 L\xs (t+\s) h(t) \, dt\\
&=\int_{a}^{T-\s}  \partial_4 L\xs (t+\s) h(t) \, dt.
\end{split}
\end{equation*}
Next, we evaluate the integral
\begin{multline*}
\int_a^T \partial_3 L\xs (t) \DC h(t) \,dt\\
=\int_a^{T-\s} \partial_3 L\xs (t) \DC h(t) \,dt\\
+\int_{T-\s}^T \partial_3 L\xs (t) \DC h(t) \,dt.
\end{multline*}
For the first integral, integrating by parts, we have
\begin{align*}
&\int_a^{T-\s} \partial_3 L\xs (t) \DC h(t) \,dt
=\gamma_1 \int_a^{T-\s} \partial_3 L\xs (t)\LC h(t)dt\\
&+ \gamma_2 \left[\int_a^b \partial_3 L\xs (t) \RCb h(t)dt
- \int_{T-\s}^b \partial_3 L\xs (t) \RCb h(t)dt \right]\\
&=\int_a^{T-\s} h(t)\left[{\gamma_1} {_tD_{T-\s}^{\a}}
\partial_3L\xs (t)+{\gamma_2} {_aD_t^{\b}} \partial_3L\xs (t)\right] dt\\
&+\int_{T-\s}^b\gamma_2 h(t)\left[{_aD_t^{\b}} \partial_3 L\xs (t)
-{_{T-\s}D_t^{\b}} \partial_3 L\xs (t)\right] dt\\
&+ \left[h(t)\left[{\gamma_1} {_tI_{T-\s}^{1-\a}} \partial_3L\xs (t)
-{\gamma_2} {_{T-\s}I_t^{1-\b}} \partial_3L\xs (t)\right]\right]_{t=T-\s}\\
&+ \left[\gamma_2 h(t)\left[-{_aI_t^{1-\b}} \partial_3L\xs (t)
+ {_{T-\s}I_t^{1-\b}} \partial_3L\xs (t)\right]\right]_{t=b}.
\end{align*}
For the second integral, in a similar way, we deduce that
\begin{align*}
&\int_{T-\s}^T \partial_3 L\xs (t) \DC h(t) \,dt \\
&=\gamma_1\left[\int_a^T \partial_3 L\xs (t)\LC h(t)dt
- \int_a^{T-\s} \partial_3 L\xs (t)\LC h(t)dt\right]\\
&+ \gamma_2 \left[\int_{T-\s}^b \partial_3 L\xs (t) \RCb h(t)dt
- \int_T^b \partial_3 L\xs (t) \RCb h(t)dt \right]\\
&=\int_a^{T-\s} \gamma_1  h(t)\left[{_tD_T^{\a}} \partial_3L\xs (t)
-{_tD_{T-\s}^{\a}} \partial_3L\xs (t)\right] dt\\
&+\int_{T-\s}^T h(t)\left[{\gamma_1}{_tD_T^{\a}} \partial_3L\xs (t)
+{\gamma_2}{_{T-\s}D_t^{\b}} \partial_3L\xs (t)\right] dt\\
&+\int_T^b\gamma_2 h(t)\left[{_{T-\s}D_t^{\b}} \partial_3 L\xs (t)
-{_TD_t^{\b}} \partial_3 L\xs (t)\right] dt\\
&+ \left[h(t)\left[-{\gamma_1} {_tI_{T-\s}^{1-\a}} \partial_3L\xs (t)
+{\gamma_2} {_{T-\s}I_t^{1-\b}} \partial_3L\xs (t)\right]\right]_{t=T-\s}\\
&+ \left[h(t)\left[{\gamma_1} {_tI_T^{1-\a}} \partial_3L\xs (t)
-{\gamma_2} {_TI_t^{1-\b}} \partial_3L\xs (t)\right]\right]_{t=T}\\
&+ \left[\gamma_2 h(t)\left[-{_{T-\s}I_t^{1-\b}} \partial_3L\xs (t)
+ {_TI_t^{1-\b}} \partial_3L\xs (t)\right]\right]_{t=b}.
\end{align*}
Replacing the above equalities into \eqref{eq_derjdelay}, we prove that
\begin{align*}
&0= \int_a^{T-\s} h(t)\left[\partial_2 L\xs (t)
+ D_{\overline{\gamma}}^{\b,\a}\partial_3 L\xs (t)+\partial_4 L\xs (t+\s)\right]dt\\
&+\int_{T-\s}^Th(t)\left[\partial_2 L\xs (t)+ D_{\overline{\gamma}}^{\b,\a}\partial_3 L\xs (t)\right]dt\\
&+ \int_T^b \gamma_2 h(t) \left[_aD_t^{\b}\partial_3 L\xs (t)-{_TD_t^{\b}\partial_3 L\xs (t)}\right]dt\\
&+ h(T)\left[\gamma_1 \, {_tI_T^{1-\a}\partial_3L\xs (t)}-{\gamma_2 \, {_TI_t^{1-\b}\partial_3L\xs (t)}}
+\partial_2 \phi(t,x(t))\right]_{t=T}\\
&+\Delta T \left[ L\xs (t)+\partial_1\phi(t,x(t))+\partial_2\phi(t,x(t))x'(t)  \right]_{t=T}\\
&+ h(b)\left[ \gamma_2  \left( _TI_t^{1-\b}\partial_3 L\xs (t)-{_aI_t^{1-\b}
\partial_3L\xs (t)}\right) \right]_{t=b}.
\end{align*}
By the arbitrariness of $h$ in $[a,b]$ and of $\triangle T$,
we obtain Eqs. \eqref{ELeq_2delay}--\eqref{ELeq_4delay}.
\end{proof}


\subsection{Example}
\label{subsec:delay_example}

Let $\alpha,\beta:[0,b]^2\to(0,1)$, $\s=1$, $f$ be a function of class
$C^1$, and $\hat{f}(t)= \DC f(t)$. Consider the following problem of
the calculus of variations:
\begin{multline*}
\mathcal{J}(x,T)
=\int_0^T\Biggl[ \left( \DC x(t)-\hat{f}(t)\right)^2+(x(t)-f(t))^2\\
+(x(t-1)-f(t-1))^2-t-2\Biggr]\,dt + T^2 \rightarrow \min
\end{multline*}
 subject to the condition $x(t)=f(t)$ for all $t\in[-1,0]$.\\
In this case, we can easily verify that $(x,T)=(f,2)$ satisfies all the
conditions in Theorem~\ref{teo:delay} and that it is actually the
(global) minimizer of the problem.


\section{Isoperimetric problems}
\label{sec:Iso}
\index{isoperimetric problem}

Isoperimetric problems are optimization problems, that consist in
minimizing or maximizing a cost functional subject to an integral
constraint. From the variational problem with dependence on a combined
Caputo derivative of variable fractional order (see Definition~\ref{Comb_Cap})
discussed in Section \ref{sec:FP}, here we study two variational problems
subject to an additional integral constraint.
In each of the problems, the terminal point in the cost integral, as well
as the terminal state, are considered to be free, and we obtain corresponding
natural boundary conditions.

In Sections~\ref{subsec:problemI} and \ref{subsec:problemII}, we study necessary
optimality conditions in order to determine the minimizers for each of the problems.
We end this section with an example (Section~\ref{subsec:exeiso}).\\
\\

For the two isoperimetric problems considered in the next sections, let $D$ be
the set given by \eqref{set} and $[x]_\gamma^{\alpha, \beta}(t)$ the
vector \eqref{operat_simpli}.

Picking up the problem of the fractional calculus of variations 
with variable-order, discussed in Section~\ref{sec:FP}, 
we consider also a differentiable Lagrangian 
$L:[a,b]\times \bR^2 \to\bR$ and the functional
$\mathcal{J}:D\rightarrow \bR$ of the form
\begin{equation}
\label{funct1_iso}
\mathcal{J}(x,T)=\int_a^T L[x]_\gamma^{\alpha, \beta}(t) dt + \phi(T,x(T)),
\end{equation}
where the terminal cost function \index{terminal cost function}
$\phi:[a,b]\times \bR\to\bR$ is of class $C^1$.

In the sequel, we need the auxiliary notation of the dual fractional derivative:
\begin{equation}
\label{aux:FD}
D_{\overline{\gamma},c}^{\b,\a}=\gamma_2 \, {_aD_t^{\b}}+\gamma_1 \, {_tD_c^{\a}},
\quad \mbox{where} \quad \overline{\gamma}=(\gamma_2,\gamma_1)
\quad \mbox{and} \quad c\in(a,b].
\end{equation}

With the functional $\mathcal{J}$, defined by \eqref{funct1_iso}, we consider
two different isoperimetrics problems.


\subsection{Necessary optimality conditions I}
\label{subsec:problemI}
The first fractional isoperimetric problem of the calculus of variations
is Problem~\ref{Problem_iso1}.

\begin{Problem}
\label{Problem_iso1}
Determine the local minimizers of $\mathcal{J}$ over all $(x,T)\in D$
satisfying a boundary condition \index{boundary conditions}
\begin{equation}
\label{bcxa_iso}
x(a)=x_a
\end{equation}
for a fixed $x_a\in \bR$ and an integral constraint of the form
\begin{equation}
\label{IsoConst_1}
\int_a^T g[x]_\gamma^{\alpha, \beta}(t)dt=\psi(T),
\end{equation}
where $g:C^{1}\left([a,b]\times \bR^2 \right)\to\bR$
and $\psi: [a,b]\to \bR$ are two differentiable functions. The terminal
time $T$ and terminal state $x(T)$ are free.
\end{Problem}

In this problem, the condition of the form \eqref{IsoConst_1} is called
an isoperimetric constraint. The next theorem gives fractional necessary
optimality conditions for Problem~\ref{Problem_iso1}.

\begin{Theorem}
\label{teo1_iso1}
Suppose that $(x,T)$ gives a local minimum  for functional \eqref{funct1_iso} on
$D$ subject to the boundary condition \eqref{bcxa_iso} and the isoperimetric
constraint \eqref{IsoConst_1}. If $(x,T)$ does not satisfy 
the Euler--Lagrange\index{Euler--Lagrange equations}
equations with respect to the isoperimetric constraint, that is,
if one of the two following conditions are not verified,
\begin{equation}
\label{eq:eq1}
\partial_2 g[x]_\gamma^{\alpha, \beta}(t) + D_{\overline{\gamma},T}^{\b,\a}
\partial_3g[x]_\gamma^{\alpha, \beta}(t)=0, \quad t\in[a,T],
\end{equation}
or
\begin{equation}
\label{eq:eq2}
\gamma_2\left[ _aD_t^{\b} \partial_3 g[x]_\gamma^{\alpha, \beta}(t)
- {_TD_t}^{\b} \partial_3 g[x]_\gamma^{\alpha, \beta}(t) \right]=0,
\quad t\in[T,b],
\end{equation}
then there exists a constant $\lambda$ such that, if we define the function
$F:[a,b]\times \bR^2\to\bR$ by $F= L-\lambda g$,
$(x,T)$ satisfies the fractional Euler--Lagrange equations
\begin{equation}
\label{ELeq_1}
\partial_2 F[x]_\gamma^{\alpha, \beta}(t)
+D{_{\overline{\gamma},T}^{\b,\a}}\partial_3 F[x]_\gamma^{\alpha, \beta}(t)=0
\end{equation}
on the interval $[a,T]$ and
\begin{equation}
\label{ELeq_2}
\gamma_2\left({\LDb}\partial_3F[x]_\gamma^{\alpha, \beta}(t)
-{ _TD{_t^{\b}}\partial_3F[x]_\gamma^{\alpha, \beta}(t)}\right)=0
\end{equation}
on the interval $[T,b]$. Moreover,
$(x,T)$ satisfies the transversality conditions
\index{transversality conditions}
\begin{equation}
\label{CT1}
\begin{cases}
F[x]_\gamma^{\alpha, \beta}(T)+\partial_1\phi(T,x(T))
+\partial_2\phi(T,x(T))x'(T)+\lambda \psi'(T)=0,\\
\left[\gamma_1 \, {_tI_T^{1-\a}\partial_3F[x]_\gamma^{\alpha, \beta}(t)}
-{\gamma_2 \, {_TI_t^{1-\b}\partial_3F[x]_\gamma^{\alpha, \beta}(t)}}\right]_{t=T}
+\partial_2 \phi(T,x(T))=0,\\
\gamma_2 \left[ {_TI_t^{1-\b}}\partial_3 F[x]_\gamma^{\alpha, \beta}(t)
-{_aI_t^{1-\b}\partial_3F[x]_\gamma^{\alpha, \beta}(t)}\right]_{t=b}=0.
\end{cases}
\end{equation}
\end{Theorem}

\begin{proof}
Consider variations of the optimal solution $(x,T)$ of the type
\begin{equation}
\label{av}
(x^{*},T^{*})=\left(x+\e_{1}{h_{1}}+\e_{2}{h_{2}},T+\e_{1}\Delta{T}\right),
\end{equation}
where, for each $i \in \{1,2\}$, $\e_{i} \in\bR$ is a small parameter,
$h_{i}\in C^1([a,b];\bR)$ satisfies $h_{i}(a)=0$, and $\triangle T \in\bR$.
The additional term $\e_{2}{h_{2}}$ must be selected so that the admissible
variations $(x^{*},T^{*})$ satisfy the isoperimetric constraint \eqref{IsoConst_1}.
For a fixed choice of $h_{i}$, let
$$
i(\e_{1},\e_{2})=\int_a^{T+\e_{1} \triangle T}
g[x^{*}]_\gamma^{\alpha, \beta}(t)dt-\psi(T+\e_1\triangle T).
$$
For $\e_{1}=\e_{2}=0$, we obtain that
$$
i(0,0)=\int_a^{T} g[x]_\gamma^{\alpha, \beta}(t)dt-\psi(T)=\psi(T)-\psi(T)=0.
$$
The derivative $\dfrac{\partial i}{\partial \e_{2}}$ is given by
\begin{equation*}
\dfrac{\partial i}{\partial \e_{2}}=\int_a^{T+\e_{1} \triangle T} \left(
\partial_2 g[x^{*}]_\gamma^{\alpha, \beta}(t) h_{2}(t)
+ \partial_3 g[x^{*}]_\gamma^{\alpha, \beta}(t) \DC h_{2}(t) \right)dt.
\end{equation*}
For $\e_{1}=\e_{2}=0$ one has
\begin{equation}
\label{funct5}
\left. \dfrac{\partial i}{\partial \e_{2}} \right|_{(0,0)}=\int_a^{T} \left(
\partial_2 g[x]_\gamma^{\alpha, \beta}(t) h_{2}(t) \
+ \partial_3 g[x]_\gamma^{\alpha, \beta}(t) \DC h_{2}(t) \right)dt.
\end{equation}
The second term in \eqref{funct5} can be written as
\begin{equation}
\label{term2}
\begin{split}
\int_a^T  &\partial_3 g[x]_\gamma^{\alpha, \beta}(t) \DC h_{2}(t) dt\\
&=\int_a^T \partial_3 g[x]_\gamma^{\alpha, \beta}(t)\left[\gamma_1
\, \LC h_{2}(t)+\gamma_2 \, \RCb h_{2}(t)\right]dt\\
&=\gamma_1 \int_a^T \partial_3 g[x]_\gamma^{\alpha, \beta}(t)\LC h_{2}(t)dt  \\
& + \gamma_2 \left[ \int_a^b \partial_3
g[x]_\gamma^{\alpha, \beta}(t) \RCb h_{2}(t)dt
- \int_T^b \partial_3 g[x]_\gamma^{\alpha, \beta}(t) \RCb h_{2}(t)dt \right].
\end{split}
\end{equation}
Using the fractional integrating by parts formula, \eqref{term2} is equal to
\begin{equation*}
\begin{split}
\int_a^T & h_{2}(t) \left[ \gamma_1 {_tD_T}^{\a} \partial_3
g[x]_\gamma^{\alpha, \beta}(t) + \gamma_2 {_aD_t}^{\b}
\partial_3 g[x]_\gamma^{\alpha, \beta}(t) \right] dt\\
+ &  \int_T^b \gamma_2 h_{2}(t) \left[ _aD_t^{\b} \partial_3
g[x]_\gamma^{\alpha, \beta}(t) - {_TD_t}^{\b} \partial_3
g[x]_\gamma^{\alpha, \beta}(t) \right] dt\\
+ & \Biggl[ h_{2}(t) \left( \gamma_1 {{_tI_T}}^{1-\a} \partial_3
g[x]_\gamma^{\alpha, \beta}(t) - \gamma_2 {_TI_t}^{1-\b} \partial_3
g[x]_\gamma^{\alpha, \beta}(t) \Biggr) \right]_{t=T}\\
+ & \Biggl[ \gamma_2 h_{2}(t)  \left( {{_TI_t}}^{1-\b} \partial_3
g[x]_\gamma^{\alpha, \beta}(t) - {_aI_t}^{1-\b} \partial_3
g[x]_\gamma^{\alpha, \beta}(t) \Biggr) \right]_{t=b}.
\end{split}
\end{equation*}
Substituting these relations into \eqref{funct5}, and considering the fractional
operator $D_{\overline{\gamma},c}^{\b,\a}$
as defined in \eqref{aux:FD}, we obtain that
\begin{equation*}
\begin{split}
\left.\dfrac{\partial i}{\partial \e_{2}} \right|_{(0,0)}&
= \int_a^T  h_{2}(t) \left[ \partial_2 g[x]_\gamma^{\alpha, \beta}(t)
+ D_{\overline{\gamma},T}^{\b,\a}\partial_3
g[x]_\gamma^{\alpha, \beta}(t) \right] dt\\
&+ \int_T^b \gamma_2 h_{2}(t) \left[ _aD_t^{\b} \partial_3
g[x]_\gamma^{\alpha, \beta}(t) - {_TD_t}^{\b} \partial_3
g[x]_\gamma^{\alpha, \beta}(t) \right] dt\\
&+ \Biggl[ h_{2}(t) \left( \gamma_1 {{_tI_T}}^{1-\a} \partial_3
g[x]_\gamma^{\alpha, \beta}(t) - \gamma_2 {_TI_t}^{1-\b} \partial_3
g[x]_\gamma^{\alpha, \beta}(t) \Biggr) \right]_{t=T}\\
&+ \Biggl[ \gamma_2 h_{2}(t)  \left( {{_TI_t}}^{1-\b} \partial_3
g[x]_\gamma^{\alpha, \beta}(t) - {_aI_t}^{1-\b} \partial_3
g[x]_\gamma^{\alpha, \beta}(t) \Biggr) \right]_{t=b}.
\end{split}
\end{equation*}
Since \eqref{eq:eq1} or \eqref{eq:eq2} fails, there exists a function $h_{2}$
such that
$$
\left.\dfrac{\partial i}{\partial \e_{2}} \right|_{(0,0)}\neq 0.
$$
In fact, if not, from the arbitrariness of the function $h_2$ and the
fundamental lemma of the calculus of the variations, \eqref{eq:eq1}
and \eqref{eq:eq2} would be verified. Thus, we may apply the implicit
function theorem, that ensures the existence of a function $\e_{2}(\cdot)$,
defined in a neighborhood of zero, such that $i(\e_1,\e_2(\e_1))=0$.
In conclusion, there exists a subfamily of variations of the form \eqref{av}
that verifies the integral constraint \eqref{IsoConst_1}. We now seek to prove
the main result. For that purpose, consider the auxiliary function
$j(\e_1,\e_2)=\mathcal{J}(x^{*},T^{*})$.

By hypothesis, function $j$ attains a local minimum at $(0,0)$ when subject
to the constraint $i(\cdot,\cdot)=0$, and we proved before that
$\nabla i(0,0)\not= (0,0)$. Applying the Lagrange multiplier rule, we ensure
the existence of a number $\lambda$ such that
$$
\nabla \left(j(0,0)-\lambda i(0,0)\right)=(0,0).
$$
In particular,
\begin{equation}
\label{funct7}
\dfrac{\partial \left(j-\lambda i\right)}{\partial \e_{1}} (0,0)=0.
\end{equation}
Let $F=L-\lambda g$. The relation \eqref{funct7} can be written as
\begin{equation}
\label{funct8}
\begin{split}
&0= \int_a^Th_{1}(t)\left[\partial_2 F[x]_\gamma^{\alpha, \beta}(t)
+ D_{\overline{\gamma},T}^{\b,\a}\partial_3 F[x]_\gamma^{\alpha, \beta}(t)\right]dt\\
&+ \int_T^b \gamma_2 h_{1}(t) \left[_aD_t^{\b}\partial_3 F[x]_\gamma^{\alpha, \beta}(t)
-{_TD_t^{\b}\partial_3 F[x]_\gamma^{\alpha, \beta}(t)}\right]dt\\
&+ h_{1}(T)\left[\gamma_1 \, {_tI_T^{1-\a}\partial_3F[x]_\gamma^{\alpha, \beta}(t)}
-{\gamma_2 \, {_TI_t^{1-\b}\partial_3F[x]_\gamma^{\alpha, \beta}(t)}}
+\partial_2 \phi(t,x(t))\right]_{t=T}\\
&+\Delta T \left[ F[x]_\gamma^{\alpha, \beta}(t)+\partial_1\phi(t,x(t))
+\partial_2\phi(t,x(t))x'(t)+\lambda\psi'(t)  \right]_{t=T}\\
&+ h_{1}(b) \gamma_2  \left[ _TI_t^{1-\b}\partial_3 F[x]_\gamma^{\alpha, \beta}(t)
-{_aI_t^{1-\b}\partial_3F[x]_\gamma^{\alpha, \beta}(t)} \right]_{t=b}.
\end{split}
\end{equation}
As $h_{1}$ and $\triangle T$ are arbitrary, we can choose $\triangle T=0$ and
$h_{1}(t)=0$ for all $t\in[T,b]$. But $h_{1}$ is arbitrary in $t\in[a,T)$.
Then, we obtain the first necessary condition \eqref{ELeq_1}:
$$
\partial_2 F[x]_\gamma^{\alpha, \beta}(t)
+D{_{\overline{\gamma},T}^{\b,\a}}\partial_3 F[x]_\gamma^{\alpha, \beta}(t)=0
\quad \forall t \in [a,T].
$$
Analogously, considering $\triangle T=0$ and $h_{1}(t)=0$
for all $t\in[a,T]\cup\{b\}$, and $h_{1}$ arbitrary on $(T,b)$,
we obtain the second necessary condition \eqref{ELeq_2}:
$$
\gamma_2\left({\LDb}\partial_3F[x]_\gamma^{\alpha, \beta}(t)
-{ _TD{_t^{\b}}\partial_3F[x]_\gamma^{\alpha, \beta}(t)}\right)=0
\quad \forall t \in [T,b].
$$
As $(x,T)$ is a solution to the necessary conditions \eqref{ELeq_1}
and \eqref{ELeq_2}, then equation \eqref{funct8} takes the form
\begin{equation}
\label{eq_derj3}
\begin{split}
&h_{1}(T)\left[\gamma_1 \, {_tI_T^{1-\a}\partial_3F[x]_\gamma^{\alpha, \beta}(t)}
-{\gamma_2 \, {_TI_t^{1-\b}\partial_3F[x]_\gamma^{\alpha, \beta}(t)}}
+\partial_2 \phi(t,x(t))\right]_{t=T}\\
&+\Delta T \left[ F[x]_\gamma^{\alpha, \beta}(t)+\partial_1\phi(t,x(t))
+\partial_2\phi(t,x(t))x'(t)+\lambda\psi'(t)  \right]_{t=T}\\
&+ h_{1}(b)\left[\gamma_2  \left( _TI_t^{1-\b}\partial_3 F[x]_\gamma^{\alpha, \beta}(t)
-{_aI_t^{1-\b}\partial_3F[x]_\gamma^{\alpha, \beta}(t)}\right) \right]_{t=b}=0.
\end{split}
\end{equation}
Transversality conditions \eqref{CT1} are obtained
for appropriate choices of variations.
\end{proof}

In the next theorem, considering the same Problem~\ref{Problem_iso1},
we rewrite the transversality conditions \eqref{CT1} in terms of the increment
on time $\Delta T$ and on the increment of space $\Delta x_T$ given by
\begin{equation}
\label{Dxt}
\Delta x_T = (x+h_1)(T+\Delta T)-x(T).
\end{equation}

\begin{Theorem}
\label{teo2_iso1}
Let $(x,T)$ be a local minimizer to the functional \eqref{funct1_iso} on $D$ subject
to the  boundary condition \eqref{bcxa_iso} and the isoperimetric constraint
\eqref{IsoConst_1}. Then $(x,T)$ satisfies the transversality conditions
\index{transversality conditions}
\begin{equation}
\label{CT2}
\begin{cases}
F[x]_\gamma^{\alpha, \beta}(T)+\partial_1\phi(T,x(T))+\lambda \psi'(T)\\
\qquad + x'(T) \left[ \gamma_2 {_TI}_t^{1-\b} \partial_3F[x]_\gamma^{\alpha, \beta}(t)
- \gamma_1 {_tI_T^{1-\a} \partial_3F[x]_\gamma^{\alpha, \beta}(t)} \right]_{t=T} =0,\\
\left[\gamma_1 \, {_tI_T^{1-\a}\partial_3F[x]_\gamma^{\alpha, \beta}(t)}
-{\gamma_2 \, {_TI_t^{1-\b}\partial_3F[x]_\gamma^{\alpha, \beta}(t)}}\right]_{t=T}
+\partial_2 \phi(T,x(T))=0,\\
\gamma_2 \left[ {_TI_t^{1-\b}}\partial_3 F[x]_\gamma^{\alpha, \beta}(t)
-{_aI_t^{1-\b}\partial_3F[x]_\gamma^{\alpha, \beta}(t)}\right]_{t=b}=0.
\end{cases}
\end{equation}
\end{Theorem}

\begin{proof}
Suppose $(x^{*},T^{*})$ is an admissible variation of the form \eqref{av}
with $\epsilon_1=1$ and $\epsilon_2=0$. Using Taylor's expansion up to first
order for a small $\Delta T$, and restricting the set of variations to those
for which $h_{1}'(T)=0$, we obtain the increment $\Delta x_T$ on $x$:
$$
(x+h_{1})\left( T + \Delta T \right)=(x+h_{1})(T)+x'(T) \Delta T + O(\Delta T)^{2}.
$$
Relation \eqref{Dxt} allows us to express $h_{1}(T)$
in terms of $\Delta T$ and $\Delta x_T$:
$$
h_{1}(T)= \Delta x_T - x'(T) \Delta T + O(\Delta T)^{2}.
$$
Substituting this expression into \eqref{eq_derj3}, and using appropriate
choices of variations, we obtain the new transversality conditions \eqref{CT2}.
\end{proof}

\begin{Theorem}
\label{teo3_iso1}
Suppose that $(x,T)$ gives a local minimum  for functional \eqref{funct1_iso} on $D$
subject to the  boundary condition \eqref{bcxa_iso} and the isoperimetric constraint
\eqref{IsoConst_1}. Then, there exists $(\lambda_{0}, \lambda)\neq(0,0)$ such
that, if we define the function $F:[a,b]\times \bR^2 \to\bR$ by
$F=\lambda_{0} L-\lambda g$, $(x,T)$ satisfies the following fractional
Euler--Lagrange equations:
\begin{equation*}
\partial_2 F[x]_\gamma^{\alpha, \beta}(t)
+D{_{\overline{\gamma},T}^{\b,\a}}\partial_3 F[x]_\gamma^{\alpha, \beta}(t)=0
\end{equation*}
on the interval $[a,T]$, and
\begin{equation*}
\gamma_2\left({\LDb}\partial_3F[x]_\gamma^{\alpha, \beta}(t)
-{ _TD{_t^{\b}}\partial_3F[x]_\gamma^{\alpha, \beta}(t)}\right)=0
\end{equation*}
on the interval $[T,b]$.
\end{Theorem}

\begin{proof}
If $(x,T)$ does not verifies \eqref{eq:eq1} or \eqref{eq:eq2},
then the hypothesis of Theorem~\ref{teo1_iso1} is satisfied and we
prove Theorem~\ref{teo3_iso1} considering $\lambda_0=1$.
If $(x,T)$ verifies \eqref{eq:eq1} and \eqref{eq:eq2},
then we prove the result by considering $\lambda=1$ and $\lambda_{0}=0$.
\end{proof}


\subsection{Necessary optimality conditions II}
\label{subsec:problemII}
We now consider a new isoperimetric type problem
with the isoperimetric constraint of form
\begin{equation}
\label{IsoConst_2}
\int_a^b g[x]_\gamma^{\alpha, \beta}(t)dt=C,
\end{equation}
where $C$ is a given real number.

\begin{Problem}
\label{Problem_iso2}
Determine the local minimizers of $\mathcal{J}$ \eqref{funct1_iso},
over all $(x,T)\in D$ satisfying a boundary condition \index{boundary conditions}
\begin{equation}
\label{bcxa}
x(a)=x_a
\end{equation}
for a fixed $x_a\in \bR$ and an integral constraint of the form \eqref{IsoConst_2}.
\end{Problem}

In the follow theorem, we give fractional necessary optimality conditions for Problem~\ref{Problem_iso2}.

\begin{Theorem}
\label{teo4}
Suppose that $(x,T)$ gives a local minimum for functional \eqref{funct1_iso} on $D$
subject to the  boundary condition \eqref{bcxa} and the isoperimetric constraint
\eqref{IsoConst_2}. If $(x,T)$ does not satisfy the Euler--Lagrange equation\index{Euler--Lagrange equations}
with respect to the isoperimetric constraint, that is, the condition
\begin{equation*}
\partial_2 g[x]_\gamma^{\alpha, \beta}(t) + D_{\overline{\gamma},b}^{\b,\a}
\partial_3g[x]_\gamma^{\alpha, \beta}(t)=0, \quad t\in[a,b],
\end{equation*}
is not satisfied, then there exists $\lambda\neq0$ such that,
if we define the function $F:[a,b]\times \bR^2 \to\bR$ by $F=L-\lambda g$,
$(x,T)$ satisfies the fractional Euler--Lagrange equations
\begin{equation}
\label{ELeq_5}
\partial_2 F[x]_\gamma^{\alpha, \beta}(t)
+ D_{\overline{\gamma},T}^{\b,\a}\partial_3 L[x]_\gamma^{\alpha, \beta}(t)
-\lambda D_{\overline{\gamma},b}^{\b,\a}\partial_3 g[x]_\gamma^{\alpha, \beta}(t)=0
\end{equation}
on the interval $[a,T]$, and
\begin{multline}
\label{ELeq_6}
\gamma_2 \left(_aD_t^{\b}\partial_3 F[x]_\gamma^{\alpha, \beta}(t)
-{_TD_t^{\b}\partial_3 L[x]_\gamma^{\alpha, \beta}(t)} \right)\\
-\lambda \left(\partial_2 g[x]_\gamma^{\alpha, \beta}(t)
+\gamma_1 {_tD_{b}^{\a}}\partial_3 g[x]_\gamma^{\alpha, \beta}(t)\right)=0
\end{multline}
on the interval $[T,b]$. Moreover, $(x,T)$ satisfies
the transversality conditions
\index{transversality conditions}
\begin{equation}
\label{CT3}
\begin{cases}
L[x]_\gamma^{\alpha, \beta}(T)+\partial_1\phi(T,x(T))+\partial_2\phi(T,x(T))x'(T) =0,\\
\left[\gamma_1 \, {_tI_T^{1-\a}\partial_3L[x]_\gamma^{\alpha, \beta}(t)}
-{\gamma_2 \, {_TI_t^{1-\b}\partial_3L[x]_\gamma^{\alpha, \beta}(t)}}
+\partial_2 \phi(t,x(t)) \right]_{t=T}=0\\
\left[-\lambda \gamma_1 {_tI_b^{1-\a}\partial_3 g[x]_\gamma^{\alpha, \beta}(t)}\right.\\
\left.\quad+\gamma_2 \left({_TI_t^{1-\b}\partial_3 L[x]_\gamma^{\alpha, \beta}(t)}
-{_aI_t^{1-\b}\partial_3F[x]_\gamma^{\alpha, \beta}(t)} \right) \right]_{t=b}=0.
\end{cases}
\end{equation}
\end{Theorem}

\begin{proof}
Similarly as done to prove Theorem~\ref{teo1_iso1}, let
$$
(x^{*},T^{*})=\left(x+\e_{1}{h_{1}}+\e_{2}{h_{2}},T+\e_{1}\Delta{T}\right)
$$
be a variation of the solution, and define
$$
i(\e_{1},\e_{2})=\int_a^b g[x^{*}]_\gamma^{\alpha, \beta}(t)dt - C.
$$
The derivative $\dfrac{\partial i}{\partial \e_{2}}$, when $\e_{1}=\e_{2}=0$, is
\begin{equation*}
\left. \dfrac{\partial i}{\partial \e_{2}} \right|_{(0,0)}=\int_a^b \left(
\partial_2 g[x]_\gamma^{\alpha, \beta}(t) h_{2}(t) \
+ \partial_3 g[x]_\gamma^{\alpha, \beta}(t) \DC h_{2}(t) \right)dt.
\end{equation*}
Integrating by parts and choosing variations such that $h_2(b)=0$, we have
$$
\left.\dfrac{\partial i}{\partial \e_{2}} \right|{(0,0)}
=\int_a^b h_{2}(t) \left[ \partial_2 g[x]_\gamma^{\alpha, \beta}(t)
+ D_{\overline{\gamma},b}^{\b,\a}\partial_3g[x]_\gamma^{\alpha, \beta}(t)
\right] dt.
$$
Thus, there exists a function $h_{2}$ such that
$$
\left.\dfrac{\partial i}{\partial \e_{2}} \right|{(0,0)}\not=0.
$$
We may apply the implicit function theorem to conclude that there exists
a subfamily of variations satisfying the integral constraint.
Consider the new function $j(\e_1,\e_2)=\mathcal{J}(x^{*},T^{*})$.
Since $j$ has a local minimum at $(0,0)$ when subject to the constraint
$i(\cdot,\cdot)=0$ and $\nabla i(0,0)\not= (0,0)$, there exists a number $\lambda$
such that
\begin{equation}
\label{eqMLgrg}
\dfrac{\partial }{\partial \e_{1}} \left(j-\lambda i\right)(0,0)=0.
\end{equation}
Let $F=L-\lambda g$. Relation \eqref{eqMLgrg} can be written as
\begin{equation*}
\begin{split}
& \int_a^T h_{1}(t)\left[\partial_2 F[x]_\gamma^{\alpha, \beta}(t)
+ D_{\overline{\gamma},T}^{\b,\a}\partial_3 L[x]_\gamma^{\alpha, \beta}(t)
-\lambda D_{\overline{\gamma},b}^{\b,\a}\partial_3
g[x]_\gamma^{\alpha, \beta}(t) \right]dt\\
&+ \int_T^b h_{1}(t) \left[\gamma_2 \left(_aD_t^{\b}\partial_3
F[x]_\gamma^{\alpha, \beta}(t)-{_TD_t^{\b}\partial_3
L[x]_\gamma^{\alpha, \beta}(t)} \right) \right.\\
&\qquad \qquad-\left. \lambda \left(\partial_2 g[x]_\gamma^{\alpha, \beta}(t)
+\gamma_1{_tD_b^{\a}} \partial_3 g[x]_\gamma^{\alpha, \beta}(t)\right)\right]dt\\
&+ h_{1}(T)\left[\gamma_1 \, {_tI_T^{1-\a}\partial_3L[x]_\gamma^{\alpha, \beta}(t)}
-{\gamma_2 \, {_TI_t^{1-\b}\partial_3 L[x]_\gamma^{\alpha, \beta}(t)}}
+\partial_2 \phi(t,x(t)) \right]_{t=T}\\
&+\Delta T \left[ L[x]_\gamma^{\alpha, \beta}(t)+\partial_1\phi(t,x(t))
+\partial_2\phi(t,x(t))x'(t)  \right]_{t=T}\\
&+ h_{1}(b) \left[-\lambda \gamma_1 {_tI_b^{1-\a}\partial_3
g[x]_\gamma^{\alpha, \beta}(t)}\right.\\
&\left. \quad \quad\quad \quad +\gamma_2 \left({_TI_t^{1-\b}
\partial_3 L[x]_\gamma^{\alpha, \beta}(t)}-{_aI_t^{1-\b}\partial_3
F[x]_\gamma^{\alpha, \beta}(t)} \right) \right]_{t=b}=0.
\end{split}
\end{equation*}
Considering appropriate choices of variations, we obtain the first \eqref{ELeq_5}
and the second \eqref{ELeq_6} necessary optimality conditions,
and also the transversality conditions \eqref{CT3}.
\end{proof}

Similarly to Theorem~\ref{teo3_iso1}, the following result holds.

\begin{Theorem}
\label{teo5}
Suppose that $(x,T)$ gives a local minimum  for functional \eqref{funct1_iso} on
$D$ subject to the  boundary condition \eqref{bcxa} and the isoperimetric
constraint \eqref{IsoConst_2}. Then there exists $(\lambda_{0}, \lambda)\neq(0,0)$
such that, if we define the function $F:[a,b]\times \bR^2 \to\bR$
by $F=\lambda_0L-\lambda g$, $(x,T)$ satisfies the fractional Euler--Lagrange equations
\begin{equation*}
\partial_2 F[x]_\gamma^{\alpha, \beta}(t)
+ D_{\overline{\gamma},T}^{\b,\a}\partial_3 L[x]_\gamma^{\alpha, \beta}(t)
-\lambda D_{\overline{\gamma},b}^{\b,\a}\partial_3 g[x]_\gamma^{\alpha, \beta}(t)=0
\end{equation*}
on the interval $[a,T]$, and
\begin{multline*}
\gamma_2 \left(_aD_t^{\b}\partial_3 F[x]_\gamma^{\alpha, \beta}(t)
-{_TD_t^{\b}\partial_3 L[x]_\gamma^{\alpha, \beta}(t)} \right)\\
-\lambda \left(\partial_2 g[x]_\gamma^{\alpha, \beta}(t)
+\gamma_1 {_tD_{b}^{\a}}\partial_3 g[x]_\gamma^{\alpha, \beta}(t)\right)=0
\end{multline*}
on the interval $[T,b]$.
\end{Theorem}


\subsection{Example}
\label{subsec:exeiso}

Let $\alpha(t,\t)=\alpha(t)$ and $\beta(t,\t)=\beta(\t)$. Define the function
$$
\psi(T)=\int_0^T \left(\dfrac{t^{1-\alpha (t)}}{2 \Gamma(2-\alpha(t))}
+\dfrac{(b-t)^{1-\beta (t)}}{2 \Gamma(2-\beta(t))} \right)^2dt
$$
on the interval $[0,b]$ with $b>0$. Consider the functional $J$ defined by
\begin{multline*}
J(x,t)= \int_0^T \Bigg[\alpha(t)+\left(\DC x(t) \right)^2\\
+ \left(\dfrac{t^{1-\alpha (t)}}{2 \Gamma(2-\alpha(t))}
+\dfrac{(b-t)^{1-\beta (t)}}{2 \Gamma(2-\beta(t))}\right)^2 \Bigg]dt
\end{multline*}
for $t \in [0,b]$ and $\gamma =(1/2,1/2)$, subject to the initial condition
$$
x(0)=0
$$
and the isoperimetric constraint
$$
\int_0^T \DC x(t)  \left(\dfrac{t^{1-\alpha (t)}}{2 \Gamma(2-\alpha(t))}
+\dfrac{(b-t)^{1-\beta (t)}}{2 \Gamma(2-\beta(t))}\right)^2 dt=\psi(T).
$$
Define $F=L-\lambda g$ with $\lambda=2$, that is,
$$
F=\alpha(t)+\left({^CD_\gamma^{\alpha(\cdot),\beta(\cdot)}} x(t)
-\frac{t^{1-\alpha(t)}}{2\Gamma(2-\alpha(t))}
-\frac{(b-t)^{1-\beta(t)}}{2\Gamma(2-\beta(t))}\right)^2.
$$
Consider the function $\overline{x}(t)=t$ with $t\in[0,b]$. Because
$$
\DC \overline{x}(t)=\frac{t^{1-\alpha(t)}}{2\Gamma(2-\alpha(t))}
+\frac{(b-t)^{1-\beta(t)}}{2\Gamma(2-\beta(t))},
$$
we have that $\overline x$ satisfies conditions \eqref{ELeq_1}, \eqref{ELeq_2}
and the two last of \eqref{CT1}. Using the first condition of \eqref{CT1}, that is,
$$
\alpha(t)+2\left(\dfrac{T^{1-\alpha (T)}}{2 \Gamma(2-\alpha(T))}
+\dfrac{(b-T)^{1-\beta (T)}}{2 \Gamma(2-\beta(T))} \right)^2=0,
$$
we obtain the optimal time $T$.


\section{Variational problems with holonomic constraints}
\label{sec:Hol}
\index{holonomic constraint}
\index{variational fractional problem! with a holonomic constraint}

In this section, we present a new variational problem subject to a new
type of constraints. A holonomic constraint is a condition of the form
$$g(t,x)=0,$$
where $x=(x_1,x_2,...,x_n)$, $n\geq2$, and $g$ is a given function
(see, e.g., \cite{Cap4:Brunt}).

Consider the space
\begin{equation}
\label{space_HL1}
U=\lbrace(x_1,x_2,T) \in C^1([a,b]) \times C^1([a,b]) \times[a,b]
: x_1(a)=x_{1a} \wedge x_2(a)=x_{2a}\rbrace
\end{equation}
for fixed reals $x_{1a},x_{2a} \in \bR$.
In this section, we consider the following variational problem:

\begin{Problem}
\label{Problem_Holon}
Find functions $x_1$ and $x_2$ that maximize or minimize the functional
$\mathcal{J}$ defined in $U$ by
\begin{equation}
\label{functHL}
\begin{split}
\mathcal{J}(x_1,x_2,T)&=\int_a^T L(t,x_1(t),x_2(t), \DC x_1(t), \DC x_2(t))dt\\
&+ \phi(T,x_1(T),x_2(T)),
\end{split}
\end{equation}
where the admissible functions satisfy the constraint
\index{holonomic constraint}
\begin{equation}
\label{HLConst}
g(t,x_1(t),x_2(t))=0, \quad t\in [a,b],
\end{equation}
called a holonomic constraint, where $g: [a,b]\times\bR^2 \rightarrow \bR$
is a continuous function and continuously differentiable with respect to second
and third arguments.\\
The terminal time $T$ and terminal states $x_1(T)$ and $x_2(T)$ are free and
the Lagrangian $L:[a,b]\times \bR^4 \rightarrow \bR$ is a continuous function
and continuously differentiable with respect to its $i$th argument,
$i \in \lbrace 2,3,4,5\rbrace$. The  terminal cost function
\index{terminal cost function} $\phi:[a,b]\times \bR^2 \to\bR$ is of class $C^1$.
\end{Problem}


\subsection{Necessary optimality conditions}
\label{subsec:neccondHolo}

The next theorem gives fractional necessary optimality conditions to the variational problem with
a holonomic constraint. To simplify the notation, we denote by $x$ the vector
$(x_1,x_2)$; by $\DC x$ we mean the two-dimensional vector $(\DC x_1,\DC x_2)$; and we use the operator
$$
[x]_\gamma^{\alpha, \beta}(t):=\left(t, x(t), \DC x(t)\right).
$$
The next result given necessary optimality conditions for Problem~\ref{Problem_Holon}.

\begin{Theorem}
\label{teo7}
Suppose that $(x,T)$ gives a local minimum to functional $\mathcal{J}$ as in
\eqref{functHL}, under the constraint \eqref{HLConst} and the
boundary conditions defined in \eqref{space_HL1} \index{boundary conditions}. If
$$
\partial_3g(t,x(t))\neq0 \quad \forall t\in [a,b],
$$
then there exists a piecewise continuous function
$\lambda : [a,b]\rightarrow \bR$ such that $(x,T)$
satisfies the following fractional Euler--Lagrange equations:\index{Euler--Lagrange equations}
\begin{equation}
\label{ELeqh_1}
\partial_2 L[x]_\gamma^{\alpha, \beta}(t) + D_{\overline{\gamma},T}^{\b,\a}
\partial_4L[x]_\gamma^{\alpha, \beta}(t) + \lambda(t) \partial_2 g(t,x(t))=0
\end{equation}
and
\begin{equation}
\label{ELeqhPP_1}
\partial_3 L[x]_\gamma^{\alpha, \beta}(t) + D_{\overline{\gamma},T}^{\b,\a}
\partial_5L[x]_\gamma^{\alpha, \beta}(t) + \lambda(t)\partial_3 g(t,x(t))=0
\end{equation}
on the interval $[a,T]$, and
\begin{equation}
\label{ELeqh_2}
\gamma_2\left(_aD_t^{\b} \partial_4 L[x]_\gamma^{\alpha, \beta}(t)
- {_TD_t}^{\b} \partial_4 L[x]_\gamma^{\alpha, \beta}(t)
+ \lambda(t)\partial_2 g(t,x(t))\right)=0
\end{equation}
and
\begin{equation}
\label{ELeqhPP_2}
_aD_t^{\b} \partial_5 L[x]_\gamma^{\alpha, \beta}(t)
- {_TD_t}^{\b} \partial_5 L[x]_\gamma^{\alpha, \beta}(t)
+\lambda(t)\partial_3 g(t,x(t))=0
\end{equation}
on the interval $[T,b]$. Moreover, $(x,T)$ satisfies the transversality conditions
\index{transversality conditions}
\begin{equation}
\label{CTh1}
\begin{cases}
L[x]_\gamma^{\alpha, \beta}(T)  + \partial_1 \phi(T,x(T))
+ \partial_2 \phi(T,x(T))x'_1(T)+ \partial_3 \phi(T,x(T))x'_2(T) =0,\\
\left[\gamma_1 \, {_tI_T^{1-\a}\partial_4L[x]_\gamma^{\alpha, \beta}(t)}
-{\gamma_2 \, {_TI_t^{1-\b}\partial_4L[x]_\gamma^{\alpha, \beta}(t)}}\right]_{t=T}
+\partial_2 \phi(T,x(T))=0,\\
\left[\gamma_1 {{_tI_T}}^{1-\a} \partial_5 L[x]_\gamma^{\alpha, \beta}(t)
- \gamma_2 {_TI_t}^{1-\b} \partial_5 L[x]_\gamma^{\alpha, \beta}(t)\right]_{t=T}
+\partial_3 \phi(T,x(T))=0,\\
\gamma_2 \left[ {_TI_t^{1-\b}}\partial_4 L[x]_\gamma^{\alpha, \beta}(t)
-{_aI_t^{1-\b}\partial_4 L[x]_\gamma^{\alpha, \beta}(t)}\right]_{t=b}=0,\\
\gamma_2 \left[ {_TI_t^{1-\b}}\partial_5 L[x]_\gamma^{\alpha, \beta}(t)
-{_aI_t^{1-\b}\partial_5 L[x]_\gamma^{\alpha, \beta}(t)}\right]_{t=b}=0.
\end{cases}
\end{equation}
\end{Theorem}

\begin{proof}
Consider admissible variations of the optimal solution $(x,T)$ of the type
\begin{equation*}
(x^{*},T^{*})=\left(x+\e h,T+\e \Delta{T}\right),
\end{equation*}
where $\e \in\bR$ is a small parameter, $h=(h_1,h_2)\in C^1([a,b])
\times C^1([a,b])$ satisfies $h_{i}(a)=0$, $i=1,2$, and $\triangle T \in\bR$.
Because
$$
\partial_3g(t,x(t))\neq0 \quad \forall t\in [a,b],
$$
by the implicit function theorem there exists a subfamily of variations of
$(x,T)$ that satisfy \eqref{HLConst}, that is, there exists a unique function
$h_2(\e ,h_1)$ such that the admissible variation ($x^{*}, T^*$) satisfies
the holonomic constraint \eqref{HLConst}:
$$
g(t,x_1(t)+\e h_1(t),x_2(t)+\e h_2(t))=0 \quad \forall t\in [a,b].
$$
Differentiating this condition with respect to $\e$
and considering $\e =0$, we obtain that
\begin{equation*}
\partial_2 g(t,x(t)) h_{1}(t) + \partial_3 g(t,x(t)) h_{2}(t) = 0,
\end{equation*}
which is equivalent to
\begin{equation}
\label{eqh1}
\dfrac{\partial_2 g(t,x(t)) h_{1}(t)}{\partial_3 g(t,x(t))}=-h_{2}(t).
\end{equation}
Define $j$ on a neighbourhood of zero by
$$
j(\e )=\int_a^{T+\e  \triangle T} L[x^{*}]_\gamma^{\alpha, \beta}(t)dt
+\phi(T+\e \triangle T, x^{*}(T+\e \triangle T)).
$$
The derivative $\dfrac{\partial j}{\partial \e }$ for $\e =0$ is
\begin{equation}
\label{eqh2}
\begin{split}
&\left. \dfrac{\partial j}{\partial \e } \right|_{\e=0}
= \int_a^{T}  \left(\partial_2 L[x]_\gamma^{\alpha, \beta}(t) h_{1}(t)
+\partial_3 L[x]_\gamma^{\alpha, \beta}(t) h_{2}(t) \right. \\
& \left.+ \partial_4 L[x]_\gamma^{\alpha, \beta}(t) \DC h_{1}(t)
+ \partial_5 L[x]_\gamma^{\alpha, \beta}(t) \DC h_{2}(t)\right)dt\\
& +L[x]_\gamma^{\alpha, \beta}(T) \triangle T + \partial_1 \phi(T,x(T))\triangle T
+ \partial_2 \phi(T,x(T))\left[h_1(T)+x'_1(T)\triangle T\right]\\
& + \partial_3 \phi(T,x(T))\left[h_2(T)+x'_2(T)\triangle T\right].
\end{split}
\end{equation}
The third term in \eqref{eqh2} can be written as
\begin{equation}
\label{term3}
\begin{split}
\int_a^T  &\partial_4 L[x]_\gamma^{\alpha, \beta}(t) \DC h_{1}(t) dt\\
&=\int_a^T \partial_4 L[x]_\gamma^{\alpha, \beta}(t)\left[\gamma_1
\, \LC h_{1}(t)+\gamma_2 \, \RCb h_{1}(t)\right]dt\\
&=\gamma_1 \int_a^T \partial_4 L[x]_\gamma^{\alpha, \beta}(t)\LC h_{1}(t)dt\\
& + \gamma_2 \left[ \int_a^b
\partial_4 L[x]_\gamma^{\alpha, \beta}(t)\RCb h_{1}(t)dt
- \int_T^b \partial_4 L[x]_\gamma^{\alpha, \beta}(t) \RCb h_{1}(t)dt \right].
\end{split}
\end{equation}
Integrating by parts, \eqref{term3} can be written as
\begin{equation*}
\begin{split}
\int_a^T h_{1}(t) & \left[ \gamma_1 {_tD_T}^{\a} \partial_4
L[x]_\gamma^{\alpha, \beta}(t) + \gamma_2 {_aD_t}^{\b} \partial_4
L[x]_\gamma^{\alpha, \beta}(t) \right] dt\\
+ &  \int_T^b \gamma_2 h_{1}(t) \left[ _aD_t^{\b} \partial_4
L[x]_\gamma^{\alpha, \beta}(t) - {_TD_t}^{\b} \partial_4
L[x]_\gamma^{\alpha, \beta}(t) \right] dt\\
+ & \Biggl[ h_{1}(t) \left( \gamma_1 {{_tI_T}}^{1-\a} \partial_4
L[x]_\gamma^{\alpha, \beta}(t) - \gamma_2 {_TI_t}^{1-\b} \partial_4
L[x]_\gamma^{\alpha, \beta}(t) \Biggr) \right]_{t=T}\\
+ & \Biggl[ \gamma_2 h_{1}(t)  \left( {{_TI_t}}^{1-\b} \partial_4
L[x]_\gamma^{\alpha, \beta}(t) - {_aI_t}^{1-\b} \partial_4
L[x]_\gamma^{\alpha, \beta}(t) \Biggr) \right]_{t=b}.
\end{split}
\end{equation*}
By proceeding similarly to the $4$th term in \eqref{eqh2},
we obtain an equivalent expression. Substituting these relations into
\eqref{eqh2} and considering the fractional operator
$D_{\overline{\gamma},c}^{\b,\a}$ as defined in \eqref{aux:FD}, we obtain that
\begin{equation}
\label{eqh3}
\begin{split}
0= \int_a^T &\Biggl[ h_{1}(t) \left[ \partial_2
L[x]_\gamma^{\alpha, \beta}(t) + D_{\overline{\gamma},T}^{\b,\a}\partial_4
L[x]_\gamma^{\alpha, \beta}(t) \right]\biggr. \\
& \qquad + \biggl. h_{2}(t) \left[ \partial_3 L[x]_\gamma^{\alpha, \beta}(t)
+ D_{\overline{\gamma},T}^{\b,\a}\partial_5
L[x]_\gamma^{\alpha, \beta}(t) \right] \biggr]dt\\
+ &  \gamma_2\int_T^b  \Biggl[h_{1}(t) \left[ _aD_t^{\b} \partial_4
L[x]_\gamma^{\alpha, \beta}(t) - {_TD_t}^{\b} \partial_4
L[x]_\gamma^{\alpha, \beta}(t) \right] \biggr.\\
& \qquad +\Biggl.h_{2}(t) \left[ _aD_t^{\b} \partial_5
L[x]_\gamma^{\alpha, \beta}(t) - {_TD_t}^{\b} \partial_5
L[x]_\gamma^{\alpha, \beta}(t) \right]dt \biggr]\\
+ & h_{1}(T)\Biggl[  \gamma_1 {{_tI_T}}^{1-\a} \partial_4
L[x]_\gamma^{\alpha, \beta}(t)\\
&\quad \quad \quad \quad - \gamma_2 {_TI_t}^{1-\b} \partial_4
L[x]_\gamma^{\alpha, \beta}(t) +\partial_2 \phi(t,x(t))\Biggr] _{t=T}
\end{split}
\end{equation}
\begin{equation*}
\begin{split}
+ & h_{2}(T)\Biggl[   \gamma_1 {{_tI_T}}^{1-\a} \partial_5
L[x]_\gamma^{\alpha, \beta}(t) - \gamma_2 {_TI_t}^{1-\b} \partial_5
L[x]_\gamma^{\alpha, \beta}(t) + \partial_3 \phi(t,x(t))\Biggr]_{t=T}\\
+ &\triangle T \bigg[ L[x]_\gamma^{\alpha, \beta}(t) + \partial_1 \phi(t,x(t))
+ \partial_2 \phi(t,x(t))x'_1(t)
+ \partial_3 \phi(t,x(t))x'_2(t) \Biggr]_{t=T}\\
+ & h_{1}(b)\Biggl[ \gamma_2   \left( {{_TI_t}}^{1-\b} \partial_4
L[x]_\gamma^{\alpha, \beta}(t) - {_aI_t}^{1-\b} \partial_4
L[x]_\gamma^{\alpha, \beta}(t) \Biggr) \right]_{t=b}\\
+ & h_{2}(b)\Biggl[ \gamma_2   \left( {{_TI_t}}^{1-\b} \partial_5
L[x]_\gamma^{\alpha, \beta}(t) - {_aI_t}^{1-\b} \partial_5
L[x]_\gamma^{\alpha, \beta}(t) \Biggr) \right]_{t=b}.
\end{split}
\end{equation*}
Define the piecewise continuous function $\lambda$ by
\begin{equation}
\label{lambdah1}
\lambda (t)
=
\begin{cases}
-\dfrac{\partial_3L[x]_\gamma^{\alpha, \beta}(t)
+ D_{\overline{\gamma},T}^{\b,\a}\partial_5L[x]_\gamma^{\alpha,
\beta}(t)}{\partial_3 g(t,x(t))}, & t\in [a,T] \\
-\dfrac{_aD_t^{\b}\partial_5L[x]_\gamma^{\alpha, \beta}(t)
-_TD_t^{\b}\partial_5L[x]_\gamma^{\alpha, \beta}(t)}{\partial_3 g(t,x(t))},
& t\in [T,b].
\end{cases}
\end{equation}
Using equations \eqref{eqh1} and \eqref{lambdah1}, we obtain that
\begin{multline*}
\lambda(t)\partial_2 g(t,x(t)) h_{1}(t)
\\
=\begin{cases}
(\partial_3L[x]_\gamma^{\alpha, \beta}(t) +D_{\overline{\gamma},T}^{\b,\a}
\partial_5L[x]_\gamma^{\alpha, \beta}(t)) h_2(t), & t\in [a,T] \\
(_aD_t^{\b}\partial_5L[x]_\gamma^{\alpha, \beta}(t)
-_TD_t^{\b}\partial_5L[x]_\gamma^{\alpha, \beta}(t)) h_2(t),  & t\in [T,b].
\end{cases}
\end{multline*}
Substituting in \eqref{eqh3}, we have
\begin{equation*}
\begin{split}
&0= \int_a^T h_{1}(t) \left[ \partial_2 L[x]_\gamma^{\alpha, \beta}(t)
+ D_{\overline{\gamma},T}^{\b,\a}\partial_4L[x]_\gamma^{\alpha, \beta}(t)
+ \lambda(t)\partial_2 g(t,x(t))\right] dt\\
+ &  \gamma_2\int_T^b  h_{1}(t) \left[ _aD_t^{\b} \partial_4
L[x]_\gamma^{\alpha, \beta}(t) - {_TD_t}^{\b} \partial_4
L[x]_\gamma^{\alpha, \beta}(t) + \lambda(t)\partial_2 g(t,x(t))\right] dt \\
+ & h_{1}(T)\Biggl[  \gamma_1 {{_tI_T}}^{1-\a} \partial_4
L[x]_\gamma^{\alpha, \beta}(t) - \gamma_2 {_TI_t}^{1-\b} \partial_4
L[x]_\gamma^{\alpha, \beta}(t) +\partial_2 \phi(t,x(t))\Biggr] _{t=T}\\
+ & h_{2}(T)\Biggl[   \gamma_1 {{_tI_T}}^{1-\a} \partial_5
L[x]_\gamma^{\alpha, \beta}(t) - \gamma_2 {_TI_t}^{1-\b} \partial_5
L[x]_\gamma^{\alpha, \beta}(t) + \partial_3 \phi(t,x(t))\Biggr]_{t=T}\\
+ &\triangle T \bigg[ L[x]_\gamma^{\alpha, \beta}(t)
+ \partial_1 \phi(t,x(t))+ \partial_2 \phi(t,x(t))x'_1(t)
+ \partial_3 \phi(t,x(t))x'_2(t) \Biggr]_{t=T}\\
+ & h_{1}(b)\Biggl[ \gamma_2   \left( {{_TI_t}}^{1-\b} \partial_4
L[x]_\gamma^{\alpha, \beta}(t) - {_aI_t}^{1-\b} \partial_4
L[x]_\gamma^{\alpha, \beta}(t) \Biggr) \right]_{t=b}\\
+ & h_{2}(b)\Biggl[ \gamma_2   \left( {{_TI_t}}^{1-\b} \partial_5
L[x]_\gamma^{\alpha, \beta}(t) - {_aI_t}^{1-\b} \partial_5
L[x]_\gamma^{\alpha, \beta}(t) \Biggr) \right]_{t=b}.
\end{split}
\end{equation*}
Considering appropriate choices of variations, we obtained the first \eqref{ELeqh_1}
and the third \eqref{ELeqh_2} necessary conditions, and also the transversality
conditions \eqref{CTh1}. The remaining conditions \eqref{ELeqhPP_1}
and \eqref{ELeqhPP_2} follow directly from \eqref{lambdah1}.
\end{proof}


\subsection{Example}
\label{subsec:exeHolo}

We end this section with a simple illustrative example.
Consider the following problem:
\begin{equation*}
\begin{gathered}
J(x,t)= \int_0^T \Biggl[\alpha(t)+\left(\DC x_1(t)
-\dfrac{t^{1-\alpha(t)}}{2 \Gamma(2-\alpha(t))}
-\dfrac{(b-t)^{1-\beta (t)}}{2 \Gamma(2-\beta(t))}\right)^2 \\
+ \left(\DC x_2(t)\right)^2\Biggr]dt \longrightarrow\min,\\
x_1(t) + x_2(t) = t+1,\\
x_1(0)=0, \quad x_2(0) = 1.
\end{gathered}
\end{equation*}
It is a simple exercise to check that $x_1(t) = t$,
$x_2(t) \equiv 1$ and $\lambda(t) \equiv 0$
satisfy our Theorem~\ref{teo7}.


\section{Fractional variational Herglotz problem}
\label{sec:Herg}

In this section, we study fractional variational problems of Herglotz type,
depending on the combined Caputo fractional derivatives $\DC$. Two different
cases are considered.\\
The variational problem of Herglotz is a generalization of the classical
variational problem. It  allows us to describe nonconservative processes,
even in case the Lagrange function is autonomous (that is, when the Lagrangian
does not depend explicitly on time). In opposite to calculus of variations,
where the cost functional is given by an integral depending only on time, space
and on the dynamics, in the Herglotz variational problem the model is given by a
differential equation involving the derivative of the objective function $z$,
and the Lagrange function depends on time, trajectories $x$ and $z$, and on the
derivative of $x$. The problem of Herglotz was posed by
\cite{Cap4:Herglotz}, but only in 1996, with the works 
\cite{Cap4:Guenther,Cap4:Guenther:book},
it has gained the attention of the mathematical community. Since then, several
papers were devoted to this subject. For example, see references
\cite{Cap4:Alm:Malin2014,Cap4:Georgieva,Cap4:Georgieva:sev,Cap4:Santos:Viet,Cap4:Santos:Disc,Cap4:Santos:Spri}.

In Section~\ref{subsec:neccondHer}, we obtain fractional Euler--Lagrange
conditions for the fractional variational problem of Herglotz, with one
variable, and the general case, for several independent variables is
discussed in Section~\ref{subsec:several}. Finally, three illustrative
examples are presented in detail (Section~\ref{subsec:exeHerg}).


\subsection{Fundamental problem of Herglotz}
\label{subsec:neccondHer}
\index{variational fractional problem! of Herglotz}
Let $\alpha,\beta: [a,b]^2\rightarrow (0,1)$ be two functions. 
The fractional Herglotz variational problem that we study is as follows.

\begin{Problem}
\label{ProblemHerg}
Determine the trajectories $x \in C^{1}\left([a,b];\bR\right)$ satisfying a
given boundary condition $x(a)=x_{a}$,\index{boundary conditions} for a
fixed $x_{a} \in \bR$, and a real $T \in (a,b]$, that extremize the value
of $z(T)$, where $z$ satisfies the differential equation
\begin{equation}
\label{funct1H}
\dot{z}(x,t)=L\left(t,x(t), \DC x(t), z(t) \right),  \quad  t \in [a,b],
\end{equation}
with dependence on a combined Caputo fractional derivative operator, subject
to the initial condition
\begin{equation}
\label{ICond}
z(a)=z_{a},
\end{equation}
where $z_a$ is a fixed real number.
\end{Problem}

In the sequel, we use the auxiliary notation
$$[x,z]_\gamma^{\alpha, \beta}(t)=\left(t,x(t), \DC x(t), z(t) \right).$$
The Lagrangian $L$ is assumed to satisfy the following hypothesis:

\begin{enumerate}
\item $L:[a,b] \times \bR^{3}\rightarrow \bR$ is differentiable,
\item $t\rightarrow \lambda(t)\partial_3L[x,z]_\gamma^{\alpha, \beta}(t)$ is such that $_TD_t^{\b} \left( \lambda(t) \partial_3L[x,z]_\gamma^{\alpha, \beta}(t)\right)$, \\
$\LDb  \left( \lambda(t) \partial_{3} L[x,z]_\gamma^{\alpha, \beta}(t) \right)$, and
$D_{\overline{\gamma}}^{\b,\a} \left( \lambda(t) \partial_{3} L[x,z]_\gamma^{\alpha, \beta}(t) \right)$  
exist and are continuous on $[a,b]$, where
$$\lambda(t)=\exp \left(-\int_a^t \partial_4 L\left[x,z \right]_\gamma^{\alpha, \beta} (\tau)d\tau \right).$$
\end{enumerate}
The following result gives necessary conditions of Euler--Lagrange type for a solution of Problem~\ref{ProblemHerg}.\index{Euler--Lagrange equations}

\begin{Theorem}
\label{mainteo_Herg}
Let $x \in C^{1}\left([a,b];\bR\right)$ be such that $z$ defined 
by Eq. \eqref{funct1H}, subject to the initial condition 
\eqref{ICond}, has an extremum at $T \in ]a,b]$.
Then, $(x,z)$ satisfies the fractional differential equations
\begin{equation}
\label{CEL1_Herg}
\partial_2 L[x,z]_\gamma^{\alpha, \beta}(t)\lambda(t)
+D{_{\overline{\gamma}}^{\b,\a}}\left(\lambda(t)
\partial_3 L[x,z]_\gamma^{\alpha, \beta}(t)\right)=0,
\end{equation}
on $[a,T]$, and
\begin{equation}
\label{CEL2_Herg}
\gamma_2\left({\LDb} \left(\lambda(t)\partial_3 
L[x,z]_\gamma^{\alpha, \beta}(t)\right)-{ _TD{_t^{\b}}\left(\lambda (t)
\partial_3 L[x,z]_\gamma^{\alpha, \beta}(t)\right)}\right)=0,
\end{equation}
on $[T,b]$.
Moreover, the following transversality conditions are satisfied:
\index{transversality conditions}
\begin{equation}
\label{CTransH}
\left\{
\begin{array}{l}
\left[\gamma_1 {_tI_T^{1-\a}} \left(\lambda (t) \partial_3
L[x,z]_\gamma^{\alpha, \beta}(t)\right)-\gamma_2 {_TI_t^{1-\b}} 
\left(\lambda (t) \partial_3L[x,z]_\gamma^{\alpha, \beta}(t)\right)\right]_{t=T}=0,\\
\gamma_2 \left[ {_TI_t^{1-\b}} \left( \lambda (t) \partial_3
L[x,z]_\gamma^{\alpha, \beta}(t)\right) 
-{_aI_t^{1-\b} \left( \lambda (t) \partial_3
L[x,z]_\gamma^{\alpha, \beta}(t)\right)}\right]_{t=b}=0.
\end{array}\right.
\end{equation}
If $T<b$, then $L[x,z]_\gamma^{\alpha, \beta}(T)=0$.
\end{Theorem}

\begin{proof}
Let $x$ be a solution of the problem. Consider an admissible variation of $x$, $\overline {x}= x+\e{h}$, where $h\in C^1([a,b];\bR)$ is an arbitrary perturbation curve and $\e \in\bR$ represents a small number $\left(\vert\epsilon\vert\ll1 \right)$. The constraint  $x(a)=x_a$ implies that all admissible variations must fulfill the condition $h(a)=0$.
On the other hand, consider an admissible variation of $z$, $\overline {z}= z+\e\theta$, where $\theta$ is a perturbation curve (not arbitrary) such that
\begin{enumerate}
\item $\theta(a)=0$, so that $z(a)=z_{a}$,
\item $\theta (T)=0$ because $z(T)$ is a maximum (or a minimum),\\
i.e., $\overline{z}(T)-z(T)\leq 0 \quad \left( \overline{z}(T)-z(T)\geq0 \right)$,
\item $\theta (t)= \dfrac {d}{d \varepsilon} z(\overline {x},t) \biggr\rvert_{\varepsilon=0}$, so that the variation satisfies the differential equation \eqref{funct1H}.
\end{enumerate}
Differentiating $\theta$ with respect to $t$, we obtain that
\begin{equation*}
\begin{split}
\dfrac{d}{dt}\theta (t)& =\dfrac{d}{dt} \dfrac{d}{d\varepsilon} z(\overline {x},t) \biggr\rvert_{\varepsilon=0}\\
&= \dfrac{d}{d\varepsilon}\dfrac{d}{dt} z(\overline {x},t) \biggr\rvert_{\varepsilon=0}\\
&= \dfrac{d}{d\varepsilon} L\left(t,x(t)+\e{h(t)}, \DC x(t)+\e\DC {h(t)}, \overline z(\overline {x},t) \right) \biggr\rvert_{\varepsilon=0},
\end{split}
\end{equation*}
and rewriting this relation, we obtain the following differential equation for $\theta$:
\begin{equation*}
\dot{\theta}(t) - \partial_{4}L[x,z]_\gamma^{\alpha, \beta}(t) \theta(t) =\partial_{2} L[x,z]_\gamma^{\alpha, \beta}(t) h(t) + \partial_{3}L[x,z]_\gamma^{\alpha, \beta}(t) \DC h(t).
\end{equation*}
Considering $\lambda(t)=\exp \left(- \displaystyle \int_a^{t} \partial_{4}L[x,z]_\gamma^{\alpha, \beta}  (\tau)d\tau \right)$, we obtain the solution for the last differential equation
\begin{multline*}
\theta(T)\lambda(T)-\theta(a)\\ 
= \int_a^{T} \left(\partial_{2} L[x,z]_\gamma^{\alpha, \beta}(t) h(t) 
+ \partial_{3}L[x,z]_\gamma^{\alpha, \beta}(t) \DC h(t) \right) \lambda(t) dt.
\end{multline*}
By hypothesis, $\theta(a)=0$. If $x$ is such that $z(x,t)$ defined 
by \eqref{funct1H} attains an extremum at $t=T$, 
then $\theta(T)$ is identically zero. Hence, we get
\begin{equation}
\label{solutionPH}
\int_a^{T} \left(\partial_{2} L[x,z]_\gamma^{\alpha, \beta}(t)h(t) 
+ \partial_{3}L[x,z]_\gamma^{\alpha, \beta}(t) \DC h(t) \right) \lambda(t) dt = 0.
\end{equation}
Considering only the second term in Eq. \eqref{solutionPH} and the definition 
of combined Caputo derivative operator, we obtain that
\begin{equation*}
\begin{split}
&\int_a^{T} \lambda(t) \partial_{3} L[x,z]_\gamma^{\alpha, \beta}(t) 
\left( \gamma_{1} \LC h(t) + \gamma_{2} \RCb h(t) \right)  dt\\
&=\gamma_{1} \int_a^{T} \lambda(t) \partial_{3} L[x,z]_\gamma^{\alpha, \beta}(t) \LC h(t)dt \\
&\quad+\gamma_{2}  \Bigg[ \int_a^{b} \lambda(t) \partial_{3} L[x,z]_\gamma^{\alpha, \beta}(t) \RCb h(t)dt\\
&\quad\quad \quad \quad \quad  
- \int_T^{b} \lambda(t) \partial_{3} L[x,z]_\gamma^{\alpha, \beta}(t) \RCb h(t)dt \Bigg]=\star.
\end{split}
\end{equation*}
Using the fractional integration by parts formula and considering 
$\overline {\gamma} =(\gamma_{_{2}}, \gamma_{1})$, we get
\begin{equation*}
\begin{split}
&\star=\int_a^{T} h(t) D_{\overline{\gamma}}^{\b,\a} \left( \lambda(t) 
\partial_{3} L[x,z]_\gamma^{\alpha, \beta}(t) \right) dt \\
& + \int_T^{b} \gamma_{2} h(t) \left[ \LDb  \left( \lambda(t) \partial_{3} 
L[x,z]_\gamma^{\alpha, \beta}(t) \right)  - _TD_t^{\b} \left( \lambda(t) \partial_{3} 
L[x,z]_\gamma^{\alpha, \beta}(t) \right) \right] dt\\
& + h(T) \left[ \gamma_{1} {}_tI_{T}^{1-\a} \left( \lambda(t) \partial_{3} 
L[x,z]_\gamma^{\alpha, \beta}(t) \right) - \gamma_{2} {}_TI_{t}^{1-\b} 
\left( \lambda(t) \partial_{3} L[x,z]_\gamma^{\alpha, \beta}(t) \right) \right]_{t=T}\\
& + h(b) \gamma_{2} \left[ _TI_{t}^{1-\b} \left( \lambda(t) \partial_{3} 
L[x,z]_\gamma^{\alpha, \beta}(t) \right) - _aI_{t}^{1-\b} \left( \lambda(t) \partial_{3} 
L[x,z]_\gamma^{\alpha, \beta}(t) \right) \right]_{t=b}.
\end{split}
\end{equation*}
Substituting this relation into expression \eqref{solutionPH}, we obtain
\begin{equation*}
\begin{split}
&0= \int_a^{T} h(t) \left[ \partial_{2} L[x,z]_\gamma^{\alpha, \beta}(t) \lambda(t) 
+  D_{\overline{\gamma}}^{\b,\a} \left( \lambda(t) \partial_{3} 
L[x,z]_\gamma^{\alpha, \beta}(t) \right)\right] dt \\
&+ \int_T^{b} \gamma_{2} h(t) \left[ \LDb  \left( \lambda(t) \partial_{3} 
L[x,z]_\gamma^{\alpha, \beta}(t) \right)  - _TD_t^{\b} \left( \lambda(t) 
\partial_{3} L[x,z]_\gamma^{\alpha, \beta}(t) \right) \right] dt\\
&+ h(T) \left[ \gamma_{1} {}_tI_{T}^{1-\a} \left( \lambda(t) \partial_{3} 
L[x,z]_\gamma^{\alpha, \beta}(t) \right) - \gamma_{2} {}_TI_{t}^{1-\b} 
\left( \lambda(t) \partial_{3} L[x,z]_\gamma^{\alpha, \beta}(t) \right) \right]_{t=T}\\
& + h(b) \gamma_{2} \left[ _TI_{t}^{1-\b} \left( \lambda(t) \partial_{3} 
L[x,z]_\gamma^{\alpha, \beta}(t) \right) - _aI_{t}^{1-\b} \left( \lambda(t) 
\partial_{3} L[x,z]_\gamma^{\alpha, \beta}(t) \right) \right]_{t=b}.
\end{split}
\end{equation*}
With appropriate choices for the variations $h(\cdot)$, we get 
the Euler--Lagrange equations \eqref{CEL1_Herg}--\eqref{CEL2_Herg} 
and the transversality conditions \eqref{CTransH}.
\end{proof}

\begin{Remark}
If $\a$ and $\b$ tend to 1, and if the Lagrangian $L$ is of class $C^2$, 
then the first Euler--Lagrange equation \eqref{CEL1_Herg} becomes
$$
\partial_2 L[x,z]_\gamma^{\alpha, \beta}(t)\lambda(t)
+(\gamma_2- \gamma_1)\frac{d}{dt}\left[\lambda(t)
\partial_{3} L[x,z]_\gamma^{\alpha, \beta}(t)\right]=0.
$$
Differentiating and considering the derivative of the lambda function, we obtain
\begin{multline*}
\lambda(t) \Bigg[\partial_2 L[x,z]_\gamma^{\alpha, \beta}(t)\\
+(\gamma_2- \gamma_1)\left[-\partial_{4} L[x,z]_\gamma^{\alpha, \beta}(t)
\partial_{3} L[x,z]_\gamma^{\alpha, \beta}(t)+\frac{d}{dt}
\partial_{3} L[x,z]_\gamma^{\alpha, \beta}(t)\right]\Bigg]=0.
\end{multline*}
As $\lambda(t)>0$, for all t, we deduce that
$$
\partial_2 L[x,z]_\gamma^{\alpha, \beta}(t)+(\gamma_2- \gamma_1)\left[
\frac{d}{dt}\partial_{3} L[x,z]_\gamma^{\alpha, \beta}(t)
-\partial_{4} L[x,z]_\gamma^{\alpha, \beta}(t)\partial_{3} 
L[x,z]_\gamma^{\alpha, \beta}(t)\right]=0.
$$
\end{Remark}


\subsection{Several independent variables}
\label{subsec:several}

Consider the following generalization of the problem of Herglotz involving $n+1$ independent variables.
Define $\Omega=\prod_{i=1}^{n} [a_i,b_i]$, with $n \in \mathbb{N}$, 
$P=[a,b]\times \Omega$ and consider the vector $s=(s_1, s_2, \ldots, s_n)\in \Omega$. 
The new problem consists in determining the trajectories 
$x \in C^{1}\left(P\right)$ that give an extremum to $z[x,T]$, 
where the functional $z$ satisfies the differential equation
\begin{multline}
\label{funct_siv1}
\dot{z}(t)=\int_{\Omega}L\left(t,s, x(t,s), \DC x(t,s),\right.\\
\left.^CD_{\gamma^1}^{\alpha_1(\cdot,\cdot),\beta_1(\cdot,\cdot)}x(t,s),\ldots, ^CD_{\gamma^n}^{\alpha_n(\cdot,\cdot),\beta_n(\cdot,\cdot)}x(t,s), z(t) \right)d^{n}s
\end{multline}
subject to the constraint
\begin{equation}
\label{herg:bound}
x(t,s)=g(t,s),\quad \mbox{ for all } \quad (t,s) \in \partial P,
\end{equation}
where $\partial P$ is the boundary of $P$ and $g$ is a given function $g:\partial P \rightarrow \bR$. We assume that
\begin{enumerate}

\item $\alpha, \alpha_i, \beta, \beta_i: [a,b]^2 \rightarrow (0,1)$ with $i=1,\ldots, n$,

\item $\gamma,\gamma^1,\ldots, \gamma^n \in [0,1]^2$,

\item $d^{n}s=ds_1 \ldots ds_n$,

\item $\DC x(t,s)$, $^CD_{\gamma^1}^{\alpha_1(\cdot,\cdot),\beta_1(\cdot,\cdot)}x(t,s),
\ldots, ^CD_{\gamma^n}^{\alpha_n(\cdot,\cdot),\beta_n(\cdot,\cdot)}x(t,s)$ exist and are continuous functions,

\item the Lagrangian $L:P\times\bR^{n+3} \rightarrow \bR$ is of class $C^1$.
\end{enumerate}

\begin{Remark}
By $\DC x(t,s)$ we mean the Caputo fractional derivative with respect to the independent variable $t$, and by $^CD_{\gamma^i}^{\alpha_i(\cdot,\cdot),\beta_i(\cdot,\cdot)}x(t,s)$ we mean the Caputo fractional derivative with respect to the independent variable $s_i$, for $i=1,\ldots,n$.
\end{Remark}

In the sequel, we use the auxiliary notation $[x,z]_{n, \gamma}^{\alpha, \beta}(t,s)$ to represent the following vector
\begin{multline*}\Big(t,s,x(t,s), \DC x(t,s), ^CD_{\gamma^1}^{\alpha_1(\cdot,\cdot),
\beta_1(\cdot,\cdot)}x(t,s),\\ \ldots, ^CD_{\gamma^n}^{\alpha_n(\cdot,\cdot),
\beta_n(\cdot,\cdot)}x(t,s), z(t) \Big).\end{multline*}
Consider the function
$$\lambda(t)=\exp \left(-\int_a^t \int_{\Omega} \partial_{2n+4}\left[x,z\right]_{n,\gamma}^{\alpha, \beta} (\tau,s) d^ns d\tau \right).$$

\begin{Theorem}
If $(x,z,T)$ is an extremizer of the functional defined by Eq. \eqref{funct_siv1},
then $(x,z,T)$ satisfies the fractional differential equations
\begin{multline}
\label{CEL1_Herg2}
\partial_{n+2} L[x,z]_{n, \gamma}^{\alpha, \beta}(t,s)\lambda(t)
+D{_{\overline{\gamma}}^{\b,\a}}\left(\lambda(t)\partial_{n+3}
L[x,z]_{n, \gamma}^{\alpha, \beta}(t,s)\right)\\
+ \sum_{i=1}^{n} D_{\overline{\gamma}^{i}}^{\beta_i(\cdot,\cdot),
\alpha_i(\cdot,\cdot)}\left(\lambda(t) \partial_{n+3+i}
L[x,z]_{n, \gamma}^{\alpha, \beta}(t,s)\right)=0
\end{multline}
on $[a,T] \times \Omega$, and
\begin{equation}
\label{CEL2_Herg2}
\gamma_2\left({\LDb} \left(\lambda(t)\partial_{n+3}
L[x,z]_{n, \gamma}^{\alpha, \beta}(t,s)\right)
-{ _TD{_t^{\b}}\left(\lambda (t) \partial_{n+3}
L[x,z]_{n, \gamma}^{\alpha, \beta}(t,s)\right)}\right)=0
\end{equation}
on $[T,b]\times \Omega$.\\
Moreover, $(x,z)$ satisfies the transversality condition
\index{transversality conditions}
\begin{multline}
\label{CTransH2}
\Bigl[\gamma_1 {_tI_T^{1-\a}} \left(\lambda (t) \partial_{n+3}
L[x,z]_{n, \gamma}^{\alpha, \beta}(t,s)\right) \\
-\gamma_2 {_TI_t^{1-\b}} \left(\lambda(t)\partial_{n+3}
L[x,z]_{n, \gamma}^{\alpha, \beta}(t,s)\right)\Bigr]_{t=T}=0,
\quad s \in\Omega.
\end{multline}
If $T<b$, then $\displaystyle \int_{\Omega}
L[x,z]_{n, \gamma}^{\alpha, \beta}(T,s)d^{n}s=0$.
\end{Theorem}

\begin{proof}
Let $x$ be a solution of the problem. Consider an admissible variation of $x$, $\overline {x}(t,s)= x(t,s)+\e{h(t,s)}$, where $h\in C^1(P)$ is an arbitrary perturbing curve and $\e \in\bR$ is such that $|\e|\ll1$. Consequently, $h(t,s)=0$ for all $(t,s)\in \partial P$ by the boundary condition \eqref{herg:bound}. \\
On the other hand, consider an admissible variation of $z$, $\overline {z}= z+\e\theta$, where $\theta$ is a perturbing curve such that
$\theta (a)=0$ and
$$\theta (t)= \dfrac {d}{d \varepsilon} z(\overline {x},t) \biggr\rvert_{\varepsilon=0}.$$
Differentiating $\theta(t)$ with respect to $t$, we obtain that
\begin{equation*}
\begin{split}
\dfrac{d}{dt}\theta (t)& =\dfrac{d}{dt} \dfrac{d}{d\varepsilon} z(\overline {x},t) \biggr\rvert_{\varepsilon=0}\\
&= \dfrac{d}{d\varepsilon}\dfrac{d}{dt} z\left(\overline {x},t\right) \biggr\rvert_{\varepsilon=0}\\
&= \dfrac{d}{d\varepsilon} \int_{\Omega}
L[\overline{x},z]_{n, \gamma}^{\alpha, \beta}(t,s) d^{n}s
\biggr\rvert_{\varepsilon=0}.
\end{split}
\end{equation*}
We conclude that
\begin{multline*}
\dot{\theta}(t)
=\int_{\Omega} \Biggl( \partial_{n+2}
L[x,z]_{n, \gamma}^{\alpha, \beta}(t,s) h(t,s)
+\partial_{n+3}
L[x,z]_{n, \gamma}^{\alpha, \beta}(t,s) \DC h(t,s)\\
+ \sum_{i=1}^{n}\partial_{n+3+i}
L[x,z]_{n, \gamma}^{\alpha, \beta}(t,s) ^CD_{\gamma^{i}}^{\alpha_i(\cdot,\cdot),
\beta_i(\cdot,\cdot)}h(t,s)+\partial_{2n+4}
L[x,z]_{n, \gamma}^{\alpha, \beta}(t,s) \theta(t)\Biggr) d^{n}s.
\end{multline*}
To simplify the notation, we define
$$
B(t)=\int_{\Omega}\partial_{2n+4}
L[x,z]_{n, \gamma}^{\alpha, \beta}(t,s)d^{n}s
$$
and
\begin{multline*}
A(t)
=\int_{\Omega}\Bigl( \partial_{n+2}
L[x,z]_{n, \gamma}^{\alpha, \beta}(t,s) h(t,s)
+\partial_{n+3} L[x,z]_{n, \gamma}^{\alpha, \beta}(t,s)
\DC h(t,s)\\
+\sum_{i=1}^{n}\partial_{n+3+i}
L[x,z]_{n, \gamma}^{\alpha, \beta}(t,s)
^CD_{\gamma^{i}}^{\alpha_i(\cdot,\cdot),
\beta_i(\cdot,\cdot)}h(t,s) \Bigr) d^{n}s.
\end{multline*}
Then, we obtain the linear differential equation
$$
\dot{\theta}(t)-B(t)\theta(t)=A(t),
$$
whose solution is
\begin{equation*}
\theta(T)\lambda(T) - \theta(a) = \int_a^{T} A(t) \lambda(t) dt.
\end{equation*}
Since $\theta(a)=\theta(T)=0$, we get
\begin{equation}
\label{solutionPH2}
\int_a^{T} A(t)\lambda(t)dt = 0.
\end{equation}
Considering only the second term in \eqref{solutionPH2}, we can write
\begin{equation*}
\begin{split}
\int_a^{T}
&\int_{\Omega}\lambda(t) \partial_{n+3}
L[x,z]_{n, \gamma}^{\alpha, \beta}(t,s)
\left( \gamma_{1} \LC h(t,s) + \gamma_{2} \RCb h(t,s) \right) d^{n}s dt\\
&=\gamma_{1} \int_a^{T} \int_{\Omega}\lambda(t) \partial_{n+3}
L[x,z]_{n, \gamma}^{\alpha, \beta}(t,s) \LC h(t,s)d^{n}s dt\\
&\quad +\gamma_{2}  \left[ \int_a^{b} \int_{\Omega}\lambda(t) \partial_{n+3}
L[x,z]_{n, \gamma}^{\alpha, \beta}(t,s) \RCb h(t,s) d^{n}s dt\right.\\
& \qquad\qquad - \left.\int_T^{b} \int_{\Omega}\lambda(t)
\partial_{n+3} L[x,z]_{n, \gamma}^{\alpha, \beta}(t,s) \RCb h(t,s) d^{n}sdt \right].
\end{split}
\end{equation*}
Let $\overline {\gamma} =(\gamma_{_{2}}, \gamma_{1})$. 
Integrating by parts (cf. Theorem~\ref{thm:FIP}), 
and since $h(a,s)=h(b,s)=0$ for all $s \in \Omega$, we obtain the following expression:
\begin{equation*}
\begin{split}
&\int_a^{T} \int_{\Omega} h(t,s) D_{ \overline{\gamma}}^{\b,\a}
\left( \lambda(t) \partial_{n+3}
L[x,z]_{n, \gamma}^{\alpha, \beta}(t,s) \right) d^{n}sdt \\
&+ \gamma_{2}\int_T^{b} \int_{\Omega} h(t,s) \left[
\LDb  \left( \lambda(t) \partial_{n+3}
L[x,z]_{n, \gamma}^{\alpha, \beta}(t,s) \right)\right.\\
&\qquad\qquad\qquad\qquad\qquad
-\left. _TD_t^{\b} \left( \lambda(t) \partial_{n+3}
L[x,z]_{n, \gamma}^{\alpha, \beta}(t,s) \right) \right] d^{n}sdt\\
& + \int_{\Omega} h(T,s) \left[ \gamma_{1} {}_tI_{T}^{1-\a} \left( \lambda(t)
\partial_{n+3} L[x,z]_{n, \gamma}^{\alpha, \beta}(t,s) \right)\right.\\
&\qquad\qquad\qquad\qquad  \left. - \gamma_{2} \, _TI_{t}^{1-\b}
\left( \lambda(t) \partial_{n+3}
L[x,z]_{n, \gamma}^{\alpha, \beta}(t,s) \right) d^n s\right]_{t=T}.
\end{split}
\end{equation*}
Doing similarly for the $(i+2)$th term in \eqref{solutionPH2}, with $i=1,\ldots, n$, 
letting $\overline {\gamma}^i =(\gamma_{_{2}}^i, \gamma_{1}^{i})$, 
and since $h(t,a_i)=h(t,b_i)=0$ for all $t \in [a,b]$, we obtain
\begin{multline*}
\int_a^{T} \int_{\Omega}\lambda(t) \partial_{n+3+i}
L[x,z]_{n, \gamma}^{\alpha, \beta}(t,s) \left( \gamma_{1}^{i}
{_{a_i}^CD_{s_i}^{\alpha_{i}(\cdot,\cdot)} }h(t,s)
+ \gamma_{2}^{i} {_{s_i}^CD_{b_i}^{\beta_{i}(\cdot,\cdot)}} h(t,s)\right) d^{n}s dt\\
=\int_a^{T} \int_{\Omega} h(t,s) D_{\overline {\gamma}^i}^{~\beta_{i}(\cdot,\cdot),
\alpha_{i}(\cdot,\cdot)} \left(\lambda(t)\partial_{n+3+i}
L[x,z]_{n, \gamma}^{\alpha, \beta}(t,s) \right) d^{n}s dt.
\end{multline*}
Substituting these relations into \eqref{solutionPH2}, we deduce that
\begin{equation*}
\begin{split}
\int_a^{T} &\int_{\Omega} h(t,s)\left[ \partial_{n+2}
L[x,z]_{n, \gamma}^{\alpha, \beta}(t,s)\lambda(t)
+ D_{ \overline{\gamma}}^{\b,\a}
\left( \lambda(t) \partial_{n+3}
L[x,z]_{n, \gamma}^{\alpha, \beta}(t,s) \right) \right.\\
&+\sum_{i=1}^{n} D_{\overline {\gamma}^i}^{~\beta_{i}(\cdot,\cdot),
\alpha_{i}(\cdot,\cdot)} \left(\lambda(t)
\partial_{n+3+i} L[x,z]_{n, \gamma}^{\alpha, \beta}(t,s) \right) d^{n}sdt \\
&+ \gamma_{2}\int_T^{b} \int_{\Omega} h(t,s) \left[
\LDb\left( \lambda(t) \partial_{n+3}
L[x,z]_{n, \gamma}^{\alpha, \beta}(t,s) \right)\right.
\end{split}
\end{equation*}
\begin{equation*}
\begin{split}
& \qquad\qquad\qquad\qquad\quad \left.
- _TD_t^{\b} \left( \lambda(t) \partial_{n+3}
L[x,z]_{n, \gamma}^{\alpha, \beta}(t,s) \right) \right] d^{n}sdt\\
&+ \int_{\Omega} h(T,s) \left[ \gamma_{1} {}_tI_{T}^{1-\a}
\left( \lambda(t) \partial_{n+3}
L[x,z]_{n, \gamma}^{\alpha, \beta}(t,s) \right)\right.\\
&\qquad\qquad\qquad\qquad\left.
-\gamma_{2} \, _TI_{t}^{1-\b} \left( \lambda(t)
\partial_{n+3} L[x,z]_{n, \gamma}^{\alpha, \beta}(t,s)
\right)d^ns \right]_{t=T}.
\end{split}
\end{equation*}
For appropriate choices with respect to $h$, we get the 
Euler--Lagrange equations \eqref{CEL1_Herg2}--\eqref{CEL2_Herg2} 
and the transversality condition \eqref{CTransH2}.
\end{proof}


\subsection{Examples}
\label{subsec:exeHerg}
We present three examples, with and without the dependence on $z$.

\begin{example}
\label{exH:1}
Consider
\begin{equation}
\label{exemp_Her1}
\begin{gathered}
\dot{z}(t)=\left(\DC x(t)\right)^2+z(t)+t^2-1, \quad t\in [0,3],\\
x(0)=1, \quad z(0)=0.
\end{gathered}
\end{equation}
In this case, $\lambda(t)=\exp(-t)$.
The necessary optimality conditions \eqref{CEL1_Herg}--\eqref{CEL2_Herg}
of Theorem~\ref{mainteo_Herg} hold for $\overline{x}(t) \equiv 1$.
If we replace $x$ by $\overline x$ in \eqref{exemp_Her1}, we obtain
\begin{gather*}
\dot{z}(t)-z(t)=t^2-1, \quad t\in [0,3],\\
z(0)=0,
\end{gather*}
whose solution is
\begin{equation}
\label{z:sol}
z(t)=\exp(t)-(t+1)^2.
\end{equation}
The last transversality condition of Theorem \ref{mainteo_Herg} asserts that
$$
L[\overline x,z]_\gamma^{\alpha, \beta}(T)=0 \Leftrightarrow \exp(T)-2T-2=0,
$$
whose solution is approximately
$$
T\approx 1.67835.
$$
We remark that $z$ \eqref{z:sol} actually attains a minimum value at this point
(see Figure~\ref{heg1},~(a)):
$$
z(1.67835)\approx -1.81685.
$$
\end{example}

\begin{example}
\label{exH:2}
Consider now
\begin{equation}
\label{exemp_Her2}
\begin{gathered}
\dot{z}(t)=(t-1)\left(x^2(t)+z^2(t)+1\right), \quad t\in [0,3],\\
x(0)=0, \quad z(0)=0.
\end{gathered}
\end{equation}
Since the first Euler--Lagrange equation \eqref{CEL1_Herg} reads
$$
(t-1)x(t)=0, \quad \forall t\in[0,T],
$$
we see that $\overline x(t)\equiv 0$ is a solution of this equation. 
The second transversality condition of \eqref{CTransH} asserts that, at $t=T$, we must have
$$
L[\overline x,z]_\gamma^{\alpha, \beta}(t)=0,
$$
that is,
$$
(t-1)(z^2(t)+1)=0,
$$
and so $T=1$ is a solution for this equation.
Substituting $x$ by $\overline x$ in  \eqref{exemp_Her2}, we get
\begin{gather*}
\dot{z}(t)=(t-1)(z^2(t)+1), \quad t\in [0,3],\\
z(0)=0.
\end{gather*}
The solution to this Cauchy problem is the function
$$
z(t)=\tan\left(\frac{t^2}{2}-t\right),
$$
(see Figure~\ref{heg1},~(b)) and the minimum value is
$$
z(1)=\tan\left(-\frac{1}{2}\right).
$$
\end{example}

\begin{example}
\label{exH:3}
For our last example, consider
\begin{equation}
\label{exemp_Her3}
\begin{gathered}
\dot{z}(t)=\left(\DC x(t) - f(t)\right)^2+t^2-1, \quad t\in [0,3],\\
x(0)=0, \quad z(0)=0,
\end{gathered}
\end{equation}
where
$$
f(t) := \frac{t^{1-\alpha(t)}}{2\Gamma(2-\alpha(t))}
-\frac{(3-t)^{1-\beta(t)}}{2\Gamma(2-\beta(t))}.
$$
In this case, $\lambda(t)\equiv 1$. We intend to find a pair $(x,z)$,
satisfying all the conditions in \eqref{exemp_Her3}, for which $z(T)$
attains a minimum value.
It is easy to verify that $\overline x (t)=t$ and $T=1$ satisfy
the necessary conditions given by Theorem \ref{mainteo_Herg}.
Replacing $x$ by $\overline x$ in \eqref{exemp_Her3},
we get a Cauchy problem of form
\begin{gather*}
\dot{z}(t)=t^2-1, \quad t\in [0,3],\\
z(0)=0,
\end{gather*}
whose solution is
$$
z(t)=\frac{t^3}{3}-t.
$$
Observe that this function attains a minimum value at $T=1$, which is $z(1)=-2/3$ (Figure~\ref{heg1},~(c)).
\end{example}

\begin{figure}[ht!]
\begin{center}
\subfigure[Extremal $z$ of Example~\ref{exH:1}.]{\includegraphics[scale=0.25]{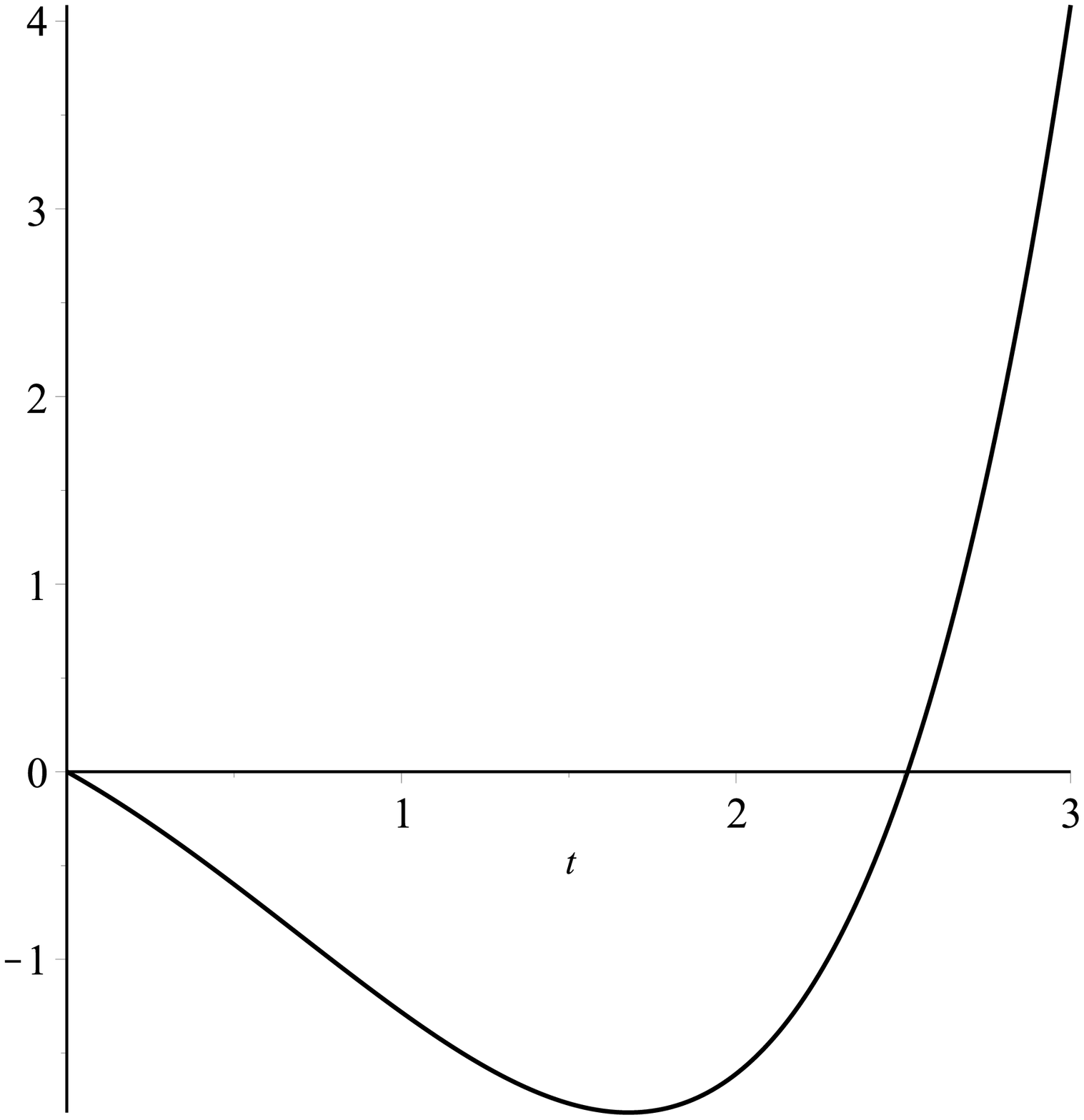}}
\subfigure[Extremal $z$ of Example~\ref{exH:2}.]{\includegraphics[scale=0.25]{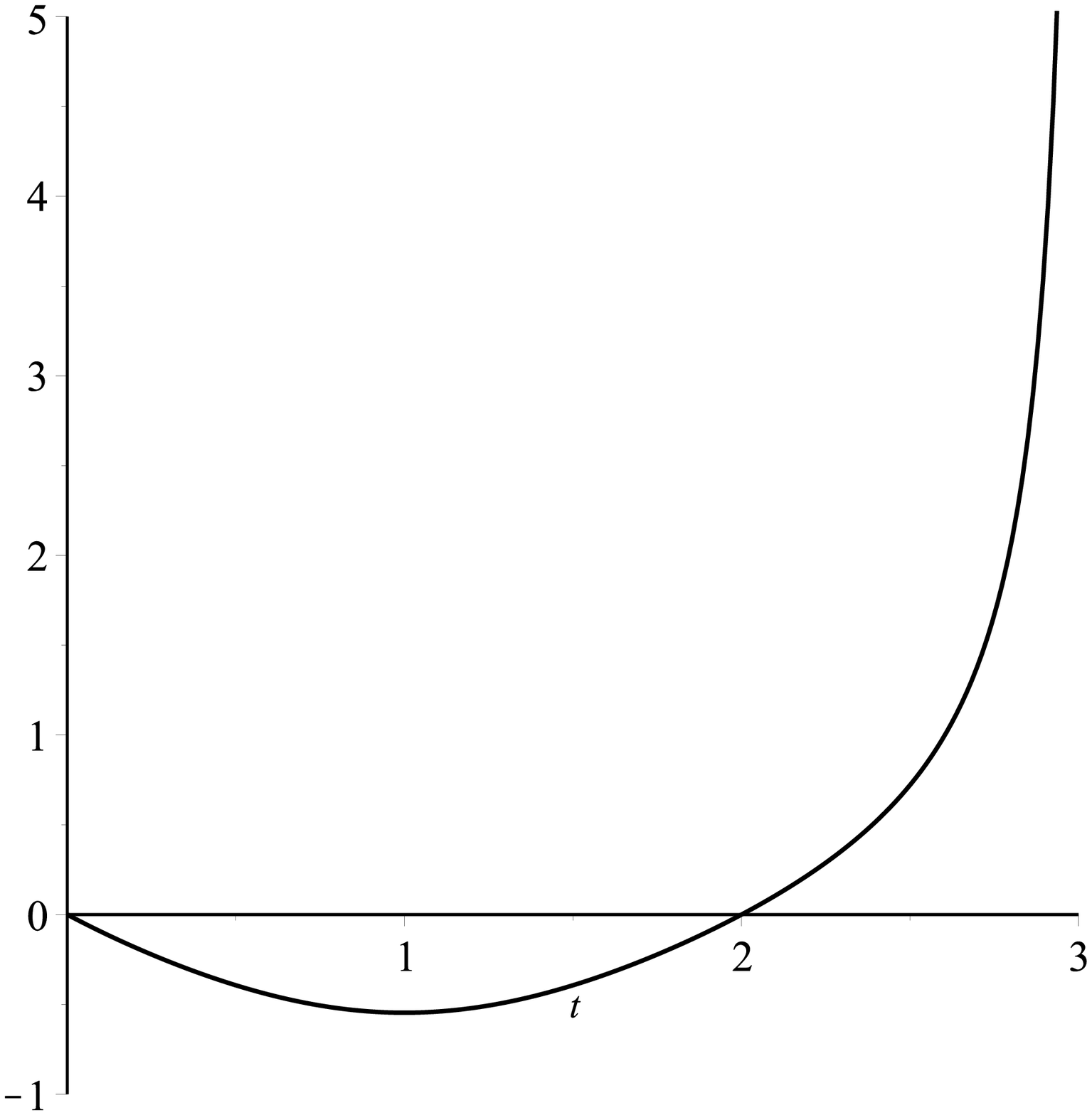}}
\subfigure[Extremal $z$ of Example~\ref{exH:3}.]{\includegraphics[scale=0.25]{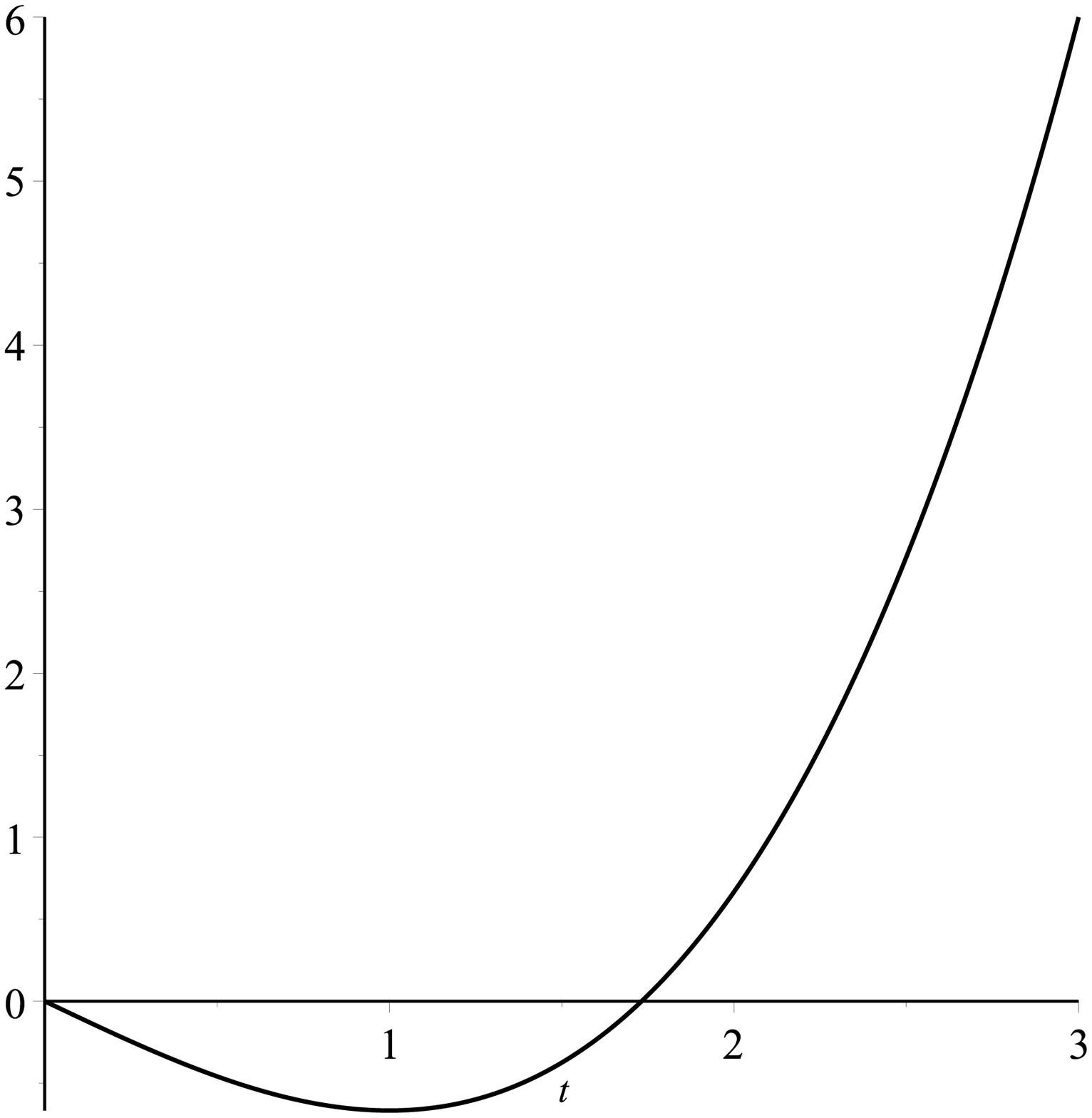}}
\caption{Graphics of function $z(\overline x,t)$.}\label{heg1}
\end{center}
\end{figure}

\begingroup
\renewcommand{\addcontentsline}[3]{}

\endgroup

%% file: Appendix.tex
\chapter*{Appendix}
\markboth{\MakeUppercase{Appendix}}{}
\addcontentsline{toc}{part}{Appendix}

In this appendix, we use a specific \textsf{MATLAB} software, the package
\textsf{Chebfun}, to obtain a few computational approximations for the main
fractional operators in this book.

\textsf{Chebfun} is an open source software package that ``aims
to provide numerical computing with functions'' in \textsf{MATLAB}
\cite{App:Matlab}.
\textsf{Chebfun} overloads \textsf{MATLAB}'s discrete operations for
matrices to analogous continuous operations for functions and operators 
\cite{App:chebfun:book}. For the mathematical underpinnings of \textsf{Chebfun},
we refer the reader to \cite{App:chebfun:book}. For the algorithmic 
backstory of \textsf{Chebfun}, we refer to \cite{App:ChebfunGuide}.

In what follows, we study some computational approximations
of Riemann--Liouville fractional integrals, of Caputo fractional derivatives
and consequently of the combined Caputo fractional derivative, all of them with
variable-order. We provide, also, the necessary \textsf{Chebfun} code for the
variable-order fractional calculus.

To implement these operators, we need two auxiliary functions: the gamma
function $\Gamma$ (Definition \ref{Gammaf}) \index{gamma function}and
the beta function $B$ (Definition \ref{Betaf})\index{beta function}.
Both functions  are available in \textsf{MATLAB} through the commands
\texttt{gamma(t)} and \texttt{beta(t,u)}, respectively.


\section*{A.1 Higher-order Riemann--Liouville fractional integrals}
\label{Chebfun:RL:int}

In this section, we discuss computational aspects to the higher-order
Riemann--Liouville fractional integrals of variable-order $\LIan x(t)$ and $\RIan x(t)$.

Considering the Definition~\ref{def:RL:fi} of higher-order Riemann--Liouville 
fractional integrals, we implemented in \textsf{Chebfun} two functions 
\texttt{leftFi(x,alpha,a)} and \texttt{rightFI(x,alpha,b)}
that approximate, respectively, the Riemann--Liouville fractional integrals
$\LIan x(t)$ and $\RIan x(t)$, through the following \textsf{Chebfun}/\textsf{MATLAB} code.
{\small
\begin{verbatim}
function r = leftFI(x,alpha,a)
g = @(t,tau) x(tau)./(gamma(alpha(t,tau)).*(t-tau).^(1-alpha(t,tau)));
r = @(t) sum(chebfun(@(tau) g(t,tau),[a t],'splitting','on'),[a t]);
end
\end{verbatim}}
\noindent and
{\small
\begin{verbatim}
function r = rightFI(x,alpha,b)
g = @(t,tau) x(tau)./(gamma(alpha(tau,t)).*(tau-t).^(1-alpha(tau,t)));
r = @(t) sum(chebfun(@(tau) g(t,tau),[t b],'splitting','on'),[t b]);
end
\end{verbatim}}

With these two functions, we illustrate their use in the following example, 
where we determine computacional approximations for Riemann--Liouville 
fractional integrals of a special power function.

\begin{example}
\label{ex:cheb02}
Let $\alpha(t,\tau) = \frac{t^2+\tau^2}{4}$ and $x(t) = t^2$ with $t \in [0,1]$.
In this case, $a = 0$, $b = 1$ and $n = 1$. We have
${_aI_{0.6}^{\alpha(\cdot,\cdot)}} x(0.6) \approx 0.2661$
and ${_{0.6}I_b^{\alpha(\cdot,\cdot)}} x(0.6) \approx 0.4619$, obtained
in \textsf{MATLAB} with our \textsf{Chebfun} functions as follows:
{\small \begin{verbatim}
a = 0; b = 1; n = 1;
alpha = @(t,tau) (t.^2+tau.^2)/4;
x = chebfun(@(t) t.^2, [0,1]);
LFI = leftFI(x,alpha,a);
RFI = rightFI(x,alpha,b);
LFI(0.6)
ans = 0.2661
RFI(0.6)
ans = 0.4619
\end{verbatim}}
\noindent Other values for ${_aI_{t}^{\alpha(\cdot,\cdot)}} x(t)$
and ${_{t}I_b^{\alpha(\cdot,\cdot)}} x(t)$
are plotted in Figure~\ref{fig:ex2}.
\begin{figure}[htb]
\centering
\includegraphics[scale=0.7]{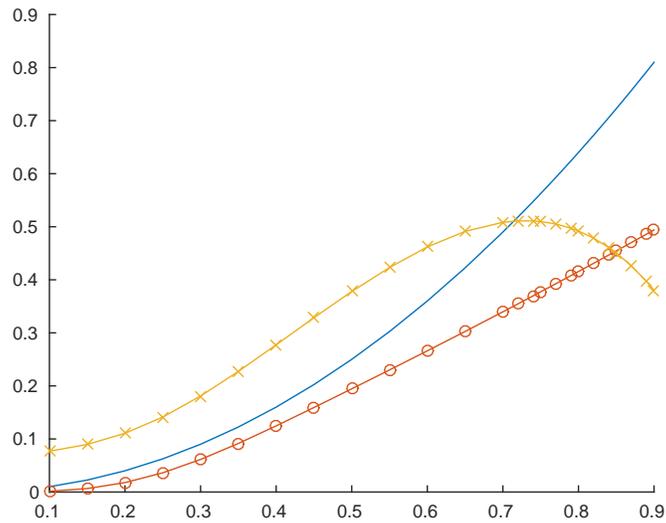}
\caption{Riemann--Liouville fractional integrals of Example~\ref{ex:cheb02}:
$x(t) = t^2$ in continuous line, left integral ${_aI_t^{\alpha(\cdot,\cdot)}} x(t)$
with ``$\circ -$'' style, and right integral ${_tI_b^{\alpha(\cdot,\cdot)}} x(t)$
with ``$\times -$'' style.}
\label{fig:ex2}
\end{figure}
\end{example}


\section*{A.2 Higher-order Caputo fractional derivatives}
\label{Chebfun:Caputo:der}

In this section, considering the Definition~\ref{HOCFD}, 
we implement in \textsf{Chebfun} two new functions 
\texttt{leftCaputo(x,alpha,a,n)} and \texttt{rightCaputo(x,alpha,b,n)}
that approximate, respectively, the higher-order Caputo 
fractional derivatives of variable-order $\LCan x(t)$ and $\RCan x(t)$.

The following code implements the left operator \eqref{eq:left:Cap:der}:
{\small
\begin{verbatim}
function r = leftCaputo(x,alpha,a,n)
dx = diff(x,n);
g = @(t,tau) dx(tau)./(gamma(n-alpha(t,tau)).
                     *(t-tau).^(1+alpha(t,tau)-n));
r = @(t) sum(chebfun(@(tau) g(t,tau),[a t],'splitting','on'),[a t]);
end
\end{verbatim}}
\noindent Similarly, we define the right operator \eqref{eq:right:Cap:der} 
with \textsf{Chebfun} in \textsf{MATLAB} as follows:
{\small
\begin{verbatim}
function r = rightCaputo(x,alpha,b,n)
dx = diff(x,n);
g = @(t,tau) dx(tau)./(gamma(n-alpha(tau,t)). 
                     *(tau-t).^(1+alpha(tau,t)-n));
r = @(t)(-1).^n.* sum(chebfun(@(tau) g(t,tau),[t b],
                     'splitting','on'),[t b]);
end
\end{verbatim}}

We use the two functions \texttt{leftCaputo} and \texttt{rightCaputo} 
to determine apro\-ximations for the Caputo fractional derivatives 
of a power function of the form $x(t)=t^\gamma$.

\begin{example}
\label{ex:cheb01}
Let $\alpha(t,\tau) = \frac{t^2}{2}$ and $x(t) = t^4$ with $t \in [0,1]$.
In this case, $a = 0$, $b = 1$ and $n = 1$. We have
${^C_aD_{0.6}^{\alpha(\cdot,\cdot)}} x(0.6) \approx 0.1857$
and ${^C_{0.6}D_b^{\alpha(\cdot,\cdot)}} x(0.6) \approx -1.0385$, obtained
in \textsf{MATLAB} with our \textsf{Chebfun} functions as follows:
{\small \begin{verbatim}
a = 0; b = 1; n = 1;
alpha = @(t,tau) t.^2/2;
x = chebfun(@(t) t.^4, [a b]);
LC = leftCaputo(x,alpha,a,n);
RC = rightCaputo(x,alpha,b,n);
LC(0.6)
ans = 0.1857
RC(0.6)
ans = -1.0385
\end{verbatim}}
\noindent See Figure~\ref{fig:ex1} for a plot
with other values of ${^C_aD_{t}^{\alpha(\cdot,\cdot)}} x(t)$
and ${^C_{t}D_b^{\alpha(\cdot,\cdot)}} x(t)$.
\begin{figure}[htb]
\centering
\includegraphics[scale=0.7]{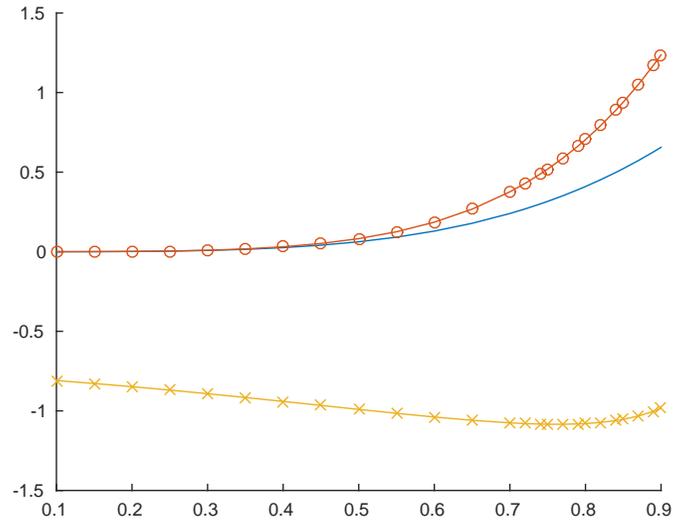}
\caption{Caputo fractional derivatives of Example~\ref{ex:cheb01}:
$x(t) = t^4$ in continuous line, left derivative ${^C_aD_t^{\alpha(\cdot,\cdot)}} x(t)$
with ``$\circ -$'' style, and right derivative ${^C_tD_b^{\alpha(\cdot,\cdot)}} x(t)$
with ``$\times -$'' style.}
\label{fig:ex1}
\end{figure}
\end{example}

\begin{example}
\label{ex:cheb01b}
In Example~\ref{ex:cheb01}, we have used the polynomial $x(t) = t^4$.
It is worth mentioning that our \textsf{Chebfun} implementation
works well for functions that are not a polynomial. For example,
let $x(t) = e^t$. In this case, we just need to change
{\small \begin{verbatim}
x = chebfun(@(t) t.^4, [a b]);
\end{verbatim}}
\noindent in Example~\ref{ex:cheb01} by
{\small \begin{verbatim}
x = chebfun(@(t) exp(t), [a b]);
\end{verbatim}}
\noindent to obtain
{\small \begin{verbatim}
LC(0.6)
ans = 0.9917
RC(0.6)
ans = -1.1398
\end{verbatim}}
\noindent See Figure~\ref{fig:ex1b} for a plot
with other values of ${^C_aD_{t}^{\alpha(\cdot,\cdot)}} x(t)$
and ${^C_{t}D_b^{\alpha(\cdot,\cdot)}} x(t)$.
\begin{figure}[htb]
\centering
\includegraphics[scale=0.7]{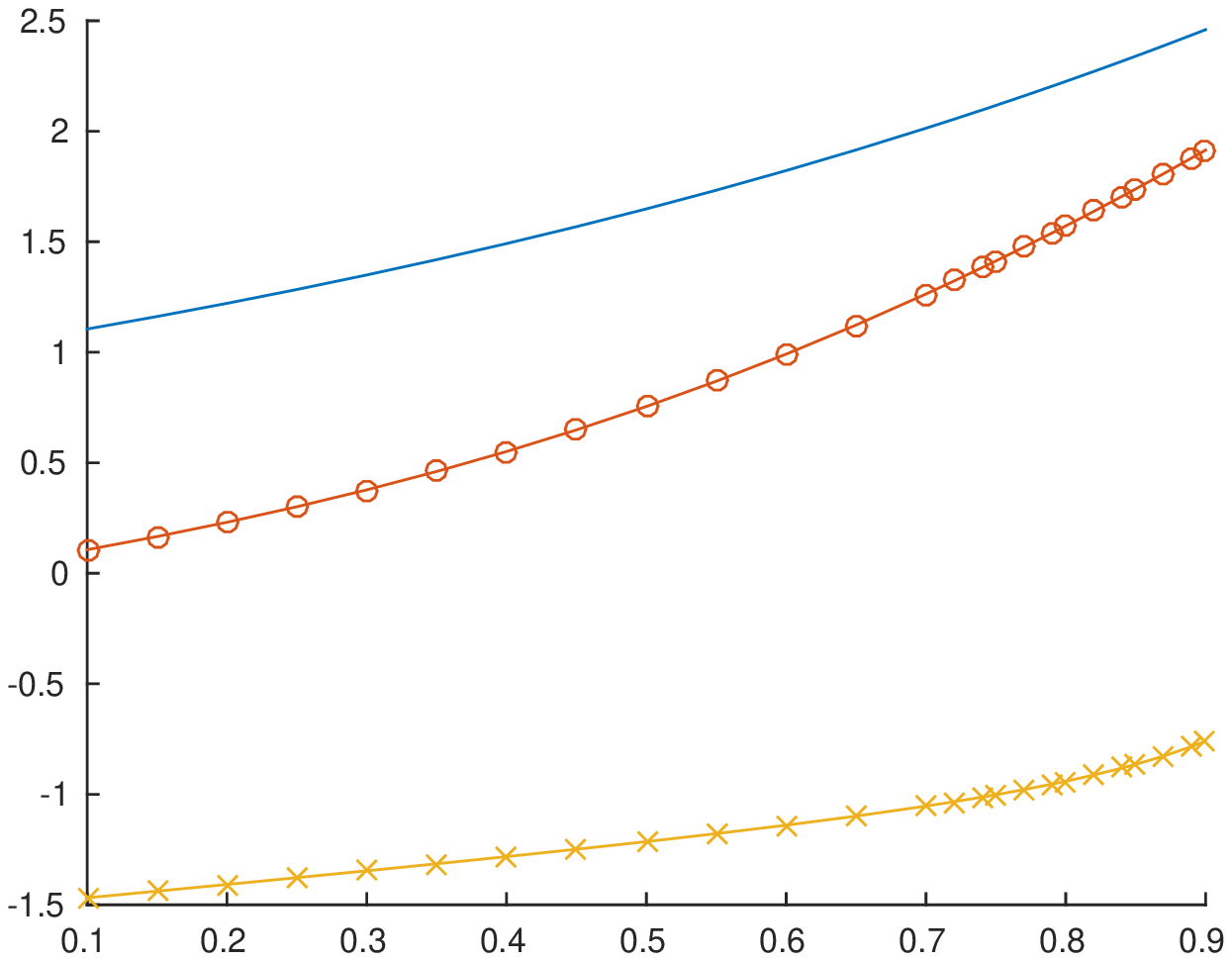}
\caption{Caputo fractional derivatives of Example~\ref{ex:cheb01b}:
$x(t) = e^t$ in continuous line, left derivative ${^C_aD_t^{\alpha(\cdot,\cdot)}} x(t)$
with ``$\circ -$'' style, and right derivative ${^C_tD_b^{\alpha(\cdot,\cdot)}} x(t)$
with ``$\times -$'' style.}
\label{fig:ex1b}
\end{figure}
\end{example}

With Lemma~\ref{lemma:power} in Section~\ref{sec:introd} we can obtain, analytically, 
the higher-order left Caputo fractional derivative of a power function of the form $x(t)=(t-a)^\gamma$.
This allows us to show the effectiveness of our computational approach, that is,
the usefulness of polynomial interpolation in Chebyshev points in fractional 
calculus of variable-order. In Lemma~\ref{lemma:power}, we assume that the fractional 
order depends only on the first variable: $\alpha_n(t,\t) := \overline{\alpha}_n(t)$, where
$\overline{\alpha}_n:[a,b]\to (n-1,n)$ is a given function.

\begin{example}
\label{ex:lemma:power}
Let us revisit Example~\ref{ex:cheb01} by choosing
$\alpha(t,\tau) = \frac{t^2}{2}$ and $x(t) = t^4$ with $t \in [0,1]$.
Table~\ref{table01} shows the approximated values obtained
by our \textsf{Chebfun} function \texttt{leftCaputo(x,alpha,a,n)}
and the exact values computed with the formula given by Le\-mma~\ref{lemma:power}.
Table~\ref{table01} was obtained using the following \texttt{MATLAB} code:
{\small \begin{verbatim}
format long
a = 0; b = 1; n = 1;
alpha = @(t,tau) t.^2/2;
x = chebfun(@(t) t.^4, [a b]);
exact = @(t) (gamma(5)./gamma(5-alpha(t))).*t.^(4-alpha(t));
approximation = leftCaputo(x,alpha,a,n);
for i = 1:9
t = 0.1*i;
E = exact(t);
A = approximation(t);
error = E - A;
[t E A error]
end
\end{verbatim}}

\begin{table}
\begin{center}
\begin{tabular}{|c|c|c|c|} \hline
$\mathbf{t}$ & \textbf{Exact Value} & \textbf{Approximation} & \textbf{Error} \\ \hline
0.1 & 1.019223177296953e-04 &  1.019223177296974e-04 & -2.046431600566390e-18 \\ \hline
0.2 & 0.001702793965464 &  0.001702793965464 & -2.168404344971009e-18 \\ \hline
0.3 & 0.009148530806348 &  0.009148530806348 &  3.469446951953614e-18 \\ \hline
0.4 & 0.031052290994593 &  0.031052290994592 &  9.089951014118469e-16 \\ \hline
0.5 & 0.082132144921157 &  0.082132144921157 &  6.522560269672795e-16 \\ \hline
0.6 & 0.185651036003120 &  0.185651036003112 &  7.938094626069869e-15 \\ \hline
0.7 & 0.376408251363662 &  0.376408251357416 &  6.246059225389899e-12 \\ \hline
0.8 & 0.704111480975332 &  0.704111480816562 &  1.587694420379648e-10 \\ \hline
0.9 & 1.236753486749357 &  1.236753486514274 &  2.350835082154390e-10 \\ \hline
\end{tabular}
\end{center}
\caption{Exact values obtained by Lemma~\ref{lemma:power}
for functions of Example~\ref{ex:lemma:power}
versus computational approximations obtained using
the \textsf{Chebfun} code.}
\label{table01}
\end{table}
\end{example}

Computational experiments similar to those of Example~\ref{ex:lemma:power},
obtained by substituting Lemma~\ref{lemma:power} by Lemma~\ref{lemma:power:2}
and our \texttt{leftCaputo} routine by the \texttt{rightCaputo} one,
reinforce the validity of our computational methods. In this case, 
we assume that the fractional order depends only on the second variable:
$\alpha_n(\t,t) := \overline{\alpha}_n(t)$, 
where $\overline{\alpha}_n : [a,b] \to (n-1,n)$ is a given function.


\section*{A.3 Higher-order combined fractional Caputo derivative}
\label{Chebfun:combined}

The higher-order combined Caputo fractional derivative combines 
both left and right Caputo fractional derivatives, that is, 
we make use of functions \texttt{leftCaputo(x,alpha,a,n)} 
and \texttt{rightCaputo(x,alpha,b,n)} provided in Section~A.1 
to define \textsf{Chebfun} computational code for the higher-order 
combined fractional Caputo derivative of variable-order:
{\small
\begin{verbatim}
function r = combinedCaputo(x,alpha,beta,gamma1,gamma2,a,b,n)
lc = leftCaputo(x,alpha,a,n);
rc = rightCaputo(x,beta,b,n);
r = @(t) gamma1 .* lc(t) + gamma2 .* rc(t);
end
\end{verbatim}}


Then we illustrate the behavior of the combined Caputo fractional 
derivative of variable-order for different values 
of $t \in (0,1)$, using \textsf{MATLAB}.

\begin{example}
\label{ex:matlab:comb}
Let $\alpha(t,\tau) = \frac{t^2 + \tau^2}{4}$,
$\beta(t,\tau) = \frac{t + \tau}{3}$ and $x(t) = t$,
$t\in [0,1]$. We have $a = 0$, $b = 1$ and $n = 1$.
For $\gamma = (\gamma_1, \gamma_2) = (0.8, 0.2)$, we have
$^{C}D_{\gamma}^{\alpha(\cdot,\cdot),\beta(\cdot,\cdot)}x(0.4) \approx 0.7144$:
{\small \begin{verbatim}
a = 0; b = 1; n = 1;
alpha = @(t,tau) (t.^2 + tau.^2)/.4;
beta = @(t,tau) (t + tau)/3;
x = chebfun(@(t) t, [0 1]);
gamma1 = 0.8;
gamma2 = 0.2;
CC = combinedCaputo(x,alpha,beta,gamma1,gamma2,a,b,n);
CC(0.4)
ans = 0.7144
\end{verbatim}}
\noindent For other values of $^{C}D_{\gamma^n}^{\alpha(\cdot,\cdot),\beta(\cdot,\cdot)}x(t)$,
for different values of $t \in (0,1)$ and
$\gamma = (\gamma_1, \gamma_2)$,
see Figure~\ref{fig:ex3} and Table~\ref{table02}.
\begin{figure}[ht!]
\centering
\includegraphics[scale=0.7]{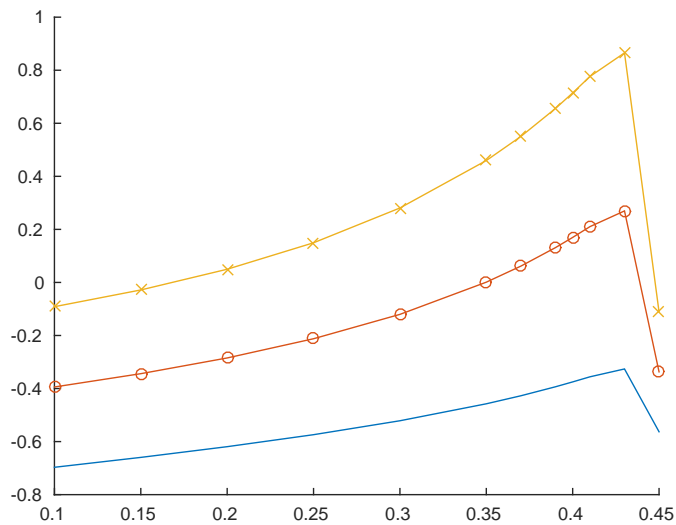}
\caption{Combined Caputo fractional derivative
$^{C}D_{\gamma}^{\alpha(\cdot,\cdot),\beta(\cdot,\cdot)}x(t)$
for $\alpha(\cdot,\cdot)$, $\beta(\cdot,\cdot)$ and $x(t)$ of Example~\ref{ex:matlab:comb}:
continuous line for $\gamma = (\gamma_1, \gamma_2) = (0.2, 0.8)$,
``$\circ -$'' style for $\gamma = (\gamma_1, \gamma_2) = (0.5, 0.5)$,
and ``$\times -$'' style for $\gamma = (\gamma_1, \gamma_2) = (0.8, 0.2)$.}
\label{fig:ex3}
\end{figure}
\begin{table}[ht!]
\begin{center}

\small{ \begin{tabular}{|c|c|c|c|} \hline
$\mathbf{t}$ & \textbf{Case~1} & \textbf{Case~2} & \textbf{Case~3}\\ \hline
0.4500 &  -0.5630    & -0.3371     &     -0.1112 \\ \hline
0.5000 & -790.4972   & -1.9752e+03 & -3.1599e+03 \\ \hline
0.5500 & -3.5738e+06 & -8.9345e+06 & -1.4295e+07 \\ \hline
0.6000 & -2.0081e+10 & -5.0201e+10 & -8.0322e+10 \\ \hline
0.6500 &  2.8464e+14 &  7.1160e+14 & 1.1386e+15  \\ \hline
0.7000 &  4.8494e+19 &  1.2124e+20 & 1.9398e+20  \\ \hline
0.7500 &  3.8006e+24 &  9.5015e+24 & 1.5202e+25  \\ \hline
0.8000 & -1.3648e+30 & -3.4119e+30 & -5.4591e+30 \\ \hline
0.8500 & -1.6912e+36 & -4.2280e+36 & -6.7648e+36 \\ \hline
0.9000 &  5.5578e+41 &  1.3895e+42 &  2.2231e+42 \\ \hline
0.9500 &  1.5258e+49 &  3.8145e+49 &  6.1033e+49 \\ \hline
0.9900 &  1.8158e+54 &  4.5394e+54 &  7.2631e+54 \\ \hline
\end{tabular}
}
\end{center}
\caption{Combined Caputo fractional derivative
$^{C}D_{\gamma}^{\alpha(\cdot,\cdot),\beta(\cdot,\cdot)}x(t)$
for $\alpha(\cdot,\cdot)$, $\beta(\cdot,\cdot)$ and $x(t)$
of Example~\ref{ex:matlab:comb}.
Case~1: $\gamma = (\gamma_1, \gamma_2) = (0.2, 0.8)$;
Case~2: $\gamma = (\gamma_1, \gamma_2) = (0.5, 0.5)$;
Case~3: $\gamma = (\gamma_1, \gamma_2) = (0.8, 0.2)$.}
\label{table02}
\end{table}

\end{example}

\begingroup
\renewcommand{\addcontentsline}[3]{}

\endgroup